\documentclass[11pt,reqno]{amsart}

\usepackage[T1]{fontenc}
\usepackage[utf8]{inputenc}
\usepackage{amsmath,amssymb,amsthm,mathtools}
\usepackage{mathrsfs}
\usepackage{array}
\usepackage[protrusion=true,expansion=false]{microtype}
\usepackage[a4paper,margin=25mm]{geometry}
\usepackage[
colorlinks=true,
linkcolor=blue,
citecolor=blue,
urlcolor=blue
]{hyperref}
\hypersetup{
  pdftitle={Completing the Arakawa--Moreau Conjecture on Maximal Ideals of Affine Vertex Algebras}, pdfauthor={Sihai Jin},
  pdfsubject={Maximal graded ideals of negative-level affine vertex algebras of types D and E},
  pdfkeywords={affine vertex algebra, minimal W-algebra, Drinfeld--Sokolov reduction, Ramond Zhu algebra, maximal ideal}
}
\usepackage[nameinlink,noabbrev]{cleveref}
\usepackage[
  backend=biber,
  style=numeric-comp,
  sorting=none,
  maxbibnames=99,
  giveninits=true,
  doi=true,
  url=false,
  isbn=false
]{biblatex}
\DeclareNameAlias{author}{family-given}
\DeclareNameAlias{editor}{family-given}
\DeclareFieldFormat[article,incollection,inproceedings]{title}{#1}
\DeclareFieldFormat[article]{pages}{#1}
\AtBeginBibliography{\raggedright}
\renewbibmacro{in:}{%
  \ifentrytype{article}
    {}
    {\printtext{\bibstring{in}\intitlepunct}}}

\allowdisplaybreaks

\renewcommand{\arraystretch}{1.32}
\newtheorem{theorem}{Theorem}[section]
\newtheorem{lemma}[theorem]{Lemma}
\newtheorem{proposition}[theorem]{Proposition}
\newtheorem{corollary}[theorem]{Corollary}
\theoremstyle{definition}

\theoremstyle{remark}
\newtheorem{remark}[theorem]{Remark}

\providecommand{\C}{\mathbb{C}}
\providecommand{\Z}{\mathbb{Z}}
\providecommand{\Vac}{\mathbf{1}}
\providecommand{\g}{\mathfrak{g}}

\providecommand{\W}{\mathcal{W}}
\providecommand{\DS}{\mathrm{DS}}
\providecommand{\wt}{\operatorname{wt}}
\providecommand{\Spec}{\operatorname{Spec}}
\providecommand{\tr}{\operatorname{tr}}
\providecommand{\gr}{\operatorname{gr}}

\title[Completing the Arakawa--Moreau conjecture]
{Completing the Arakawa--Moreau Conjecture on Maximal Ideals of Affine Vertex Algebras}
\author{Sihai Jin}
\address{Department of Mathematics, Sichuan University}
\email{jinsihai@stu.scu.edu.cn}
\date{}

\begin{document}
\begin{abstract}
	Arakawa and Moreau constructed explicit singular vectors in a family of
	negative-level universal affine vertex algebras of types $D$ and $E$ and
	conjectured that the ideals generated by these vectors are maximal.
	Previous work established the $n=0$ cases for
	$D_4,E_6,E_7,$ and $E_8$, as well as the level $-2$ statement for
	$D_\ell$, $\ell\ge5$.  We prove all cases left open by these results:
	the level $-1$ statement for $D_\ell$, $\ell\ge5$, and the $n>0$
	negative-level cases for $D_4,E_6,E_7,$ and $E_8$.  Together with the
	previously known cases, this completes Arakawa--Moreau Conjecture~1.
	
	The proof identifies the exact images of the prescribed singular vectors
	under minimal Drinfeld--Sokolov reduction, proves simplicity of the reduced
	quotients by a uniform Ramond--Zhu/Casimir-gap argument combined with Li
	spectral flow, and then lifts simplicity to the affine quotients using
	exactness and a nonvanishing theorem for the reduction functor.  We isolate
	this mechanism as a general minimal-reduction maximality principle and a
	DS--Ramond--Casimir maximality theorem.  Within the same framework, we also
	give an alternative reduction-theoretic proof of the known
	$D_\ell$ level $-2$ theorem and a new rank-reduction proof of the known
	maximal-ideal theorem for the collapsing family $V^{2-2r}(D_{2r})$.
	Consequently, every candidate quotient in Arakawa--Moreau Conjecture~1 is
	the corresponding simple affine vertex algebra.
\end{abstract}

\maketitle
\tableofcontents
\section{Introduction}

\subsection{Background and the conjecture}

Let $\mathfrak g$ be a finite-dimensional simple Lie algebra and let
$V^k(\mathfrak g)$ be the universal affine vertex algebra at level $k$,
with simple graded quotient $L_k(\mathfrak g)$.  Determining the maximal
graded ideal of $V^k(\mathfrak g)$ is a basic presentation problem in
affine representation theory: an explicit set of generators gives a
concrete realization of $L_k(\mathfrak g)$ and makes it possible to study
its modules, associated variety, and quantum Hamiltonian reductions.  We
use the standard conventions for affine vertex algebras and their modules
from \cite{Kac1998,FrenkelBenZvi2004,FrenkelZhu1992}, and for quantized
Drinfeld--Sokolov reduction and the representation theory of affine
$W$-algebras from
\cite{KacRoanWakimoto2003,KacWakimoto2004,Arakawa2004,Arakawa2005,
Arakawa2007}.

Arakawa and Moreau introduced a family of candidate maximal ideals whose
construction is governed by the Joseph ideal and the minimal nilpotent
orbit \cite{ArakawaMoreau2018}.  The finite-$W$-algebra and Slodowy-slice
background underlying this construction goes back to
\cite{Premet2002,GanGinzburg2002}.  Let $\theta$ be the highest root of
$\mathfrak g$.  In the relevant types, the complement of the Cartan square
and the trivial representation in $S^2(\mathfrak g)$ contains distinguished
Kostant components.  After symmetrization, highest-weight vectors in these
components give affine singular vectors at specific negative integral
levels.  Quotienting $V^k(\mathfrak g)$ by the ideals generated by the
prescribed powers produces affine vertex algebras whose associated
varieties are the closures of the minimal nilpotent orbits
\cite[Proposition~5.2]{ArakawaMoreau2018}; their minimal
Drinfeld--Sokolov reductions are lisse
\cite[Theorem~6.1(3)]{ArakawaMoreau2018}.  The remaining question is
whether these candidate quotients are already simple.

We recall the precise form of the conjecture.  For
$\mathfrak g\in\{D_4,E_6,E_7,E_8\}$, set
\begin{equation}\label{eq:intro-exceptional-levels}
  k_n=n-\frac{h^\vee}{6}-1,
  \qquad n\in\mathbb Z_{\ge0}.
\end{equation}
For $E_6,E_7,E_8$, let $w$ be the distinguished highest-weight vector in
the relevant Kostant component; for $D_4$, let $w_1,w_3,w_4$ be the three
triality-related vectors.  For the series $D_\ell$, $\ell\ge5$, set
\begin{equation*}
  k_n=n-2,
  \qquad n\in\mathbb Z_{\ge0},
\end{equation*}
and let $w_A,w_D$ denote the two Kostant components used in the
Arakawa--Moreau construction.  The conjecture asserts that, whenever
$k_n<0$, the quotient by the corresponding singular vectors is simple:
for the exceptional types the generators are the powers
$\sigma(w_i)^{n+1}$, while in type $D_\ell$ they are
\begin{equation*}
  \sigma(w_A)^{n+\ell-3},
  \qquad
  \sigma(w_D)^{n+1}.
\end{equation*}
Throughout the paper, these powers are the associative powers in
$U(\mathfrak g[t^{-1}]t^{-1})$ under the PBW vector-space identification
$V^k(\mathfrak g)\cong U(\mathfrak g[t^{-1}]t^{-1})\mathbf1$, exactly as
in \cite[Section~4]{ArakawaMoreau2018}; they are not iterated normally
ordered products unless this is stated explicitly.
Equivalently, these vectors generate the full maximal graded ideal of the
universal affine vertex algebra
\cite[Conjecture~1]{ArakawaMoreau2018}.

The conjecture contains an infinite-rank family and a finite exceptional
family.  For $D_\ell$, $\ell\ge5$, the inequality $k_n<0$ leaves precisely
$n=0,1$, that is, the levels $-2$ and $-1$.  In the exceptional cases it
leaves $n=0,\ldots,h^\vee/6$.  Arakawa and Moreau proved the $n=0$ cases
for $D_4,E_6,E_7,E_8$ in the proof of their Theorem~3.1
\cite{ArakawaMoreau2018}; the level $-2$ statement for $D_4$ also appears
in \cite[Theorem~4.2]{Perse2013}.  In addition, Adamovi\'c, Kac,
M\"oseneder Frajria, Papi, and Per\v{s}e proved that, for every
$\ell\ge5$, the maximal ideal of $V^{-2}(D_\ell)$ is generated by the two
prescribed Arakawa--Moreau vectors
\cite[Section~6, especially Corollary~6.6(1)]
{AdamovicKacMosenederFrajriaPapiPerse2020}.
In the notation of that reference, their vectors $w_1$ and $w_2$ agree,
up to nonzero scalar multiples, with
$\sigma(w_D)$ and $\sigma(w_A)^{\ell-3}$, respectively.

Thus the cases left open before the present work were the level $-1$
statement for $D_\ell$, $\ell\ge5$ (including $D_5$), together with the
$n>0$ negative-level cases for $D_4,E_6,E_7,$ and $E_8$.  These are the
new maximal-ideal results proved in this paper.  For completeness, and to
exhibit the uniform reduction mechanism, we also give an alternative proof
of the known $D_\ell$ level $-2$ theorem.

Related recent work addresses nearby but logically distinct questions.
Adamovi\'c, Per\v{s}e, and Vukorepa, as well as Jiang and Song, determine
maximal ideals at non-admissible levels in type $A$
\cite{AdamovicPerseVukorepa2026,JiangSong2025}.  Jiang and Song also
determine the weights of singular vectors of minimal conformal weight for
nonsimple simply-laced universal affine vertex algebras
\cite{JiangSongSingular2025}.  Although this result includes types $D$
and $E$, it does not prove that the ideals generated by the
Arakawa--Moreau vectors are maximal.  Work on non-admissible minimal
$W$-algebras instead concerns lisse or rationality properties and explicit
realizations of the simple reductions
\cite{Kawasetsu2018,ArakawaVanEkeren2023,
	ArakawaCreutzigKawasetsu2025,
	CreutzigFasquelKovalchukLinshawNakatsuka2025}; in particular, the recent
type-$D$ orbifold realization at level $-1$ does not identify the kernel of
$V^{-1}(D_\ell)\to L_{-1}(D_\ell)$.  Shan--Yan--Zhao formulate geometric
predictions for associated varieties at non-admissible integer and rational
levels \cite{ShanYanZhao2025,ShanYanZhao2026}.  Such associated-variety
statements do not by themselves determine explicit generators of a maximal
affine ideal.  Apart from the previously known cases explicitly cited
above, these results do not settle the affine-kernel statements proved in
this paper.

\subsection{Main result}

Our new results, combined with the previously known cases, yield the
following complete form of the conjecture.

\begin{theorem}[Completion of the Arakawa--Moreau maximal-ideal conjecture]
\label{thm:main}
Conjecture~1 of \cite{ArakawaMoreau2018} holds at every negative level in
its stated range: the maximal graded ideals are generated by the prescribed
singular vectors.  More precisely:
\begin{enumerate}
\item for every $\ell\ge5$,
\[
\ker\bigl(V^{-2}(D_\ell)\to L_{-2}(D_\ell)\bigr)
 =\bigl\langle\sigma(w_A)^{\ell-3},\sigma(w_D)\bigr\rangle,
\]
\[
\ker\bigl(V^{-1}(D_\ell)\to L_{-1}(D_\ell)\bigr)
 =\bigl\langle\sigma(w_A)^{\ell-2},\sigma(w_D)^2\bigr\rangle;
\]
\item for the exceptional cases, with $k_n$ as in
\eqref{eq:intro-exceptional-levels}, the prescribed ideals are maximal
for
\[
\begin{array}{c|c|c}
\mathfrak g & n & \text{generators}\\ \hline
D_4 & 0,1 & \sigma(w_1)^{n+1},\ \sigma(w_3)^{n+1},\
                 \sigma(w_4)^{n+1}\\
E_6 & 0,1,2 & \sigma(w)^{n+1}\\
E_7 & 0,1,2,3 & \sigma(w)^{n+1}\\
E_8 & 0,1,2,3,4,5 & \sigma(w)^{n+1}.
\end{array}
\]
\end{enumerate}
Consequently, Arakawa--Moreau Conjecture~1 is proved in full.
\end{theorem}

The provenance of the cases in Theorem~\ref{thm:main} is as follows.
In part~(1), the level $-2$ equality was proved in
\cite[Section~6, Corollary~6.6(1)]
{AdamovicKacMosenederFrajriaPapiPerse2020}
and is reproved here by a different method, whereas the level $-1$
equality is new.  In part~(2), the cases $n=0$ were proved by Arakawa and
Moreau in the proof of \cite[Theorem~3.1]{ArakawaMoreau2018}, whereas all
cases $n>0$ are new.  Thus the new results of this paper settle every case
that remained open and, together with the cited results, complete the
conjecture.

Theorem~\ref{thm:main} concerns the maximal-ideal statement formulated as
Conjecture~1 in \cite{ArakawaMoreau2018}.  It should not be confused with
the separate ``if and only if'' classification of associated varieties in
Conjecture~2 of that paper.  For every case listed above, the theorem
identifies the candidate quotient with the corresponding simple affine
vertex algebra; hence the singular-vector presentations and associated
varieties obtained in the original construction apply directly to the
simple quotients.

\subsection{Strategy of proof}

The proof has a common structural core, but three points require more than
an associated-graded or dimension argument.

\smallskip
\noindent\emph{Exact reduction of the generated ideal.}
For each singular vector we compute its actual class under minimal
Drinfeld--Sokolov reduction, not merely its principal symbol.  An extremal
PBW-line argument shows that the relevant weight space in the universal
minimal $W$-algebra is one-dimensional.  Once nonvanishing is verified,
the reduced class is therefore the prescribed power of an extremal affine
current.  This yields an exact presentation of the reduced quotient.

\smallskip
\noindent\emph{Simplicity of the reduced quotient.}
Lisse-ness alone does not imply simplicity.  We pass to the finite-dimensional
Ramond Zhu algebra, using twisted Zhu theory
\cite{DongLiMason1998}, the relationship with finite $W$-algebras developed
in \cite{DeSoleKac2006}, and the explicit Ramond-sector relations of
\cite{KacMosenederFrajriaPapi2025}.  The quadratic relation for the
weight-$3/2$ generators gives, after taking traces, a weighted Casimir
identity.  Li spectral flow \cite{Li2012} restricts the possible lowest
Ramond energies.  A finite type-dependent Casimir estimate then creates a
gap: every nonzero ideal would have to contain an extremal
$\mathfrak g^\natural$-constituent, and the low-degree weight-space
calculation identifies that constituent with the vacuum line.  Hence the
reduced quotient is simple.  The polynomial controlling the quadratic
relation and the collapsing endpoint are taken from
\cite{AdamovicKacMosenederFrajriaPapiPerse2018}.

\smallskip
\noindent\emph{Lifting simplicity to the affine quotient.}
The implication
\[
  H^0_{\mathrm{DS},f_\theta}(M)=0\quad\Longrightarrow\quad M=0
\]
is false without hypotheses.  We therefore prove a nonvanishing statement
for every nonzero graded submodule that occurs here.  Choosing a
lowest-degree affine highest-weight vector places an irreducible quotient
in the category on which minimal reduction is exact and nonzero.  Applied
to the kernel of the map from the candidate quotient to
$L_k(\mathfrak g)$, this shows that simplicity of the reduced quotient
forces simplicity of the affine quotient.  In
Theorem~\ref{thm:common-minimal-reduction-maximality} we isolate the resulting
general principle: at a negative integral level, a graded candidate quotient
with nonzero simple minimal reduction is already the simple affine quotient.
Theorem~\ref{thm:common-DS-Ramond-Casimir-maximality} combines this lifting
principle with the Ramond--Casimir gap into a reusable maximality theorem.
Section~\ref{sec:collapsing-application} then applies the first principle
outside the Arakawa--Moreau family, giving a rank-reduction proof of the
known maximal-ideal theorem at the collapsing levels
$V^{2-2r}(D_{2r})$.

These three steps are formalized once in
Section~\ref{sec:common-framework}.  The dependence on the Lie type is
confined to the following finite data:
\begin{enumerate}
\item the exact extremal BRST image of each singular vector;
\item the scalar contraction in the Ramond--Casimir trace identity;
\item the maximal and next-to-maximal Casimir values allowed by the
      nilpotence relations, together with a low-degree exclusion of the
      competing extremal weights.
\end{enumerate}
For the series $D_\ell$ these verifications are uniform in $\ell$.  The
level $-2$ endpoint, whose affine maximal-ideal theorem is already known
\cite[Section~6, Corollary~6.6(1)]
{AdamovicKacMosenederFrajriaPapiPerse2020}, is included in order to give an
alternative proof within the present framework.  It is treated separately
at the local computational stage, where the minimal $W$-algebra collapses
to its surviving $A_1$-affine part; the final simplicity and affine-lifting
arguments are the same as in the other cases.

\subsection{Organization of the paper}

Section~\ref{sec:common-framework} develops the common reduction,
Ramond Zhu, Casimir-gap, spectral-flow, and affine-lifting framework, and
culminates in the general minimal-reduction and DS--Ramond--Casimir
maximality theorems.
Section~\ref{sec:case-dl} proves the previously open level $-1$ theorem for
the complete series $D_\ell$, $\ell\ge5$, and gives an alternative proof
of the known level $-2$ theorem.  Section~\ref{sec:case-d4} treats the
triality case $D_4$, and Sections~\ref{sec:case-e6}--\ref{sec:case-e8}
treat $E_6,E_7,E_8$, respectively.  Section~\ref{sec:collapsing-application}
gives the independent rank-reduction application at the type-$D$ collapsing
levels.  The appendices collect the root data, classical BRST computations,
and finite Casimir tables that are genuinely type dependent.
\section{Common framework}
\label{sec:common-framework}

All vector spaces, Lie algebras, and vertex algebras are over $\mathbb C$.
We normalize the invariant form $(\,\cdot\mid\cdot\,)$ so that long roots
have squared length~$2$, use Bourbaki root numbering, and write
\[
  Y(a,z)=\sum_{r\in\mathbb Z}a_{(r)}z^{-r-1}.
\]
For a simple Lie algebra $\mathfrak g$ with highest root $\theta$, choose an
$\mathfrak{sl}_2$-triple $(e_\theta,\theta^\vee,f_\theta)$ and set
$x_\theta=\theta^\vee/2$.  The associated minimal grading is
\begin{equation*}
  \mathfrak g=\mathfrak g_{-1}\oplus\mathfrak g_{-1/2}
  \oplus\mathfrak g_0\oplus\mathfrak g_{1/2}\oplus\mathfrak g_1,
  \qquad
  \mathfrak g_0=\mathfrak g^\natural\oplus\mathbb Cx_\theta.
\end{equation*}
The type-dependent descriptions of $\mathfrak g^\natural$ and
$\mathfrak g_{-1/2}$ are recorded in the case sections.

\subsection{Affine and minimal \texorpdfstring{$W$}{W}-algebras}

Let $V^k(\mathfrak g)$ be the universal affine vertex algebra and
$L_k(\mathfrak g)$ its simple graded quotient.  We use
\begin{equation*}
  [a_\lambda b]=[a,b]+k(a\mid b)\lambda,
  \qquad a,b\in\mathfrak g.
\end{equation*}
The universal and simple minimal $W$-algebras are
\begin{align*}
  \mathcal W^k(\mathfrak g,f_\theta)
  &=H^0_{\mathrm{DS},f_\theta}\!\left(V^k(\mathfrak g)\right),\\
  \mathcal W_k(\mathfrak g,f_\theta)
  &=H^0_{\mathrm{DS},f_\theta}\!\left(L_k(\mathfrak g)\right)
\end{align*}
whenever the second reduction is nonzero.  The universal algebra is strongly
generated by currents $J^{\{a\}}$ ($a\in\mathfrak g^\natural$), fields
$G^{\{u\}}$ ($u\in\mathfrak g_{-1/2}$), and the conformal vector $\omega$,
of weights $1,3/2,2$, respectively.  Our affine and BRST conventions are
those of \cite{Kac1998,FrenkelBenZvi2004,KacRoanWakimoto2003,
KacWakimoto2004,KacWakimoto2005,Arakawa2007}.

\subsection{Exactness, nonvanishing, and associated varieties}
\label{subsec:common-DS}

Minimal reduction is exact on the affine category $\mathcal O_k$
\cite[Theorem~6.7.1]{Arakawa2005}.  We use the convention of
\cite[Section~2.12]{Arakawa2005}: an object of $\mathcal O_k$ is an
$\widehat{\mathfrak h}$-weight module with finite-dimensional weight spaces
whose support is contained in a finite union of sets
$\mu_i-\widehat Q_+$, where $\widehat Q_+$ is the positive affine root cone.
The following observation verifies both conditions for every affine module
used below.

\begin{lemma}[Graded subquotients lie in the exactness category]
\label{lem:common-graded-subquotient-category-O}
Let $M$ be a graded subquotient of $V^k(\mathfrak g)$ such that
\[
 M=\bigoplus_{r\in d_0+\mathbb Z_{\ge0}}M[r],\qquad
 \dim M[r]<\infty,
\]
and every $M[r]$ is a finite-dimensional $\mathfrak g$-module.  Then
$M\in\mathcal O_k$.  In fact,
\[
 \operatorname{Supp}_{\widehat{\mathfrak h}}M
 \subset k\Lambda_0-\widehat Q_+.
\]
Consequently, minimal reduction is exact on every short exact sequence of
such modules.
\end{lemma}

\begin{proof}
Write
$\widehat{\mathfrak h}=\mathfrak h\oplus\mathbb CK\oplus\mathbb CD$.
Each $M[r]$ is a finite-dimensional $\mathfrak g$-module, so $\mathfrak h$
acts semisimply on it; moreover, $K$ acts as $k$ and the inherited affine
grading diagonalizes $D$.  Hence $M$ is an
$\widehat{\mathfrak h}$-weight module.  A fixed affine weight occurs in a
single homogeneous component, and its weight space is contained in a
finite-dimensional $\mathfrak h$-weight space of that component.  Thus all
$\widehat{\mathfrak h}$-weight spaces of $M$ are finite-dimensional.

It remains to verify the support condition.  A PBW monomial in the vacuum
module has the form
\[
 x_1(-m_1)\cdots x_s(-m_s)\mathbf1,
 \qquad m_i\in\mathbb Z_{>0},
\]
where $x_i\in\mathfrak g_{\alpha_i}$ and $\alpha_i$ may be zero for a
Cartan vector.  Its affine weight is
\[
 k\Lambda_0-\sum_{i=1}^s(m_i\delta-\alpha_i).
\]
For $\alpha_i\ne0$, the element $m_i\delta-\alpha_i$ is a positive affine
root, while for $\alpha_i=0$ it is the positive imaginary root
$m_i\delta$.  Therefore
\[
 \operatorname{Supp}_{\widehat{\mathfrak h}}V^k(\mathfrak g)
 \subset k\Lambda_0-\widehat Q_+.
\]
Taking a submodule or quotient cannot create new weights, so the same
inclusion holds for every graded subquotient $M$.  The defining support
condition for $\mathcal O_k$ is thus satisfied with the single upper weight
$k\Lambda_0$, proving the assertion.
\end{proof}

For an irreducible affine highest-weight module
$L(\widehat\lambda)$, the form of the nonvanishing criterion needed here is
\begin{equation}
  \alpha_0^\vee=K-\theta^\vee,
  \label{eq:common-affine-simple-coroot}
\end{equation}
\begin{equation}
  \widehat\lambda(\alpha_0^\vee)<0
  \quad\Longrightarrow\quad
  H^0_{\mathrm{DS},f_\theta}\!\left(L(\widehat\lambda)\right)\ne0.
  \label{eq:common-DS-nonvanishing}
\end{equation}
This is the nonvanishing part of the irreducible minimal-reduction theorem
\cite[Theorem~6.7.4]{Arakawa2005}: outside the nonnegative-integral affine-wall
values, the zeroth reduction of an irreducible highest-weight module is
nonzero and irreducible.  The displayed strict inequality is the specialization
needed throughout this paper.
For the finitely generated quotients considered below, associated varieties
satisfy \cite[Theorem~6.1(3)]{ArakawaMoreau2018}
\begin{equation}
  X_{H^0_{\mathrm{DS},f_\theta}(M)}
  =X_M\cap\mathcal S_{f_\theta}.
  \label{eq:common-associated-variety}
\end{equation}
Thus a quotient with associated variety $\overline{\mathbb O}_{\min}$ has
lisse minimal reduction.  For the finite-order Ramond twist, the standard
Zhu filtration gives the surjection
\eqref{eq:common-C2-Ramond-Zhu}; hence a lisse vertex algebra has
finite-dimensional Ramond Zhu algebra.  We record this implication as
\begin{equation}
 V\text{ lisse}\quad\Longrightarrow\quad
 \dim A_{\mathrm R}(V)<\infty.
 \label{eq:common-lisse-Ramond-finite}
\end{equation}

\begin{lemma}[Nonvanishing of a reduced singular generator]
\label{lem:common-reduced-singular-generator}
Let $k\in\mathbb Z_{<0}$, let $0\ne s\in V^k(\mathfrak g)$ be an affine
singular vector of positive conformal degree, and put
\[
 N=U(\widehat{\mathfrak g}^{\prime})s.
\]
Let $\widehat\lambda$ be the affine highest weight of $s$, including the
external derivation grading.  Assume
\begin{equation}
 \widehat\lambda(\alpha_0^\vee)<0.
 \label{eq:common-singular-generator-negative-wall}
\end{equation}
Then
\begin{equation*}
 H^0_{\mathrm{DS},f_\theta}(N)\ne0,
 \qquad
 [s]_{\mathrm{DS}}\ne0
 \quad\text{in}\quad
 \mathcal W^k(\mathfrak g,f_\theta),
\end{equation*}
and $H^0_{\mathrm{DS},f_\theta}(N)$ is cyclic, generated by
$[s]_{\mathrm{DS}}$.
\end{lemma}

\begin{proof}
The graded affine highest-weight module $N$ has an irreducible
highest-weight quotient
\[
 N\twoheadrightarrow L(\widehat\lambda).
\]
By \eqref{eq:common-singular-generator-negative-wall} and the irreducible
nonvanishing criterion \eqref{eq:common-DS-nonvanishing},
\[
 H^0_{\mathrm{DS},f_\theta}\!\left(L(\widehat\lambda)\right)\ne0.
\]
Lemma~\ref{lem:common-graded-subquotient-category-O} places $N$ and its
quotients in the category on which minimal reduction is exact.  Hence
exactness applied to the displayed surjection gives
\[
 H^0_{\mathrm{DS},f_\theta}(N)
 \twoheadrightarrow
 H^0_{\mathrm{DS},f_\theta}\!\left(L(\widehat\lambda)\right),
\]
so $H^0_{\mathrm{DS},f_\theta}(N)$ is nonzero.

Shift the scalar action of the derivation so that $s$ has $D$-value zero,
and denote the resulting module by $N^\circ$ and its highest weight by
$\widehat\lambda^\circ$.  There is a canonical surjection
\[
 M(\widehat\lambda^\circ)\twoheadrightarrow N^\circ
\]
from the affine Verma module.  Both $M(\widehat\lambda^\circ)$ and
$N^\circ$ are objects of $\mathcal O_k$ (the scalar shift of $D$ only
changes the origin of the external grading).  By
\cite[Theorem~6.3]{KacWakimoto2004}, the minimal reduction of the source is
a $W$-Verma module generated by the BRST class of its affine
highest-weight vector.  Exactness applied to
$M(\widehat\lambda^\circ)\twoheadrightarrow N^\circ$ therefore shows
that $H^0_{\mathrm{DS},f_\theta}(N)$ is generated by
$[s]_{\mathrm{DS}}$.  Since this cyclic module is nonzero, its cyclic
generator cannot vanish.

Finally, the short exact sequence
\[
 0\longrightarrow N\longrightarrow V^k(\mathfrak g)
 \longrightarrow V^k(\mathfrak g)/N\longrightarrow0
\]
lies in $\mathcal O_k$ by
Lemma~\ref{lem:common-graded-subquotient-category-O}.  Exactness therefore
gives an injection
\[
 H^0_{\mathrm{DS},f_\theta}(N)
 \hookrightarrow
 H^0_{\mathrm{DS},f_\theta}(V^k(\mathfrak g))
 =\mathcal W^k(\mathfrak g,f_\theta),
\]
so the same nonzero generator $[s]_{\mathrm{DS}}$ is nonzero in the
universal minimal $W$-algebra.
\end{proof}

\begin{proposition}[Exact reduction of a generated quotient]
\label{prop:common-exact-generated-quotient}
Let $s_1,\ldots,s_r$ be affine singular vectors in $V^k(\mathfrak g)$,
let $N_i=U(\widehat{\mathfrak g}^{\prime})s_i$, set
$N=N_1+\cdots+N_r$, and put $Q=V^k(\mathfrak g)/N$.  Assume that all
these modules lie in a category on which minimal reduction is exact and
that $H^0_{\mathrm{DS},f_\theta}(N_i)$ is cyclic, generated by the class
$\bar s_i$.  Then
\begin{equation}
 H^0_{\mathrm{DS},f_\theta}(Q)
 \cong
 \mathcal W^k(\mathfrak g,f_\theta)
 \big/\langle\bar s_1,\ldots,\bar s_r\rangle.
 \label{eq:common-exact-generated-quotient}
\end{equation}
\end{proposition}

\begin{proof}
The sum map $\bigoplus_iN_i\twoheadrightarrow N$ remains surjective after
minimal reduction.  Hence the image of $H^0_{\mathrm{DS},f_\theta}(N)$
in the adjoint $\mathcal W^k(\mathfrak g,f_\theta)$-module is contained
in the ideal generated by the $\bar s_i$.  Conversely, every $\bar s_i$
belongs to that image.  Since an adjoint submodule is a vertex-algebra
ideal, the image is exactly $\langle\bar s_1,\ldots,\bar s_r\rangle$.
Exactness applied to $0\to N\to V^k(\mathfrak g)\to Q\to0$ gives
\eqref{eq:common-exact-generated-quotient}.
\end{proof}

\begin{proposition}[Extremal PBW line]
\label{prop:common-extremal-PBW-line}
Let \(\mathfrak a\) be a simply laced simple ideal of
\(\mathfrak g^\natural\), let \(\eta\) be its highest root, and choose
\(0\ne e_\eta\in\mathfrak a_\eta\).  Extend
\[
 d_\eta(\mu)=\langle\mu,\eta^\vee\rangle
\]
by zero on the other simple ideals of \(\mathfrak g^\natural\).  Assume that
all weights \(\nu\) of \(\mathfrak g_{-1/2}\) satisfy
\begin{equation}
 d_\eta(\nu)\le1.
 \label{eq:common-extremal-G-bound}
\end{equation}
Then, at every noncritical level \(k\) and for every integer \(r\ge1\),
\begin{equation*}
 \mathcal W^k(\mathfrak g,f_\theta)_r[r\eta]
 =\mathbb C\bigl(J^{\{e_\eta\}}_{(-1)}\bigr)^r\mathbf1
 =\mathbb C:J^{\{e_\eta\}}{}^r:.
\end{equation*}
In particular, every nonzero vector of conformal weight \(r\) and finite
weight \(r\eta\) is a nonzero scalar multiple of this current power.
\end{proposition}

\begin{proof}
Use the PBW basis of the universal minimal \(W\)-algebra formed by ordered
monomials in derivatives of the currents, the \(G\)-fields, and the conformal
vector.  A current weight \(\mu\) is zero or a root of one of the simple
ideals.  Since \(\mathfrak a\) is simply laced,
\[
 d_\eta(\mu)\le2,
\]
with equality only for \(\mu=\eta\); currents from the other ideals have
charge zero.  Therefore
\[
 d_\eta(\mu)\le2\le2(1+j)
\]
for a factor \(\partial^jJ^{\{a\}}\), and equality throughout occurs only
for the undifferentiated current \(J^{\{e_\eta\}}\).  By
\eqref{eq:common-extremal-G-bound}, every factor
\(\partial^jG^{\{u\}}\) satisfies the strict estimate
\[
 d_\eta(\operatorname{wt}u)\le1
 <2\left(\frac32+j\right),
\]
and every Virasoro factor has charge zero and positive conformal weight.
Thus a PBW monomial \(X\) satisfies
\[
 d_\eta(\operatorname{wt}X)\le2\Delta(X).
\]
For simultaneous weight \((\Delta,\operatorname{wt})=(r,r\eta)\), both
sides are \(2r\).  Every factor must attain equality, so every factor is
\(J^{\{e_\eta\}}\), and exactly \(r\) factors occur.  PBW independence
shows that the resulting vector is nonzero.
\end{proof}

\subsection{Ramond Zhu relation}

The Ramond automorphism fixes $J^{\{a\}}$ and $\omega$ and sends
$G^{\{u\}}$ to $-G^{\{u\}}$.  We denote the associated twisted Zhu algebra
by $A_{\mathrm R}(V)$, following the finite-order twisted Zhu construction
of \cite{DongLiMason1998}.  In this paper $\mathfrak g$ is an ordinary simple
Lie algebra, regarded as purely even.  Accordingly, the bracket in the
Ramond Zhu algebra is the ordinary associative commutator.  This is
compatible with the identification of the Ramond Zhu algebra with Premet's
ordinary finite $W$-algebra
\cite[Equation~(5.1) and Remark~5.8]{KacMosenederFrajriaPapi2025}.  If $I$ is Ramond-stable, then
\[
 A_{\mathrm R}(V/I)\cong A_{\mathrm R}(V)/\langle[I]_{\mathrm R}\rangle.
\]
Moreover,
\begin{equation}
 V/C_2(V)\twoheadrightarrow\operatorname{gr}A_{\mathrm R}(V).
 \label{eq:common-C2-Ramond-Zhu}
\end{equation}
All Casimir elements below use the restriction of the invariant form
normalized at the beginning of this section.  Thus, on each simple ideal
$\mathfrak a\subset\mathfrak g^\natural$, dual bases are taken with
respect to this restricted form, and the corresponding Casimir eigenvalue
on $L_{\mathfrak a}(\lambda)$ is
$C_{\mathfrak a}(\lambda)=(\lambda\mid\lambda+2\rho_{\mathfrak a})$.
Affine levels of the current subalgebras are expressed with the same
restricted form.  These conventions are used uniformly in every
contraction and Casimir table below.

Choose dual bases $\{x_i\}$ and $\{x^i\}$ of $\mathfrak g^\natural$ and
write $\Omega^\natural$ for the quadratic Casimir.  In the universal
Ramond Zhu algebra, the classes of the $G$-generators satisfy
\begin{equation}
 [u,v]=\langle u,v\rangle
 \left(\Omega^\natural-2(k+h^\vee)L-\frac12p_{\mathfrak g}(k)\right)
 +Q(u,v),
 \label{eq:common-Ramond-Zhu}
\end{equation}
where $L=[\omega]_{\mathrm R}$ and
\begin{equation}
 Q(u,v)=\sum_{i,j}\langle[x_i,u],[v,x^j]\rangle
       (x^ix_j+x_jx^i).
 \label{eq:common-quadratic-term}
\end{equation}
The type-dependent polynomial and Casimir normalization are stated where the
relation is used; see \cite{AdamovicKacMosenederFrajriaPapiPerse2018,
KacMosenederFrajriaPapi2025}.

\begin{proposition}[Uniform Ramond--Casimir trace formula]
\label{prop:common-Ramond-Casimir-trace}
Suppose
$\mathfrak g^\natural=\mathfrak a_1\oplus\cdots\oplus\mathfrak a_t$
is semisimple and write
$\Omega^\natural=\Omega_1+\cdots+\Omega_t$.  Let $M$ be a nonzero
finite-dimensional module over a quotient of the Ramond Zhu algebra on
which $L$ acts by $h$.  Put
\[
 N=\dim M,\qquad \tau_j=\operatorname{tr}_M(\Omega_j).
\]
If the type-dependent contraction calculation gives
\begin{equation}
 \operatorname{tr}_M Q(u,v)
 =\langle u,v\rangle\sum_{j=1}^t\gamma_j\tau_j,
 \label{eq:common-Q-trace-coefficients}
\end{equation}
then
\begin{equation*}
 \sum_{j=1}^t(1+\gamma_j)\frac{\tau_j}{N}
 =2(k+h^\vee)h+\frac12p_{\mathfrak g}(k).
\end{equation*}
\end{proposition}

\begin{proof}
Take the trace of \eqref{eq:common-Ramond-Zhu}.  The left-hand side is an
ordinary associative commutator and has trace zero.  Substitute
\eqref{eq:common-Q-trace-coefficients}, choose $u,v$ with
$\langle u,v\rangle\ne0$, and divide by $N$.
\end{proof}

\subsection{Li twisting}
\label{subsec:common-Li-twist}

For a weight-one primary vector $H$ with semisimple rational zero mode, Li's
operator is
\begin{equation*}
 \Delta(H,z)=z^{H_{(0)}}
 \exp\!\left(\sum_{r\ge1}\frac{H_{(r)}}{-r}(-z)^{-r}\right).
\end{equation*}
If $H_{(1)}H=\kappa\mathbf1$, then the twisted Hamiltonian is
\begin{equation}
  L_0^\Delta=L_0+H_0+\frac\kappa2.
  \label{eq:common-Li-Hamiltonian}
\end{equation}
For a current $H=J^{\{x\}}$ at level $r$, Cartan and root modes transform as
\begin{equation}
 J_0^{\{h\},\Delta}=J_0^{\{h\}}+r(x\mid h),
 \qquad
 J_0^{\{e_\gamma\},\Delta}=J_{\gamma(x)}^{\{e_\gamma\}}.
 \label{eq:common-Li-mode-shifts}
\end{equation}
Every vertex-algebra ideal is stable coefficientwise under the twisted
fields; explicitly,
\begin{equation}
 Y(\Delta(H,z)a,z)I\subset I((z^{1/N}))
 \label{eq:common-Li-ideal-stability}
\end{equation}
for a suitable common denominator $N$.  Hence the same vector space is a
submodule of the Li-twisted adjoint module \cite{Li2012}.

\begin{lemma}[Lowest-energy Ramond--Zhu module]
\label{lem:common-lowest-Ramond-Zhu-module}
Let $V$ be a nonzero conformally graded vertex algebra whose Li-twisted
adjoint module is an $L_0^{\mathrm R}$-semisimple positive-energy module
\[
 V^{\mathrm R}=\bigoplus_{m\ge0}V^{\mathrm R}[h_0+m].
\]
Assume that the relevant Ramond Zhu algebra $B$ is finite-dimensional.
If $0\ne I\subset V$ is an ideal, then its twisted underlying space is a
Ramond submodule and there are $m\in\mathbb Z_{\ge0}$ and a nonzero
finite-dimensional $B$-module
\[
 M\subset I^{\mathrm R}[h_0+m]
\]
on which the conformal class $L=[\omega]_{\mathrm R}$ acts as
$h_0+m$.  If a semisimple Lie algebra $\mathfrak a\subset B$ acts on
$M$, then any simple quotient of $M$ splits as an $\mathfrak a$-module.
\end{lemma}

\begin{proof}
Equation~\eqref{eq:common-Li-ideal-stability} makes the underlying vector
space of $I$ a Ramond submodule.  Its nonempty spectrum has a least element
$h_0+m$.  The lowest eigenspace is a module over the Ramond Zhu algebra,
and $L$ acts there by that eigenvalue.  For $0\ne v$ in this eigenspace,
the cyclic module $Bv$ is finite-dimensional.  The last assertion follows
from complete reducibility for finite-dimensional modules over a semisimple
Lie algebra.
\end{proof}

\begin{proposition}[Casimir-gap simplicity criterion]
\label{prop:common-Casimir-gap-simplicity}
Let $\pi:V\twoheadrightarrow V_{\mathrm s}$ be a nonzero unital
homomorphism from a conformally graded vertex algebra to a simple vertex
algebra.  Suppose the following data are available.
\begin{enumerate}
\item The Li-twisted adjoint module is $L_0^{\mathrm R}$-semisimple and
\[
 V^{\mathrm R}=\bigoplus_{m\ge0}V^{\mathrm R}[h_0+m],
 \qquad
 (V^{\mathrm R}[h_0])_{\lambda_0}=\mathbb C\mathbf1.
\]
\item The Ramond Zhu algebra $B$ is finite-dimensional, and every simple
$B$-module is completely reducible over a semisimple Lie algebra
$\mathfrak a$.
\item Let $\mathcal C\in B$ be the image of a weighted Casimir element of
$Z(U(\mathfrak a))$.  Assume that it has average
\[
 \frac{\operatorname{tr}_S\mathcal C}{\dim S}=A(h)
\]
on every nonzero finite-dimensional simple $B$-module $S$ on which $L$
acts as $h$.  All allowed irreducible $\mathfrak a$-constituents satisfy
$c_{\mathcal C}\le C_{\max}$, and for every $m\ge1$,
\[
 A(h_0+m)\ge A(h_0+1)>C_{\max}.
\]
\item For every nonzero finite-dimensional $B$-submodule
$M\subset V^{\mathrm R}[h_0]$ and every simple quotient $S$ of $M$, the
finite-dimensional and low-degree analysis forces the restriction of $S$
to $\mathfrak a$ to contain an irreducible constituent of highest weight
$\lambda_0$.
\end{enumerate}
Then $\pi$ is an isomorphism; in particular, $V$ is simple.
\end{proposition}

\begin{proof}
Let $I=\ker\pi$ and suppose $I\ne0$.  By
Lemma~\ref{lem:common-lowest-Ramond-Zhu-module}, a simple quotient $S$ of
a cyclic Zhu submodule at the lowest ideal energy has
$h=h_0+m$, $m\ge0$.  The Casimir bound and the strict inequality at
$h_0+1$ force $m=0$.  By the final hypothesis, $S$ contains the
$\lambda_0$-constituent.  Since $M$ is finite-dimensional and
$\mathfrak a$ is semisimple, the surjection $M\twoheadrightarrow S$ splits
as an $\mathfrak a$-module.  Hence $M$, and therefore the lowest ideal
space, contains a $\lambda_0$-constituent.  Its highest-weight line lies in
$(V^{\mathrm R}[h_0])_{\lambda_0}=\mathbb C\mathbf1$, so
$\mathbf1\in I$, contradicting that $\pi$ is a nonzero unital map.
\end{proof}

\begin{remark}[How the criterion is checked]

Every application below verifies the four hypotheses in the same order:
(i) the exact reduced quotient and lisse-ness give a finite Ramond Zhu
algebra; (ii) Li twisting gives the integral spectral lattice and identifies
the extremal vacuum line; (iii) the Ramond trace relation gives the affine
function \(A(h)\), while the root-system appendix supplies
\(C_{\max}\); and (iv) the low original conformal degrees exclude all
nonvacuum constituents attaining the required average at \(h_0\).  Each
case theorem cites these four ingredients explicitly, so no type-dependent
assumption is imported silently into Proposition~\ref{prop:common-Casimir-gap-simplicity}.
\end{remark}

\subsection{Detection, lifting, and maximality principles}

\begin{lemma}[Detection by minimal reduction]
\label{lem:common-DS-detection}
Let $k\in\mathbb Z_{<0}$, and let $M\ne0$ be an object of the relevant affine
category with a grading
\[
 M=\bigoplus_{r\in d_0+\mathbb Z_{\ge0}}M[r],
 \qquad \dim M[r]<\infty,
 \qquad a(n)M[r]\subset M[r-n],
\]
such that every $M[r]$ is a finite-dimensional $\mathfrak g$-module.  Then
$H^0_{\mathrm{DS},f_\theta}(M)\ne0$.
\end{lemma}

\begin{proof}
Choose a nonzero $\mathfrak g$-highest-weight vector $v$ in the lowest
nonzero homogeneous subspace, of highest weight $\mu\in P_+$.  Positive
modes annihilate $v$, so it is an affine highest-weight vector.  Its cyclic
module has an irreducible quotient of highest weight
$\widehat\lambda=k\Lambda_0+\mu-d\delta$, and
\[
 \widehat\lambda(\alpha_0^\vee)
 =k-\langle\mu,\theta^\vee\rangle<0.
\]
The quotient has nonzero minimal reduction by
\eqref{eq:common-DS-nonvanishing}.  Exactness applied first to the cyclic
surjection and then to the inclusion of the cyclic module into $M$ proves
the assertion.
\end{proof}

\begin{proposition}[Affine lifting]
\label{prop:common-affine-lifting}
Let $Q$ be a conformally graded quotient of $V^k(\mathfrak g)$ at a negative
integral level, with finite-dimensional homogeneous subspaces that are
finite-dimensional $\mathfrak g$-modules, and let
$Q\twoheadrightarrow L_k(\mathfrak g)$ be the canonical surjection.  If the
induced map
\[
 H^0_{\mathrm{DS},f_\theta}(Q)
 \longrightarrow
 H^0_{\mathrm{DS},f_\theta}(L_k(\mathfrak g))
\]
is an isomorphism, then $Q\cong L_k(\mathfrak g)$.
\end{proposition}

\begin{proof}
Let $K$ be the kernel.  By
Lemma~\ref{lem:common-graded-subquotient-category-O}, the exact sequence
$0\to K\to Q\to L_k(\mathfrak g)\to0$ lies in the exactness category.
Exactness gives $H^0_{\mathrm{DS},f_\theta}(K)=0$.  If $K\ne0$, it satisfies
the hypotheses of Lemma~\ref{lem:common-DS-detection}, a contradiction.
\end{proof}

Write
\[
 \operatorname{Rad}V^k(\mathfrak g)
 :=\ker\bigl(V^k(\mathfrak g)\to L_k(\mathfrak g)\bigr)
\]
for the unique maximal proper graded ideal.

\begin{theorem}[Minimal-reduction maximality principle]
\label{thm:common-minimal-reduction-maximality}
Let $\mathfrak g$ be a finite-dimensional simple Lie algebra, let
$f_\theta$ be a minimal nilpotent element, and let $k\in\mathbb Z_{<0}$.
Let
\[
 N\subset\operatorname{Rad}V^k(\mathfrak g)
\]
be a graded ideal and put $Q=V^k(\mathfrak g)/N$.  Assume that every
homogeneous subspace of $Q$ is finite-dimensional and is a
finite-dimensional $\mathfrak g$-module.  If
\[
 H^0_{\mathrm{DS},f_\theta}(Q)
\]
is a nonzero simple vertex algebra, then
\[
 Q\cong L_k(\mathfrak g).
\]
Equivalently,
\[
 N=\operatorname{Rad}V^k(\mathfrak g).
\]
\end{theorem}

\begin{proof}
Since $N\subset\operatorname{Rad}V^k(\mathfrak g)$, the canonical quotient
map factors through a surjection
\[
 q:Q\twoheadrightarrow L_k(\mathfrak g).
\]
The corresponding short exact sequence lies in the category on which
minimal reduction is exact, by
Lemma~\ref{lem:common-graded-subquotient-category-O}.  Hence $q$ induces a
surjection
\[
 H^0_{\mathrm{DS},f_\theta}(Q)
 \twoheadrightarrow
 H^0_{\mathrm{DS},f_\theta}(L_k(\mathfrak g)).
\]
The target is nonzero by Lemma~\ref{lem:common-DS-detection}.  Since the
source is simple, the induced map is an isomorphism.
Proposition~\ref{prop:common-affine-lifting} therefore gives
$Q\cong L_k(\mathfrak g)$, and the equality of ideals follows.
\end{proof}

\begin{corollary}[Singular-vector maximality principle]

Let $s_1,\ldots,s_r$ be affine singular vectors contained in
$\operatorname{Rad}V^k(\mathfrak g)$, where $k\in\mathbb Z_{<0}$.  Set
\[
 N_i=U(\widehat{\mathfrak g}^{\prime})s_i,
 \qquad
 N=N_1+\cdots+N_r,
 \qquad
 Q=V^k(\mathfrak g)/N.
\]
Assume that the modules involved lie in a category on which minimal
reduction is exact, that each
$H^0_{\mathrm{DS},f_\theta}(N_i)$ is cyclic with generator $\bar s_i$, and
that
\[
 \mathcal W^k(\mathfrak g,f_\theta)
 \big/\langle\bar s_1,\ldots,\bar s_r\rangle
\]
is nonzero and simple.  Then
\[
 N=\operatorname{Rad}V^k(\mathfrak g).
\]
\end{corollary}

\begin{proof}
Proposition~\ref{prop:common-exact-generated-quotient} identifies the
minimal reduction of $Q$ with the displayed simple quotient.  The conclusion
follows from Theorem~\ref{thm:common-minimal-reduction-maximality}.
\end{proof}

\begin{theorem}[DS--Ramond--Casimir maximality theorem]
\label{thm:common-DS-Ramond-Casimir-maximality}
Let $\mathfrak g$ be a finite-dimensional simple Lie algebra,
$f_\theta$ a minimal nilpotent element, and $k\in\mathbb Z_{<0}$.  Let
$s_1,\ldots,s_r\in\operatorname{Rad}V^k(\mathfrak g)$ be affine singular
vectors, set
\[
 N=\sum_{i=1}^rU(\widehat{\mathfrak g}^{\prime})s_i,
 \qquad
 Q=V^k(\mathfrak g)/N,
\]
and assume that
\[
 Q=\bigoplus_{n\ge0}Q[n],
 \qquad \dim Q[n]<\infty,
\]
with every $Q[n]$ a finite-dimensional $\mathfrak g$-module.  Suppose that
exact minimal
reduction gives
\[
 H^0_{\mathrm{DS},f_\theta}(Q)
 \cong
 \overline{\mathcal W}
 :=\mathcal W^k(\mathfrak g,f_\theta)
   \big/\langle\bar s_1,\ldots,\bar s_r\rangle,
\]
where $\overline{\mathcal W}\ne0$.  Assume further that the following data
are available.
\begin{enumerate}
\item The Li-twisted adjoint module is $L_0^{\mathrm R}$-semisimple and
\[
 \overline{\mathcal W}^{\mathrm R}
 =\bigoplus_{m\ge0}\overline{\mathcal W}^{\mathrm R}[h_0+m],
 \qquad
 \bigl(\overline{\mathcal W}^{\mathrm R}[h_0]\bigr)_{\lambda_0}
 =\mathbb C\mathbf1.
\]
\item The Ramond Zhu algebra $B=A_{\mathrm R}(\overline{\mathcal W})$ is
finite-dimensional, and its finite-dimensional simple modules are completely
reducible over a semisimple Lie algebra $\mathfrak a\subset B$.
\item Let $\mathcal C\in B$ be the image of a weighted Casimir element of
$Z(U(\mathfrak a))$.  Assume that it has average
\[
 \frac{\operatorname{tr}_S\mathcal C}{\dim S}=A(h)
\]
on every nonzero finite-dimensional simple $B$-module $S$ on which the
conformal class acts as $h$.  Every allowed irreducible
$\mathfrak a$-constituent has Casimir eigenvalue at most $C_{\max}$, and
\[
 A(h_0+m)\ge A(h_0+1)>C_{\max}
 \qquad(m\ge1).
\]
\item For every nonzero finite-dimensional $B$-submodule
$M\subset\overline{\mathcal W}^{\mathrm R}[h_0]$ and every simple quotient
$S$ of $M$, the restriction of $S$ to $\mathfrak a$ contains an irreducible
constituent of highest weight $\lambda_0$.
\end{enumerate}
Then
\[
 Q\cong L_k(\mathfrak g),
 \qquad
 N=\operatorname{Rad}V^k(\mathfrak g).
\]
\end{theorem}

\begin{proof}
The quotient map $Q\twoheadrightarrow L_k(\mathfrak g)$ induces a nonzero
unital homomorphism
\[
 \overline{\mathcal W}
 \longrightarrow
 H^0_{\mathrm{DS},f_\theta}(L_k(\mathfrak g))
 =\mathcal W_k(\mathfrak g,f_\theta).
\]
This target is the simple minimal $W$-algebra and is nonzero by
Lemma~\ref{lem:common-DS-detection}.  Hypotheses
(1)--(4) are precisely those of
Proposition~\ref{prop:common-Casimir-gap-simplicity}; hence
$\overline{\mathcal W}$ is simple.
Theorem~\ref{thm:common-minimal-reduction-maximality} now gives the asserted
affine isomorphism and equality of ideals.
\end{proof}

\begin{remark}

The preceding theorem separates the argument into a formal part and a
finite type-dependent part.  The formal part consists of exact reduction,
Ramond--Zhu detection of ideals, and affine lifting.  The only local data to
be verified in an application are the exact reduced singular vectors, the
Ramond--Casimir trace coefficients, the Casimir gap, and the ground-energy
exclusion.  The case sections below verify exactly these inputs.  Consequently, once
reduced simplicity has been established, each affine maximality statement
below is an instance of
Theorem~\ref{thm:common-minimal-reduction-maximality}; the present theorem
packages the reduced and affine steps into a single reusable statement.
\end{remark}
\section{The series \texorpdfstring{$D_\ell$}{D-l}, \texorpdfstring{$\ell\ge5$}{l>=5}}
\label{sec:case-dl}

\subsection{Type data for \texorpdfstring{$D_\ell$}{D-l}}
\label{dl:sec:Dl-preliminaries}

Fix
\begin{equation}
  \ell\ge5,\qquad
  \mathfrak g=D_\ell\cong\mathfrak{so}_{2\ell},\qquad
  m=\ell-2,\qquad K=m-1=\ell-3.
  \label{dl:eq:Dl-global-rank-notation}
\end{equation}
The affine central element is denoted by $\mathbf K$, whereas $K$ is the
positive integer in \eqref{dl:eq:Dl-global-rank-notation}.  We use the standard
orthogonal realization and Bourbaki numbering
\begin{equation}
 \alpha_i=\varepsilon_i-\varepsilon_{i+1}\ (1\le i\le\ell-1),
 \qquad \alpha_\ell=\varepsilon_{\ell-1}+\varepsilon_\ell,
 \label{dl:eq:Dl-simple-roots}
\end{equation}
with highest root $\theta=\varepsilon_1+\varepsilon_2$.  The two distinguished
roots are
\begin{equation}
 \theta_A=\varepsilon_1-\varepsilon_2,
 \qquad \theta_D=\varepsilon_3+\varepsilon_4.
 \label{dl:eq:Dl-two-natural-highest-roots}
\end{equation}
For the minimal grading associated with $f_\theta=e_{-\theta}$,
\begin{equation*}
 \mathfrak g^\natural=\mathfrak s\oplus\mathfrak d,
 \qquad \mathfrak s\cong A_1,
 \qquad \mathfrak d\cong D_m,
 \qquad
 U:=\mathfrak g_{-\frac12}
 \cong L_{\mathfrak s}(\varpi)\boxtimes L_{\mathfrak d}(\eta_1),
\end{equation*}
where $\varpi=\theta_A/2$ and $\eta_1=\varepsilon_3$ is the vector highest
weight of $D_m$.  When $m=3$ we use
\begin{equation}
 D_3\cong A_3,
 \label{dl:eq:Dl-D3-A3-identification}
\end{equation}
under which $L_{\mathfrak d}(\eta_1)$ is the six-dimensional module of
highest weight $\omega_2$.  On $U$ we use
\begin{equation}
 \langle u,v\rangle=(e_\theta\mid[u,v]),
 \label{dl:eq:Dl-skew-form}
\end{equation}
and the Casimir normalization
\begin{equation}
 c_2(\varpi)=\frac32,
 \qquad c_2(\eta_1)=2m-1.
 \label{dl:eq:Dl-basic-Casimir-values}
\end{equation}
All further root, weight-coset, and Casimir data are collected in
Appendix~\ref{dl:app:Dl-root-weight-Casimir-data}.  In particular, that
appendix records the full root system and the two Kostant weights used below.  These conventions agree with
\cite{Bourbaki2002,KacRoanWakimoto2003,KacWakimoto2004,
AdamovicKacMosenederFrajriaPapiPerse2018}.

\subsubsection{A uniform package for the two negative levels}

For \(n\in\{0,1\}\), set
\begin{equation}
 k_n=n-2,
 \qquad
 K_n=m+n-2=\ell+n-4.
 \label{dl:eq:Dl-uniform-level-parameters}
\end{equation}
Then the two natural current levels in the minimal reduction are
\begin{equation*}
 k_{A,n}^{\natural}=K_n,
 \qquad
 k_{D,n}^{\natural}=n,
\end{equation*}
and the Arakawa--Moreau generators have the uniform exponents
\begin{equation}
 s_{A,n}=\sigma(w_A)^{K_n+1},
 \qquad
 s_{D,n}=\sigma(w_D)^{n+1}.
 \label{dl:eq:Dl-uniform-singular-generators}
\end{equation}
The Li twist by \(x=\varpi\) has vacuum energy
\begin{equation*}
 h_{0,n}=\frac{K_n}{4}.
\end{equation*}
Thus the two cases differ only by the level-zero collapse of the
\(D_m\)-current factor when \(n=0\):
\begin{equation*}
\begin{array}{c|c|c|c|c}
 n&k_n&K_n&(s_{A,n},s_{D,n})&(k_{A,n}^{\natural},k_{D,n}^{\natural})\\ \hline
 0&-2&m-2&(\sigma(w_A)^{m-1},\sigma(w_D))&(m-2,0)\\
 1&-1&m-1&(\sigma(w_A)^m,\sigma(w_D)^2)&(m-1,1).
\end{array}
\end{equation*}
This parameter package will also make the two final affine-lifting
arguments identical.

\subsubsection{Level, Ramond, and spectral-flow data}

All formal conventions are those of Section~\ref{sec:common-framework}.  The
minimal $W$-algebra generators are
\begin{equation}
  J^{\{a\}},\ a\in\mathfrak g^\natural;
  \qquad G^{\{u\}},\ u\in U;
  \qquad \omega,
  \label{dl:eq:Dl-minimal-W-generators}
\end{equation}
and the two current levels are
\begin{equation*}
  k_A^\natural=k+m,\qquad k_D^\natural=k+2.
\end{equation*}
At $k=-1$ they are
\begin{equation}
  k_A^\natural=m-1=K,\qquad k_D^\natural=1.
  \label{dl:eq:Dl-level-minus-one-natural-levels}
\end{equation}
Since $p_{D_\ell}(k)=(k+2)(k+m)$, the specialized Ramond relation is
\begin{equation}
 [u,v]=\langle u,v\rangle
 \left(\Omega_A+\Omega_D-2(2m+1)L-\frac{m-1}{2}\right)+Q(u,v).
 \label{dl:eq:Dl-level-minus-one-Ramond-Zhu-relation}
\end{equation}

\subsubsection{Li twist}

For the coweight $x$ used below and $H=J^{\{x\}}$, the formulas of
Section~\ref{subsec:common-Li-twist} specialize to
\begin{align}
 L_0^{\mathrm R}&=L_0+x_0+\frac{m-1}{4},
 \label{dl:eq:Dl-Li-Ramond-Hamiltonian}\\
 J_0^{\{h\},\mathrm R}&=J_0^{\{h\}}+(m-1)(x\mid h),
 \label{dl:eq:Dl-Li-Cartan-shift}\\
 J_0^{\{e_\gamma\},\mathrm R}&=J_{\gamma(x)}^{\{e_\gamma\}}.
 \label{dl:eq:Dl-Li-root-mode-shift}
\end{align}
\begin{lemma}[Ideals under Li twisting]

Every ideal considered in this section is a submodule of the Li-twisted
adjoint module on the same vector space.
\end{lemma}
\begin{proof}
Apply the coefficientwise stability statement in
Section~\ref{subsec:common-Li-twist}.
\end{proof}

\subsection{Singular vectors and their exact minimal reductions}
\label{dl:sec:Dl-singular-vectors-exact-reductions}

In this section we identify the exact minimal reductions of the two
affine singular vectors that define the \(D_\ell\), \(n=1\) candidate
quotient of Arakawa and Moreau.  We retain the notation
\[
  m=\ell-2,
  \qquad
  K=m-1
\]
from \eqref{dl:eq:Dl-global-rank-notation}.

\subsubsection{The two Kostant components}

Let \(\Omega\in S^2(\mathfrak g)\) be the quadratic Casimir element.
For \(\mathfrak g=D_\ell\), the Kostant decomposition is
\begin{equation*}
  S^2(\mathfrak g)
  =
  L_{\mathfrak g}(2\theta)
  \oplus
  \mathbb C\Omega
  \oplus
  \mathcal K,
  \qquad
  \mathcal K=\mathcal K_A\oplus\mathcal K_D,
\end{equation*}
where
\begin{equation}
  \mathcal K_A
  \cong
  L_{\mathfrak g}(\theta+\theta_A)
  =
  L_{\mathfrak g}(2\omega_1)
  \label{dl:eq:Dl-first-Kostant-component}
\end{equation}
and
\begin{equation*}
  \mathcal K_D
  \cong
  L_{\mathfrak g}(\theta+\theta_D).
\end{equation*}
More explicitly,
\begin{equation}
  \theta+\theta_D
  =
  \begin{cases}
    \omega_4+\omega_5,&\ell=5,\\
    \omega_4,&\ell\ge6.
  \end{cases}
  \label{dl:eq:Dl-second-Kostant-highest-weight-cases}
\end{equation}
See \cite[Section~2 and Theorem~4.2(4)]{ArakawaMoreau2018}.

Choose highest-weight vectors
\begin{equation*}
\begin{aligned}
  0\ne w_A&\in\mathcal K_A,
  &
  \operatorname{wt}_{\mathfrak g}(w_A)
  &=\theta+\theta_A,\\
  0\ne w_D&\in\mathcal K_D,
  &
  \operatorname{wt}_{\mathfrak g}(w_D)
  &=\theta+\theta_D.
\end{aligned}
\end{equation*}
Let
\begin{equation*}
  \widetilde\sigma:S^2(\mathfrak g)
  \longrightarrow
  U\!\left(\mathfrak g[t^{-1}]t^{-1}\right)
\end{equation*}
be the degree-two symmetrization map determined by
\begin{equation*}
  \widetilde\sigma(xy)
  =
  \frac12
  \bigl(
    x(-1)y(-1)+y(-1)x(-1)
  \bigr).
\end{equation*}
Under the PBW vector-space identification
\begin{equation}
  U\!\left(\mathfrak g[t^{-1}]t^{-1}\right)
  \overset{\sim}{\longrightarrow}
  V^k(\mathfrak g),
  \qquad
  u\longmapsto u\mathbf1,
  \label{dl:eq:Dl-affine-PBW-identification}
\end{equation}
we write
\begin{equation*}
  \sigma(w)=\widetilde\sigma(w)\mathbf1.
\end{equation*}
Following the convention of \cite{ArakawaMoreau2018}, powers of
\(\sigma(w)\) below mean associative powers of
\(\widetilde\sigma(w)\) in
\(U(\mathfrak g[t^{-1}]t^{-1})\), applied to the vacuum.  Thus we
define
\begin{equation*}
  s_A
  =
  \widetilde\sigma(w_A)^m\mathbf1
  \in V^{-1}(\mathfrak g)_{2m},
  \qquad
  s_D
  =
  \widetilde\sigma(w_D)^2\mathbf1
  \in V^{-1}(\mathfrak g)_4.
\end{equation*}
We henceforth abbreviate these two vectors as
\(s_A=\sigma(w_A)^m\) and \(s_D=\sigma(w_D)^2\).  This is PBW-power
notation and not an iterated normally ordered product.  Their finite
weights and affine degrees are
\begin{equation*}
\begin{aligned}
  \deg s_A&=2m,
  &
  \operatorname{wt}_{\mathfrak g}(s_A)
  &=m(\theta+\theta_A),\\
  \deg s_D&=4,
  &
  \operatorname{wt}_{\mathfrak g}(s_D)
  &=2(\theta+\theta_D).
\end{aligned}
\end{equation*}

By \cite[Theorem~4.2(4)(a)]{ArakawaMoreau2018},
\(\sigma(w_A)^{r+1}\) is singular precisely at
\begin{equation*}
  k=r-\ell+2.
\end{equation*}
Taking \(r=\ell-3=m-1\) gives \(k=-1\) and exponent \(m\).
Similarly, by
\cite[Theorem~4.2(4)(b)]{ArakawaMoreau2018},
\(\sigma(w_D)^{r+1}\) is singular precisely at
\begin{equation*}
  k=r-2.
\end{equation*}
Taking \(r=1\) gives \(k=-1\) and exponent \(2\).  Consequently,
\begin{equation*}
  s_A\text{ and }s_D
  \text{ are nonzero affine singular vectors in }
  V^{-1}(\mathfrak g).
\end{equation*}
Their nonvanishing also follows from the affine PBW theorem.

Set
\begin{equation*}
  N_A=U(\widehat{\mathfrak g})s_A,
  \qquad
  N_D=U(\widehat{\mathfrak g})s_D
  \subset V^{-1}(\mathfrak g).
\end{equation*}
Since the affine currents strongly generate \(V^{-1}(\mathfrak g)\),
the spaces \(N_A\) and \(N_D\) are graded vertex-algebra ideals as well
as affine submodules.

\begin{lemma}[Properness of the singular submodules]
\label{dl:lem:Dl-singular-submodules-proper}
The submodules \(N_A\) and \(N_D\), and their sum
\(N=N_A+N_D\), are proper graded ideals of
\(V^{-1}(\mathfrak g)\).  In particular,
\begin{equation}
  N_A,\ N_D,\ N
  \subset
  \ker\!\left(
    V^{-1}(\mathfrak g)\longrightarrow L_{-1}(\mathfrak g)
  \right).
  \label{dl:eq:Dl-singular-submodules-in-maximal-ideal}
\end{equation}
\end{lemma}

\begin{proof}
Let \(X\in\{A,D\}\), and let \(d_X>0\) be the conformal degree of
\(s_X\).  Since \(s_X\) is an affine highest-weight vector, the
triangular decomposition of \(\widehat{\mathfrak g}\) gives
\begin{equation*}
  N_X
  =
  U\!\left(\mathfrak g[t^{-1}]t^{-1}\right)
  U(\mathfrak g)\,s_X.
\end{equation*}
The zero modes preserve conformal degree, while every negative affine
mode raises it.  Therefore
\begin{equation*}
  (N_X)_r=0
  \qquad
  (r<d_X).
\end{equation*}
In particular, \(\mathbf1\notin N_X\), so \(N_X\) is proper.  The same
argument gives
\[
  N_r=0
  \qquad
  \bigl(r<\min\{2m,4\}\bigr),
\]
and hence \(N\) is proper.  Since \(L_{-1}(\mathfrak g)\) is the simple
graded quotient of the universal affine vacuum module, every proper
graded submodule is contained in its maximal proper graded submodule.
This proves \eqref{dl:eq:Dl-singular-submodules-in-maximal-ideal}.
\end{proof}

\subsubsection{Nonvanishing under minimal reduction}

The affine highest weights of the two singular vectors, including the
external derivation grading, are
\begin{equation*}
\begin{aligned}
  \widehat\lambda_A
  &=
  -\Lambda_0
  +
  m(\theta+\theta_A)
  -
  2m\delta,\\
  \widehat\lambda_D
  &=
  -\Lambda_0
  +
  2(\theta+\theta_D)
  -
  4\delta.
\end{aligned}
\end{equation*}
The roots \(\theta_A\) and \(\theta_D\) are perpendicular to
\(\theta\), so
\begin{equation*}
  \langle\theta+\theta_A,\theta^\vee\rangle=2,
  \qquad
  \langle\theta+\theta_D,\theta^\vee\rangle=2.
\end{equation*}
Using \eqref{eq:common-affine-simple-coroot}, we obtain
\begin{equation}
\begin{aligned}
  \widehat\lambda_A(\alpha_0^\vee)
  &=
  -1-2m<0,\\
  \widehat\lambda_D(\alpha_0^\vee)
  &=
  -1-4=-5<0.
\end{aligned}
\label{dl:eq:Dl-two-singular-negative-affine-coordinates}
\end{equation}

\begin{lemma}
\label{dl:lem:Dl-reduced-singular-vectors-nonzero}
The BRST classes of the two singular vectors are nonzero:
\begin{equation*}
  [s_A]_{\mathrm{DS}}\ne0,
  \qquad
  [s_D]_{\mathrm{DS}}\ne0
  \quad\text{in}\quad
  \mathcal W^{-1}(\mathfrak g,f_\theta).
\end{equation*}
\end{lemma}

\begin{proof}
We give the argument simultaneously for
\(X\in\{A,D\}\).  The highest-weight module \(N_X\) has an irreducible
highest-weight quotient
\begin{equation*}
  N_X\twoheadrightarrow L(\widehat\lambda_X).
\end{equation*}
By \eqref{dl:eq:Dl-two-singular-negative-affine-coordinates} and the
nonvanishing criterion
\eqref{eq:common-DS-nonvanishing},
\begin{equation*}
  H^0_{\mathrm{DS},f_\theta}
  \!\left(L(\widehat\lambda_X)\right)\ne0.
\end{equation*}
Exactness on \(\mathcal O_{-1}\) gives a surjection
\begin{equation*}
  H^0_{\mathrm{DS},f_\theta}(N_X)
  \twoheadrightarrow
  H^0_{\mathrm{DS},f_\theta}
  \!\left(L(\widehat\lambda_X)\right),
\end{equation*}
so the source is nonzero.

Shift the scalar action of the derivation so that \(s_X\) has
\(D\)-value zero, and denote the resulting module by \(N_X^\circ\).
There is a surjection
\begin{equation}
  M(\widehat\lambda_X^\circ)
  \twoheadrightarrow
  N_X^\circ
  \label{dl:eq:Dl-Verma-to-singular-submodule}
\end{equation}
from the affine Verma module.  By
\cite[Theorem~6.3]{KacWakimoto2004}, its reduction is a
\(W\)-Verma module generated by the BRST class of the affine
highest-weight vector.  Exactness applied to
\eqref{dl:eq:Dl-Verma-to-singular-submodule} therefore shows that
\begin{equation}
  H^0_{\mathrm{DS},f_\theta}(N_X)
  \text{ is generated by }[s_X]_{\mathrm{DS}}.
  \label{dl:eq:Dl-individual-reduction-cyclic}
\end{equation}
Since this cyclic module is nonzero, its cyclic generator is nonzero.
Finally, exactness applied to
\(N_X\hookrightarrow V^{-1}(\mathfrak g)\) identifies it with a
submodule of the universal minimal \(W\)-algebra, proving the claim.
\end{proof}

\subsubsection{The two extremal PBW lines}

Choose nonzero highest-root vectors
\begin{equation}
  e_A=e_{\theta_A}\in\mathfrak s,
  \qquad
  e_D=e_{\theta_D}\in\mathfrak d,
  \qquad
  J_A=J^{\{e_A\}},
  \qquad
  J_D=J^{\{e_D\}}.
  \label{dl:eq:Dl-two-highest-root-currents}
\end{equation}

For a \(\mathfrak g^\natural\)-weight \(\mu\), set
\begin{equation*}
  q_A(\mu)=\langle\mu,\theta_A^\vee\rangle,
  \qquad
  q_D(\mu)=\langle\mu,\theta_D^\vee\rangle.
\end{equation*}

\begin{lemma}[\(A_1\)-extremal PBW line]
\label{dl:lem:Dl-A1-extremal-PBW-line}
One has
\begin{equation*}
  \mathcal W^{-1}(\mathfrak g,f_\theta)_m[m\theta_A]
  =
  \mathbb C
  \bigl(J_A{}_{(-1)}\bigr)^m\mathbf1
  =
  \mathbb C:J_A^m:.
\end{equation*}
\end{lemma}

\begin{proof}
Apply Proposition~\ref{prop:common-extremal-PBW-line} to the simple ideal
\(\mathfrak s\cong A_1\), with \(\eta=\theta_A\).  Currents from
\(\mathfrak d\) have zero \(q_A\)-charge, while the weights of the
\(A_1\)-adjoint module have charge at most two, with equality only on
\(\mathbb Ce_A\).  Since
\(\mathfrak g_{-1/2}\cong\mathbf2\boxtimes V_{\mathrm{vec}}\), all its
weights have \(q_A\)-charge at most one.  The common proposition gives the
claimed line (and is independent of the affine level).
\end{proof}

\begin{lemma}[\(D_m\)-extremal PBW line]
\label{dl:lem:Dl-D-extremal-PBW-line}
One has
\begin{equation*}
  \mathcal W^{-1}(\mathfrak g,f_\theta)_2[2\theta_D]
  =
  \mathbb C
  \bigl(J_D{}_{(-1)}\bigr)^2\mathbf1
  =
  \mathbb C:J_D^2:.
\end{equation*}
\end{lemma}

\begin{proof}
Apply Proposition~\ref{prop:common-extremal-PBW-line} to the simply laced
ideal \(\mathfrak d\cong D_m\), with \(\eta=\theta_D\).  Currents from
\(\mathfrak s\) have zero \(q_D\)-charge, and a root of \(D_m\) has
\(q_D\)-charge at most two, with equality only for \(\theta_D\).  The
\(D_m\)-weights of the vector representation are \(\pm\varepsilon_j\), so
all weights of \(\mathfrak g_{-1/2}\) have \(q_D\)-charge at most one.
The common proposition with \(r=2\) proves the assertion.
\end{proof}

\subsubsection{Exact reductions of the singular vectors}

The restriction of \(\theta\) to
\(\mathfrak h^\natural\) is zero, while the restrictions of
\(\theta_A\) and \(\theta_D\) are their respective highest roots.
Therefore,
\begin{equation}
\begin{aligned}
  \operatorname{wt}_{\mathfrak g^\natural}
  \bigl([s_A]_{\mathrm{DS}}\bigr)
  &=m\theta_A,\\
  \operatorname{wt}_{\mathfrak g^\natural}
  \bigl([s_D]_{\mathrm{DS}}\bigr)
  &=2\theta_D.
\end{aligned}
\label{dl:eq:Dl-reduced-singular-natural-weights}
\end{equation}
Indeed, the standard BRST representatives are the affine
highest-weight vectors tensored with the ghost vacuum, and the latter
has \(\mathfrak h^\natural\)-weight zero.

For an affine highest-weight representative of finite weight \(\mu\)
and affine degree \(d\), the minimal-reduction Hamiltonian acts by
\(d-\mu(x_\theta)\).  Since
\begin{equation*}
  \theta(x_\theta)=1,
  \qquad
  \theta_A(x_\theta)=\theta_D(x_\theta)=0,
\end{equation*}
we obtain
\begin{equation}
\begin{aligned}
  \Delta_W\bigl([s_A]_{\mathrm{DS}}\bigr)
  &=
  2m-m=m,\\
  \Delta_W\bigl([s_D]_{\mathrm{DS}}\bigr)
  &=
  4-2=2.
\end{aligned}
\label{dl:eq:Dl-reduced-singular-conformal-weights}
\end{equation}

\begin{proposition}[Exact reduced singular vectors]
\label{dl:prop:Dl-exact-reduced-singular-vectors}
There exist nonzero scalars
\(c_A,c_D\in\mathbb C^\times\) such that
\begin{equation*}
  [s_A]_{\mathrm{DS}}
  =
  c_A
  \bigl(J_A{}_{(-1)}\bigr)^m\mathbf1
  =
  c_A:J_A^m:
\end{equation*}
and
\begin{equation*}
  [s_D]_{\mathrm{DS}}
  =
  c_D
  \bigl(J_D{}_{(-1)}\bigr)^2\mathbf1
  =
  c_D:J_D^2:.
\end{equation*}
\end{proposition}

\begin{proof}
Lemma~\ref{dl:lem:Dl-reduced-singular-vectors-nonzero} shows that both
classes are nonzero.  Equations
\eqref{dl:eq:Dl-reduced-singular-natural-weights} and
\eqref{dl:eq:Dl-reduced-singular-conformal-weights} place
\([s_A]_{\mathrm{DS}}\) in the one-dimensional space of
Lemma~\ref{dl:lem:Dl-A1-extremal-PBW-line}, and place
\([s_D]_{\mathrm{DS}}\) in the one-dimensional space of
Lemma~\ref{dl:lem:Dl-D-extremal-PBW-line}.  Each class is therefore a
scalar multiple of the indicated PBW vector, and each scalar is
nonzero.
\end{proof}

\subsubsection{The reduced candidate quotient}

Set
\begin{equation*}
  N=N_A+N_D,
  \qquad
  Q_\ell=V^{-1}(\mathfrak g)/N.
\end{equation*}
Equivalently,
\begin{equation*}
  Q_\ell
  =
  V^{-1}(D_\ell)
  \Big/
  \left\langle
    \sigma(w_A)^m,
    \sigma(w_D)^2
  \right\rangle.
\end{equation*}

\begin{lemma}[Category-\(\mathcal O\) scope]
\label{dl:lem:Dl-candidate-modules-in-category-O}
The graded affine modules
\begin{equation*}
  N_A,\qquad N_D,\qquad N=N_A+N_D,\qquad Q_\ell
\end{equation*}
belong to \(\mathcal O_{-1}\).  Their homogeneous subspaces are
finite-dimensional \(\mathfrak g\)-modules.
\end{lemma}

\begin{proof}
The conformal grading of \(V^{-1}(\mathfrak g)\) is bounded below,
and each homogeneous subspace is a finite-dimensional
\(\mathfrak g\)-module by the affine PBW theorem.  The singular
submodules \(N_A,N_D\), their sum, and the quotient \(Q_\ell\) inherit
these properties.  They are weight modules of level \(-1\), and the
action of the positive affine Borel is locally finite.  Hence they are
objects of the affine category \(\mathcal O_{-1}\).
\end{proof}
Define the ideal
\begin{equation}
  \mathcal J_\ell
  =
  \left\langle
    :J_A^m:,
    :J_D^2:
  \right\rangle
  \subset
  \mathcal W^{-1}(\mathfrak g,f_\theta)
  \label{dl:eq:Dl-reduced-current-ideal}
\end{equation}
and the quotient
\begin{equation*}
  \widetilde{\mathcal W}_\ell
  =
  \mathcal W^{-1}(\mathfrak g,f_\theta)/
  \mathcal J_\ell.
\end{equation*}

\begin{proposition}[Exact reduction of the candidate quotient]
\label{dl:prop:Dl-exact-reduction-candidate-quotient}
There is a canonical isomorphism
\begin{equation}
  H^0_{\mathrm{DS},f_\theta}(Q_\ell)
  \cong
  \widetilde{\mathcal W}_\ell.
  \label{dl:eq:Dl-exact-reduction-Q}
\end{equation}
Moreover,
\begin{equation}
  \widetilde{\mathcal W}_\ell\ne0,
  \qquad
  \widetilde{\mathcal W}_\ell
  \text{ is lisse}.
  \label{dl:eq:Dl-Wtilde-nonzero-lisse}
\end{equation}
\end{proposition}

\begin{proof}
Lemma~\ref{dl:lem:Dl-candidate-modules-in-category-O}, the cyclicity statement
\eqref{dl:eq:Dl-individual-reduction-cyclic}, and
Proposition~\ref{dl:prop:Dl-exact-reduced-singular-vectors} verify the
hypotheses of Proposition~\ref{prop:common-exact-generated-quotient}; this
gives \eqref{dl:eq:Dl-exact-reduction-Q}.  Moreover,
\cite[Proposition~5.2]{ArakawaMoreau2018} gives
\begin{equation*}
 X_{Q_\ell}=\overline{\mathbb O}_{\min}.
\end{equation*}
Hence \eqref{eq:common-associated-variety}, together with
\cite[Theorem~6.1(3)]{ArakawaMoreau2018}, proves that the reduced quotient is
nonzero and lisse.
\end{proof}

By Lemma~\ref{dl:lem:Dl-singular-submodules-proper}, the two singular
vectors lie in the maximal proper graded ideal of
\(V^{-1}(\mathfrak g)\).  Hence the canonical map to the simple affine
vertex algebra factors through a surjection
\begin{equation*}
  Q_\ell\twoheadrightarrow L_{-1}(\mathfrak g).
\end{equation*}
Exactness of minimal reduction and
\eqref{eq:common-DS-nonvanishing} then give a natural surjective
homomorphism
\begin{equation*}
  \pi_\ell:
  \widetilde{\mathcal W}_\ell
  \twoheadrightarrow
  \mathcal W_{-1}(\mathfrak g,f_\theta).
\end{equation*}
The next sections prove that \(\pi_\ell\) is an isomorphism.

\subsection{The Ramond Zhu algebra and a weighted Casimir identity}
\label{dl:sec:Dl-Ramond-weighted-Casimir}

We now pass from the lisse quotient
\(\widetilde{\mathcal W}_\ell\) constructed in
Section~\ref{dl:sec:Dl-singular-vectors-exact-reductions} to its
finite-dimensional Ramond Zhu algebra.  The two reduced current
relations impose different integrability bounds on the
\(A_1\)- and \(D_m\)-factors.  The Ramond relation then yields a
weighted trace identity involving the two quadratic Casimirs.

\subsubsection{The Ramond Zhu algebra and nilpotency relations}

Set
\begin{equation*}
  B_\ell
  =
  A_{\mathrm R}
  \!\left(\widetilde{\mathcal W}_\ell\right).
\end{equation*}
By \eqref{dl:eq:Dl-Wtilde-nonzero-lisse} and
\eqref{eq:common-lisse-Ramond-finite},
\(B_\ell\) is a nonzero finite-dimensional unital associative
algebra:
\begin{equation*}
  0<\dim B_\ell<\infty.
\end{equation*}

Let
\begin{equation*}
  e_A=[J_A]_{\mathrm R},
  \qquad
  e_D=[J_D]_{\mathrm R}
  \quad\text{in }B_\ell,
\end{equation*}
where \(J_A\) and \(J_D\) were defined in
\eqref{dl:eq:Dl-two-highest-root-currents}.  Since a root space is
isotropic for the invariant bilinear form,
\begin{equation}
  [J_A{}_\lambda J_A]=0,
  \qquad
  [J_D{}_\lambda J_D]=0.
  \label{dl:eq:Dl-highest-root-current-self-brackets}
\end{equation}
Equivalently,
\begin{equation*}
  (J_X)_{(r)}J_X=0,
  \qquad
  r\in\mathbb Z_{\ge0},
  \qquad
  X\in\{A,D\}.
\end{equation*}

\begin{lemma}[Zhu images of the current powers]
\label{dl:lem:Dl-Zhu-images-current-powers}
Let \(X\in\{A,D\}\).  In the Ramond Zhu algebra of every conformal
quotient of \(\mathcal W^{-1}(\mathfrak g,f_\theta)\), one has
\begin{equation}
  \bigl[:J_X^r:\bigr]_{\mathrm R}
  =
  [J_X]_{\mathrm R}^{\,r},
  \qquad
  r\in\mathbb Z_{\ge1}.
  \label{dl:eq:Dl-Zhu-images-current-powers}
\end{equation}
\end{lemma}

\begin{proof}
The Ramond automorphism fixes \(J_X\), and \(J_X\) has conformal
weight one.  The assertion is clear for \(r=1\).  Suppose it holds
for some \(r\ge1\).  By
\eqref{dl:eq:Dl-highest-root-current-self-brackets} and the noncommutative
Wick formula,
\begin{equation*}
  [J_X{}_\lambda :J_X^r:]=0.
\end{equation*}
Hence the weight-one Ramond Zhu product satisfies
\begin{align*}
  [J_X]_{\mathrm R}
  *
  \bigl[:J_X^r:\bigr]_{\mathrm R}
  &=
  \bigl[
    (J_X)_{(-1)}:J_X^r:
  \bigr]_{\mathrm R}
  +
  \bigl[
    (J_X)_{(0)}:J_X^r:
  \bigr]_{\mathrm R}
  \notag\\
  &=
  \bigl[:J_X^{r+1}:\bigr]_{\mathrm R}.
\end{align*}
The induction hypothesis proves
\eqref{dl:eq:Dl-Zhu-images-current-powers}.
\end{proof}

The defining relations of
\(\widetilde{\mathcal W}_\ell\) in
\eqref{dl:eq:Dl-reduced-current-ideal}, together with
Lemma~\ref{dl:lem:Dl-Zhu-images-current-powers}, give
\begin{equation*}
  e_A^m=0,
  \qquad
  e_D^2=0
  \quad\text{in }B_\ell.
\end{equation*}

The current zero modes define a Lie algebra homomorphism
\begin{equation*}
  \mathfrak g^\natural
  =
  \mathfrak s\oplus\mathfrak d
  \longrightarrow B_\ell.
\end{equation*}
Consequently, every finite-dimensional \(B_\ell\)-module is a
finite-dimensional \(\mathfrak g^\natural\)-module and is completely
reducible over \(\mathfrak g^\natural\).

\begin{lemma}[Allowed finite-dimensional types]
\label{dl:lem:Dl-allowed-finite-types}
Let \(M\) be a finite-dimensional \(B_\ell\)-module.  Every irreducible
\(\mathfrak g^\natural\)-constituent of \(M\) has the form
\begin{equation*}
  L_{\mathfrak s}(a\varpi)
  \boxtimes
  L_{\mathfrak d}(\lambda_D),
\end{equation*}
where
\begin{equation*}
  0\le a\le m-1
\end{equation*}
and
\begin{equation}
  \langle\lambda_D,\theta_D^\vee\rangle\le1.
  \label{dl:eq:Dl-D-highest-root-level-bound}
\end{equation}
\end{lemma}

\begin{proof}
Let \(v\) be a highest-weight vector of an irreducible constituent.

For the \(A_1\)-factor, suppose that its highest weight is
\(a\varpi\).  In the \(\mathfrak{sl}_2\)-string through \(v\),
\begin{equation*}
  e_A^r f_A^r v\ne0,
  \qquad
  0\le r\le a.
\end{equation*}
If \(a\ge m\), then \(e_A^m f_A^m v\ne0\), contradicting
\(e_A^m=0\).  Hence \(a\le m-1\).

For the \(D_m\)-factor, consider the
\(\mathfrak{sl}_2\)-subalgebra associated with the highest root
\(\theta_D\).  If
\[
  q=\langle\lambda_D,\theta_D^\vee\rangle\ge2,
\]
then its root string gives
\begin{equation*}
  e_D^2f_D^2v\ne0,
\end{equation*}
contradicting \(e_D^2=0\).  Thus \(q\le1\).
\end{proof}

For \(m\ge4\), the dominant \(D_m\)-weights satisfying
\eqref{dl:eq:Dl-D-highest-root-level-bound} are
\begin{equation}
  0,
  \qquad
  \eta_1,
  \qquad
  \eta_{m-1},
  \qquad
  \eta_m,
  \label{dl:eq:Dl-D-level-one-dominant-weights}
\end{equation}
where the last two are the spinor weights.  When \(m=3\), under
\(D_3\cong A_3\), the corresponding list is
\begin{equation}
  0,
  \qquad
  \omega_1,
  \qquad
  \omega_2,
  \qquad
  \omega_3.
  \label{dl:eq:Dl-D3-level-one-dominant-weights}
\end{equation}
At this stage all these possibilities are allowed.  The two spinor
cosets will be excluded in Section~\ref{dl:sec:Dl-weight-cosets-lowest-Ramond}.

\subsubsection{Invariant trace forms}

Let \(M\) be a finite-dimensional \(B_\ell\)-module, and write
\begin{equation*}
  \rho:
  \mathfrak g^\natural
  \longrightarrow
  \operatorname{End}_{\mathbb C}(M)
\end{equation*}
for the induced representation.  Put
\begin{equation*}
  N_M=\dim M,
  \qquad
  \tau_A=\operatorname{tr}_M\rho(\Omega_A),
  \qquad
  \tau_D=\operatorname{tr}_M\rho(\Omega_D).
\end{equation*}

For each simple summand
\(\mathfrak a\in\{\mathfrak s,\mathfrak d\}\), invariance of the trace
form implies that there is a scalar \(I_{\mathfrak a}\) such that
\begin{equation}
  \operatorname{tr}_M
  \bigl(
    \rho(x)\rho(y)
  \bigr)
  =
  I_{\mathfrak a}(x\mid y),
  \qquad
  x,y\in\mathfrak a.
  \label{dl:eq:Dl-Dynkin-index-trace-form}
\end{equation}
If \(x\in\mathfrak s\) and \(y\in\mathfrak d\), then
\begin{equation}
  \operatorname{tr}_M
  \bigl(
    \rho(x)\rho(y)
  \bigr)=0.
  \label{dl:eq:Dl-mixed-trace-form-zero}
\end{equation}
Indeed, for fixed \(y\in\mathfrak d\), the functional
\[
  x\longmapsto
  \operatorname{tr}_M\bigl(\rho(x)\rho(y)\bigr)
\]
vanishes on
\([\mathfrak s,\mathfrak s]=\mathfrak s\) by invariance and the
commutativity of the two simple ideals.

Taking the trace of the Casimir gives
\begin{equation*}
  I_{\mathfrak s}
  =
  \frac{\tau_A}{\dim\mathfrak s}
  =
  \frac{\tau_A}{3},
  \qquad
  I_{\mathfrak d}
  =
  \frac{\tau_D}{\dim\mathfrak d}
  =
  \frac{\tau_D}{m(2m-1)}.
\end{equation*}

\subsubsection{The trace of the quadratic Ramond term}

Choose dual bases adapted to the orthogonal decomposition
\(\mathfrak g^\natural=\mathfrak s\oplus\mathfrak d\).  The double sum
in \eqref{eq:common-quadratic-term} decomposes as
\begin{equation*}
  Q(u,v)
  =
  Q_{AA}(u,v)+Q_{DD}(u,v)+Q_{\mathrm{mix}}(u,v),
\end{equation*}
where the two indices of \(Q_{AA}\) lie in \(\mathfrak s\), those of
\(Q_{DD}\) lie in \(\mathfrak d\), and
\(Q_{\mathrm{mix}}\) is the sum of the two mixed index ranges.  The
mixed term need not vanish as an element of the enveloping algebra.

\begin{lemma}[Trace of the quadratic term]
\label{dl:lem:Dl-trace-quadratic-Ramond-term}
For all \(u,v\in U\),
\begin{equation}
  \operatorname{tr}_M\rho\bigl(Q_{\mathrm{mix}}(u,v)\bigr)=0,
  \label{dl:eq:Dl-trace-quadratic-mixed-term}
\end{equation}
\begin{equation}
  \operatorname{tr}_M\rho\bigl(Q_{AA}(u,v)\bigr)
  =
  \tau_A\langle u,v\rangle,
  \label{dl:eq:Dl-trace-quadratic-A-term}
\end{equation}
and
\begin{equation}
  \operatorname{tr}_M\rho\bigl(Q_{DD}(u,v)\bigr)
  =
  \frac2m
  \tau_D\langle u,v\rangle.
  \label{dl:eq:Dl-trace-quadratic-D-term}
\end{equation}
Consequently,
\begin{equation}
  \operatorname{tr}_M\rho\bigl(Q(u,v)\bigr)
  =
  \left(
    \tau_A+\frac2m\tau_D
  \right)
  \langle u,v\rangle.
  \label{dl:eq:Dl-trace-total-quadratic-term}
\end{equation}
\end{lemma}

\begin{proof}
Every summand of \(Q_{\mathrm{mix}}(u,v)\) is a scalar multiple of
\[
  xy+yx,
  \qquad
  x\in\mathfrak s,\quad y\in\mathfrak d.
\]
Since the two simple ideals commute and
\eqref{dl:eq:Dl-mixed-trace-form-zero} gives
\(\operatorname{tr}_M\rho(x)\rho(y)=0\), one has
\begin{equation*}
  \operatorname{tr}_M
  \rho(xy+yx)
  =
  2\operatorname{tr}_M
  \bigl(\rho(x)\rho(y)\bigr)
  =0.
\end{equation*}
Summing proves \eqref{dl:eq:Dl-trace-quadratic-mixed-term}.

Let \(\mathfrak a\) be either
\(\mathfrak s\) or \(\mathfrak d\), and choose an orthonormal basis
\(\{x_r\}\) of \(\mathfrak a\).  For this choice, the original and
dual basis vectors coincide.  Equations
\eqref{eq:common-quadratic-term} and
\eqref{dl:eq:Dl-Dynkin-index-trace-form} give
\begin{equation*}
  \operatorname{tr}_M
  \rho\bigl(Q_{\mathfrak a\mathfrak a}(u,v)\bigr)
  =
  2I_{\mathfrak a}
  \sum_r
  \langle[x_r,u],[v,x_r]\rangle.
\end{equation*}

The skew form \eqref{dl:eq:Dl-skew-form} is
\(\mathfrak g^\natural\)-invariant.  Since
\([v,x_r]=-[x_r,v]\), we obtain
\begin{align*}
  \sum_r
  \langle[x_r,u],[v,x_r]\rangle
  &=
  -\sum_r
  \langle x_ru,x_rv\rangle
  \notag\\
  &=
  \left\langle
    u,
    \sum_r x_r^2v
  \right\rangle.
\end{align*}
Thus the contraction is the Casimir eigenvalue of
\(\mathfrak a\) on \(U\), multiplied by
\(\langle u,v\rangle\).

As an \(\mathfrak s\)-module, \(U\) is a direct sum of copies of
\(L_{\mathfrak s}(\varpi)\), and hence
\begin{equation*}
  \Omega_A|_U
  =
  \frac32\,\operatorname{id}_U.
\end{equation*}
Using
\(I_{\mathfrak s}=\tau_A/3\), we get
\begin{equation*}
  2I_{\mathfrak s}\frac32=\tau_A,
\end{equation*}
which proves \eqref{dl:eq:Dl-trace-quadratic-A-term}.

As a \(\mathfrak d\)-module, \(U\) is a direct sum of copies of the
vector representation
\(L_{\mathfrak d}(\eta_1)\).  By
\eqref{dl:eq:Dl-basic-Casimir-values},
\begin{equation*}
  \Omega_D|_U
  =
  (2m-1)\operatorname{id}_U.
\end{equation*}
Using
\(I_{\mathfrak d}=\tau_D/[m(2m-1)]\), we obtain
\begin{equation*}
  2I_{\mathfrak d}(2m-1)
  =
  \frac2m\tau_D,
\end{equation*}
which proves \eqref{dl:eq:Dl-trace-quadratic-D-term}.
Adding the two contributions gives
\eqref{dl:eq:Dl-trace-total-quadratic-term}.
\end{proof}

\subsubsection{The weighted Casimir trace identity}

We now combine the preceding contraction with the Ramond relation at
level \(-1\).  Recall from
\eqref{dl:eq:Dl-level-minus-one-Ramond-Zhu-relation} that
\begin{equation*}
  [u,v]
  =
  \langle u,v\rangle
  \left(
    \Omega_A+\Omega_D
    -2(2m+1)L
    -\frac{m-1}{2}
  \right)
  +Q(u,v).
\end{equation*}
As explained after \eqref{eq:common-Ramond-Zhu}, because
$D_\ell$ is an ordinary simple Lie algebra its Ramond Zhu algebra is
identified with Premet's ordinary finite $W$-algebra, and
\[
  [\rho(u),\rho(v)]
  =
  \rho(u)\rho(v)-\rho(v)\rho(u).
\]
In particular,
\begin{equation*}
  \operatorname{tr}_M
  [\rho(u),\rho(v)]
  =0.
\end{equation*}

\begin{proposition}[Weighted Casimir trace identity]
\label{dl:prop:Dl-weighted-Casimir-trace-identity}
Let \(M\) be a nonzero finite-dimensional \(B_\ell\)-module on which
the central element \(L=[\omega]_{\mathrm R}\) acts by the scalar
\(h\in\mathbb C\).  Then
\begin{equation}
  \frac{\operatorname{tr}_M\rho(\Omega_A)}{\dim M}
  +
  \frac{m+2}{2m}
  \frac{\operatorname{tr}_M\rho(\Omega_D)}{\dim M}
  =
  (2m+1)h+\frac{m-1}{4}.
  \label{dl:eq:Dl-weighted-Casimir-trace-identity}
\end{equation}
Equivalently, in terms of \(\ell=m+2\),
\begin{equation}
  \frac{\operatorname{tr}_M\rho(\Omega_A)}{\dim M}
  +
  \frac{\ell}{2(\ell-2)}
  \frac{\operatorname{tr}_M\rho(\Omega_D)}{\dim M}
  =
  (2\ell-3)h+\frac{\ell-3}{4}.
  \label{dl:eq:Dl-weighted-Casimir-trace-identity-ell}
\end{equation}
\end{proposition}

\begin{proof}
Lemma~\ref{dl:lem:Dl-trace-quadratic-Ramond-term} gives
$\gamma_A=1$ and $\gamma_D=2/m$ in
Proposition~\ref{prop:common-Ramond-Casimir-trace}.  At $k=-1$,
$h^\vee=2m+2$ and $p_{D_\ell}(-1)=m-1$.  Thus
\[
 2\frac{\tau_A}{N_M}+\left(1+\frac2m\right)\frac{\tau_D}{N_M}
 =2(2m+1)h+\frac{m-1}{2}.
\]
Division by two gives \eqref{dl:eq:Dl-weighted-Casimir-trace-identity};
substituting $m=\ell-2$ gives
\eqref{dl:eq:Dl-weighted-Casimir-trace-identity-ell}.
\end{proof}

\begin{remark}[General-level form]

The same calculation applied to
\eqref{eq:common-Ramond-Zhu} gives
\begin{equation*}
  \frac{\operatorname{tr}_M\rho(\Omega_A)}{\dim M}
  +
  \frac{m+2}{2m}
  \frac{\operatorname{tr}_M\rho(\Omega_D)}{\dim M}
  =
  \bigl(k+h^\vee(D_\ell)\bigr)h
  +
  \frac14p_{D_\ell}(k),
\end{equation*}
whenever the two simple factors are represented with the
normalizations of Section~\ref{dl:sec:Dl-preliminaries}.  At \(k=-1\),
this specializes to
\eqref{dl:eq:Dl-weighted-Casimir-trace-identity}.
\end{remark}

The trace identity is only one part of the later numerical estimate.
The relation \(e_A^m=0\) already gives the sharp bound
\begin{equation}
  c_2(a\varpi)
  =
  \frac{a(a+2)}2
  \le
  \frac{m^2-1}{2}.
  \label{dl:eq:Dl-A1-Casimir-upper-bound}
\end{equation}
For the \(D_m\)-factor, however, the level-one condition still permits
the two spinor types in
\eqref{dl:eq:Dl-D-level-one-dominant-weights}.  The next section proves
that these spinor cosets do not occur in the adjoint Ramond module of
\(\widetilde{\mathcal W}_\ell\).  This will reduce the maximal
\(D_m\)-Casimir to the vector value \(2m-1\).

\subsection{Weight cosets and the lowest Ramond space}
\label{dl:sec:Dl-weight-cosets-lowest-Ramond}

The level-one condition for the \(D_m\)-factor obtained in
Section~\ref{dl:sec:Dl-Ramond-weighted-Casimir} still permits the two
spinor highest weights.  In the adjoint module of
\(\widetilde{\mathcal W}_\ell\), however, the strong generators occupy
only the root and vector cosets of the \(D_m\) weight lattice.  This
excludes the spinor types.  We then use the Li realization of the
Ramond sector to identify the entire lowest Ramond eigenspace.

Throughout this section,
\begin{equation*}
  m=\ell-2,
  \qquad
  K=m-1,
  \qquad
  x=\varpi^\vee=\varpi.
\end{equation*}

\subsubsection{The root and vector cosets}

Let \(Q_D\) and \(P_D\) denote the root and weight lattices of
\(\mathfrak d\cong D_m\), respectively.  For \(m\ge4\), let
\(\eta_1,\ldots,\eta_m\) be the fundamental weights of \(D_m\), with
\(\eta_1\) the vector highest weight and
\(\eta_{m-1},\eta_m\) the two spinor highest weights.

\begin{lemma}[The \(D_m\)-weight cosets of the adjoint module]
\label{dl:lem:Dl-adjoint-D-weight-cosets}
Every \(D_m\)-weight occurring in
\(\widetilde{\mathcal W}_\ell\), and hence in its Li-twisted adjoint
module \(\widetilde{\mathcal W}_\ell^{\mathrm R}\), belongs to
\begin{equation*}
  Q_D+\mathbb Z\eta_1.
\end{equation*}
For \(m\ge4\), the image of this subgroup in \(P_D/Q_D\) consists of
the two classes
\begin{equation}
  Q_D,
  \qquad
  \eta_1+Q_D.
  \label{dl:eq:Dl-root-and-vector-classes}
\end{equation}
\end{lemma}

\begin{proof}
The \(D_m\)-weights of the strong generators listed in
\eqref{dl:eq:Dl-minimal-W-generators} have the following form:
\begin{equation*}
\begin{aligned}
  \operatorname{wt}_{\mathfrak d}
  \bigl(J^{\{a\}}\bigr)
  &=0,
  &&a\in\mathfrak s,\\
  \operatorname{wt}_{\mathfrak d}
  \bigl(J^{\{b\}}\bigr)
  &\in Q_D,
  &&b\in\mathfrak d,\\
  \operatorname{wt}_{\mathfrak d}
  \bigl(G^{\{u\}}\bigr)
  &\in\eta_1+Q_D,
  &&u\in U,\\
  \operatorname{wt}_{\mathfrak d}(\omega)
  &=0.
\end{aligned}
\end{equation*}
The third line follows from
\(U\cong\mathbf2\boxtimes L_{\mathfrak d}(\eta_1)\).
Derivatives do not change finite weights.  Therefore every PBW
monomial in the universal minimal \(W\)-algebra has \(D_m\)-weight in
\(Q_D+\mathbb Z\eta_1\).  Passing to the quotient
\(\widetilde{\mathcal W}_\ell\) cannot create a new weight.

The Li twist is taken with an element of the \(A_1\)-factor, so it does
not change the \(D_m\)-Cartan zero modes.  Hence the same weight
restriction holds in the twisted adjoint module.

Finally, the vector class has order two in \(P_D/Q_D\).  Thus the
subgroup generated by \(\eta_1+Q_D\) consists precisely of the two
classes in \eqref{dl:eq:Dl-root-and-vector-classes}.
\end{proof}

\begin{remark}[The case \(D_3\cong A_3\)]
\label{dl:rem:Dl-D3-weight-cosets}
When \(m=3\), we use the identification
\(D_3\cong A_3\) from
\eqref{dl:eq:Dl-D3-A3-identification}.  The vector representation has
highest weight \(\omega_2\), and the subgroup generated by its class is
\begin{equation}
  Q(A_3)
  \cup
  \bigl(\omega_2+Q(A_3)\bigr).
  \label{dl:eq:Dl-D3-root-vector-cosets}
\end{equation}
The two four-dimensional spinor representations correspond to
\(\omega_1\) and \(\omega_3\), whose classes do not belong to
\eqref{dl:eq:Dl-D3-root-vector-cosets}.
\end{remark}

\begin{corollary}[Exclusion of spinor constituents]
\label{dl:cor:Dl-spinor-constituents-excluded}
Let \(S\) be a finite-dimensional
\(\mathfrak g^\natural\)-stable subspace of
\(\widetilde{\mathcal W}_\ell^{\mathrm R}\), and suppose that every
irreducible \(D_m\)-constituent of \(S\) has highest weight
\(\lambda_D\) satisfying
\begin{equation}
  \langle\lambda_D,\theta_D^\vee\rangle\le1.
  \label{dl:eq:Dl-section5-level-one-assumption}
\end{equation}
Then the \(D_m\)-highest weight of every irreducible constituent is
\begin{equation*}
  \lambda_D=0
  \quad\text{or}\quad
  \lambda_D=\eta_1.
\end{equation*}
For \(m=3\), the second possibility is understood as the
\(A_3\)-weight \(\omega_2\).
\end{corollary}

\begin{proof}
Assume first that \(m\ge4\).  By
\eqref{dl:eq:Dl-D-level-one-dominant-weights}, the dominant weights
satisfying \eqref{dl:eq:Dl-section5-level-one-assumption} are
\[
  0,\qquad
  \eta_1,\qquad
  \eta_{m-1},\qquad
  \eta_m.
\]
The last two lie in the spinor cosets.  Each irreducible constituent
embeds into $S$, and its highest weight is itself a weight occurring in
$S$.  Lemma~\ref{dl:lem:Dl-adjoint-D-weight-cosets} therefore forces that
highest weight to lie in the root or vector coset.  Hence only \(0\) and
\(\eta_1\) remain.

When \(m=3\), the list
\eqref{dl:eq:Dl-D3-level-one-dominant-weights} consists of
\(0,\omega_1,\omega_2,\omega_3\).  Remark
\ref{dl:rem:Dl-D3-weight-cosets} excludes
\(\omega_1\) and \(\omega_3\), leaving \(0\) and \(\omega_2\).
\end{proof}

In particular, whenever the finite-dimensional top space of a
lower-bounded Ramond submodule is viewed as a \(B_\ell\)-module,
Lemma~\ref{dl:lem:Dl-allowed-finite-types} and Corollary
\ref{dl:cor:Dl-spinor-constituents-excluded} imply that its
\(D_m\)-Casimir eigenvalues are bounded above by
\begin{equation}
  c_2(\eta_1)=2m-1.
  \label{dl:eq:Dl-D-Casimir-bound-after-spinor-exclusion}
\end{equation}

\subsubsection{The Ramond spectral lattice}

Let
\(\widetilde{\mathcal W}_\ell^{\mathrm R}\) denote the Li-twisted
adjoint module defined by \eqref{eq:common-Li-ideal-stability} with
\(H=J^{\{x\}}\).  By
\eqref{dl:eq:Dl-Li-Ramond-Hamiltonian},
\begin{equation}
  L_0^{\mathrm R}
  =
  L_0+x_0+\frac K4.
  \label{dl:eq:Dl-section5-Ramond-Hamiltonian}
\end{equation}

\begin{proposition}[Ramond spectral lattice]
\label{dl:prop:Dl-Ramond-spectral-lattice}
The operator \(L_0^{\mathrm R}\) acts semisimply on
\(\widetilde{\mathcal W}_\ell^{\mathrm R}\), and
\begin{equation}
  \operatorname{Spec}L_0^{\mathrm R}
  \subset
  \frac K4+\mathbb Z_{\ge0}.
  \label{dl:eq:Dl-Ramond-spectrum-inclusion}
\end{equation}
The lower endpoint is attained by the vacuum:
\begin{equation}
  L_0^{\mathrm R}\mathbf1
  =
  \frac K4\mathbf1.
  \label{dl:eq:Dl-Ramond-vacuum-energy}
\end{equation}
\end{proposition}

\begin{proof}
Every ordinary conformal homogeneous subspace of
\(\widetilde{\mathcal W}_\ell\) is finite-dimensional and
\(\mathfrak g^\natural\)-stable.  Since finite-dimensional
\(\mathfrak g^\natural\)-modules are completely reducible, \(x_0\)
acts semisimply on every such subspace.  It commutes with \(L_0\), so
\(L_0^{\mathrm R}\) is semisimple.

It remains to compute the possible eigenvalues of \(L_0+x_0\).
Images of PBW monomials in derivatives of the strong generators span
\(\widetilde{\mathcal W}_\ell\).  Choose an
\(\mathfrak h^\natural\)-weight basis of these generators.

For an \(A_1\)-root vector \(a_\gamma\), a factor
\(\partial^rJ^{\{a_\gamma\}}\) contributes
\begin{equation*}
  1+r+\gamma(x)
  \in\mathbb Z_{\ge0}
\end{equation*}
to \(L_0+x_0\), since
\(\gamma(x)\in\{-1,1\}\).  Equality with zero occurs only for
\(r=0\) and \(\gamma=-\theta_A\).  An \(A_1\)-Cartan current
contributes \(1+r\), and every \(D_m\)-current contributes \(1+r\).

If \(u_\mu\in U\) has \(A_1\)-weight \(\mu\), then
\begin{equation*}
  \partial^rG^{\{u_\mu\}}
  \quad\text{contributes}\quad
  \frac32+r+\mu(x)
  \in\{1+r,2+r\}.
\end{equation*}
Finally,
\begin{equation*}
  \partial^r\omega
  \quad\text{contributes}\quad
  2+r.
\end{equation*}
Thus every PBW factor has nonnegative integral excess energy, and the
only factor of excess zero is the undifferentiated negative
\(A_1\)-root current.  Adding the contributions proves
\eqref{dl:eq:Dl-Ramond-spectrum-inclusion}.  Since
\(L_0\mathbf1=x_0\mathbf1=0\),
\eqref{dl:eq:Dl-section5-Ramond-Hamiltonian} gives
\eqref{dl:eq:Dl-Ramond-vacuum-energy}.
\end{proof}

Set
\begin{equation*}
  T_\ell
  =
  \widetilde{\mathcal W}_\ell^{\mathrm R}
  \!\left[\frac K4\right].
\end{equation*}
Proposition~\ref{dl:prop:Dl-Ramond-spectral-lattice} shows that
\(T_\ell\) is the lowest Ramond eigenspace and contains the vacuum.

\subsubsection{The affine \(A_1\)-subalgebra}

Choose an \(\mathfrak{sl}_2\)-triple
\begin{equation*}
  (e_A,h_A,f_A)
  \subset\mathfrak s,
  \qquad
  [e_A,f_A]=h_A,
  \qquad
  (e_A\mid f_A)=1.
\end{equation*}
Write
\begin{equation*}
  E_A=J^{\{e_A\}},
  \qquad
  H_A=J^{\{h_A\}},
  \qquad
  F_A=J^{\{f_A\}}.
\end{equation*}
Thus \(E_A=J_A\).

The \(A_1\)-currents have level \(K\) by
\eqref{dl:eq:Dl-level-minus-one-natural-levels}.  They therefore define
a unital vertex-algebra homomorphism
\begin{equation*}
  \iota:
  V^K(A_1)
  \longrightarrow
  \widetilde{\mathcal W}_\ell.
\end{equation*}

\begin{proposition}[The integrable \(A_1\)-subalgebra]
\label{dl:prop:Dl-integrable-A1-subalgebra}
The homomorphism \(\iota\) factors through an injective homomorphism
\begin{equation*}
  \overline\iota:
  L_K(A_1)
  \hookrightarrow
  \widetilde{\mathcal W}_\ell.
\end{equation*}
In particular,
\begin{equation}
  (F_A)_{(-1)}^j\mathbf1\ne0
  \quad(0\le j\le K),
  \qquad
  (F_A)_{(-1)}^{K+1}\mathbf1=0.
  \label{dl:eq:Dl-negative-current-string}
\end{equation}
\end{proposition}

\begin{proof}
Since \(m=K+1\), the defining relation
\(:J_A^m:=0\) of
\(\widetilde{\mathcal W}_\ell\) is
\begin{equation*}
  (E_A)_{(-1)}^{K+1}\mathbf1=0.
\end{equation*}
For positive integral level \(K\), this vector generates the maximal
proper ideal of \(V^K(A_1)\); see \cite{Kac1998}.  Hence
\(\iota\) factors through the simple affine vertex algebra
\[
  L_K(A_1)
  =
  V^K(A_1)/
  \left\langle
    (E_A)_{(-1)}^{K+1}\mathbf1
  \right\rangle.
\]
The induced map is nonzero because it preserves the vacuum.  Since
\(L_K(A_1)\) is simple, it is injective.

The Chevalley involution of \(A_1\), which exchanges \(e_A\) and
\(-f_A\), extends to an automorphism of \(V^K(A_1)\) and preserves its
maximal ideal.  Applying it to the defining relation gives
\begin{equation*}
  (F_A)_{(-1)}^{K+1}\mathbf1=0
  \quad\text{in }L_K(A_1).
\end{equation*}

For \(0\le j\le K\), set
\(
  \mathsf e=(E_A)_{(1)}
\)
and
\(
  \mathsf f=(F_A)_{(-1)}
\).
The affine commutation relations give
\begin{equation*}
  \mathsf e^{\,j}\mathsf f^{\,j}\mathbf1
  =
  j!\,K(K-1)\cdots(K-j+1)\mathbf1,
\end{equation*}
with the empty product equal to one.  The right-hand side is nonzero,
so \((F_A)_{(-1)}^j\mathbf1\ne0\) for \(0\le j\le K\).
Injectivity of \(\overline\iota\) transfers these statements to
\(\widetilde{\mathcal W}_\ell\) and proves
\eqref{dl:eq:Dl-negative-current-string}.
\end{proof}

\subsubsection{Identification of the lowest Ramond space}

\begin{proposition}[The lowest Ramond space]
\label{dl:prop:Dl-lowest-Ramond-space}
The lowest Ramond eigenspace is
\begin{equation}
  T_\ell
  =
  \operatorname{span}_{\mathbb C}
  \left\{
    (F_A)_{(-1)}^j\mathbf1
    \,\middle|\,
    0\le j\le K
  \right\}.
  \label{dl:eq:Dl-lowest-Ramond-space-basis}
\end{equation}
Under the twisted zero-mode action,
\begin{equation}
  T_\ell
  \cong
  L_{\mathfrak s}(K\varpi)
  \boxtimes
  \mathbf1_{\mathfrak d}.
  \label{dl:eq:Dl-lowest-Ramond-module-type}
\end{equation}
In particular, \(T_\ell\) is an irreducible
\(\mathfrak g^\natural\)-module of dimension \(K+1=m\).
\end{proposition}

\begin{proof}
The proof of Proposition
\ref{dl:prop:Dl-Ramond-spectral-lattice} shows more than the spectral
inclusion: among all derivatives of the strong generators, the only
factor of excess energy zero is the undifferentiated state \(F_A\).
Consequently, the images of the zero-excess PBW monomials span
\(T_\ell\).  The current $F_A$ has regular operator product with itself,
since $[f_A,f_A]=0$ and $(f_A\mid f_A)=0$; hence these monomials are
precisely its ordinary $(-1)$-powers.  Therefore
\begin{equation*}
  T_\ell
  =
  \operatorname{span}_{\mathbb C}
  \left\{
    (F_A)_{(-1)}^j\mathbf1
    \,\middle|\,
    j\ge0
  \right\}.
\end{equation*}
Proposition~\ref{dl:prop:Dl-integrable-A1-subalgebra} truncates this span
exactly at \(j=K\), proving
\eqref{dl:eq:Dl-lowest-Ramond-space-basis}.

By \eqref{dl:eq:Dl-Li-root-mode-shift} and
\eqref{dl:eq:Dl-Li-Cartan-shift}, the twisted zero modes of the
\(A_1\)-current triple are
\begin{equation*}
  \mathsf e
  =
  (E_A)_{(1)},
  \qquad
  \mathsf f
  =
  (F_A)_{(-1)},
  \qquad
  \mathsf h
  =
  (H_A)_{(0)}+K.
\end{equation*}
The affine current relations imply
\begin{equation*}
  [\mathsf h,\mathsf e]=2\mathsf e,
  \qquad
  [\mathsf h,\mathsf f]=-2\mathsf f,
  \qquad
  [\mathsf e,\mathsf f]=\mathsf h.
\end{equation*}
Moreover,
\begin{equation*}
  \mathsf e\mathbf1=0,
  \qquad
  \mathsf h\mathbf1=K\mathbf1,
  \qquad
  \mathsf f^j\mathbf1
  =
  (F_A)_{(-1)}^j\mathbf1.
\end{equation*}
Thus the vacuum is a highest-weight vector of
\(\mathfrak s\)-weight \(K\varpi\), and the vectors in
\eqref{dl:eq:Dl-lowest-Ramond-space-basis} form the irreducible
\((K+1)\)-dimensional module
\(L_{\mathfrak s}(K\varpi)\).

For \(b\in\mathfrak d\), the twisted zero mode is the ordinary mode
\(J^{\{b\}}_{(0)}\).  It kills the vacuum and commutes with every
\(F_A{}_{(-1)}\), because the two simple ideals
\(\mathfrak s\) and \(\mathfrak d\) commute.  Hence
\(\mathfrak d\) acts trivially on \(T_\ell\), proving
\eqref{dl:eq:Dl-lowest-Ramond-module-type}.
\end{proof}

\begin{corollary}[The shifted vacuum line]
\label{dl:cor:Dl-shifted-vacuum-line}
The highest-weight line of \(T_\ell\) is the vacuum line:
\begin{equation*}
  T_\ell[K\varpi]
  =
  \mathbb C\mathbf1.
\end{equation*}
\end{corollary}

\begin{proof}
This follows immediately from
\eqref{dl:eq:Dl-lowest-Ramond-module-type}.  It may also be seen directly
from the shift
\begin{equation*}
  J_0^{\{h\},\mathrm R}
  =
  J_0^{\{h\}}+K(x\mid h):
\end{equation*}
the original degree-zero, weight-zero vacuum acquires shifted
\(A_1\)-weight \(K\varpi\), and
\((\widetilde{\mathcal W}_\ell)_0=\mathbb C\mathbf1\).
\end{proof}

The results of this section provide the sharp representation-theoretic
input for the simplicity argument.  If a nonzero ideal has a lowest
Ramond eigenspace \(S\), then \(S\) is a nonzero \(B_\ell\)-module
contained in the twisted adjoint module.  Since \(B_\ell\) is
finite-dimensional, every nonzero \(v\in S\) generates a nonzero
finite-dimensional cyclic module \(B_\ell v\).  The \(A_1\)-Casimir on
such a cyclic module is bounded by
\eqref{dl:eq:Dl-A1-Casimir-upper-bound}, while
Corollary~\ref{dl:cor:Dl-spinor-constituents-excluded} gives the sharp
\(D_m\)-bound
\eqref{dl:eq:Dl-D-Casimir-bound-after-spinor-exclusion}.  Section~\ref{dl:sec:Dl-simplicity-reduced-quotient}
combines these bounds with the weighted trace identity to prove that
the quotient \(\widetilde{\mathcal W}_\ell\) is simple.

\subsection{Reduced simplicity at level \texorpdfstring{$-1$}{-1}}
\label{dl:sec:Dl-simplicity-reduced-quotient}

The preceding sections verify the hypotheses of the common
Casimir-gap criterion.  We record the resulting argument in a form
that will be reused at the endpoint level.

\begin{theorem}[Simplicity of the reduced candidate quotient]
\label{dl:thm:Dl-reduced-candidate-simple}
For every \(\ell\ge5\), the vertex algebra
\(\widetilde{\mathcal W}_\ell\) is simple, and the canonical map is an
isomorphism
\begin{equation*}
 \widetilde{\mathcal W}_\ell
 \cong
 \mathcal W_{-1}(D_\ell,f_\theta).
\end{equation*}
\end{theorem}

\begin{proof}
Apply Proposition~\ref{prop:common-Casimir-gap-simplicity} to the
canonical nonzero surjection
\(
 \widetilde{\mathcal W}_\ell\twoheadrightarrow
 \mathcal W_{-1}(D_\ell,f_\theta)
\).
The spectral origin is
\(
 h_0=K/4=(m-1)/4
\)
by Proposition~\ref{dl:prop:Dl-Ramond-spectral-lattice}, and
Corollary~\ref{dl:cor:Dl-shifted-vacuum-line} gives the distinguished
one-dimensional line
\(
 T_\ell[K\varpi]=\mathbb C\mathbf1
\).
Take
\begin{equation*}
 \mathcal C
 =
 \Omega_A+\frac{m+2}{2m}\Omega_D,
 \qquad
 A(h)=(2m+1)h+\frac{m-1}{4}.
\end{equation*}
Proposition~\ref{dl:prop:Dl-weighted-Casimir-trace-identity} gives the
required trace identity.  The allowed constituents satisfy
\begin{equation}
 c_{\mathcal C}\le C_{\max}
 :=
 \frac{m^2-1}{2}
 +\frac{m+2}{2m}(2m-1)
 \label{dl:eq:Dl-minus-one-Cmax}
\end{equation}
by \eqref{dl:eq:Dl-A1-Casimir-upper-bound} and
Corollary~\ref{dl:cor:Dl-spinor-constituents-excluded}.  Moreover,
\begin{equation*}
 A(h_0+1)-C_{\max}
 =
 (2m+1)-\frac{m+2}{2m}(2m-1)
 =\frac{2m^2-m+2}{2m}>0.
\end{equation*}
Finally, Proposition~\ref{dl:prop:Dl-lowest-Ramond-space} identifies
the lowest space with the irreducible module
\(
 L_{\mathfrak s}(K\varpi)\boxtimes\mathbf1_{\mathfrak d}
\).  Hence every nonzero ground $B_\ell$-submodule is the entire lowest
space, and every simple quotient contains its highest-weight line
\(\mathbb C\mathbf1\).  All hypotheses of the criterion hold, proving
both simplicity and the displayed isomorphism.
\end{proof}

\begin{corollary}[The maximal ideal after minimal reduction]

For every \(\ell\ge5\),
\begin{equation*}
 \ker\!\left(
   \mathcal W^{-1}(D_\ell,f_\theta)
   \longrightarrow
   \mathcal W_{-1}(D_\ell,f_\theta)
 \right)
 =
 \left\langle:J_A^m:,:J_D^2:\right\rangle.
\end{equation*}
\end{corollary}

\begin{proof}
Combine the definition of \(\widetilde{\mathcal W}_\ell\) with
Theorem~\ref{dl:thm:Dl-reduced-candidate-simple}.
\end{proof}

\subsection{The maximal ideal of the affine vertex algebra}

We apply the common affine-lifting proposition to the canonical quotient.

\subsubsection{Uniform affine lifting}

The singular-vector quotient admits a canonical surjection
\begin{equation*}
 q_\ell:Q_\ell\twoheadrightarrow L_{-1}(D_\ell).
\end{equation*}
By Proposition~\ref{dl:prop:Dl-exact-reduction-candidate-quotient} and
Theorem~\ref{dl:thm:Dl-reduced-candidate-simple}, the induced map on minimal
reductions is an isomorphism:
\begin{equation}
 H^0_{\mathrm{DS},f_\theta}(Q_\ell)
 \cong\mathcal W_{-1}(D_\ell,f_\theta).
 \label{dl:eq:Dl-reduction-Q-is-simple-W}
\end{equation}
Thus Proposition~\ref{prop:common-affine-lifting} applies.

\begin{theorem}[The affine maximal ideal at level \(-1\)]
\label{dl:thm:Dl-main-maximal-ideal}
For every \(\ell\ge5\), the quotient \(Q_\ell\) is simple:
\begin{equation*}
  Q_\ell\cong L_{-1}(D_\ell).
\end{equation*}
Equivalently,
\begin{equation*}
  \ker\!\left(
    V^{-1}(D_\ell)
    \longrightarrow
    L_{-1}(D_\ell)
  \right)
  =
  \left\langle
    \sigma(w_A)^{\ell-2},
    \sigma(w_D)^2
  \right\rangle.
\end{equation*}
Here the powers of \(\sigma(w_A)\) and \(\sigma(w_D)\) are the
associative PBW powers fixed in
\eqref{dl:eq:Dl-affine-PBW-identification}.
\end{theorem}

\begin{proof}
Apply Proposition~\ref{prop:common-affine-lifting} to
$q_\ell:Q_\ell\twoheadrightarrow L_{-1}(D_\ell)$ using
\eqref{dl:eq:Dl-reduction-Q-is-simple-W}.
\end{proof}

\begin{remark}[The rank-five case]

The level \(-1\) theorem includes \(D_5\).  Here
\(\mathfrak d\cong D_3\cong A_3\), and the second Kostant component has
highest weight \(\omega_4+\omega_5\), as in
\eqref{dl:eq:Dl-second-Kostant-highest-weight-cases}.  The uniform
Casimir argument requires no separate low-rank proof.
\end{remark}

\subsection{An alternative proof at the endpoint level
	\texorpdfstring{$k=-2$}{k=-2}}
\label{dl:sec:Dl-level-minus-two}

The maximal-ideal statement at level $-2$ was proved in
\cite[Section~6, Corollary~6.6(1)]
{AdamovicKacMosenederFrajriaPapiPerse2020}.
Under the notation and normalization of that reference, its generators
$w_1,w_2$ agree, up to nonzero scalar multiples, with
$\sigma(w_D),\sigma(w_A)^{\ell-3}$, respectively.  We include the endpoint
in order to give an alternative proof using minimal reduction and the
general lifting principle developed above.  It must be treated separately
because one natural current level collapses and the reduced quotient has a
different generating set.  Set
\begin{equation*}
	K_0=m-2=\ell-4.
\end{equation*}
At level $-2$, the two natural current levels are
\begin{equation*}
	k_A^\natural=K_0,
	\qquad
	k_D^\natural=0.
\end{equation*}
\subsubsection{Endpoint singular vectors and the candidate quotient}

Define the PBW-power states
\begin{equation*}
  s_{A,-2}
  =
  \widetilde\sigma(w_A)^{m-1}\mathbf1,
  \qquad
  s_{D,-2}
  =
  \widetilde\sigma(w_D)\mathbf1
  =
  \sigma(w_D)
  \quad\text{in }V^{-2}(D_\ell).
\end{equation*}
Their finite weights and conformal degrees are
\begin{equation*}
\begin{aligned}
  \operatorname{wt}_{\mathfrak g}(s_{A,-2})
  &=(m-1)(\theta+\theta_A),
  &\deg s_{A,-2}&=2(m-1),\\
  \operatorname{wt}_{\mathfrak g}(s_{D,-2})
  &=\theta+\theta_D,
  &\deg s_{D,-2}&=2.
\end{aligned}
\end{equation*}

\begin{proposition}[Endpoint singular vectors]

The states \(s_{A,-2}\) and \(s_{D,-2}\) are nonzero affine
singular vectors in \(V^{-2}(D_\ell)\).
\end{proposition}

\begin{proof}
For the first Kostant component, take \(n=m-2=\ell-4\ge1\) in
\cite[Theorem~4.2(4)(a)]{ArakawaMoreau2018}.  The singularity level is
\[
  k=n-\ell+2=-2,
\]
and the exponent is \(n+1=m-1\), which gives \(s_{A,-2}\).

For the second component, take $n=0$ in
\cite[Theorem~4.2(4)(b)]{ArakawaMoreau2018}.  Its parameter range is
$n\in\mathbb Z_{\ge0}$, and this specialization gives the level
$k=n-2=-2$ and exponent $n+1=1$.  Hence $s_{D,-2}=\sigma(w_D)$ is an
affine singular vector.  For completeness, the endpoint mode identity
used in the proof of that theorem reads
\begin{equation*}
  f_\theta(1)\sigma(w_D)
  =
  (k+2)e_{\theta_D}(-1)\mathbf1.
\end{equation*}
At $k=-2$ its right-hand side vanishes, in agreement with the direct
criterion of \cite[Lemma~4.1]{ArakawaMoreau2018}.  Both states are
nonzero by the affine PBW theorem.
\end{proof}

Set
\begin{equation*}
  N_{A,-2}=U(\widehat{\mathfrak g})s_{A,-2},
  \qquad
  N_{D,-2}=U(\widehat{\mathfrak g})s_{D,-2},
  \qquad
  N_{-2}=N_{A,-2}+N_{D,-2},
\end{equation*}
and define
\begin{equation*}
  Q_{\ell,-2}
  =
  V^{-2}(D_\ell)/N_{-2}
  =
  V^{-2}(D_\ell)
  \Big/
  \left\langle
    \sigma(w_A)^{\ell-3},
    \sigma(w_D)
  \right\rangle.
\end{equation*}
The same triangular-decomposition argument as in
Lemma~\ref{dl:lem:Dl-singular-submodules-proper} shows that
\(N_{A,-2}\), \(N_{D,-2}\), and \(N_{-2}\) have strictly positive
lowest conformal degree.  They are therefore proper graded ideals and
are contained in the maximal proper graded ideal of
\(V^{-2}(D_\ell)\).

\begin{proposition}[Associated variety at the endpoint]
\label{dl:prop:Dl-minus-two-associated-variety}
One has
\begin{equation}
  X_{Q_{\ell,-2}}
  =
  \overline{\mathbb O}_{\min}.
  \label{dl:eq:Dl-minus-two-associated-variety}
\end{equation}
\end{proposition}

\begin{proof}
Under the standard identification
\[
  R_{V^{-2}(\mathfrak g)}
  =
  V^{-2}(\mathfrak g)/C_2(V^{-2}(\mathfrak g))
  \cong S(\mathfrak g),
\]
the quotient map induces a right-exact sequence
\begin{equation*}
  \frac{N_{-2}}
  {N_{-2}\cap C_2(V^{-2}(\mathfrak g))}
  \longrightarrow
  S(\mathfrak g)
  \longrightarrow
  R_{Q_{\ell,-2}}
  \longrightarrow0.
\end{equation*}
Let \(I_{-2}\) denote the image of the first arrow, equivalently the
kernel of the second.  In the affine PBW realization,
\[
  C_2(V^{-2}(\mathfrak g))
  =
  \mathfrak g[t^{-1}]t^{-2}V^{-2}(\mathfrak g).
\]
Zero modes generate the finite-dimensional \(\mathfrak g\)-orbits of
the two highest-weight symbols, \((-1)\)-modes become multiplication
in \(S(\mathfrak g)\), and all modes \(a(-r)\), \(r\ge2\), vanish in
the \(C_2\)-quotient.  Consequently,
\begin{equation*}
  I_{-2}
  =
  S(\mathfrak g)
  \left(
    U(\mathfrak g)w_A^{m-1}
    +
    U(\mathfrak g)w_D
  \right).
\end{equation*}
In particular, \(I_{-2}\) is a Poisson ideal stable under the adjoint
group \(G\).

Let \(J_W\subset S(\mathfrak g)\) be the ideal generated by
\(W=\mathcal K_A\oplus\mathcal K_D\).  Since
\(w_A^{m-1}\in J_W\) and \(w_D\in J_W\), one has
\begin{equation}
  I_{-2}\subset J_W.
  \label{dl:eq:Dl-minus-two-I-contained-JW}
\end{equation}
On the other hand, \(w_A^{m-1}\in I_{-2}\) implies
\(w_A\in\sqrt{I_{-2}}\).  Since \(I_{-2}\) is \(G\)-stable, its
radical is also \(G\)-stable; hence it contains the irreducible
\(G\)-module \(\mathcal K_A=U(\mathfrak g)w_A\).  Moreover,
\(w_D\in I_{-2}\) and the \(G\)-stability of \(I_{-2}\) imply
\(\mathcal K_D\subset I_{-2}\).  Therefore
\begin{equation}
  J_W\subset\sqrt{I_{-2}}.
  \label{dl:eq:Dl-minus-two-JW-contained-radical}
\end{equation}
Equations
\eqref{dl:eq:Dl-minus-two-I-contained-JW} and
\eqref{dl:eq:Dl-minus-two-JW-contained-radical} give
\[
  \sqrt{I_{-2}}=\sqrt{J_W}.
\]
By \cite[Lemma~2.1]{ArakawaMoreau2018},
\(\sqrt{J_W}\) is the prime defining ideal of
\(\overline{\mathbb O}_{\min}\).  This proves
\eqref{dl:eq:Dl-minus-two-associated-variety}.
\end{proof}

\subsubsection{Exact minimal reduction at level \(-2\)}

The affine highest weights of the endpoint singular vectors are
\begin{equation*}
\begin{aligned}
  \widehat\lambda_{A,-2}
  &=-2\Lambda_0+(m-1)(\theta+\theta_A)-2(m-1)\delta,\\
  \widehat\lambda_{D,-2}
  &=-2\Lambda_0+(\theta+\theta_D)-2\delta.
\end{aligned}
\end{equation*}
Their affine-wall coordinates are
\begin{equation}
\begin{aligned}
  \widehat\lambda_{A,-2}(\alpha_0^\vee)
  &=-2-2(m-1)=-2m<0,\\
  \widehat\lambda_{D,-2}(\alpha_0^\vee)
  &=-2-2=-4<0.
\end{aligned}
\label{dl:eq:Dl-minus-two-negative-coordinates}
\end{equation}

\begin{proposition}[Exact endpoint reductions]
\label{dl:prop:Dl-minus-two-exact-singular-reductions}
There exist nonzero scalars
\(c_{A,-2},c_{D,-2}\in\mathbb C^\times\) such that
\begin{equation*}
  [s_{A,-2}]_{\mathrm{DS}}
  =
  c_{A,-2}:J_A^{m-1}:,
  \qquad
  [s_{D,-2}]_{\mathrm{DS}}
  =
  c_{D,-2}J_D.
\end{equation*}
\end{proposition}

\begin{proof}
The proof of Lemma~\ref{dl:lem:Dl-reduced-singular-vectors-nonzero}
applies in \(\mathcal O_{-2}\): each singular submodule has an
irreducible highest-weight quotient, whose reduction is nonzero by
\eqref{dl:eq:Dl-minus-two-negative-coordinates}; exactness and the
reduction of affine Verma modules then show that the cyclic BRST
classes of \(s_{A,-2}\) and \(s_{D,-2}\) are nonzero.

Their natural weights and reduced conformal weights are
\begin{equation*}
\begin{aligned}
  \operatorname{wt}_{\mathfrak g^\natural}
  ([s_{A,-2}]_{\mathrm{DS}})
  &=(m-1)\theta_A,
  &\Delta_W([s_{A,-2}]_{\mathrm{DS}})&=m-1,\\
  \operatorname{wt}_{\mathfrak g^\natural}
  ([s_{D,-2}]_{\mathrm{DS}})
  &=\theta_D,
  &\Delta_W([s_{D,-2}]_{\mathrm{DS}})&=1.
\end{aligned}
\end{equation*}
The factorwise charge estimate in
Lemma~\ref{dl:lem:Dl-A1-extremal-PBW-line} is independent of the affine
level and gives
\begin{equation*}
  \mathcal W^{-2}(D_\ell,f_\theta)_{m-1}
  [(m-1)\theta_A]
  =
  \mathbb C:J_A^{m-1}:.
\end{equation*}
The conformal weight-one subspace of the universal minimal
\(W\)-algebra is the current space
\(J^{\{\mathfrak g^\natural\}}\), whose \(\theta_D\)-root space is
\(\mathbb CJ_D\).  Nonvanishing of both BRST classes proves the
claim.
\end{proof}

Define
\begin{equation*}
  \mathcal J_{\ell,-2}
  =
  \left\langle:J_A^{m-1}:,J_D\right\rangle
  \subset\mathcal W^{-2}(D_\ell,f_\theta)
\end{equation*}
and
\begin{equation*}
  \widetilde{\mathcal W}_{\ell,-2}
  =
  \mathcal W^{-2}(D_\ell,f_\theta)/\mathcal J_{\ell,-2}.
\end{equation*}

\begin{proposition}[Reduction of the endpoint quotient]
\label{dl:prop:Dl-minus-two-reduction-candidate}
There is a canonical isomorphism
\begin{equation}
  H^0_{\mathrm{DS},f_\theta}(Q_{\ell,-2})
  \cong
  \widetilde{\mathcal W}_{\ell,-2}.
  \label{dl:eq:Dl-minus-two-reduction-isomorphism}
\end{equation}
Moreover, \(\widetilde{\mathcal W}_{\ell,-2}\) is nonzero and lisse.
\end{proposition}

\begin{proof}
The endpoint singular submodules, their sum, and $Q_{\ell,-2}$ lie in
$\mathcal O_{-2}$ by the same PBW-grading argument as
Lemma~\ref{dl:lem:Dl-candidate-modules-in-category-O}.  Their reductions are
cyclic, and Proposition~\ref{dl:prop:Dl-minus-two-exact-singular-reductions}
identifies the cyclic generators.  Proposition~\ref{prop:common-exact-generated-quotient}
gives \eqref{dl:eq:Dl-minus-two-reduction-isomorphism}.  Finally,
Proposition~\ref{dl:prop:Dl-minus-two-associated-variety} and
\eqref{eq:common-associated-variety} show that the reduction is lisse; it is
nonzero because it maps onto the nonzero simple reduction of
$L_{-2}(D_\ell)$.
\end{proof}

\subsubsection{Collapse of the \(D_m\)-currents and the \(G\)-fields}

\begin{lemma}[Vanishing of the level-zero current factor]
\label{dl:lem:Dl-minus-two-D-currents-vanish}
In \(\widetilde{\mathcal W}_{\ell,-2}\),
\begin{equation*}
  J^{\{a\}}=0
  \qquad
  (a\in\mathfrak d).
\end{equation*}
\end{lemma}

\begin{proof}
The kernel of the quotient map, intersected with the weight-one
\(\mathfrak d\)-current space, is a \(\mathfrak d\)-submodule under
current zero modes because
\[
  J^{\{a\}}_{(0)}J^{\{b\}}=J^{\{[a,b]\}}.
\]
It contains the highest-root current
\(J_D=J^{\{e_{\theta_D}\}}\).  The adjoint representation of the
simple algebra \(\mathfrak d\) is generated by its highest-root
vector, so the kernel contains the whole current space.
\end{proof}

\begin{lemma}[Vanishing of the half-integer generators]
\label{dl:lem:Dl-minus-two-G-fields-vanish}
In \(\widetilde{\mathcal W}_{\ell,-2}\),
\begin{equation*}
  G^{\{u\}}=0
  \qquad
  (u\in U).
\end{equation*}
Consequently, the quotient is strongly generated by the
\(A_1\)-currents and \(\omega\).
\end{lemma}

\begin{proof}
The universal current--\(G\) relation is
\begin{equation*}
  [J^{\{a\}}{}_{\lambda}G^{\{u\}}]
  =
  G^{\{[a,u]\}},
  \qquad
  a\in\mathfrak g^\natural,
  \quad u\in U;
\end{equation*}
see \cite{AdamovicKacMosenederFrajriaPapiPerse2018}.
By Lemma~\ref{dl:lem:Dl-minus-two-D-currents-vanish}, its left-hand side
is zero for \(a\in\mathfrak d\).  Since
\(U\cong\mathbf2\boxtimes L_{\mathfrak d}(\eta_1)\) and the vector
representation is nontrivial and irreducible,
\([\mathfrak d,U]=U\).  Hence all \(G\)-generators vanish.  The strong
generation assertion follows from
\eqref{dl:eq:Dl-minimal-W-generators}.
\end{proof}

\subsubsection{Ramond trace identity and reduced simplicity}

Set
\begin{equation*}
  B_{\ell,-2}
  =
  A_{\mathrm R}(\widetilde{\mathcal W}_{\ell,-2}).
\end{equation*}
This is a nonzero finite-dimensional unital associative algebra by
Proposition~\ref{dl:prop:Dl-minus-two-reduction-candidate}.  Put
\(e_A=[J_A]_{\mathrm R}\).  The defining current relation gives
\begin{equation*}
  e_A^{m-1}=e_A^{K_0+1}=0.
\end{equation*}
Thus every irreducible \(A_1\)-constituent of a finite-dimensional
\(B_{\ell,-2}\)-module has highest weight \(a\varpi\) with
\(0\le a\le K_0\), and
\begin{equation}
  c_2(a\varpi)
  \le
  c_2(K_0\varpi)
  =
  \frac{K_0(K_0+2)}2
  =
  \frac{m(m-2)}2.
  \label{dl:eq:Dl-minus-two-A-Casimir-bound}
\end{equation}

\begin{proposition}[Endpoint Casimir trace identity]
\label{dl:prop:Dl-minus-two-Casimir-trace}
Let \(M\) be a nonzero finite-dimensional
\(B_{\ell,-2}\)-module on which
\(L=[\omega]_{\mathrm R}\) acts by \(h\).  Then
\begin{equation*}
  \frac{\operatorname{tr}_M\rho(\Omega_A)}{\dim M}
  =
  2mh.
\end{equation*}
\end{proposition}

\begin{proof}
At \(k=-2\), one has
\[
  k+h^\vee(D_\ell)=2m,
  \qquad
  p_{D_\ell}(-2)=0.
\]
The images of the \(G\)-generators and the \(D_m\)-currents vanish by
Lemmas~\ref{dl:lem:Dl-minus-two-D-currents-vanish} and
\ref{dl:lem:Dl-minus-two-G-fields-vanish}.  Hence the Ramond relation
\eqref{eq:common-Ramond-Zhu} becomes
\begin{equation}
  0
  =
  \langle u,v\rangle(\Omega_A-4mL)
  +Q_{AA}(u,v).
  \label{dl:eq:Dl-minus-two-reduced-Ramond-relation}
\end{equation}
Let \(\{a_r\}\) be an orthonormal basis of \(\mathfrak s\), and
write \(\tau_A=\operatorname{tr}_M\rho(\Omega_A)\).  Invariance of
the trace form gives
\[
  \operatorname{tr}_M(\rho(a)\rho(b))
  =
  I_A(a\mid b),
  \qquad
  I_A=\frac{\tau_A}{3}.
\]
The contraction calculation of
Lemma~\ref{dl:lem:Dl-trace-quadratic-Ramond-term}, now applied directly to
this module, yields
\begin{align*}
  \operatorname{tr}_M\rho(Q_{AA}(u,v))
  &=
  2I_A\sum_r
  \langle[a_r,u],[v,a_r]\rangle\\
  &=
  2I_A\frac32\langle u,v\rangle
  =
  \tau_A\langle u,v\rangle.
\end{align*}
Taking the trace of
\eqref{dl:eq:Dl-minus-two-reduced-Ramond-relation} therefore gives
\[
  0
  =
  \langle u,v\rangle
  \left(
    2\tau_A-4mh\dim M
  \right).
\]
Choose \(u,v\) with \(\langle u,v\rangle\ne0\) and divide by
\(2\dim M\).
\end{proof}

Use the same coweight \(x=\varpi\) as in the level \(-1\) argument.  Since the
\(A_1\)-current level is now \(K_0\), Li twisting gives
\begin{equation*}
  L_0^{\mathrm R,-2}
  =
  L_0+x_0+\frac{K_0}{4}.
\end{equation*}
The surviving strong generators have excess energies
\begin{equation}
\begin{array}{c|cccc}
  \text{generator}&E_A&H_A&F_A&\omega\\ \hline
  \Delta+x_0&2&1&0&2.
\end{array}
\label{dl:eq:Dl-minus-two-excess-table}
\end{equation}
Derivatives add positive integers.  Every ordinary conformal
homogeneous subspace of \(\widetilde{\mathcal W}_{\ell,-2}\) is
finite-dimensional and stable under the \(A_1\)-zero modes.  Since
finite-dimensional \(A_1\)-modules are completely reducible, \(x_0\)
acts semisimply on every homogeneous subspace.  Moreover, \(x_0\)
commutes with \(L_0\).  Hence
\begin{equation*}
  L_0^{\mathrm R,-2}
  =
  L_0+x_0+\frac{K_0}{4}
  \quad\text{acts semisimply on the twisted adjoint module.}
\end{equation*}
The PBW spanning theorem together with
\eqref{dl:eq:Dl-minus-two-excess-table} then places its spectrum in
\(K_0/4+\mathbb Z_{\ge0}\).

\begin{proposition}[Endpoint lowest Ramond space]
\label{dl:prop:Dl-minus-two-lowest-space}
The spectrum of \(L_0^{\mathrm R,-2}\) is contained in
\(K_0/4+\mathbb Z_{\ge0}\), and the lowest eigenspace is
\begin{equation*}
  T_{\ell,-2}
  =
  \operatorname{span}_{\mathbb C}
  \left\{
    (F_A)_{(-1)}^j\mathbf1
    \;\middle|\;
    0\le j\le K_0
  \right\}
  \cong
  L_{\mathfrak s}(K_0\varpi).
\end{equation*}
In particular, \(T_{\ell,-2}\) is irreducible.
\end{proposition}

\begin{proof}
The excess table shows that only the undifferentiated state \(F_A\)
has excess zero, so the lowest eigenspace is spanned by its
\((-1)\)-powers.  Since \(m-1=K_0+1\), the relation
\(:J_A^{m-1}:=0\) says
\[
  (E_A)_{(-1)}^{K_0+1}\mathbf1=0.
\]
The current homomorphism
\(V^{K_0}(A_1)\to\widetilde{\mathcal W}_{\ell,-2}\) therefore factors
through \(L_{K_0}(A_1)\), since the maximal ideal of the positive
integral-level universal affine algebra is generated by
\((E_A)_{(-1)}^{K_0+1}\mathbf1\); see \cite{Kac1998}.  The induced map is
injective because it preserves the vacuum and \(L_{K_0}(A_1)\) is simple.
Consequently the standard integrable \(A_1^{(1)}\)-string relations give
\((F_A)_{(-1)}^{K_0+1}\mathbf1=0\), while the powers with
\(0\le j\le K_0\) are nonzero.  Under the twisted zero modes the
vacuum has highest weight \(K_0\varpi\), giving the stated
irreducible module.
\end{proof}

\begin{theorem}[Simplicity of the endpoint reduction]
\label{dl:thm:Dl-minus-two-reduced-simple}
For every \(\ell\ge5\), the canonical map is an isomorphism
\begin{equation*}
 \widetilde{\mathcal W}_{\ell,-2}
 \cong
 \mathcal W_{-2}(D_\ell,f_\theta)
 \cong
 L_{\ell-4}(A_1).
\end{equation*}
In particular, \(\widetilde{\mathcal W}_{\ell,-2}\) is simple.
\end{theorem}

\begin{proof}
The natural map
\(
 \widetilde{\mathcal W}_{\ell,-2}\twoheadrightarrow
 \mathcal W_{-2}(D_\ell,f_\theta)
\)
is nonzero because
\(( -2\Lambda_0)(\alpha_0^\vee)=-2<0\).
Apply Proposition~\ref{prop:common-Casimir-gap-simplicity} with
\begin{equation*}
 h_0=\frac{K_0}{4}=\frac{m-2}{4},
 \qquad
 \mathfrak a=\mathfrak s,
 \qquad
 \mathcal C=\Omega_A,
 \qquad
 A(h)=2mh.
\end{equation*}
The finite-dimensionality of the Ramond Zhu algebra follows from
Proposition~\ref{dl:prop:Dl-minus-two-reduction-candidate}; the trace
identity is Proposition~\ref{dl:prop:Dl-minus-two-Casimir-trace}; and
\eqref{dl:eq:Dl-minus-two-A-Casimir-bound} gives
\begin{equation*}
 C_{\max}=\frac{m(m-2)}2=A(h_0),
 \qquad
 A(h_0+1)=C_{\max}+2m>C_{\max}.
\end{equation*}
Proposition~\ref{dl:prop:Dl-minus-two-lowest-space} identifies the
lowest Ramond space with the irreducible module
\(L_{\mathfrak s}(K_0\varpi)\).  Thus every nonzero ground submodule and
every simple quotient contains its highest-weight line, which is the vacuum
line.  The common criterion therefore proves that the canonical map is an
isomorphism.  The final identification with \(L_{\ell-4}(A_1)\) is
the type-\(D\) collapsing-level theorem
\cite[Theorem~3.3 and Proposition~3.4]{AdamovicKacMosenederFrajriaPapiPerse2018}.
\end{proof}

\subsubsection{Uniform affine lifting at the endpoint}

The endpoint quotient admits a canonical surjection
\begin{equation*}
 q_{\ell,-2}:Q_{\ell,-2}\twoheadrightarrow L_{-2}(D_\ell).
\end{equation*}
Theorem~\ref{dl:thm:Dl-minus-two-reduced-simple} identifies its minimal
reduction with the nonzero simple algebra
$\mathcal W_{-2}(D_\ell,f_\theta)$.  Proposition~\ref{prop:common-affine-lifting}
therefore applies.

\begin{theorem}[The affine maximal ideal at level \(-2\)]
\label{dl:thm:Dl-minus-two-affine-maximal-ideal}
For every \(\ell\ge5\),
\begin{equation*}
  Q_{\ell,-2}\cong L_{-2}(D_\ell).
\end{equation*}
Equivalently,
\begin{equation*}
  \ker\!\left(
    V^{-2}(D_\ell)\longrightarrow L_{-2}(D_\ell)
  \right)
  =
  \left\langle
    \sigma(w_A)^{\ell-3},
    \sigma(w_D)
  \right\rangle.
\end{equation*}
\end{theorem}

\begin{proof}[Alternative proof]
Apply Proposition~\ref{prop:common-affine-lifting} to
$q_{\ell,-2}:Q_{\ell,-2}\twoheadrightarrow L_{-2}(D_\ell)$.
\end{proof}

\begin{theorem}[The two negative levels of type \(D_\ell\)]

Let \(\ell\ge5\) and \(n\in\{0,1\}\).  With the uniform parameters
\eqref{dl:eq:Dl-uniform-level-parameters},
\begin{equation*}
 \ker\!\left(
   V^{k_n}(D_\ell)\longrightarrow L_{k_n}(D_\ell)
 \right)
 =
 \left\langle
   \sigma(w_A)^{K_n+1},
   \sigma(w_D)^{n+1}
 \right\rangle.
\end{equation*}
Equivalently, the Arakawa--Moreau maximal-ideal conjecture holds for
all negative levels in the series \(D_\ell\), \(\ell\ge5\).
\end{theorem}

\begin{proof}
	For $n=0$, this is the previously known result
	\cite[Corollary~6.6(1)]
	{AdamovicKacMosenederFrajriaPapiPerse2020}, with the alternative proof
	given in Theorem~\ref{dl:thm:Dl-minus-two-affine-maximal-ideal}.
	For $n=1$, it is the new result
	Theorem~\ref{dl:thm:Dl-main-maximal-ideal}.  Formula
	\eqref{dl:eq:Dl-uniform-singular-generators} identifies these with the two
	generator families prescribed in
	\cite[Sections~4--5]{ArakawaMoreau2018}.
\end{proof}
\section{The triality case \texorpdfstring{$D_4$}{D4}}
\label{sec:case-d4}

\subsection{Triality data}
\label{d4:sec:D4-preliminaries}

Let $\mathfrak g=D_4\cong\mathfrak{so}_8$ with the Bourbaki realization of
Appendix~\ref{d4:app:D4-root-weight-Casimir-data}.  Its highest root is
\begin{equation}
 \theta=\varepsilon_1+\varepsilon_2
 =\alpha_1+2\alpha_2+\alpha_3+\alpha_4.
 \label{d4:eq:D4-highest-root}
\end{equation}
The outer nodes
\begin{equation}
 \mathcal A=\{1,3,4\}
 \label{d4:eq:D4-outer-node-index-set}
\end{equation}
are permuted by $\operatorname{Out}(D_4)\cong\mathfrak S_3$.  For the
minimal grading associated with $f_\theta=e_{-\theta}$ one has
\begin{equation*}
 \mathfrak g^\natural
 =\mathfrak s_1\oplus\mathfrak s_3\oplus\mathfrak s_4,
 \qquad \mathfrak s_a\cong A_1,
 \qquad
 U:=\mathfrak g_{-\frac12}
 \cong\boxtimes_{a\in\mathcal A}L_{\mathfrak s_a}(\varpi_a),
\end{equation*}
where $\varpi_a$ is the fundamental weight of the $a$th factor.  The three
outer weights restrict as
\begin{equation}
 \omega_a|_{\mathfrak h^\natural}=\varpi_a,
 \qquad a\in\mathcal A.
 \label{d4:eq:D4-outer-weight-restrictions}
\end{equation}
For $u,v\in U$ set
\begin{equation}
 \langle u,v\rangle=(e_\theta\mid[u,v]);
 \label{d4:eq:D4-symplectic-form}
\end{equation}
this is the invariant symplectic form used in the Ramond relation.  With our
normalization,
\begin{equation}
 c_2(r\varpi_a)=\frac{r(r+2)}2,
 \qquad c_2(\varpi_a)=\frac32.
 \label{d4:eq:D4-A1-Casimir-formula}
\end{equation}
The explicit roots, the eight weights of $U$, and all Casimir tables appear
in Appendix~\ref{d4:app:D4-root-weight-Casimir-data}.  These are the standard
minimal-grading data of \cite{KacRoanWakimoto2003,KacWakimoto2004,
AdamovicKacMosenederFrajriaPapiPerse2018}.

\subsubsection{Level, Ramond, and spectral-flow data}

All formal conventions are those of Section~\ref{sec:common-framework}.  At
$k=-1$, each $A_1$ current factor has level one, so
\begin{equation}
 [J^{\{a\}}{}_{\lambda}J^{\{b\}}]
 =J^{\{[a,b]\}}+\delta_{r,s}(a\mid b)\lambda,
 \qquad a\in\mathfrak s_r,\ b\in\mathfrak s_s.
 \label{d4:eq:D4-A1-current-relation-level-minus-one}
\end{equation}
The minimal $W$-algebra has the generators
\begin{equation}
 J^{\{a\}},\ a\in\mathfrak g^\natural;
 \qquad G^{\{u\}},\ u\in U;
 \qquad \omega.
 \label{d4:eq:D4-minimal-W-generators}
\end{equation}
Since $p_{D_4}(k)=(k+2)^2$, the Ramond relation at $k=-1$ is
\begin{equation*}
 [u,v]=\langle u,v\rangle(\Omega^\natural-10L-\tfrac12)+Q(u,v).
\end{equation*}

\subsubsection{Li twist}
\label{d4:subsec:D4-Li-delta-operator}
Take
\begin{equation}
 x=\varpi_1^\vee,\qquad H=J^{\{x\}}.
 \label{d4:eq:D4-Li-choice-x}
\end{equation}
Then
\begin{align}
 L_0^{\mathrm R}&=L_0+x_0+\frac14,
 \label{d4:eq:D4-Li-Ramond-Hamiltonian}\\
 J_0^{\{h\},\mathrm R}&=J_0^{\{h\}}+(x\mid h),
 \label{d4:eq:D4-Li-Cartan-shift}\\
 J_0^{\{e_\gamma\},\mathrm R}&=J_{\gamma(x)}^{\{e_\gamma\}},
 \label{d4:eq:D4-Li-root-mode-shift}
\end{align}
with
\begin{align}
 \gamma(x)&\in\{-1,0,1\},
 \label{d4:eq:D4-Li-current-charges}\\
 \mu(x)&\in\{-\tfrac12,\tfrac12\}.
 \label{d4:eq:D4-Li-half-space-charges}
\end{align}

\subsection{Triality singular vectors and their minimal reductions}
\label{d4:sec:D4-triality-singular-vectors}

In this section we identify the exact minimal reductions of the three
triality-related affine singular vectors at level \(-1\).  We retain the
outer-node index set
\(\mathcal A=\{1,3,4\}\) from
\eqref{d4:eq:D4-outer-node-index-set}.

\subsubsection{The three distinguished singular vectors}

Let \(\Omega\in S^2(\mathfrak g)\) be the quadratic Casimir.  The
Kostant decomposition for \(D_4\) takes the form
\begin{equation*}
  S^2(\mathfrak g)
  =L_{\mathfrak g}(2\theta)
   \oplus\mathbb C\Omega
   \oplus\mathcal K,
  \qquad
  \mathcal K
  =\bigoplus_{a\in\mathcal A}\mathcal K_a,
  \qquad
  \mathcal K_a\cong L_{\mathfrak g}(2\omega_a).
\end{equation*}
The three summands \(\mathcal K_a\) are permuted transitively by
triality; see \cite[Section~2]{ArakawaMoreau2018}.  Choose a highest-weight vector
\begin{equation*}
  0\ne w_a\in\mathcal K_a,
  \qquad
  \operatorname{wt}_{\mathfrak g}(w_a)=2\omega_a,
  \qquad a\in\mathcal A.
\end{equation*}
We use the degree-two symmetrization map
\begin{equation*}
  \sigma:S^2(\mathfrak g)\longrightarrow V^k(\mathfrak g)_2,
  \qquad
  \sigma(xy)
  =\frac12\bigl(x_{(-1)}y_{(-1)}+y_{(-1)}x_{(-1)}\bigr)\mathbf1.
\end{equation*}
Following the convention of \cite[Section~4]{ArakawaMoreau2018}, we use
throughout this subsection the PBW vector-space identification
\begin{equation*}
  U\bigl(\mathfrak g[t^{-1}]t^{-1}\bigr)
  \xrightarrow{\sim}V^{-1}(\mathfrak g),
  \qquad u\longmapsto u\mathbf1.
\end{equation*}
Let \(\widetilde\sigma(w_a)\) be the element of the enveloping algebra
corresponding to the state \(\sigma(w_a)\).  For \(a\in\mathcal A\), define
\begin{equation*}
  s_a=\widetilde\sigma(w_a)^2\mathbf1
      \in V^{-1}(\mathfrak g)_4,
\end{equation*}
where the square is the associative product in
\(U(\mathfrak g[t^{-1}]t^{-1})\).  As in
\cite[Theorem~4.2(1)(b)]{ArakawaMoreau2018}, we abbreviate this PBW state
by \(s_a=\sigma(w_a)^2\).  It is not being defined here as the normally
ordered product \(\sigma(w_a)_{(-1)}\sigma(w_a)\).
Thus
\begin{equation*}
  \deg s_a=4,
  \qquad
  \operatorname{wt}_{\mathfrak g}(s_a)=4\omega_a.
\end{equation*}

By \cite[Section~4 and Theorem~4.2(1)(b)]{ArakawaMoreau2018}, where
\(V^k(\mathfrak g)\) is identified with
\(U(\mathfrak g[t^{-1}]t^{-1})\), the associative PBW power
\(\sigma(w_a)^{n+1}\) is singular precisely at
\begin{equation*}
  k=n-\frac{h^\vee}{6}-1.
\end{equation*}
For \(D_4\), one has \(h^\vee=6\).  Taking \(n=1\) gives \(k=-1\), and
therefore
\begin{equation*}
  s_a\text{ is a nonzero affine singular vector in }
  V^{-1}(\mathfrak g),
  \qquad a\in\mathcal A.
\end{equation*}
The nonvanishing also follows directly from the affine PBW theorem.  Set
\begin{equation*}
  N_a=U(\widehat{\mathfrak g})s_a
  \subset V^{-1}(\mathfrak g),
  \qquad a\in\mathcal A,
\end{equation*}
where
\begin{equation*}
  \widehat{\mathfrak g}
  =\mathfrak g[t,t^{-1}]\oplus\mathbb CK\oplus\mathbb CD.
\end{equation*}
Since the affine currents strongly generate the vacuum algebra, every
\(N_a\) is a graded vertex-algebra ideal as well as an affine submodule.

\subsubsection{Nonvanishing under minimal reduction}

The affine highest weight of \(s_a\), including its external
\(D\)-grading, is
\begin{equation*}
  \widehat\lambda_a
  =-\Lambda_0+4\omega_a-4\delta.
\end{equation*}
The \(\delta\)-term has no effect on evaluation at the affine simple
coroot.  Since the coefficient of every outer simple root in
\eqref{d4:eq:D4-highest-root} is one,
\begin{equation}
  \langle\omega_a,\theta^\vee\rangle=1,
  \qquad a\in\mathcal A.
  \label{d4:eq:D4-outer-weight-highest-coroot-pairing}
\end{equation}
Using \eqref{eq:common-affine-simple-coroot}, we obtain
\begin{align}
  \widehat\lambda_a(\alpha_0^\vee)
  &=-1-\langle4\omega_a,\theta^\vee\rangle
    \notag\\
  &=-5<0.
  \label{d4:eq:D4-singular-highest-weight-negative-wall}
\end{align}

\begin{lemma}
\label{d4:lem:D4-reduced-singular-vectors-nonzero}
For every \(a\in\mathcal A\), the BRST class of \(s_a\) is nonzero:
\begin{equation*}
  [s_a]_{\mathrm{DS}}\ne0
  \quad\text{in}\quad
  \mathcal W^{-1}(\mathfrak g,f_\theta).
\end{equation*}
\end{lemma}

\begin{proof}
The highest-weight module \(N_a\) has an irreducible highest-weight
quotient
\begin{equation*}
  N_a\twoheadrightarrow L(\widehat\lambda_a).
\end{equation*}
By \eqref{d4:eq:D4-singular-highest-weight-negative-wall} and the
nonvanishing criterion \eqref{eq:common-DS-nonvanishing},
\begin{equation*}
  H^0_{\mathrm{DS},f_\theta}
  \bigl(L(\widehat\lambda_a)\bigr)\ne0.
\end{equation*}
Exactness on \(\mathcal O_{-1}\) therefore gives a surjection
\begin{equation*}
  H^0_{\mathrm{DS},f_\theta}(N_a)
  \twoheadrightarrow
  H^0_{\mathrm{DS},f_\theta}
  \bigl(L(\widehat\lambda_a)\bigr),
\end{equation*}
so the source is nonzero.

To identify its cyclic vector, shift the scalar action of \(D\) so that
\(s_a\) has \(D\)-value zero, and denote the resulting highest weight by
\(\widehat\lambda_a^\circ\).  There is a canonical surjection
\begin{equation}
  M(\widehat\lambda_a^\circ)\twoheadrightarrow N_a^\circ
  \label{d4:eq:D4-Verma-to-individual-singular-submodule}
\end{equation}
from the affine Verma module.  By
\cite[Theorem~6.3]{KacWakimoto2004}, its reduction is a \(W\)-Verma
module generated by the class of the affine highest-weight vector.
Exactness applied to
\eqref{d4:eq:D4-Verma-to-individual-singular-submodule} shows that
\begin{equation}
  H^0_{\mathrm{DS},f_\theta}(N_a)
  \text{ is generated by }[s_a]_{\mathrm{DS}}.
  \label{d4:eq:D4-individual-reduction-cyclic}
\end{equation}
Since this cyclic module is nonzero, its cyclic generator is nonzero.
Finally, exactness applied to the inclusion
\(N_a\hookrightarrow V^{-1}(\mathfrak g)\) gives an injection
\begin{equation*}
  H^0_{\mathrm{DS},f_\theta}(N_a)
  \hookrightarrow
  \mathcal W^{-1}(\mathfrak g,f_\theta),
\end{equation*}
which proves the assertion in the universal minimal \(W\)-algebra.
\end{proof}

\subsubsection{The extremal PBW lines}

For \(a\in\mathcal A\), choose a nonzero highest-root vector
\begin{equation}
  e_a=e_{\beta_a}\in(\mathfrak s_a)_{\beta_a},
  \qquad
  J_a=J^{\{e_a\}}.
  \label{d4:eq:D4-highest-root-currents}
\end{equation}
For a \(\mathfrak g^\natural\)-weight \(\mu\), put
\begin{equation*}
  q_a(\mu)=\langle\mu,\beta_a^\vee\rangle.
\end{equation*}
In particular,
\begin{equation*}
  q_a(4\varpi_a)=4.
\end{equation*}

\begin{lemma}[Extremal PBW lines]
\label{d4:lem:D4-extremal-PBW-lines}
For every \(a\in\mathcal A\),
\begin{equation*}
  \mathcal W^{-1}(\mathfrak g,f_\theta)_2[4\varpi_a]
  =\mathbb C\bigl(J_a{}_{(-1)}\bigr)^2\mathbf1
  =\mathbb C:J_a^2:.
\end{equation*}
\end{lemma}

\begin{proof}
Fix \(a\in\mathcal A\) and apply
Proposition~\ref{prop:common-extremal-PBW-line} to the corresponding
\(A_1\)-ideal \(\mathfrak s_a\), whose highest root is
\(\beta_a=2\varpi_a\).  Currents in the other two factors have zero
\(q_a\)-charge, while the adjoint weights of \(\mathfrak s_a\) have charge
at most two, with equality only on \(\mathbb Ce_a\).  Every weight of
\(U\cong\mathbf2\boxtimes\mathbf2\boxtimes\mathbf2\) has
\(q_a\)-charge at most one.  Taking \(r=2\) gives
\(r\beta_a=4\varpi_a\) and proves the stated line.
\end{proof}

\subsubsection{Exact reductions of the singular vectors}

We next compute the simultaneous weights of the reduced classes.  The
restriction formula \eqref{d4:eq:D4-outer-weight-restrictions} gives
\begin{equation}
  \operatorname{wt}_{\mathfrak g^\natural}
  \bigl([s_a]_{\mathrm{DS}}\bigr)
  =4\varpi_a.
  \label{d4:eq:D4-reduced-singular-natural-weight}
\end{equation}
Indeed, the BRST representative is the affine highest-weight vector
tensored with the ghost vacuum, and the latter has
\(\mathfrak h^\natural\)-weight zero.

For the conformal weight, the minimal-reduction Hamiltonian acts on an
affine highest-weight representative of finite weight \(\mu\) and
affine degree \(d\) by \(d-\mu(x_\theta)\).  Since
\eqref{d4:eq:D4-outer-weight-highest-coroot-pairing} implies
\begin{equation*}
  \omega_a(x_\theta)
  =\frac12\langle\omega_a,\theta^\vee\rangle
  =\frac12,
\end{equation*}
we obtain
\begin{align}
  \Delta_W\bigl([s_a]_{\mathrm{DS}}\bigr)
  &=4-(4\omega_a)(x_\theta)
    \notag\\
  &=2.
  \label{d4:eq:D4-reduced-singular-conformal-weight}
\end{align}

\begin{proposition}[Exact reduced singular vectors]
\label{d4:prop:D4-exact-reduced-singular-vectors}
For every \(a\in\mathcal A\), there exists a nonzero scalar
\(c_a\in\mathbb C^\times\) such that
\begin{equation}
  [s_a]_{\mathrm{DS}}
  =c_a\bigl(J_a{}_{(-1)}\bigr)^2\mathbf1
  =c_a:J_a^2:.
  \label{d4:eq:D4-exact-reduced-singular-vectors}
\end{equation}
\end{proposition}

\begin{proof}
By Lemma~\ref{d4:lem:D4-reduced-singular-vectors-nonzero}, the left-hand
side is nonzero.  Equations
\eqref{d4:eq:D4-reduced-singular-natural-weight} and
\eqref{d4:eq:D4-reduced-singular-conformal-weight} place it in the
one-dimensional simultaneous weight space of
Lemma~\ref{d4:lem:D4-extremal-PBW-lines}.  It is therefore a scalar multiple
of \(:J_a^2:\), and the scalar is nonzero.
\end{proof}

The triality group permutes the three identities in
\eqref{d4:eq:D4-exact-reduced-singular-vectors}.  Thus one may normalize the
choices of \(w_a\) and \(e_a\) compatibly with triality, although no
particular value of the nonzero constants \(c_a\) will be needed.

\subsubsection{The reduced triality quotient}

Set
\begin{equation*}
  N=N_1+N_3+N_4,
  \qquad
  Q=V^{-1}(\mathfrak g)/N.
\end{equation*}
Equivalently,
\begin{equation*}
  Q
  =V^{-1}(D_4)
   \Big/
   \left\langle
     \sigma(w_1)^2,
     \sigma(w_3)^2,
     \sigma(w_4)^2
   \right\rangle.
\end{equation*}
Define the ideal
\begin{equation*}
  \mathcal J
  =\left\langle
     :J_1^2:,
     :J_3^2:,
     :J_4^2:
   \right\rangle
  \subset\mathcal W^{-1}(\mathfrak g,f_\theta)
\end{equation*}
and the quotient
\begin{equation*}
  \widetilde{\mathcal W}
  =\mathcal W^{-1}(\mathfrak g,f_\theta)/\mathcal J.
\end{equation*}

\begin{proposition}[Exact reduction of the triality quotient]
\label{d4:prop:D4-exact-reduction-triality-quotient}
There is a canonical isomorphism of vertex algebras
\begin{equation}
  H^0_{\mathrm{DS},f_\theta}(Q)
  \cong\widetilde{\mathcal W}.
  \label{d4:eq:D4-exact-reduction-Q}
\end{equation}
Moreover,
\begin{equation}
  \widetilde{\mathcal W}\ne0,
  \qquad
  \widetilde{\mathcal W}\text{ is lisse}.
  \label{d4:eq:D4-Wtilde-nonzero-lisse}
\end{equation}
\end{proposition}

\begin{proof}
The three singular submodules lie in the exactness category, their reductions
are cyclic by \eqref{d4:eq:D4-individual-reduction-cyclic}, and their cyclic
generators are identified in
Proposition~\ref{d4:prop:D4-exact-reduced-singular-vectors}.  Proposition~\ref{prop:common-exact-generated-quotient}
therefore gives \eqref{d4:eq:D4-exact-reduction-Q}.  The associated variety is
\begin{equation*}
 X_Q=\overline{\mathbb O}_{\min}
\end{equation*}
by \cite[Proposition~5.2]{ArakawaMoreau2018}; now
\eqref{eq:common-associated-variety} and
\cite[Theorem~6.1(3)]{ArakawaMoreau2018} give nonvanishing and lisse-ness.
\end{proof}

Since the three singular vectors lie in the maximal proper ideal of
\(V^{-1}(\mathfrak g)\), the canonical map to the simple affine vertex
algebra factors through a surjection
\begin{equation*}
  Q\twoheadrightarrow L_{-1}(\mathfrak g).
\end{equation*}
Exactness of reduction and the nonvanishing statement in
\eqref{eq:common-DS-nonvanishing} consequently give a natural
surjective homomorphism
\begin{equation}
  \pi:\widetilde{\mathcal W}
  \twoheadrightarrow
  \mathcal W_{-1}(\mathfrak g,f_\theta).
  \label{d4:eq:D4-Wtilde-surjects-simple-W}
\end{equation}
The next two sections prove that \(\pi\) is an isomorphism by combining
the Ramond Zhu algebra with Li spectral flow.

\subsection{The Ramond Zhu algebra and a Casimir trace identity}
\label{d4:sec:D4-Ramond-Casimir}

We now pass from the lisse quotient
\(\widetilde{\mathcal W}\) constructed in
Section~\ref{d4:sec:D4-triality-singular-vectors} to a finite-dimensional
associative algebra.  The three triality-related current relations impose
independent level-one bounds on the three \(A_1\)-factors, while the
Ramond relation yields a uniform trace identity for their total Casimir.
These two facts will be combined with Li spectral flow in the next
section.

\subsubsection{The Ramond Zhu algebra and the three square-zero relations}

Set
\begin{equation*}
  B=A_{\mathrm R}(\widetilde{\mathcal W}).
\end{equation*}
By \eqref{d4:eq:D4-Wtilde-nonzero-lisse} and
\eqref{eq:common-lisse-Ramond-finite}, \(B\) is a finite-dimensional
unital associative algebra:
\begin{equation}
  0<\dim B<\infty.
  \label{d4:eq:D4-Ramond-Zhu-B-finite}
\end{equation}
For \(a\in\mathcal A\), let
\begin{equation*}
  e_a=[J_a]_{\mathrm R}\in B,
\end{equation*}
where \(J_a=J^{\{e_a\}}\) is the highest-root current introduced in
\eqref{d4:eq:D4-highest-root-currents}.  The harmless repetition of the
symbol \(e_a\) identifies the highest-root vector of \(\mathfrak s_a\)
with its image in the Ramond Zhu algebra.

For every \(a\in\mathcal A\), the level-one current relation
\eqref{d4:eq:D4-A1-current-relation-level-minus-one} gives
\begin{equation*}
  [J_a{}_{\lambda}J_a]
  =J^{\{[e_a,e_a]\}}+(e_a\mid e_a)\lambda=0.
\end{equation*}
Indeed, \([e_a,e_a]=0\), and the invariant form pairs the root space
\((\mathfrak s_a)_{\beta_a}\) only with
\((\mathfrak s_a)_{-\beta_a}\).  Thus
\begin{equation}
  (J_a)_{(m)}J_a=0,
  \qquad m\in\mathbb Z_{\ge0}.
  \label{d4:eq:D4-highest-root-current-products-zero}
\end{equation}

\begin{lemma}
\label{d4:lem:D4-Zhu-image-current-square}
For every \(a\in\mathcal A\),
\begin{equation*}
  \bigl[:J_a^2:\bigr]_{\mathrm R}=e_a^2
\end{equation*}
in the Ramond Zhu algebra of every conformal quotient of
\(\mathcal W^{-1}(\mathfrak g,f_\theta)\).
\end{lemma}

\begin{proof}
The Ramond automorphism fixes \(J_a\), and \(J_a\) has conformal weight
one.  Hence its Ramond Zhu product with itself is
\begin{equation*}
  [J_a]_{\mathrm R}*[J_a]_{\mathrm R}
  =\bigl[(J_a)_{(-1)}J_a\bigr]_{\mathrm R}
   +\bigl[(J_a)_{(0)}J_a\bigr]_{\mathrm R}.
\end{equation*}
The second term vanishes by
\eqref{d4:eq:D4-highest-root-current-products-zero}, while the first is
\(\bigl[:J_a^2:\bigr]_{\mathrm R}\).
\end{proof}

Since the defining ideal \(\mathcal J\) of
\(\widetilde{\mathcal W}\) contains all three states \(:J_a^2:\),
Lemma~\ref{d4:lem:D4-Zhu-image-current-square} implies
\begin{equation}
  e_a^2=0
  \qquad\text{in }B,
  \qquad a\in\mathcal A.
  \label{d4:eq:D4-three-highest-root-square-zero-relations}
\end{equation}

The current zero modes give a Lie algebra homomorphism
\(\mathfrak g^\natural\to B\).  Therefore every finite-dimensional
\(B\)-module is a finite-dimensional
\(\mathfrak g^\natural\)-module and is completely reducible over
\(\mathfrak g^\natural\).

\begin{lemma}[Allowed \(A_1^3\)-types]
\label{d4:lem:D4-allowed-A1-cubed-types}
Let \(M\) be a finite-dimensional \(B\)-module.  Every irreducible
\(\mathfrak g^\natural\)-constituent of \(M\) has the form
\begin{equation}
  L_{\mathfrak g^\natural}(\lambda)
  =L_{\mathfrak s_1}(m_1\varpi_1)
   \boxtimes
   L_{\mathfrak s_3}(m_3\varpi_3)
   \boxtimes
   L_{\mathfrak s_4}(m_4\varpi_4),
  \qquad
  m_a\in\{0,1\}.
  \label{d4:eq:D4-allowed-A1-cubed-types}
\end{equation}
Consequently, the total Casimir \(\Omega^\natural\) acts on every such
constituent by
\begin{equation}
  c_2^\natural(\lambda)
  =\frac32(m_1+m_3+m_4)
  \le\frac92.
  \label{d4:eq:D4-total-Casimir-bound-allowed-types}
\end{equation}
\end{lemma}

\begin{proof}
Fix \(a\in\mathcal A\).  Suppose that the \(\mathfrak s_a\)-highest
weight of an irreducible constituent is \(m_a\varpi_a\).  In the
standard \(\mathfrak{sl}_2\)-string generated by a highest-weight vector
\(v\), one has
\begin{equation*}
  e_a^r f_a^{m_a}v
  =c_{m_a,r}f_a^{m_a-r}v,
  \qquad
  0\le r\le m_a,
  \qquad
  c_{m_a,r}\ne0.
\end{equation*}
If \(m_a\ge2\), then \(e_a^2f_a^{m_a}v\ne0\), contradicting
\eqref{d4:eq:D4-three-highest-root-square-zero-relations}.  Hence
\(m_a\le1\).  Applying this argument to all three summands proves
\eqref{d4:eq:D4-allowed-A1-cubed-types}.

By \eqref{d4:eq:D4-A1-Casimir-formula}, the Casimir of \(\mathfrak s_a\)
acts on \(L_{\mathfrak s_a}(m_a\varpi_a)\) by
\(m_a(m_a+2)/2\).  For \(m_a\in\{0,1\}\), this is respectively
\(0\) or \(3/2\).  Summing over \(a\in\mathcal A\) gives
\eqref{d4:eq:D4-total-Casimir-bound-allowed-types}.
\end{proof}

\subsubsection{Trace forms for the three simple summands}

Let \(M\) be a finite-dimensional \(B\)-module, and write
\begin{equation*}
  \rho:\mathfrak g^\natural\longrightarrow
  \operatorname{End}_{\mathbb C}(M)
\end{equation*}
for the induced representation.  Put
\begin{equation*}
  N_M=\dim M,
  \qquad
  \tau_a=\operatorname{tr}_M\rho(\Omega_a),
  \qquad
  \tau=\sum_{a\in\mathcal A}\tau_a
       =\operatorname{tr}_M\rho(\Omega^\natural).
\end{equation*}
For each simple factor \(\mathfrak s_a\), invariance of the trace form
implies that there is a scalar \(I_a\in\mathbb C\) such that
\begin{equation}
  \operatorname{tr}_M\bigl(\rho(x)\rho(y)\bigr)
  =I_a(x\mid y),
  \qquad x,y\in\mathfrak s_a.
  \label{d4:eq:D4-Dynkin-index-each-A1}
\end{equation}
If \(x\in\mathfrak s_a\) and \(y\in\mathfrak s_b\) with \(a\ne b\),
then
\begin{equation}
  \operatorname{tr}_M\bigl(\rho(x)\rho(y)\bigr)=0.
  \label{d4:eq:D4-mixed-trace-forms-zero}
\end{equation}
Indeed, for fixed \(y\in\mathfrak s_b\), the functional
\(x\mapsto\operatorname{tr}_M(\rho(x)\rho(y))\) vanishes on
\([\mathfrak s_a,\mathfrak s_a]=\mathfrak s_a\) by invariance and the
commutativity of the two summands.

Choose an orthonormal basis \(\{x_{a,r}\}_{r=1}^3\) of
\(\mathfrak s_a\).  Taking the trace of its Casimir gives
\begin{equation}
  \tau_a=3I_a,
  \qquad
  I_a=\frac{\tau_a}{3}.
  \label{d4:eq:D4-index-Casimir-trace-relation}
\end{equation}

\subsubsection{The trace of the quadratic Ramond term}

\begin{lemma}[Trace of the quadratic term]
\label{d4:lem:D4-trace-quadratic-Ramond-term}
For all \(u,v\in U\),
\begin{equation}
  \operatorname{tr}_M\rho\bigl(Q(u,v)\bigr)
  =\tau\langle u,v\rangle.
  \label{d4:eq:D4-trace-quadratic-Ramond-term}
\end{equation}
\end{lemma}

\begin{proof}
Use the union of orthonormal bases
\(\{x_{a,r}\mid a\in\mathcal A,\ 1\le r\le3\}\) in the definition
\eqref{eq:common-quadratic-term}.  By
\eqref{d4:eq:D4-mixed-trace-forms-zero}, all terms involving two distinct
simple summands vanish after taking the trace.  For the contribution
from \(\mathfrak s_a\), equations
\eqref{d4:eq:D4-Dynkin-index-each-A1} and
\eqref{d4:eq:D4-index-Casimir-trace-relation} give
\begin{align}
  \operatorname{tr}_M\rho\bigl(Q_a(u,v)\bigr)
  &=2I_a\sum_{r=1}^3
    \langle[x_{a,r},u],[v,x_{a,r}]\rangle.
  \label{d4:eq:D4-trace-Qa-first-contraction}
\end{align}
The skew form \eqref{d4:eq:D4-symplectic-form} is
\(\mathfrak g^\natural\)-invariant.  Since
\([v,x_{a,r}]=-x_{a,r}v\), we obtain
\begin{align*}
  \sum_{r=1}^3
  \langle[x_{a,r},u],[v,x_{a,r}]\rangle
  &=-\sum_{r=1}^3
    \langle x_{a,r}u,x_{a,r}v\rangle
    \notag\\
  &=\left\langle
      u,\sum_{r=1}^3x_{a,r}^2v
    \right\rangle.
\end{align*}
As a module for \(\mathfrak s_a\), the space
\(U\cong\mathbf2\boxtimes\mathbf2\boxtimes\mathbf2\) is a direct sum
of copies of the fundamental representation
\(L_{\mathfrak s_a}(\varpi_a)\).  Hence \(\Omega_a\) acts on all of
\(U\) by the scalar
\begin{equation*}
  c_2(\varpi_a)=\frac32.
\end{equation*}
It follows that
\begin{equation}
  \sum_{r=1}^3
  \langle[x_{a,r},u],[v,x_{a,r}]\rangle
  =\frac32\langle u,v\rangle.
  \label{d4:eq:D4-quadratic-contraction-A1-value}
\end{equation}
Substituting
\eqref{d4:eq:D4-index-Casimir-trace-relation} and
\eqref{d4:eq:D4-quadratic-contraction-A1-value} into
\eqref{d4:eq:D4-trace-Qa-first-contraction}, we find
\begin{equation*}
  \operatorname{tr}_M\rho\bigl(Q_a(u,v)\bigr)
  =2\frac{\tau_a}{3}\frac32\langle u,v\rangle
  =\tau_a\langle u,v\rangle.
\end{equation*}
Summing over \(a\in\mathcal A\) proves
\eqref{d4:eq:D4-trace-quadratic-Ramond-term}.
\end{proof}

\subsubsection{The Casimir trace identity}

We now combine the preceding contraction with the Ramond relation at
level \(-1\).

\begin{proposition}[Casimir trace identity]
\label{d4:prop:D4-Casimir-trace-identity}
Let \(M\) be a nonzero finite-dimensional \(B\)-module on which the
central element \(L=[\omega]_{\mathrm R}\) acts as the scalar
\(h\in\mathbb C\).  Then
\begin{equation}
  \frac{\operatorname{tr}_M\rho(\Omega^\natural)}{\dim M}
  =5h+\frac14.
  \label{d4:eq:D4-Casimir-trace-identity}
\end{equation}
\end{proposition}

\begin{proof}
Lemma~\ref{d4:lem:D4-trace-quadratic-Ramond-term} gives coefficient
$\gamma=1$ for the total triality Casimir in
Proposition~\ref{prop:common-Ramond-Casimir-trace}.  Since
$h^\vee(D_4)=6$ and $p_{D_4}(-1)=1$, the common formula reads
$2\,\operatorname{tr}_M(\Omega^\natural)/\dim M=10h+1/2$, which is
\eqref{d4:eq:D4-Casimir-trace-identity}.
\end{proof}

\begin{remark}

The same calculation applied to the general Ramond relation
\eqref{eq:common-Ramond-Zhu} gives, whenever the three
simple current factors are represented with the normalizations of
Section~\ref{d4:sec:D4-preliminaries},
\begin{equation*}
  \frac{\operatorname{tr}_M\rho(\Omega^\natural)}{\dim M}
  =(k+6)h+\frac{(k+2)^2}{4}.
\end{equation*}
At \(k=-1\), this specializes to
\eqref{d4:eq:D4-Casimir-trace-identity}.
\end{remark}

Combining Lemma~\ref{d4:lem:D4-allowed-A1-cubed-types} with
Proposition~\ref{d4:prop:D4-Casimir-trace-identity}, we obtain the two
numerical facts needed in the simplicity argument:
\begin{equation}
  \frac{\operatorname{tr}_M\rho(\Omega^\natural)}{\dim M}
  =5h+\frac14,
  \qquad
  \frac{\operatorname{tr}_M\rho(\Omega^\natural)}{\dim M}
  \le\frac92.
  \label{d4:eq:D4-trace-identity-and-Casimir-upper-bound}
\end{equation}
The first equality records the Ramond energy, whereas the second follows
from the three independent square-zero relations.  Section~\ref{d4:sec:D4-Li-spectral-flow-simplicity} will show
that these two constraints force the lowest energy of every nonzero
Ramond-twisted ideal to equal \(1/4\).

\subsection{Li spectral flow and simplicity of the reduced quotient}
\label{d4:sec:D4-Li-spectral-flow-simplicity}

We now combine the Casimir trace identity of
Section~\ref{d4:sec:D4-Ramond-Casimir} with the Li twist fixed in
Subsection~\ref{d4:subsec:D4-Li-delta-operator}.  The spectral-flow
realization of the Ramond sector gives a discrete lower-bounded spectrum.
The trace identity then forces the lowest energy of a hypothetical
nonzero ideal to coincide with the Ramond vacuum energy.  At that energy,
the three square-zero relations leave only a small list of
\(A_1^3\)-types, all of which are excluded by the original conformal
grading.

Throughout this section, we use the surjection
\begin{equation*}
  \pi:\widetilde{\mathcal W}
  \twoheadrightarrow
  \mathcal W_{-1}(\mathfrak g,f_\theta)
\end{equation*}
constructed in \eqref{d4:eq:D4-Wtilde-surjects-simple-W}.

\subsubsection{The Ramond spectral lattice}

Let
\[
  x=\varpi_1^\vee,
  \qquad
  H=J^{\{x\}},
\]
as in \eqref{d4:eq:D4-Li-choice-x}.  The corresponding Li twist is the
Ramond twist, and its conformal Hamiltonian is
\begin{equation}
  L_0^{\mathrm R}=L_0+x_0+\frac14
  \label{d4:eq:D4-section5-Ramond-Hamiltonian}
\end{equation}
by \eqref{d4:eq:D4-Li-Ramond-Hamiltonian}.

\begin{proposition}[Ramond spectral lattice]
\label{d4:prop:D4-Ramond-spectral-lattice}
The operator \(L_0^{\mathrm R}\) acts semisimply on the Li-twisted
adjoint module \(\widetilde{\mathcal W}^{\mathrm R}\), and
\begin{equation}
  \operatorname{Spec}L_0^{\mathrm R}
  \subset
  \frac14+\mathbb Z_{\ge0}.
  \label{d4:eq:D4-Ramond-spectrum-inclusion}
\end{equation}
The lower endpoint is attained by the vacuum:
\begin{equation}
  L_0^{\mathrm R}\mathbf1=\frac14\mathbf1.
  \label{d4:eq:D4-Ramond-vacuum-energy}
\end{equation}
\end{proposition}

\begin{proof}
Every ordinary conformal homogeneous subspace of the universal minimal
\(W\)-algebra is finite-dimensional and
\(\mathfrak g^\natural\)-stable.  The same is true after passing to the
conformal quotient \(\widetilde{\mathcal W}\).  Since finite-dimensional
\(\mathfrak g^\natural\)-modules are completely reducible, \(x_0\) acts
semisimply on each homogeneous subspace and commutes with \(L_0\).
Thus \eqref{d4:eq:D4-section5-Ramond-Hamiltonian} acts semisimply.

It remains to determine the possible eigenvalues of \(L_0+x_0\).
The images of PBW monomials in derivatives of
\[
  J^{\{a\}},
  \qquad
  G^{\{u\}},
  \qquad
  \omega
\]
span \(\widetilde{\mathcal W}\); see
\eqref{d4:eq:D4-minimal-W-generators}.  If
\(a_\gamma\in\mathfrak g^\natural\) is a root vector of weight
\(\gamma\), then
\begin{equation*}
  \partial^rJ^{\{a_\gamma\}}
  \quad\text{contributes}\quad
  1+r+\gamma(x)
  \in\mathbb Z_{\ge0}
\end{equation*}
to \(L_0+x_0\), because
\(\gamma(x)\in\{-1,0,1\}\) by
\eqref{d4:eq:D4-Li-current-charges}.  A Cartan current has charge zero and
contributes \(1+r\).

If \(u_\mu\in U\) has weight \(\mu\), then
\begin{equation*}
  \partial^rG^{\{u_\mu\}}
  \quad\text{contributes}\quad
  \frac32+r+\mu(x)
  \in\mathbb Z_{>0},
\end{equation*}
because \(\mu(x)=\pm\frac12\) by
\eqref{d4:eq:D4-Li-half-space-charges}.  Finally,
\begin{equation*}
  \partial^r\omega
  \quad\text{contributes}\quad
  2+r\in\mathbb Z_{>0}.
\end{equation*}
Adding these contributions over the factors of a PBW monomial shows that
every \(L_0+x_0\)-eigenvalue lies in \(\mathbb Z_{\ge0}\).  Formula
\eqref{d4:eq:D4-section5-Ramond-Hamiltonian} proves
\eqref{d4:eq:D4-Ramond-spectrum-inclusion}.  Since
\(L_0\mathbf1=x_0\mathbf1=0\), it also gives
\eqref{d4:eq:D4-Ramond-vacuum-energy}.
\end{proof}

Write
\begin{equation*}
  T=\widetilde{\mathcal W}^{\mathrm R}
    \!\left[\frac14\right]
\end{equation*}
for the lowest Ramond eigenspace.  It is nonzero by
\eqref{d4:eq:D4-Ramond-vacuum-energy}.

\subsubsection{Shifted weights and the Ramond bottom}

The Cartan shift
\eqref{d4:eq:D4-Li-Cartan-shift} identifies the Ramond
\(\mathfrak g^\natural\)-weight with a translate of the original weight.
If a vector has original \(\mathfrak g^\natural\)-weight \(\nu\), then
its shifted weight is
\begin{equation}
  \lambda=\nu+\varpi_1.
  \label{d4:eq:D4-original-to-shifted-weight}
\end{equation}
Indeed, the additional Cartan eigenvalue is
\((x\mid h)=\varpi_1(h)\).

\begin{proposition}[Degree--weight relation at the Ramond bottom]
\label{d4:prop:D4-degree-weight-relation}
Let \(0\ne v\in T\) be homogeneous for the original conformal grading,
and suppose that its shifted
\(\mathfrak g^\natural\)-weight is \(\lambda\).  If its original
conformal degree is \(d\), then
\begin{equation}
  \nu=\lambda-\varpi_1
  \label{d4:eq:D4-original-weight-from-shifted-weight}
\end{equation}
and
\begin{equation}
  d=\frac12-\lambda(x).
  \label{d4:eq:D4-degree-from-shifted-weight}
\end{equation}
Equivalently,
\begin{equation}
  d+\nu(x)=0.
  \label{d4:eq:D4-zero-excess-energy-relation}
\end{equation}
\end{proposition}

\begin{proof}
Equation \eqref{d4:eq:D4-original-weight-from-shifted-weight} follows from
\eqref{d4:eq:D4-original-to-shifted-weight}.  Since \(v\in T\),
\eqref{d4:eq:D4-section5-Ramond-Hamiltonian} gives
\[
  d+\nu(x)+\frac14=\frac14.
\]
This proves \eqref{d4:eq:D4-zero-excess-energy-relation}.  Substituting
\(\nu=\lambda-\varpi_1\) and using
\[
  \varpi_1(x)
  =(\varpi_1\mid\varpi_1)
  =\frac12
\]
gives \eqref{d4:eq:D4-degree-from-shifted-weight}.
\end{proof}

The root-mode shift is also relevant at the Ramond bottom.  For a root
vector \(e_\gamma\in\mathfrak g^\natural\), the twisted zero mode is
\[
  J_0^{\{e_\gamma\},\mathrm R}
  =J_{\gamma(x)}^{\{e_\gamma\}}
\]
by \eqref{d4:eq:D4-Li-root-mode-shift}.  In particular, if
\(e_1,f_1,h_1\) is the standard \(\mathfrak{sl}_2\)-triple in
\(\mathfrak s_1\), then
\begin{equation}
  E_1^{\mathrm R}=J_1^{\{e_1\}},
  \qquad
  F_1^{\mathrm R}=J_{-1}^{\{f_1\}},
  \qquad
  H_1^{\mathrm R}=J_0^{\{h_1\}}+1.
  \label{d4:eq:D4-first-A1-twisted-zero-modes}
\end{equation}
For the positive-root vectors of \(\mathfrak s_3\) and
\(\mathfrak s_4\), the twisted zero modes remain their ordinary zero
modes.  Consequently, the vacuum is a shifted
\(\mathfrak g^\natural\)-highest-weight vector of weight
\(\varpi_1\).

\begin{corollary}[The shifted vacuum line]
\label{d4:cor:D4-shifted-vacuum-line}
The shifted weight-\(\varpi_1\) space at the Ramond bottom is
one-dimensional:
\begin{equation*}
  T[\varpi_1]=\mathbb C\mathbf1.
\end{equation*}
In particular, the vacuum is the unique shifted highest-weight vector
of weight \(\varpi_1\) in \(T\).
\end{corollary}

\begin{proof}
The vacuum belongs to \(T[\varpi_1]\) by
\eqref{d4:eq:D4-Ramond-vacuum-energy},
\eqref{d4:eq:D4-Li-Cartan-shift}, and
\eqref{d4:eq:D4-first-A1-twisted-zero-modes}.  Conversely, if
\(v\in T[\varpi_1]\), Proposition
\ref{d4:prop:D4-degree-weight-relation} gives
\[
  \nu=0,
  \qquad
  d=\frac12-\varpi_1(x)=0.
\]
The degree-zero subspace of the nonzero conformal quotient
\(\widetilde{\mathcal W}\) is \(\mathbb C\mathbf1\), proving the claim.
\end{proof}

For completeness, the twisted \(\mathfrak s_1\)-action on the vacuum is
nontrivial.  Indeed, if
\(F_1^{\mathrm R}\mathbf1=0\), then
\[
  [E_1^{\mathrm R},F_1^{\mathrm R}]\mathbf1=0,
\]
whereas \eqref{d4:eq:D4-first-A1-twisted-zero-modes} gives
\(H_1^{\mathrm R}\mathbf1=\mathbf1\).  Thus
\(F_1^{\mathrm R}\mathbf1\ne0\), and the twisted vacuum generates the
fundamental \(\mathfrak s_1\)-module, with the other two
\(A_1\)-factors acting trivially.  Only the uniqueness of its highest
line will be needed below.

We also record the two low original degrees used in the exclusion
argument:
\begin{equation}
  \widetilde{\mathcal W}_0=\mathbb C\mathbf1,
  \qquad
  \widetilde{\mathcal W}_{\frac12}=0.
  \label{d4:eq:D4-low-original-conformal-degrees}
\end{equation}
These statements follow because the universal minimal \(W\)-algebra is
of CFT type and is generated in conformal weights \(1\), \(3/2\), and
\(2\); passage to a nonzero conformal quotient cannot create new
vectors.

\subsubsection{Lowest spaces of twisted ideals}

Proposition~\ref{d4:prop:D4-Ramond-spectral-lattice}, the finite-dimensionality
of $B$ in \eqref{d4:eq:D4-Ramond-Zhu-B-finite}, and
Lemma~\ref{lem:common-lowest-Ramond-Zhu-module} show that every nonzero ideal
has a finite-dimensional cyclic $B$-module at a lowest energy
$\frac14+m$, $m\in\mathbb Z_{\ge0}$.

\subsubsection{Simplicity of the reduced quotient}

\begin{theorem}[Simplicity of the reduced quotient]
\label{d4:thm:D4-reduced-quotient-simple}
The natural map
\begin{equation*}
  \pi:\widetilde{\mathcal W}
  \longrightarrow
  \mathcal W_{-1}(D_4,f_\theta)
\end{equation*}
is an isomorphism.  Equivalently,
\begin{equation*}
  \mathcal W^{-1}(D_4,f_\theta)
  \Big/
  \left\langle
    :J_1^2:,\ :J_3^2:,\ :J_4^2:
  \right\rangle
  \cong
  \mathcal W_{-1}(D_4,f_\theta).
\end{equation*}
\end{theorem}

\begin{proof}
We apply Proposition~\ref{prop:common-Casimir-gap-simplicity} with
\[
 h_0=\frac14,
 \qquad
 \lambda_0=\varpi_1,
 \qquad
 A(h)=5h+\frac14,
 \qquad
 C_{\max}=\frac92.
\]
The spectral lattice and extremal line are
Proposition~\ref{d4:prop:D4-Ramond-spectral-lattice} and
Corollary~\ref{d4:cor:D4-shifted-vacuum-line}; the Ramond Zhu algebra is
finite-dimensional by \eqref{d4:eq:D4-Ramond-Zhu-B-finite}.  The trace
identity and the global Casimir bound are
\eqref{d4:eq:D4-trace-identity-and-Casimir-upper-bound}, and
\[
 A(h_0+1)=\frac{13}{2}>\frac92=C_{\max}.
\]
It remains only to verify the ground-space forcing condition.  Let $S$ be
a simple quotient of a nonzero ground $B$-submodule.  Complete reducibility
over $A_1^3$ embeds every $A_1^3$-constituent of $S$ back into the ambient
ground space.  An allowed constituent has highest weight
$\lambda=m_1\varpi_1+m_3\varpi_3+m_4\varpi_4$ with
$m_a\in\{0,1\}$.  By
Proposition~\ref{d4:prop:D4-degree-weight-relation}, $m_1=0$ would place its
highest line in original degree $1/2$, which is zero by
\eqref{d4:eq:D4-low-original-conformal-degrees}.  If $m_1=1$, the line lies
in degree zero, so its original weight must vanish; hence
$m_3=m_4=0$ and $\lambda=\varpi_1$.  Thus every such simple quotient
contains the extremal constituent, whose highest line is
$\mathbb C\mathbf1$.  The common criterion proves that $\pi$ is an
isomorphism.
\end{proof}

Combining Theorem~\ref{d4:thm:D4-reduced-quotient-simple} with
Proposition~\ref{d4:prop:D4-exact-reduction-triality-quotient}, we obtain
\begin{equation*}
  H^0_{\mathrm{DS},f_\theta}(Q)
  \cong
  \mathcal W_{-1}(D_4,f_\theta).
\end{equation*}
The final affine step follows from the common detection and lifting result,
Proposition~\ref{prop:common-affine-lifting}.

\subsection{The maximal ideal of the affine vertex algebra}

We apply Proposition~\ref{prop:common-affine-lifting} to the triality quotient.

\subsubsection{Uniform affine lifting}

The triality quotient admits the canonical surjection
\begin{equation*}
 q:Q\twoheadrightarrow L_{-1}(D_4).
\end{equation*}
Proposition~\ref{d4:prop:D4-exact-reduction-triality-quotient} and
Theorem~\ref{d4:thm:D4-reduced-quotient-simple} identify the induced map on
minimal reductions with an isomorphism.  Hence
Proposition~\ref{prop:common-affine-lifting} applies.

\begin{theorem}[The maximal ideal at level \(-1\)]
\label{d4:thm:D4-main-maximal-ideal}
The quotient \(Q\) is simple:
\begin{equation*}
  Q\cong L_{-1}(D_4).
\end{equation*}
Equivalently,
\begin{equation*}
  \ker\!\left(
    V^{-1}(D_4)\longrightarrow L_{-1}(D_4)
  \right)
  =
  \left\langle
    \sigma(w_1)^2,\,
    \sigma(w_3)^2,\,
    \sigma(w_4)^2
  \right\rangle.
\end{equation*}
\end{theorem}

\begin{proof}
Apply Proposition~\ref{prop:common-affine-lifting} to
$q:Q\twoheadrightarrow L_{-1}(D_4)$.
\end{proof}

\begin{corollary}

The Arakawa--Moreau maximal-ideal conjecture holds for
\[
  \mathfrak g=D_4,
  \qquad
  n=1,
  \qquad
  k=n-\frac{h^\vee(D_4)}6-1=-1.
\]
\end{corollary}

\begin{proof}
For \(D_4\), the relevant Kostant module has the three triality
components indexed by \(1,3,4\), and the conjectural generators at
\(n=1\) are precisely
\(\sigma(w_1)^2,\sigma(w_3)^2,\sigma(w_4)^2\).
The assertion is therefore Theorem~\ref{d4:thm:D4-main-maximal-ideal};
see \cite{ArakawaMoreau2018}.
\end{proof}

\begin{remark}

The boundary case \(n=0\), corresponding to \(k=-2\), was already
known.  Theorem~\ref{d4:thm:D4-main-maximal-ideal} settles the next
negative integral level \(k=-1\), which is the only case with
\(n>0\) and \(k=n-2<0\).
\end{remark}
\section{Type \texorpdfstring{$E_6$}{E6}}
\label{sec:case-e6}

\subsection{Type data}\label{e6:sec:preliminaries}

Let $\mathfrak g=E_6$ with Bourbaki numbering and $(\theta\mid\theta)=2$.
The highest root is
\begin{equation}
 \theta=\alpha_1+2\alpha_2+2\alpha_3+3\alpha_4+2\alpha_5+\alpha_6.
 \label{e6:eq:E6-highest-root}
\end{equation}
For $x=\theta^\vee/2$ and $f=f_\theta$, the minimal grading has
\begin{equation*}
 \mathfrak g^\natural\cong A_5,
 \qquad
 \mathfrak g_{\frac12}\cong\mathfrak g_{-\frac12}
 \cong L_{A_5}(\varpi_3)\cong\bigwedge^3\mathbb C^6.
\end{equation*}
The latter is a $20$-dimensional symplectic module, with form
\begin{equation*}
 \langle u,v\rangle_{\rm ne}=(f\mid[u,v]).
\end{equation*}
We denote by $\vartheta$ the highest root of $A_5$.  Thus
$h^\vee(E_6)=12$, and all root and Casimir calculations used below are
listed in Appendix~\ref{e6:app:explicit-calculations}.  This subsection
is the specialization of the common minimal-grading conventions in
Section~\ref{sec:common-framework}.
\label{e6:subsec:E6-minimal-grading}

\subsubsection{Level and reduction data}

We use Section~\ref{sec:common-framework}.  The PBW symmetrization is
\begin{equation}
\label{e6:eq:symmetrization-map}
 \sigma(a_1\cdots a_r)=\frac1{r!}\sum_{\pi\in\mathfrak S_r}
 a_{\pi(1)}(-1)\cdots a_{\pi(r)}(-1)\mathbf1.
\end{equation}
For $E_6$, $\mathfrak g^\natural\cong A_5$, the natural current level is
\begin{equation}
 k^\natural=k+3,
 \label{e6:eq:natural-level}
\end{equation}
and the two new cases are
\begin{equation}
 k_n=n-3,\qquad n=1,2.
 \label{e6:eq:distinguished-levels}
\end{equation}
Exactness and nonvanishing are supplied by
Section~\ref{subsec:common-DS}; in particular
\begin{equation*}
 H^0_{\mathrm{DS},f_\theta}(V^k(\mathfrak g)/I)
 \cong\mathcal W^k(\mathfrak g,f_\theta)/H^0_{\mathrm{DS},f_\theta}(I).
\end{equation*}
\begin{corollary}
\label{e6:cor:negative-level-nonvanishing}
For $k=-2,-1$ and $\lambda\in P_+$,
$H^0_{\mathrm{DS},f_\theta}(L_k(\lambda))\ne0$.
\end{corollary}
\begin{proof}
Apply \eqref{eq:common-DS-nonvanishing}.
\end{proof}

\subsection{Singular vectors in the affine vertex algebra of type \texorpdfstring{$E_6$}{E6}}
\label{e6:sec:singular-vectors}

In this section we recall the quadratic highest-weight vector associated with the
Joseph ideal of $E_6$ and determine exactly when its normally ordered powers become
affine singular vectors.  We keep the normalization and the Bourbaki numbering fixed
in Section~\ref{e6:sec:preliminaries}.

\subsubsection{A quadratic highest-weight vector in \texorpdfstring{$S^2(\mathfrak g)$}{S2(g)}}

Let $\omega_1,\ldots,\omega_6$ denote the fundamental weights of
$\mathfrak g=E_6$.  To shorten the root formulas, for
$\mathbf c=(c_1,\ldots,c_6)\in\mathbb Z^6$ we write
\[
  [c_1c_2c_3c_4c_5c_6]
  =\sum_{i=1}^{6}c_i\alpha_i.
\]
Thus, by \eqref{e6:eq:E6-highest-root},
\[
  \theta=[122321].
\]
Set
\begin{equation*}
  \theta_1=[101111]
  =\alpha_1+\alpha_3+\alpha_4+\alpha_5+\alpha_6.
\end{equation*}
This is the highest root of the root subsystem of type $A_5$ generated by
$\alpha_1,\alpha_3,\alpha_4,\alpha_5,\alpha_6$; thus $\theta_1=\vartheta$ in the notation of
Subsection~\ref{e6:subsec:E6-minimal-grading}.  Consider the following three
pairs of positive roots:
\begin{equation*}
\begin{aligned}
  (\beta_1,\delta_1)&=([010000],[011210]),\\
  (\beta_2,\delta_2)&=([010100],[011110]),\\
  (\beta_3,\delta_3)&=([010110],[011100]).
\end{aligned}
\end{equation*}
They satisfy
\begin{equation}\label{e6:eq:root-pair-sum}
  \beta_j+\delta_j=\theta-\theta_1,
  \qquad j=1,2,3,
\end{equation}
and all the roots $\beta_j+\theta_1$ and $\delta_j+\theta_1$ are positive.

We choose the root vectors compatibly so that
\begin{equation}\label{e6:eq:Chevalley-normalization-w}
\begin{aligned}
  [e_{\beta_j},e_{\theta_1}]&=e_{\beta_j+\theta_1},\\
  [e_{\delta_j},e_{\theta_1}]&=e_{\delta_j+\theta_1},\\
  [e_{\delta_j},[e_{\beta_j},e_{\theta_1}]]&=e_\theta
\end{aligned}
\qquad (j=1,2,3).
\end{equation}
Such a compatible choice is the one used in the construction of the Joseph ideal;
see \cite[Chapter~IV, Proposition~11]{Garfinkle1982} and
\cite[Theorem~4.2 and Table~2]{ArakawaMoreau2018}.

Define
\begin{equation}\label{e6:eq:def-w}
  w
  =e_\theta e_{\theta_1}
   -\sum_{j=1}^{3}
    e_{\beta_j+\theta_1}e_{\delta_j+\theta_1}
  \in S^2(\mathfrak g).
\end{equation}

\begin{proposition}
The vector $w$ is a nonzero $\mathfrak g$-highest-weight vector of weight
\[
  \theta+\theta_1=\omega_1+\omega_6.
\]
Moreover,
\begin{equation}\label{e6:eq:S2-E6-decomposition}
  S^2(\mathfrak g)
  \cong
  L_{\mathfrak g}(2\theta)
  \oplus L_{\mathfrak g}(0)
  \oplus L_{\mathfrak g}(\omega_1+\omega_6),
\end{equation}
and $w$ spans the highest-weight line of the last summand.
\end{proposition}

\begin{proof}
By \eqref{e6:eq:root-pair-sum}, every monomial occurring in \eqref{e6:eq:def-w} has
weight $\theta+\theta_1$.  The coefficient of $e_\theta e_{\theta_1}$ is nonzero,
so $w\neq0$.  The normalizations \eqref{e6:eq:Chevalley-normalization-w}, together
with the root-string relations, imply that the positive simple-root vectors
annihilate $w$.  Finally,
\[
  \bigl\langle\theta+\theta_1,\alpha_i^\vee\bigr\rangle
  =\delta_{i1}+\delta_{i6},
\]
so the highest weight is $\omega_1+\omega_6$.  A term-by-term verification with a compatible Chevalley normalization is given in
Appendix~\ref{e6:app:Joseph-root-calculation}.  The decomposition
\eqref{e6:eq:S2-E6-decomposition} and the explicit highest-weight vector
\eqref{e6:eq:def-w} are the $E_6$ case of the construction in
\cite[Section~2 and Theorem~4.2]{ArakawaMoreau2018}.
\end{proof}

Let
\begin{equation}\label{e6:eq:def-u}
  u=\sigma(w)\in V^k(\mathfrak g)_2,
\end{equation}
where $\sigma$ is the symmetrization map \eqref{e6:eq:symmetrization-map}.  Explicitly,
\begin{align}
  u=\frac12\Bigl(&e_\theta(-1)e_{\theta_1}(-1)
                  +e_{\theta_1}(-1)e_\theta(-1) \notag\\
   &-\sum_{j=1}^{3}
     \bigl(
       e_{\beta_j+\theta_1}(-1)e_{\delta_j+\theta_1}(-1)
       +e_{\delta_j+\theta_1}(-1)e_{\beta_j+\theta_1}(-1)
     \bigr)
  \Bigr)\mathbf1.
  \label{e6:eq:explicit-u}
\end{align}
The vector $u$ is a $\mathfrak g$-highest-weight vector of conformal weight $2$
and $\mathfrak g$-weight $\omega_1+\omega_6$.

\begin{lemma}\label{e6:lem:PBW-powers-u}
For every $r\geq1$, the state $u^r$ is nonzero.  It has conformal weight $2r$,
$\mathfrak g$-weight $r(\omega_1+\omega_6)$, and PBW symbol
\begin{equation}\label{e6:eq:PBW-symbol-u-power}
  \operatorname{gr}_{2r}(u^r)=w^r
  \quad\text{in}\quad
  S\bigl(\mathfrak g[t^{-1}]t^{-1}\bigr),
\end{equation}
where $w$ is regarded as an element supported in loop degree $-1$.
\end{lemma}

\begin{proof}
The PBW symbol of $u$ is $w$.  Compatibility of the PBW filtration with the
$(-1)$-product gives \eqref{e6:eq:PBW-symbol-u-power}.  Since the symmetric algebra
$S(\mathfrak g[t^{-1}]t^{-1})$ is an integral domain and $w\neq0$, one has
$w^r\neq0$, hence $u^r\neq0$.  The assertions about conformal weight and
$\mathfrak g$-weight follow from additivity under the $(-1)$-product.
\end{proof}

\subsubsection{The affine singularity calculation}

We first record the standard affine highest-weight criterion in the form needed
below.

\begin{lemma}\label{e6:lem:affine-singularity-criterion}
Let $v\in V^k(\mathfrak g)$ be a $\mathfrak g$-highest-weight vector for the
zero-mode action.  Then $v$ is a singular vector for the affine Kac--Moody algebra
$\widehat{\mathfrak g}^{\,\prime}$ if and only if
\[
  f_\theta(1)v=0.
\]
\end{lemma}

\begin{proof}
The positive nilpotent subalgebra of the affine Kac--Moody algebra is generated
by the affine Chevalley generators
\[
  e_i(0),\quad 1\leq i\leq6,
  \qquad\text{and}\qquad
  e_0=f_\theta(1).
\]
The first six generators annihilate $v$ because $v$ is a
$\mathfrak g$-highest-weight vector.  Hence the remaining condition is precisely
$f_\theta(1)v=0$.
\end{proof}

The following computation is the point at which the level enters.

\begin{proposition}\label{e6:prop:f-theta-on-u-powers}
For every integer $r\geq1$,
\begin{equation}\label{e6:eq:f-theta-on-u-powers}
  f_\theta(1)u^r
  =r(k-r+4)\,
   \bigl(u^{r-1}\bigr)_{(-1)}e_{\theta_1}(-1)\mathbf1.
\end{equation}
Consequently, $u^r$ is an affine singular vector in $V^k(E_6)$ if and only if
\begin{equation}\label{e6:eq:singular-level-for-r}
  k=r-4.
\end{equation}
\end{proposition}

\begin{proof}
The root relations following from \eqref{e6:eq:Chevalley-normalization-w} give
\begin{equation}\label{e6:eq:key-double-brackets}
  [[f_\theta,e_{\beta_j+\theta_1}],e_{\delta_j+\theta_1}]
  =[[f_\theta,e_{\delta_j+\theta_1}],e_{\beta_j+\theta_1}]
  =-e_{\theta_1}
\end{equation}
for $j=1,2,3$.  They also imply
\begin{equation*}
  (e_{\theta_1}(-1)\mathbf1)_{(-1)}u
  =u_{(-1)}e_{\theta_1}(-1)\mathbf1.
\end{equation*}
Since $u$ has weight $\theta+\theta_1$, one has
\begin{equation*}
  [f_\theta,e_\theta](0)u=-2u.
\end{equation*}
Furthermore, for every $j$,
\begin{equation}\label{e6:eq:other-zero-modes-on-u}
  [f_\theta,e_{\beta_j+\theta_1}](0)u
  =[f_\theta,e_{\delta_j+\theta_1}](0)u=0.
\end{equation}
Indeed, the corresponding roots, when nonzero, are orthogonal to the highest
weight $\theta+\theta_1$.

Applying the affine commutator formula to \eqref{e6:eq:explicit-u}, using
\eqref{e6:eq:key-double-brackets}--\eqref{e6:eq:other-zero-modes-on-u}, and then
inducting on $r$, gives
\begin{align*}
  f_\theta(1)u^r
  &=\sum_{j=0}^{r-1}(k+3-2j)\,
    \bigl(u^{r-1}\bigr)_{(-1)}e_{\theta_1}(-1)\mathbf1\\
  &=r(k-r+4)\,
    \bigl(u^{r-1}\bigr)_{(-1)}e_{\theta_1}(-1)\mathbf1.
\end{align*}
A complete mode-by-mode induction is recorded in
Appendix~\ref{e6:app:affine-mode-calculation}.  This is the $E_6$ specialization of
the calculation in the proof of
\cite[Theorem~4.2]{ArakawaMoreau2018}.

By Lemma~\ref{e6:lem:PBW-powers-u}, the PBW symbol of the state on the right-hand
side of \eqref{e6:eq:f-theta-on-u-powers} is
$w^{r-1}e_{\theta_1}$, which is nonzero.  Therefore
$f_\theta(1)u^r=0$ if and only if $k-r+4=0$.  The conclusion now follows from
Lemma~\ref{e6:lem:affine-singularity-criterion}.
\end{proof}

\subsubsection{The singular vectors at levels \texorpdfstring{$-2$}{-2} and
\texorpdfstring{$-1$}{-1}}

Recall from \eqref{e6:eq:distinguished-levels} that $k_n=n-3$ for $n=1,2$.
Define
\begin{equation*}
  s_n=u^{n+1}\in V^{k_n}(E_6),
  \qquad n=1,2.
\end{equation*}

\begin{theorem}
The states $s_1$ and $s_2$ are nonzero affine singular vectors.  More precisely:
\begin{enumerate}
  \item At level $k_1=-2$,
  \[
    s_1=u^2
  \]
  has conformal weight $4$ and $E_6$-highest weight
  $2\omega_1+2\omega_6$.

  \item At level $k_2=-1$,
  \[
    s_2=u^3
  \]
  has conformal weight $6$ and $E_6$-highest weight
  $3\omega_1+3\omega_6$.
\end{enumerate}
\end{theorem}

\begin{proof}
For $r=n+1$, equation \eqref{e6:eq:singular-level-for-r} becomes
\[
  k=r-4=n-3=k_n.
\]
Thus Proposition~\ref{e6:prop:f-theta-on-u-powers} shows that $s_n$ is singular in
$V^{k_n}(E_6)$.  Its nonvanishing and its conformal and finite-dimensional
highest weights follow from Lemma~\ref{e6:lem:PBW-powers-u}.
\end{proof}

Let $I_n$ be the graded vertex-algebra ideal generated by $s_n$.  Since
$V^{k_n}(E_6)$ is generated by the affine currents $a(-1)\mathbf1$, the
subspace
\begin{equation*}
  U(\widehat{\mathfrak g}^{\,\prime})s_n
\end{equation*}
is stable under all modes of the generating currents and hence is a
vertex-algebra ideal.  Conversely, every vertex-algebra ideal containing $s_n$
is stable under those modes.  Therefore
\begin{equation}\label{e6:eq:In-affine-submodule}
  I_n
  =\langle s_n\rangle
  =U(\widehat{\mathfrak g}^{\,\prime})s_n.
\end{equation}
Set
\begin{equation}\label{e6:eq:def-V-tilde-n}
  \widetilde V_n
  =V^{k_n}(E_6)/I_n
  =V^{n-3}(E_6)/\langle u^{n+1}\rangle,
  \qquad n=1,2.
\end{equation}
Because $s_n$ is a positive-energy affine singular vector, the
$\widehat{\mathfrak g}^{\,\prime}$-submodule in
\eqref{e6:eq:In-affine-submodule} is proper.  Indeed, the PBW triangular
decomposition and the singularity of $s_n$ show that every vector in
$U(\widehat{\mathfrak g}^{\,\prime})s_n$ is a linear combination of
monomials in zero modes from the negative finite nilpotent subalgebra and
negative loop modes applied to $s_n$.  Such operators preserve or increase
the standard loop-degree grading.  Since $s_n$ has positive degree $2(n+1)$,
the vacuum vector cannot belong to this submodule.  Hence $I_n$ is contained
in the unique maximal graded ideal of $V^{k_n}(E_6)$, and there is a natural
surjective homomorphism
\begin{equation}\label{e6:eq:V-tilde-to-simple}
  \pi_n:\widetilde V_n\twoheadrightarrow L_{k_n}(E_6).
\end{equation}
The remaining task is to prove that $\pi_n$ is injective.  In the next section we
first identify the minimal Drinfeld--Sokolov reductions of the quotients
$\widetilde V_n$.

\subsection{Minimal \texorpdfstring{$W$}{W}-algebra reductions of the affine quotients}

We now apply the minimal Drinfeld--Sokolov reduction to the quotients
$\widetilde V_n$ introduced in \eqref{e6:eq:def-V-tilde-n}.  The main point is to
identify the BRST class of the affine singular vector $s_n=u^{n+1}$.  We shall
obtain a canonical surjection from an explicit quotient of the universal minimal
$W$-algebra onto $H^0_{\mathrm{DS},f_\theta}(\widetilde V_n)$.  We deliberately do not assert at this stage that this surjection is an
isomorphism.  In the next section, a cyclicity theorem for the reduction of
highest-weight submodules will identify its kernel, after which an independent
Ramond-sector argument will prove simplicity.

\subsubsection{The BRST classes of the affine singular vectors}

Let $\mathbf1_{\mathrm{gh}}$ denote the vacuum vector of the charged and neutral
fermionic systems in the BRST complex.  Since $s_n$ is an affine singular vector,
the vector
\[
  s_n\otimes\mathbf1_{\mathrm{gh}}
  \in C^0_{\mathrm{DS},f_\theta}\bigl(V^{k_n}(\mathfrak g)\bigr)
\]
is BRST closed; this follows directly from the BRST differential for a singular
highest-weight vector, as observed in the first sentence of the proof of
\cite[Lemma~7.3]{KacWakimoto2004}.  We write
\begin{equation*}
  \xi_n
  =\bigl[s_n\otimes\mathbf1_{\mathrm{gh}}\bigr]
  \in
  \mathcal W^{k_n}(\mathfrak g,f_\theta),
  \qquad n=1,2.
\end{equation*}

Recall that $\theta_1=\vartheta$ is the highest root of
$\mathfrak g^\natural\cong\mathfrak{sl}_6$.  The following proposition identifies
$\xi_n$ up to a nonzero scalar.  This is sufficient for all ideal-theoretic
arguments below.

\begin{proposition}\label{e6:prop:BRST-image-s-n}
For $n=1,2$, there exists a constant $c_n\in\mathbb C^\times$ such that
\begin{equation}\label{e6:eq:BRST-image-s-n}
  \xi_n
  =c_n\,
   \mathopen{:}J^{\{e_\vartheta\}}{}^{\,n+1}\mathclose{:}
  \quad\text{in}\quad
  \mathcal W^{k_n}(E_6,f_\theta).
\end{equation}
In particular, $\xi_n\neq0$.
\end{proposition}

\begin{proof}
Set \(r=n+1\).  The ghost vacuum has conformal weight zero, and every
monomial in \(u\) has \(\operatorname{ad}x\)-degree one because
\(\theta(x)=1\) and \(\vartheta(x)=0\).  Hence \(\xi_n\) has reduced
conformal weight \(2r-r=r\).  Its finite
\(\mathfrak g^\natural\cong A_5\)-weight is \(r\vartheta\), and it is a
highest-weight vector.

The affine highest weight of $s_n$ is
\[
 \widehat\lambda_n
 =k_n\Lambda_0+r(\theta+\vartheta)-2r\delta.
\]
Since $\vartheta$ belongs to the root system of
$\mathfrak g^\natural$, it is orthogonal to $\theta$, and therefore
\[
 \widehat\lambda_n(\alpha_0^\vee)
 =k_n-r\langle\theta+\vartheta,\theta^\vee\rangle
 =(n-3)-2(n+1)=-n-5<0.
\]
Applying Lemma~\ref{lem:common-reduced-singular-generator} to
$I_n=U(\widehat{\mathfrak g}^{\prime})s_n$ gives
$\xi_n=[s_n]_{\mathrm{DS}}\ne0$; it also gives the cyclicity of
$H^0_{\mathrm{DS},f_\theta}(I_n)$ that will be used below.

Finally,
\(\mathfrak g_{-1/2}\cong L_{A_5}(\varpi_3)\), and all its weights \(\nu\)
satisfy
\(\langle\nu,\vartheta^\vee\rangle\le
\langle\varpi_3,\vartheta^\vee\rangle=1\).  Therefore
Proposition~\ref{prop:common-extremal-PBW-line}, applied to \(A_5\), shows
that the simultaneous weight space \((r,r\vartheta)\) is the line spanned
by \(\mathopen{:}J^{\{e_\vartheta\}}{}^r\mathclose{:}\).  Since $\xi_n$
is nonzero, this proves \eqref{e6:eq:BRST-image-s-n}, with \(c_n\ne0\).
\end{proof}

\begin{remark}
The constants $c_n$ depend on the choices of Chevalley root vectors and on the
normalization of the BRST representatives.  Their values are irrelevant: replacing
a generator of a vertex-algebra ideal by a nonzero scalar multiple does not change
the ideal.  Thus no additional normalization calculation is needed later.
\end{remark}

\subsubsection{Exact reduction sequences}

Let $I_n=\langle s_n\rangle$ as in Section~\ref{e6:sec:singular-vectors}.  Applying the exactness of minimal reduction recalled in
Subsection~\ref{subsec:common-DS} to the short exact sequence
\[
  0\longrightarrow I_n
   \longrightarrow V^{k_n}(E_6)
   \longrightarrow \widetilde V_n
   \longrightarrow0
\]
gives an exact sequence of
$\mathcal W^{k_n}(E_6,f_\theta)$-modules
\begin{equation}\label{e6:eq:DS-exact-sequence-I-n}
  0\longrightarrow H^0_{\mathrm{DS},f_\theta}(I_n)
   \longrightarrow \mathcal W^{k_n}(E_6,f_\theta)
   \longrightarrow
   H^0_{\mathrm{DS},f_\theta}(\widetilde V_n)
   \longrightarrow0.
\end{equation}
The first term is a vertex-algebra ideal of the universal minimal $W$-algebra.
By Proposition~\ref{e6:prop:BRST-image-s-n}, it contains the normally ordered power
$\mathopen{:}J^{\{e_\vartheta\}}{}^{n+1}\mathclose{:}$.  We therefore introduce
\begin{equation}\label{e6:eq:def-W-tilde-n}
  \widetilde{\mathcal W}_n
  =\mathcal W^{k_n}(E_6,f_\theta)
   \Big/
   \left\langle
     \mathopen{:}J^{\{e_\vartheta\}}{}^{\,n+1}\mathclose{:}
   \right\rangle,
  \qquad n=1,2.
\end{equation}

\begin{theorem}\label{e6:thm:canonical-surjection-reduction}
For $n=1,2$, there is a canonical surjective homomorphism of conformal vertex
algebras
\begin{equation}\label{e6:eq:rho-n}
  \rho_n:\widetilde{\mathcal W}_n
  \twoheadrightarrow
  H^0_{\mathrm{DS},f_\theta}(\widetilde V_n).
\end{equation}
The target is nonzero.  More precisely, the natural affine quotient map
$\pi_n$ in \eqref{e6:eq:V-tilde-to-simple} induces a surjection
\begin{equation}\label{e6:eq:reduced-map-to-simple-W}
  H^0_{\mathrm{DS},f_\theta}(\widetilde V_n)
  \twoheadrightarrow
  \mathcal W_{k_n}(E_6,f_\theta),
\end{equation}
and the vertex algebra on the right is nonzero and simple.
\end{theorem}

\begin{proof}
The inclusion
\[
  \left\langle
    \mathopen{:}J^{\{e_\vartheta\}}{}^{n+1}\mathclose{:}
  \right\rangle
  \subset H^0_{\mathrm{DS},f_\theta}(I_n)
\]
follows from Proposition~\ref{e6:prop:BRST-image-s-n}.  Hence the quotient map in
\eqref{e6:eq:DS-exact-sequence-I-n} factors through the quotient
\eqref{e6:eq:def-W-tilde-n}, giving the surjection \eqref{e6:eq:rho-n}.

Next apply the exact reduction functor to the surjection
$\pi_n:\widetilde V_n\twoheadrightarrow L_{k_n}(E_6)$.  This gives
\eqref{e6:eq:reduced-map-to-simple-W}.  By
Corollary~\ref{e6:cor:negative-level-nonvanishing},
$\mathcal W_{k_n}(E_6,f_\theta)$ is nonzero and simple.  Therefore both
$H^0_{\mathrm{DS},f_\theta}(\widetilde V_n)$ and
$\widetilde{\mathcal W}_n$ are nonzero.
\end{proof}

\begin{remark}
Theorem~\ref{e6:thm:canonical-surjection-reduction} uses only the inclusion of the
explicit class \eqref{e6:eq:BRST-image-s-n} in
$H^0_{\mathrm{DS},f_\theta}(I_n)$.  Lemma~\ref{lem:common-reduced-singular-generator}
also supplies the cyclicity of this reduced highest-weight submodule.  In the
next section we use that stronger conclusion to identify the kernel exactly;
the subsequent simplicity proof is independent and uses the Ramond Zhu
algebra and Li spectral flow.
\end{remark}

\subsubsection{The integrable affine \texorpdfstring{$A_5$}{A5} subalgebra}

At $k=k_n=n-3$, the affine level of the weight-one
$\mathfrak g^\natural\cong\mathfrak{sl}_6$ currents is $k^\natural=n$ by
\eqref{e6:eq:natural-level}.  Thus there is a canonical homomorphism
\begin{equation}\label{e6:eq:Vn-A5-to-Wtilde}
  V^n(\mathfrak{sl}_6)\longrightarrow\widetilde{\mathcal W}_n,
  \qquad
  e_\vartheta(-1)\mathbf1
  \longmapsto J^{\{e_\vartheta\}}.
\end{equation}
At the positive integral level $n$, the maximal graded ideal of
$V^n(\mathfrak{sl}_6)$ is generated by the standard integrability singular vector
$e_\vartheta(-1)^{n+1}\mathbf1$; see, for example, \cite{Kac1998}.  The defining
relation of $\widetilde{\mathcal W}_n$ therefore makes
\eqref{e6:eq:Vn-A5-to-Wtilde} factor through the simple integrable affine vertex
algebra $L_n(\mathfrak{sl}_6)$.

\begin{corollary}
For $n=1,2$, the weight-one currents define an embedding
\begin{equation}\label{e6:eq:Ln-A5-embedding}
  L_n(\mathfrak{sl}_6)
  \hookrightarrow
  \widetilde{\mathcal W}_n.
\end{equation}
The same integrable affine vertex algebra embeds into every nonzero quotient of
$\widetilde{\mathcal W}_n$.
\end{corollary}

\begin{proof}
The factor map from $L_n(\mathfrak{sl}_6)$ to
$\widetilde{\mathcal W}_n$ is nonzero because it sends the vacuum to the vacuum.
Since $L_n(\mathfrak{sl}_6)$ is simple, the map is injective.  If
$\widetilde{\mathcal W}_n\twoheadrightarrow A$ is a nonzero vertex-algebra
quotient, its restriction to $L_n(\mathfrak{sl}_6)$ is again unital and hence
nonzero, so it is injective for the same reason.
\end{proof}

The embeddings \eqref{e6:eq:Ln-A5-embedding} provide the integrable affine
$A_5$ structure needed for the spectral-flow analysis in the next section.

\subsection{Ramond reduction, spectral flow, and simplicity}
\label{e6:sec:Ramond-spectral-flow-simplicity}

We now prove that the two explicit quotients
$\widetilde{\mathcal W}_n$, $n=1,2$, are simple.  The proof has four parts.  We
first strengthen Theorem~\ref{e6:thm:canonical-surjection-reduction} to an
isomorphism and deduce lisse-ness.  We then derive a Casimir trace identity in
the Ramond Zhu algebra, determine the relevant finite-dimensional
$A_5$-weights, and finally use Li spectral flow to force every nonzero ideal to
contain the vacuum.

\subsubsection{Exact identification and lisse-ness}
\label{e6:subsec:exact-identification-lisse}

Recall that $I_n\subset V^{k_n}(E_6)$ is the vertex-algebra ideal generated by
$s_n$.  In the proof of Proposition~\ref{e6:prop:BRST-image-s-n} we verified
\[
 \widehat\lambda_n(\alpha_0^\vee)=-n-5<0.
\]
Hence Lemma~\ref{lem:common-reduced-singular-generator} shows that
$H^0_{\mathrm{DS},f_\theta}(I_n)$ is cyclic, generated by
$[s_n]_{\mathrm{DS}}$.  Proposition~\ref{e6:prop:BRST-image-s-n} therefore gives
\begin{equation*}
  H^0_{\mathrm{DS},f_\theta}(I_n)
  =
  \left\langle
    \mathopen{:}J^{\{e_\vartheta\}}{}^{\,n+1}\mathclose{:}
  \right\rangle.
\end{equation*}

\begin{proposition}\label{e6:prop:exact-reduction-Wtilde}
For $n=1,2$, the canonical homomorphism $\rho_n$ in \eqref{e6:eq:rho-n} is an
isomorphism:
\begin{equation}\label{e6:eq:exact-reduction-Wtilde}
  H^0_{\mathrm{DS},f_\theta}(\widetilde V_n)
  \cong
  \widetilde{\mathcal W}_n.
\end{equation}
Moreover, $\widetilde{\mathcal W}_n$ is nonzero and lisse.
\end{proposition}

\begin{proof}
The cyclicity argument preceding the proposition and
Proposition~\ref{e6:prop:BRST-image-s-n} verify the hypotheses of
Proposition~\ref{prop:common-exact-generated-quotient}; hence
\eqref{e6:eq:exact-reduction-Wtilde} holds.  Furthermore,
\cite[Proposition~5.2]{ArakawaMoreau2018} gives
\begin{equation*}
 X_{\widetilde V_n}=\overline{\mathbb O}_{\min},
 \qquad n=1,2,
\end{equation*}
and \eqref{eq:common-associated-variety} makes the reduction lisse.
Nonvanishing follows from Theorem~\ref{e6:thm:canonical-surjection-reduction}.
\end{proof}

Let $\sigma_{\mathrm R}$ denote the order-two Ramond automorphism of
$\widetilde{\mathcal W}_n$ defined on the standard generators by
\begin{equation}\label{e6:eq:Ramond-automorphism}
  \sigma_{\mathrm R}\bigl(J^{\{a\}}\bigr)=J^{\{a\}},
  \qquad
  \sigma_{\mathrm R}\bigl(G^{\{v\}}\bigr)=-G^{\{v\}},
  \qquad
  \sigma_{\mathrm R}(\omega)=\omega.
\end{equation}
Write
\begin{equation*}
  B_n=A_{\sigma_{\mathrm R}}(\widetilde{\mathcal W}_n)
\end{equation*}
for the Ramond Zhu algebra.  This is the Hamiltonian-twisted Zhu algebra
for the conformal Hamiltonian, since
$\sigma_{\mathrm R}=\exp(2\pi iL_0)$ on the standard generators.  Let
\[
  R_{\widetilde{\mathcal W}_n}
  =\widetilde{\mathcal W}_n/C_2(\widetilde{\mathcal W}_n)
\]
be its $C_2$-Poisson algebra.  Lisse-ness implies that
$R_{\widetilde{\mathcal W}_n}$ is finite-dimensional.  For the grading
induced by the conformal Hamiltonian, the zero-parameter twisted Zhu--Poisson
algebra is precisely $R_{\widetilde{\mathcal W}_n}$, and
\cite[Proposition~2.17(c)]{DeSoleKac2006} gives a canonical surjection
\begin{equation*}
  R_{\widetilde{\mathcal W}_n}
  \twoheadrightarrow
  \operatorname{gr}B_n.
\end{equation*}
Thus $\operatorname{gr}B_n$ is finite-dimensional.  Since the standard Zhu
filtration is exhaustive and separated, it follows that $B_n$ itself is
finite-dimensional.
We denote by
\[
  L=[\omega]_{\mathrm R}\in B_n
\]
the central conformal element.

\subsubsection{The Ramond Zhu relation and the Casimir trace identity}

Put
\[
  \mathfrak a=\mathfrak g^\natural\cong\mathfrak{sl}_6,
  \qquad
  U=\mathfrak g_{-\frac12}\cong L_{A_5}(\varpi_3).
\]
On $U$ we use the nondegenerate $\mathfrak a$-invariant skew form
\begin{equation}\label{e6:eq:symplectic-form-minus-half}
  \langle u,v\rangle=(e_\theta\mid[u,v]),
  \qquad u,v\in U.
\end{equation}
Choose dual bases $\{a_i\}_{i=1}^{35}$ and $\{a^i\}_{i=1}^{35}$ of
$\mathfrak a$ and set
\begin{equation*}
  \Omega_{A_5}=\sum_{i=1}^{35}a_i a^i\in U(\mathfrak a).
\end{equation*}
For $u,v\in U$, define
\begin{equation}\label{e6:eq:def-Q-uv}
  Q(u,v)
  =\sum_{i,j=1}^{35}
    \left\langle[a_i,u],[v,a_j]\right\rangle
    \bigl(a^i a^j+a^j a^i\bigr).
\end{equation}

The Ramond Zhu relation for a universal minimal $W$-algebra is given in
\cite[Equation~(5.1)]{KacMosenederFrajriaPapi2025}.  For an ordinary simple Lie
algebra the same formula is interpreted in the ordinary finite $W$-algebra, by
\cite[Remark~5.8]{KacMosenederFrajriaPapi2025}.  Combining this formula with
$p_{E_6}(k)=(k+3)(k+4)$ from
\cite[Table~4]{AdamovicKacMosenederFrajriaPapiPerse2018} gives, after passage to $B_n$,
\begin{equation*}
  [u,v]
  =\langle u,v\rangle
   \left(
     \Omega_{A_5}
     -2(k_n+12)L
     -\frac12(k_n+3)(k_n+4)
   \right)
   +Q(u,v).
\end{equation*}
Here $u,v$ denote the Ramond Zhu classes of the fields $G^{\{u\}}$ and
$G^{\{v\}}$.  By the cited remark, the bracket is the ordinary associative
commutator in the finite $W$-algebra.  Consequently,
\begin{equation*}
  \operatorname{tr}_M[\rho(u),\rho(v)]=0
\end{equation*}
for every finite-dimensional $B_n$-module $M$.

Let $M$ be such a module, and write
\begin{equation*}
  N_M=\dim M,
  \qquad
  \tau_M=\operatorname{tr}_M\rho(\Omega_{A_5}).
\end{equation*}
Since $\mathfrak{sl}_6$ is simple, there is a scalar $I_M$ such that
\begin{equation}\label{e6:eq:trace-form-index}
  \operatorname{tr}_M\bigl(\rho(a)\rho(b)\bigr)
  =I_M(a\mid b),
  \qquad a,b\in\mathfrak a.
\end{equation}
Taking the trace of the Casimir gives
\begin{equation*}
  \tau_M=35I_M.
\end{equation*}
The Casimir eigenvalue on the $20$-dimensional module
$U=L_{A_5}(\varpi_3)$ is
\begin{equation}\label{e6:eq:Casimir-omega3-A5}
  c_2(\varpi_3)
  =(\varpi_3\mid\varpi_3+2\rho_{A_5})
  =\frac{21}{2}.
\end{equation}

\begin{lemma}\label{e6:lem:trace-Q-E6}
For every finite-dimensional $B_n$-module $M$ and all $u,v\in U$,
\begin{equation}\label{e6:eq:trace-Q-E6}
  \operatorname{tr}_M\rho(Q(u,v))
  =\frac35\tau_M\langle u,v\rangle.
\end{equation}
\end{lemma}

\begin{proof}
Taking the trace in \eqref{e6:eq:def-Q-uv} and using
\eqref{e6:eq:trace-form-index} yields
\begin{align*}
  \operatorname{tr}_M\rho(Q(u,v))
  &=2I_M\sum_{i=1}^{35}
    \left\langle[a_i,u],[v,a^i]\right\rangle.
\end{align*}
The invariance of \eqref{e6:eq:symplectic-form-minus-half} converts the sum into
the action of the quadratic Casimir on $U$.  Hence, by
\eqref{e6:eq:Casimir-omega3-A5},
\[
  \sum_{i=1}^{35}
  \left\langle[a_i,u],[v,a^i]\right\rangle
  =\frac{21}{2}\langle u,v\rangle.
\]
Therefore
\[
  \operatorname{tr}_M\rho(Q(u,v))
  =21I_M\langle u,v\rangle
  =\frac{21}{35}\tau_M\langle u,v\rangle,
\]
which is \eqref{e6:eq:trace-Q-E6}.
\end{proof}

\begin{proposition}[Casimir trace identity]\label{e6:prop:E6-Casimir-trace-identity}
Let $M$ be a nonzero finite-dimensional $B_n$-module on which $L$ acts as a
scalar $h$.  Then
\begin{equation}\label{e6:eq:E6-Casimir-trace-identity}
  \frac{\operatorname{tr}_M\rho(\Omega_{A_5})}{\dim M}
  =\frac58
   \left(
     2(k_n+12)h
     +\frac12(k_n+3)(k_n+4)
   \right).
\end{equation}
Equivalently,
\begin{align*}
  n=1:\quad
  \frac{\operatorname{tr}_M\rho(\Omega_{A_5})}{\dim M}
  &=\frac58(20h+1),\\
  n=2:\quad
  \frac{\operatorname{tr}_M\rho(\Omega_{A_5})}{\dim M}
  &=\frac58(22h+3).
\end{align*}
\end{proposition}

\begin{proof}
Lemma~\ref{e6:lem:trace-Q-E6} gives $\gamma=3/5$.  Applying
Proposition~\ref{prop:common-Ramond-Casimir-trace} with
$h^\vee(E_6)=12$ and $p_{E_6}(k)=(k+3)(k+4)$ gives
\[
 \frac85\frac{\tau_M}{N_M}
 =2(k_n+12)h+\frac12(k_n+3)(k_n+4),
\]
which is \eqref{e6:eq:E6-Casimir-trace-identity}.  The two displayed
specializations follow from $k_1=-2$ and $k_2=-1$.
\end{proof}

\subsubsection{Allowed finite-dimensional \texorpdfstring{$A_5$}{A5}-weights}

The defining relation of $\widetilde{\mathcal W}_n$ becomes
\begin{equation*}
  e_\vartheta^{\,n+1}=0
  \qquad\text{in }B_n.
\end{equation*}
Indeed, for a weight-one current the Zhu class of a normally ordered power is
the corresponding associative power.  Hence every irreducible
$\mathfrak{sl}_6$-constituent $L_{A_5}(\mu)$ of a finite-dimensional
$B_n$-module satisfies
\begin{equation*}
  \langle\mu,\vartheta^\vee\rangle\le n.
\end{equation*}
Writing $\mu=\sum_{i=1}^{5}m_i\varpi_i$, this is simply
\begin{equation*}
  m_1+\cdots+m_5\le n.
\end{equation*}

For later use, let
\begin{equation*}
  c_2(\mu)=(\mu\mid\mu+2\rho_{A_5})
\end{equation*}
be the Casimir eigenvalue on $L_{A_5}(\mu)$.  The inverse Cartan matrix in our
normalization is
\begin{equation}\label{e6:eq:A5-inverse-Cartan}
  A_{A_5}^{-1}
  =\frac16
  \begin{pmatrix}
    5&4&3&2&1\\
    4&8&6&4&2\\
    3&6&9&6&3\\
    2&4&6&8&4\\
    1&2&3&4&5
  \end{pmatrix}.
\end{equation}
A direct substitution in
$c_2(\mu)=\boldsymbol m^{\mathsf T}A_{A_5}^{-1}
(\boldsymbol m+2\boldsymbol 1)$ gives the following sharp bounds.  The full
polynomial and the complete level-two table are recorded in
Appendix~\ref{e6:app:A5-Casimir-table}.

\begin{lemma}\label{e6:lem:A5-Casimir-bounds}
Let $M$ be a finite-dimensional $B_n$-module.
\begin{enumerate}
  \item For $n=1$,
  \begin{equation*}
    c_2(\mu)\le\frac{21}{2},
  \end{equation*}
  with equality only for $\mu=\varpi_3$.  If no copy of
  $L_{A_5}(\varpi_3)$ occurs, then
  \begin{equation*}
    c_2(\mu)\le\frac{28}{3}.
  \end{equation*}

  \item For $n=2$,
  \begin{equation*}
    c_2(\mu)\le24,
  \end{equation*}
  with equality only for $\mu=2\varpi_3$.  If no copy of
  $L_{A_5}(2\varpi_3)$ occurs, then
  \begin{equation*}
    c_2(\mu)\le\frac{131}{6}.
  \end{equation*}
\end{enumerate}
The same bounds hold for the dimension-weighted average
$\operatorname{tr}_M\rho(\Omega_{A_5})/\dim M$.
\end{lemma}

\begin{proof}
For $n=1$, the permitted weights are
$0,\varpi_1,\ldots,\varpi_5$, and their Casimir values are
\[
  0,\quad
  \frac{35}{6},\quad
  \frac{28}{3},\quad
  \frac{21}{2},\quad
  \frac{28}{3},\quad
  \frac{35}{6}.
\]

For $n=2$, the additional values, grouped by the Dynkin-diagram involution, are
\begin{equation*}
\begin{array}{c|c}
\mu & c_2(\mu)\\ \hline
2\varpi_3 & 24\\
\varpi_2+\varpi_3,\ \varpi_3+\varpi_4 & 131/6\\
2\varpi_2,\ 2\varpi_4 & 64/3\\
\varpi_2+\varpi_4 & 20\\
\varpi_1+\varpi_3,\ \varpi_3+\varpi_5 & 52/3\\
\varpi_1+\varpi_2,\ \varpi_4+\varpi_5 & 33/2\\
\varpi_1+\varpi_4,\ \varpi_2+\varpi_5 & 95/6\\
2\varpi_1,\ 2\varpi_5 & 40/3\\
\varpi_1+\varpi_5 & 12.
\end{array}
\end{equation*}
Together with the level-one values, this proves all assertions.  Since
finite-dimensional $\mathfrak{sl}_6$-modules are completely reducible, a trace
average is a dimension-weighted average of these Casimir eigenvalues.
\end{proof}

\subsubsection{Li spectral flow and the extremal vacuum line}

Identify weights and coweights of $A_5$ using the normalized invariant form and
set
\begin{equation*}
  \eta=\varpi_3^\vee\in\mathfrak h_{A_5},
  \qquad
  H=J^{\{\eta\}}\in(\widetilde{\mathcal W}_n)_1.
\end{equation*}
By \eqref{e6:eq:A5-inverse-Cartan},
\begin{equation*}
  (\eta\mid\eta)=\frac32.
\end{equation*}
Since the affine $A_5$ level is $n$,
\begin{equation*}
  H_{(1)}H=\frac{3n}{2}\mathbf1.
\end{equation*}

We use Li's delta operator in the convention
\begin{equation}\label{e6:eq:Li-delta-E6}
  \Delta(H,z)
  =z^{H_{(0)}}
   \exp\left(
     \sum_{m\ge1}\frac{H_{(m)}}{-m}(-z)^{-m}
   \right).
\end{equation}
For a $\widetilde{\mathcal W}_n$-module $(M,Y_M)$, set
\begin{equation*}
  Y_M^{\mathrm R}(a,z)
  =Y_M\bigl(\Delta(H,z)a,z\bigr).
\end{equation*}
This is a module twisted by
$\exp(2\pi iH_{(0)})$; see \cite{Li2012}.

\begin{lemma}
The automorphism $\exp(2\pi iH_{(0)})$ is the Ramond automorphism
$\sigma_{\mathrm R}$ in \eqref{e6:eq:Ramond-automorphism}.
\end{lemma}

\begin{proof}
Every root $\alpha$ of $A_5$ has $\alpha(\eta)\in\mathbb Z$, so the current
fields are fixed.  Every weight $\mu$ of the minuscule module
$L_{A_5}(\varpi_3)$ belongs to $\varpi_3+Q(A_5)$ and therefore
\[
  \mu(\eta)
  \equiv(\varpi_3\mid\varpi_3)
  =\frac32
  \pmod{\mathbb Z}.
\]
Thus the exponential acts by $-1$ on all fields $G^{\{v\}}$ and fixes the
conformal vector.
\end{proof}

Apply this twist to the adjoint module and write
$\widetilde{\mathcal W}_n^{\mathrm R}$ for the resulting Ramond module.  A
direct calculation from \eqref{e6:eq:Li-delta-E6} gives
\begin{align}
  L_0^{\mathrm R}
  &=L_0+H_0+\frac{3n}{4},
  \notag\\
  J_0^{\{h\},\mathrm R}
  &=J_0^{\{h\}}+n(\eta\mid h),
  \qquad h\in\mathfrak h_{A_5}.
  \label{e6:eq:Li-shift-Cartan-E6}
\end{align}

\begin{proposition}\label{e6:prop:E6-Ramond-spectrum-vacuum-line}
For $n=1,2$, the operator $L_0^{\mathrm R}$ acts semisimply on
$\widetilde{\mathcal W}_n^{\mathrm R}$, and
\begin{equation}\label{e6:eq:E6-Ramond-spectrum}
  \operatorname{Spec}_{\widetilde{\mathcal W}_n^{\mathrm R}}L_0^{\mathrm R}
  \subset
  \frac{3n}{4}+\mathbb Z_{\ge0}.
\end{equation}
Moreover,
\begin{equation}\label{e6:eq:E6-extremal-vacuum-line}
  \left(
    \widetilde{\mathcal W}_n^{\mathrm R}
    \left[\frac{3n}{4}\right]
  \right)_{n\varpi_3}
  =\mathbb C\mathbf1,
\end{equation}
where the subscript refers to the shifted $A_5$ action in
\eqref{e6:eq:Li-shift-Cartan-E6}.
\end{proposition}

\begin{proof}
Every ordinary conformal homogeneous subspace of
$\widetilde{\mathcal W}_n$ is finite-dimensional and stable under the
$A_5$-zero-mode action.  Since finite-dimensional $A_5$-modules are
completely reducible, $H_0$ acts semisimply on each such subspace and
commutes with $L_0$.  Hence
$L_0^{\mathrm R}=L_0+H_0+3n/4$ acts semisimply on the twisted adjoint
module.

The quotient $\widetilde{\mathcal W}_n$ is spanned by normally ordered monomials
in derivatives of the fields
\[
  J^{\{a_\alpha\}},\qquad
  G^{\{v_\mu\}},\qquad
  \omega,
\]
chosen to be $H_0$-eigenvectors.  Since $\eta=\varpi_3^\vee$ is minuscule,
\[
  \alpha(\eta)\in\{-1,0,1\}
\]
for roots of $A_5$.  The weights of $L_{A_5}(\varpi_3)$ have
\[
  \mu(\eta)\in
  \left\{-\frac32,-\frac12,\frac12,\frac32\right\}.
\]
The charge multiplicities and the Li-shift calculation are displayed explicitly
in Appendix~\ref{e6:app:Li-charge-calculation}.  Consequently, the contributions of
$\partial^rJ^{\{a_\alpha\}}$,
$\partial^sG^{\{v_\mu\}}$, and $\partial^t\omega$ to
$L_0^{\mathrm R}-3n/4=L_0+H_0$ are respectively
\[
  1+r+\alpha(\eta),
  \qquad
  \frac32+s+\mu(\eta),
  \qquad
  2+t,
\]
all of which are nonnegative integers.  This proves
\eqref{e6:eq:E6-Ramond-spectrum}.

The vacuum has twisted conformal weight $3n/4$ and shifted $A_5$-weight
$n\varpi_3$.  A nonempty spanning monomial can have twisted conformal weight
$3n/4$ only if every one of its factors has zero contribution.  Hence every
current factor has root charge $-1$, every $G$-factor has charge $-3/2$, and no
Virasoro factor occurs.  The total original $H_0$-charge is therefore strictly
negative.  Such a monomial cannot have the same shifted $A_5$-weight as the
vacuum, because equality with $n\varpi_3$ would force its original total weight,
and hence its $\eta$-charge, to be zero.  Thus only the empty monomial contributes
to the indicated weight space.
\end{proof}

\subsubsection{Simplicity of the two reduced quotients}

\begin{theorem}\label{e6:thm:E6-Wtilde-simple}
For $n=1,2$, the vertex algebra $\widetilde{\mathcal W}_n$ is simple.  More
precisely,
\begin{equation*}
  \widetilde{\mathcal W}_n
  \cong
  \mathcal W_{k_n}(E_6,f_\theta).
\end{equation*}
Consequently,
\begin{equation}\label{e6:eq:E6-exact-reduction-simple-W}
  H^0_{\mathrm{DS},f_\theta}(\widetilde V_n)
  \cong
  \mathcal W_{k_n}(E_6,f_\theta).
\end{equation}
\end{theorem}

\begin{proof}
For each $n=1,2$, apply
Proposition~\ref{prop:common-Casimir-gap-simplicity} to the natural nonzero
surjection
$\widetilde{\mathcal W}_n\twoheadrightarrow
\mathcal W_{k_n}(E_6,f_\theta)$.  The spectral lattice and extremal line are
provided by Proposition~\ref{e6:prop:E6-Ramond-spectrum-vacuum-line}, with
\[
 h_0=\frac{3n}{4},
 \qquad
 \lambda_0=n\varpi_3.
\]
The Ramond Zhu algebra $B_n$ is finite-dimensional by the lisse argument in
Subsubsection~\ref{e6:subsec:exact-identification-lisse}.  The remaining numerical data
are
\[
\begin{array}{c|c|c|c|c}
 n&A_n(h)&C_{\max}&A_n(h_0)&C_{\mathrm{next}}\\ \hline
 1&\frac58(20h+1)&\frac{21}{2}&10&\frac{28}{3}\\[1mm]
 2&\frac58(22h+3)&24&\frac{45}{2}&\frac{131}{6}
\end{array}
\]
by Proposition~\ref{e6:prop:E6-Casimir-trace-identity} and
Lemma~\ref{e6:lem:A5-Casimir-bounds}.  In both cases
$A_n(h_0+1)>C_{\max}$, so a lowest ideal energy must equal $h_0$.
Moreover, $A_n(h_0)>C_{\mathrm{next}}$; hence every simple Zhu quotient at
the ground energy contains the unique maximal-Casimir constituent
$L_{A_5}(n\varpi_3)$.  Its shifted highest line is
$\mathbb C\mathbf1$ by
\eqref{e6:eq:E6-extremal-vacuum-line}.  The common criterion therefore
proves simplicity and the first isomorphism.  Combining it with
Proposition~\ref{e6:prop:exact-reduction-Wtilde} gives
\eqref{e6:eq:E6-exact-reduction-simple-W}.
\end{proof}

The isomorphism \eqref{e6:eq:E6-exact-reduction-simple-W} is precisely the
hypothesis needed for Proposition~\ref{prop:common-affine-lifting}.

\subsection{Maximal ideals at levels \texorpdfstring{$-2$}{-2} and
\texorpdfstring{$-1$}{-1}}

We now apply the common affine-lifting result to $\widetilde V_n$.

\subsubsection{Uniform affine lifting}

For $n=1,2$, recall the canonical surjection
\[
 \pi_n:\widetilde V_n\twoheadrightarrow L_{k_n}(E_6),
 \qquad k_n=n-3.
\]
Theorem~\ref{e6:thm:E6-Wtilde-simple} and
Proposition~\ref{e6:prop:exact-reduction-Wtilde} show that the induced map on
minimal reductions is an isomorphism.  Hence
Proposition~\ref{prop:common-affine-lifting} applies.

\begin{theorem}\label{e6:thm:E6-affine-quotients-simple}
For $n=1,2$, the homomorphism $\pi_n$ is an isomorphism.  Equivalently,
\begin{equation*}
  V^{n-3}(E_6)\big/\langle u^{n+1}\rangle
  \cong
  L_{n-3}(E_6).
\end{equation*}
\end{theorem}

\begin{proof}
Apply Proposition~\ref{prop:common-affine-lifting} to
$\pi_n:\widetilde V_n\twoheadrightarrow L_{k_n}(E_6)$.
\end{proof}

We can now state the result in the explicit form announced in the introduction.

\begin{corollary}[Main theorem]
Let $u=\sigma(w)$ be the quadratic highest-weight state defined in
\eqref{e6:eq:def-u}.  Then
\begin{align*}
  \operatorname{Rad}V^{-2}(E_6)
  &=\langle u^2\rangle,
  &
  V^{-2}(E_6)\big/\langle u^2\rangle
  &\cong L_{-2}(E_6),\\
  \operatorname{Rad}V^{-1}(E_6)
  &=\langle u^3\rangle,
  &
  V^{-1}(E_6)\big/\langle u^3\rangle
  &\cong L_{-1}(E_6).
\end{align*}
In particular, the affine singular vectors $u^2$ and $u^3$ generate the full
maximal ideals at levels $-2$ and $-1$, respectively.
\end{corollary}

\begin{proof}
For $n=1$, Theorem~\ref{e6:thm:E6-affine-quotients-simple} gives
\[
  V^{-2}(E_6)/\langle u^2\rangle\cong L_{-2}(E_6),
\]
so $\langle u^2\rangle$ is the unique maximal graded ideal.  The case $n=2$
is identical.
\end{proof}

\begin{remark}
The proof does not require an exhaustive search for further singular vectors or
calculations in higher PBW degree.  Once the reduced quotients are known to be
simple, exactness together with
Lemma~\ref{lem:common-DS-detection} excludes every possible remaining
affine ideal at once.
\end{remark}
\section{Type \texorpdfstring{$E_7$}{E7}}

\subsection{Type data}
\label{e7:sec:preliminaries}

Let $\mathfrak g=E_7$ and use the minimal grading defined by
$x_\theta=\theta^\vee/2$.  Then
\begin{equation*}
 \mathfrak g^\natural\cong D_6,
 \qquad
 U:=\mathfrak g_{-\frac12}\cong L_{D_6}(\omega_6),
\end{equation*}
where the choice of the half-spin node $\omega_6$ is fixed throughout.  On
$U$ we use the nondegenerate invariant form
\begin{equation}
 \langle u,v\rangle=(e_\theta\mid[u,v]).
 \label{e7:eq:E7-half-spin-symplectic-form}
\end{equation}
The basic numerical data are
\begin{equation}
 h^\vee(E_7)=18,
 \qquad h^\vee(D_6)=10,
 \qquad \dim D_6=66,
 \qquad c_2(\omega_6)=\frac{33}{2}.
 \label{e7:eq:E7-half-spin-Casimir-data}
\end{equation}
For later spectral calculations we use the standard orthogonal coordinates
\begin{equation}
 \omega_5=\frac12(\varepsilon_1+\cdots+\varepsilon_5-\varepsilon_6),
 \qquad
 \omega_6=\frac12(\varepsilon_1+\cdots+\varepsilon_5+\varepsilon_6).
 \label{e7:eq:D6-spin-weight-coordinates}
\end{equation}
The full $D_6$ weight and Casimir data are given in
Appendix~\ref{e7:app:D6-weight-Casimir-data}; see also
\cite{KacWakimoto2004,KacWakimoto2005,
ArakawaCreutzigKawasetsuLinshaw2017}.

\subsubsection{Level, Ramond, and spectral-flow data}

All formal conventions are those of Section~\ref{sec:common-framework}.  For
\begin{equation*}
 k_n=n-4,\qquad n=1,2,3,
\end{equation*}
the $D_6$ currents have level $n$:
\begin{equation}
 [J^{\{a\}}{}_{\lambda}J^{\{b\}}]
 =J^{\{[a,b]\}}+n(a\mid b)\lambda.
 \label{e7:eq:E7-D6-current-relation-n}
\end{equation}
The generators are
\begin{equation}
 J^{\{a\}},\ a\in\mathfrak g^\natural;
 \qquad G^{\{u\}},\ u\in\mathfrak g_{-1/2};
 \qquad \omega.
 \label{e7:eq:E7-minimal-W-strong-generators}
\end{equation}
Since $p_{E_7}(k)=(k+4)(k+6)$, the Ramond relation is
\begin{equation}
 [u,v]=\langle u,v\rangle
 \left(\Omega_{D_6}-2(n+14)L-\frac{n(n+2)}2\right)+Q(u,v).
 \label{e7:eq:E7-n-Ramond-Zhu-relation}
\end{equation}
\label{e7:subsec:E7-Li-delta-operator}
For $x=\omega_6^\vee$ and $H=J^{\{x\}}$, the Li twist satisfies
\begin{align}
 L_0^{\mathrm R}&=L_0+x_0+\frac{3n}{4},
 \label{e7:eq:E7-Li-Ramond-conformal-shift}\\
 J_0^{\{h\},\mathrm R}&=J_0^{\{h\}}+n(x\mid h).
 \label{e7:eq:E7-Li-Cartan-shift}
\end{align}

\subsection{The distinguished singular vectors and their minimal reductions}
\label{e7:sec:distinguished-singular-vectors}

We now introduce the affine singular vectors that determine the quotients
studied in this paper and identify their images under minimal
Drinfeld--Sokolov reduction.  The principal result of this section is the
explicit formula
\[
  [s_n]_{\mathrm{DS}}
  =c_n:J^{\{e_\vartheta\}}{}^{n+1}:,
  \qquad c_n\in\mathbb C^\times,
\]
which converts the affine problem into a single highest-root relation in the
universal minimal \(W\)-algebra.

\subsubsection{The distinguished affine singular vectors}

Recall the decomposition
\[
  S^2(\mathfrak g)
  =L_{E_7}(2\theta)\oplus\mathbb C\Omega\oplus W,
\]
where \(\Omega\) is the quadratic Casimir and, for \(\mathfrak g=E_7\),
\begin{equation}
  W\cong L_{E_7}(\theta+\vartheta).
  \label{e7:eq:E7-Joseph-module-W}
\end{equation}
Here \(\vartheta\) is the highest root of
\(\mathfrak g^\natural\cong D_6\), regarded as a weight of \(E_7\) by
extension from \(\mathfrak h^\natural\).  Fix a nonzero highest-weight vector
\[
  w\in W_{\theta+\vartheta}.
\]

Let
\begin{equation*}
  \sigma:S^2(\mathfrak g)\longrightarrow V^k(\mathfrak g)_2
\end{equation*}
be the standard PBW symmetrization map.  Thus, for \(a,b\in\mathfrak g\),
\[
  \sigma(ab)
  =\frac12\bigl(a(-1)b(-1)+b(-1)a(-1)\bigr)\mathbf1.
\]
For
\[
  n\in\{1,2,3\},
  \qquad
  k_n=n-4,
\]
define
\begin{equation*}
  s_n=\sigma(w)^{n+1}\in V^{k_n}(E_7),
\end{equation*}
where the power is taken in the PBW realization
\(V^{k_n}(E_7)\cong U(E_7[t^{-1}]t^{-1})\mathbf1\).  Let
\begin{equation*}
  N_n=U(\widehat{E}_7)s_n,
  \qquad
  Q_n=V^{k_n}(E_7)/N_n.
\end{equation*}
Since the affine currents strongly generate \(V^{k_n}(E_7)\), the
\(\widehat E_7\)-submodule \(N_n\) is also the vertex-algebra ideal generated
by \(s_n\).

\begin{proposition}[Arakawa--Moreau]

For every \(n\in\{1,2,3\}\), the vector \(s_n\) is a nonzero affine singular
vector of \(V^{k_n}(E_7)\).  Its conformal degree and finite
\(E_7\)-highest weight are
\begin{equation}
  \deg s_n=2(n+1),
  \qquad
  \operatorname{wt}_{E_7}(s_n)
  =(n+1)(\theta+\vartheta).
  \label{e7:eq:E7-sn-degree-and-weight}
\end{equation}
\end{proposition}

\begin{proof}
For the Deligne exceptional series outside type \(A\),
\cite[Theorem~4.2(1)(b)]{ArakawaMoreau2018} states that
\(\sigma(w)^{n+1}\) is singular precisely at
\[
  k=n-\frac{h^\vee}{6}-1.
\]
Since \(h^\vee(E_7)=18\), this level is \(k_n=n-4\).  The map \(\sigma\)
lands in conformal degree two and preserves the finite \(E_7\)-weight.
Equation \eqref{e7:eq:E7-sn-degree-and-weight} follows by additivity.
\end{proof}

The affine highest weight of the cyclic module \(N_n\), including the action
of the derivation \(D\), is
\begin{equation}
  \widehat\lambda_n
  =k_n\Lambda_0
   +(n+1)(\theta+\vartheta)
   -2(n+1)\delta.
  \label{e7:eq:E7-affine-highest-weight-lambda-n}
\end{equation}
We shall also use the shifted weight
\begin{equation*}
  \widehat\lambda_n^\circ
  =\widehat\lambda_n+2(n+1)\delta
  =k_n\Lambda_0+(n+1)(\theta+\vartheta),
\end{equation*}
which satisfies \(\widehat\lambda_n^\circ(D)=0\).  Let \(N_n^\circ\) denote the same
\(\widehat E_7'=E_7[t,t^{-1}]\oplus\mathbb CK\)-module as \(N_n\),
with the action of \(D\) shifted by \(2(n+1)\,\mathrm{id}\).  Then
\(N_n^\circ\) has highest weight \(\widehat\lambda_n^\circ\).
The BRST differential and the induced \(W\)-action depend only on the
\(\widehat E_7'\)-action; the shift of \(D\) merely changes the origin
of the external grading.  Hence there is a canonical identification
\begin{equation*}
  H^0_{\mathrm{DS},f_\theta}(N_n^\circ)
  \cong H^0_{\mathrm{DS},f_\theta}(N_n)
\end{equation*}
as ungraded \(W\)-modules.

\subsubsection{Nonvanishing of the reduced singular vector}

Let
\[
  M(\widehat\lambda_n^\circ)
\]
be the affine Verma module with highest weight
\(\widehat\lambda_n^\circ\).  There is a canonical surjection
\begin{equation*}
  M(\widehat\lambda_n^\circ)\twoheadrightarrow N_n^\circ
\end{equation*}
which sends the Verma highest-weight vector to \(s_n\).  By
\cite[Theorem~6.3]{KacWakimoto2004}, the minimal reduction of the source is a
\(W\)-Verma module generated by the BRST class of that highest-weight
vector.  Exactness therefore implies that
\begin{equation}
  H^0_{\mathrm{DS},f_\theta}(N_n)
  \quad\text{is generated by}\quad
  [s_n]_{\mathrm{DS}}.
  \label{e7:eq:E7-reduction-Nn-cyclic}
\end{equation}

We next prove that this cyclic vector is nonzero.

\begin{lemma}
\label{e7:lem:E7-reduced-singular-vector-nonzero}
For every \(n\in\{1,2,3\}\),
\begin{equation*}
  [s_n]_{\mathrm{DS}}\ne0.
\end{equation*}
\end{lemma}

\begin{proof}
The highest-weight module \(N_n\) has an irreducible highest-weight quotient
\[
  N_n\twoheadrightarrow L(\widehat\lambda_n).
\]
Since \(E_7\) is simply laced,
\[
  \alpha_0^\vee=K-\theta^\vee.
\]
Moreover, \(\vartheta\) belongs to the root system of
\(\mathfrak g^\natural\), and hence
\[
  \langle\vartheta,\theta^\vee\rangle=0.
\]
It follows from \eqref{e7:eq:E7-affine-highest-weight-lambda-n} that
\begin{align*}
  \widehat\lambda_n(\alpha_0^\vee)
  &=k_n-(n+1)
    \langle\theta+\vartheta,\theta^\vee\rangle
    \notag\\
  &=n-4-2(n+1)
    =-n-6<0.
\end{align*}
By the minimal-reduction nonvanishing criterion recalled in
\eqref{eq:common-DS-nonvanishing},
\begin{equation*}
  H^0_{\mathrm{DS},f_\theta}
  \bigl(L(\widehat\lambda_n)\bigr)\ne0.
\end{equation*}
Exactness applied to
\(N_n\twoheadrightarrow L(\widehat\lambda_n)\) gives a surjection
\[
  H^0_{\mathrm{DS},f_\theta}(N_n)
  \twoheadrightarrow
  H^0_{\mathrm{DS},f_\theta}
  \bigl(L(\widehat\lambda_n)\bigr).
\]
Thus \(H^0_{\mathrm{DS},f_\theta}(N_n)\ne0\).  In view of the cyclicity
statement \eqref{e7:eq:E7-reduction-Nn-cyclic}, its cyclic generator
\([s_n]_{\mathrm{DS}}\) cannot vanish.
\end{proof}

\subsubsection{The extremal line in the universal minimal
\texorpdfstring{\(W\)}{W}-algebra}

We determine the finite weight and conformal weight of the reduced vector.
Because \(\theta\) restricts trivially to \(\mathfrak h^\natural\), while
\(\vartheta\) is the highest root of \(D_6\),
\begin{equation}
  \operatorname{wt}_{D_6}
  \bigl([s_n]_{\mathrm{DS}}\bigr)
  =(n+1)\vartheta.
  \label{e7:eq:E7-reduced-singular-vector-D6-weight}
\end{equation}
The reduced conformal Hamiltonian acts on the affine highest-weight vector
by the affine degree minus the \(x_\theta\)-eigenvalue.  Since
\[
  \theta(x_\theta)=1,
  \qquad
  \vartheta(x_\theta)=0,
\]
we obtain
\begin{align}
  \Delta_W\bigl([s_n]_{\mathrm{DS}}\bigr)
  &=2(n+1)
    -(n+1)(\theta+\vartheta)(x_\theta)
    \notag\\
  &=n+1.
  \label{e7:eq:E7-reduced-singular-vector-conformal-weight}
\end{align}

The following PBW estimate identifies the simultaneous weight space
containing this vector.

\begin{lemma}[The extremal PBW line]
\label{e7:lem:E7-extremal-PBW-line}
For every \(n\ge0\) and every noncritical level \(k\),
\begin{equation*}
  \mathcal W^k(E_7,f_\theta)_{n+1}
  \bigl[(n+1)\vartheta\bigr]
  =
  \mathbb C
  \bigl(J^{\{e_\vartheta\}}_{-1}\bigr)^{n+1}\mathbf1.
\end{equation*}
\end{lemma}

\begin{proof}
Apply Proposition~\ref{prop:common-extremal-PBW-line} to
\(\mathfrak g^\natural\cong D_6\), with \(\eta=\vartheta\).  The current
weights are roots or zero, and equality in
\(\langle\mu,\vartheta^\vee\rangle\le2\) occurs only for
\(\mu=\vartheta\).  Moreover,
\(\mathfrak g_{-1/2}\cong L_{D_6}(\omega_6)\), and every one of its weights
\(\nu\) satisfies
\(\langle\nu,\vartheta^\vee\rangle\le
\langle\omega_6,\vartheta^\vee\rangle=1\).  The common proposition with
\(r=n+1\) gives the claim.
\end{proof}

\subsubsection{The exact reduced singular vector}

\begin{proposition}[Reduction of the distinguished singular vector]
\label{e7:prop:E7-reduction-distinguished-singular-vector}
For every \(n\in\{1,2,3\}\), there exists a nonzero scalar
\(c_n\in\mathbb C^\times\) such that
\begin{equation*}
  [s_n]_{\mathrm{DS}}
  =c_n
   \bigl(J^{\{e_\vartheta\}}_{-1}\bigr)^{n+1}\mathbf1
  =c_n:J^{\{e_\vartheta\}}{}^{n+1}:.
\end{equation*}
\end{proposition}

\begin{proof}
By Lemma~\ref{e7:lem:E7-reduced-singular-vector-nonzero}, the vector on the
left-hand side is nonzero.  Equations
\eqref{e7:eq:E7-reduced-singular-vector-D6-weight} and
\eqref{e7:eq:E7-reduced-singular-vector-conformal-weight} place it in the
one-dimensional simultaneous weight space described by
Lemma~\ref{e7:lem:E7-extremal-PBW-line}.  Hence it is a scalar multiple of the
stated highest-root-current power, and the scalar is nonzero.
\end{proof}

\subsubsection{The reduced quotients}

Define
\begin{equation}
  \widetilde{\mathcal W}_n
  =
  \mathcal W^{k_n}(E_7,f_\theta)
  \Big/
  \left\langle
    :J^{\{e_\vartheta\}}{}^{n+1}:
  \right\rangle.
  \label{e7:eq:E7-definition-Wtilde-n}
\end{equation}

\begin{proposition}
\label{e7:prop:E7-reduction-of-Qn}
For every \(n\in\{1,2,3\}\), there is a canonical isomorphism of conformal
vertex algebras
\begin{equation}
  H^0_{\mathrm{DS},f_\theta}(Q_n)
  \cong
  \widetilde{\mathcal W}_n.
  \label{e7:eq:E7-reduction-Qn-is-Wtilde}
\end{equation}
\end{proposition}

\begin{proof}
The reduction of $N_n$ is cyclic by
\eqref{e7:eq:E7-reduction-Nn-cyclic}, and
Proposition~\ref{e7:prop:E7-reduction-distinguished-singular-vector}
identifies its generator.  Proposition~\ref{prop:common-exact-generated-quotient}
applied with $r=1$ gives \eqref{e7:eq:E7-reduction-Qn-is-Wtilde}.
\end{proof}

The geometric properties of these quotients will be needed in the next
section.

\begin{corollary}

For every \(n\in\{1,2,3\}\),
\begin{equation}
  \widetilde{\mathcal W}_n\ne0,
  \qquad
  \widetilde{\mathcal W}_n\ \text{is lisse}.
  \label{e7:eq:E7-Wtilde-nonzero-lisse}
\end{equation}
Consequently, its Ramond Zhu algebra is finite-dimensional.
\end{corollary}

\begin{proof}
For \(E_7\), the module \(W\) in
\eqref{e7:eq:E7-Joseph-module-W} is irreducible, so the quotient \(Q_n\) is the
specific quotient denoted \(\widetilde V_{k_n}(E_7)\) in
\cite{ArakawaMoreau2018}.  By
\cite[Proposition~5.2]{ArakawaMoreau2018},
\begin{equation*}
  X_{Q_n}=\overline{\mathbb O_{\min}}.
\end{equation*}
The nonvanishing and lisse assertions now follow from
\cite[Theorem~6.1(3)]{ArakawaMoreau2018} together with
Proposition~\ref{e7:prop:E7-reduction-of-Qn}.  Finite-dimensionality of the
Ramond Zhu algebra follows from
\eqref{eq:common-lisse-Ramond-finite}.
\end{proof}

There is also a natural surjection
\begin{equation}
  \widetilde{\mathcal W}_n
  \twoheadrightarrow
  \mathcal W_{k_n}(E_7,f_\theta),
  \label{e7:eq:E7-Wtilde-to-simple-W}
\end{equation}
induced by the canonical map
\(Q_n\twoheadrightarrow L_{k_n}(E_7)\).  The main task of the later
simplicity argument will be to prove that
\eqref{e7:eq:E7-Wtilde-to-simple-W} is an isomorphism for
\(n=1,2,3\).

The relation in \eqref{e7:eq:E7-definition-Wtilde-n} will next be passed to the
Ramond Zhu algebra, where it yields a highest-root nilpotency bound and a
uniform Casimir trace identity.

\subsection{The Ramond Zhu algebra and a Casimir trace identity}
\label{e7:sec:Ramond-Casimir-identity}

In the preceding section we identified the minimal reduction of the affine
quotient $Q_n$ with the lisse vertex algebra
\[
  \widetilde{\mathcal W}_n
  =
  \mathcal W^{k_n}(E_7,f_\theta)
  \Big/
  \left\langle
    :J^{\{e_\vartheta\}}{}^{n+1}:
  \right\rangle,
  \qquad
  k_n=n-4,
  \quad n\in\{1,2,3\}.
\]
The purpose of this section is to extract from this single defining relation
finite-dimensional information that will later control ideals of
$\widetilde{\mathcal W}_n$.  We first pass the highest-root relation to the
Ramond Zhu algebra, and then take the trace of the quadratic Ramond
commutator relation.  The resulting Casimir identity is uniform in
$n=1,2,3$.

\subsubsection{Finite-dimensionality and the highest-root relation}

Set
\begin{equation*}
  B_n=A_{\mathrm R}(\widetilde{\mathcal W}_n).
\end{equation*}
By \eqref{e7:eq:E7-Wtilde-nonzero-lisse} and
\eqref{eq:common-lisse-Ramond-finite}, we have
\begin{equation*}
  0<\dim B_n<\infty.
\end{equation*}
The strict positivity follows from the nonzero vacuum class.

Fix a nonzero highest-root vector
\begin{equation*}
  e=e_\vartheta\in(D_6)_\vartheta,
  \qquad
  J=J^{\{e\}}.
\end{equation*}
The current relation \eqref{e7:eq:E7-D6-current-relation-n} gives
\begin{equation*}
  [J_\lambda J]
  =J^{\{[e,e]\}}+n(e\mid e)\lambda=0.
\end{equation*}
Indeed, $[e,e]=0$, while the invariant form pairs the root space
$(D_6)_\vartheta$ only with $(D_6)_{-\vartheta}$.  Equivalently,
\begin{equation}
  J_{(m)}J=0,
  \qquad m\ge0.
  \label{e7:eq:E7-highest-root-current-products-zero}
\end{equation}

The following elementary observation identifies the Zhu image of every
normal power of $J$.

\begin{lemma}
\label{e7:lem:E7-Zhu-image-normal-power}
For every integer $r\ge1$,
\begin{equation*}
  \bigl[:J^r:\bigr]_{\mathrm R}
  =[J]_{\mathrm R}^{\,r}
\end{equation*}
in the Ramond Zhu algebra of every conformal quotient in which the state
$:J^r:$ is defined.
\end{lemma}

\begin{proof}
The Ramond automorphism fixes $J$, and $J$ has conformal weight one.  Hence,
for every Ramond-invariant state $v$, its Zhu product is
\begin{equation}
  [J]_{\mathrm R}*[v]_{\mathrm R}
  =[J_{(-1)}v]_{\mathrm R}+[J_{(0)}v]_{\mathrm R}.
  \label{e7:eq:E7-weight-one-Ramond-Zhu-product}
\end{equation}
Equation \eqref{e7:eq:E7-highest-root-current-products-zero}, together with the
noncommutative Wick formula, implies
\begin{equation*}
  J_{(m)}:J^r:=0,
  \qquad m\ge0.
\end{equation*}
Taking $v=:J^r:$ in
\eqref{e7:eq:E7-weight-one-Ramond-Zhu-product} therefore gives
\[
  [J]_{\mathrm R}*\bigl[:J^r:\bigr]_{\mathrm R}
  =\bigl[:J^{r+1}:\bigr]_{\mathrm R}.
\]
The assertion follows by induction on $r$.
\end{proof}

Since the defining ideal of $\widetilde{\mathcal W}_n$ contains
$:J^{n+1}:$, Lemma~\ref{e7:lem:E7-Zhu-image-normal-power} yields
\begin{equation}
  e_\vartheta^{\,n+1}=0
  \qquad\text{in }B_n.
  \label{e7:eq:E7-highest-root-nilpotency-Bn}
\end{equation}
Here and below the same symbol $e_\vartheta$ denotes the image in $B_n$ of
the Ramond Zhu class of $J^{\{e_\vartheta\}}$.

The nilpotency relation gives an immediate bound on the $D_6$-types that can
occur in finite-dimensional $B_n$-modules.

\begin{lemma}[Highest-root level bound]
\label{e7:lem:E7-Bn-constituent-level-bound}
Let $M$ be a finite-dimensional $B_n$-module.  If the irreducible
$D_6$-module $L_{D_6}(\lambda)$ occurs in the restriction of $M$ to $D_6$,
then
\begin{equation}
  \langle\lambda,\vartheta^\vee\rangle\le n.
  \label{e7:eq:E7-Bn-constituent-level-bound}
\end{equation}
\end{lemma}

\begin{proof}
Let
\[
  m=\langle\lambda,\vartheta^\vee\rangle.
\]
Consider the $\mathfrak{sl}_2$-subalgebra generated by
$e_\vartheta$, $f_\vartheta$, and $\vartheta^\vee$.  If $v_\lambda$ is a
$D_6$-highest-weight vector, then
$f_\vartheta^m v_\lambda\ne0$, and the standard $\mathfrak{sl}_2$ formulas
show that
\begin{equation}
  e_\vartheta^r f_\vartheta^m v_\lambda
  =c_{m,r}f_\vartheta^{m-r}v_\lambda,
  \qquad
  0\le r\le m,
  \qquad
  c_{m,r}\ne0.
  \label{e7:eq:E7-highest-root-string-formula}
\end{equation}
If $m\ge n+1$, taking $r=n+1$ in
\eqref{e7:eq:E7-highest-root-string-formula} contradicts
\eqref{e7:eq:E7-highest-root-nilpotency-Bn}.  Hence $m\le n$.
\end{proof}

Thus every finite-dimensional $B_n$-module is, as a $D_6$-module, a direct
sum of irreducibles whose highest weights lie in the finite set
\begin{equation*}
  P_n^+(D_6)
  =
  \{\lambda\in P_+(D_6)
    \mid
    \langle\lambda,\vartheta^\vee\rangle\le n\}.
\end{equation*}
The relevant Casimir values in these finite sets are recorded in
Appendix~\ref{e7:app:D6-weight-Casimir-data}.

\subsubsection{The trace of the quadratic Ramond term}

Let $M$ be a finite-dimensional $B_n$-module, and denote the induced
$D_6$-representation by
\begin{equation*}
  \rho:D_6\longrightarrow\operatorname{End}_{\mathbb C}(M).
\end{equation*}
Put
\begin{equation*}
  N_M=\dim M,
  \qquad
  \tau_M=\operatorname{tr}_M\rho(\Omega_{D_6}).
\end{equation*}
Since $D_6$ is simple, the trace form of $\rho$ is a scalar multiple of the
normalized invariant form.  Thus there exists $I_M\in\mathbb C$ such that
\begin{equation}
  \operatorname{tr}_M\bigl(\rho(a)\rho(b)\bigr)
  =I_M(a\mid b),
  \qquad a,b\in D_6.
  \label{e7:eq:E7-Dynkin-index-definition}
\end{equation}
Choose an orthonormal basis
$\{a_i\}_{i=1}^{66}$ of $D_6$.  Taking the trace of the Casimir gives
\begin{equation}
  \tau_M=66I_M,
  \qquad
  I_M=\frac{\tau_M}{66}.
  \label{e7:eq:E7-Dynkin-index-Casimir-trace}
\end{equation}

\begin{lemma}[Trace of the quadratic term]
\label{e7:lem:E7-trace-quadratic-term}
For every $u,v\in U=\mathfrak g_{-\frac12}\cong L_{D_6}(\omega_6)$,
\begin{equation*}
  \operatorname{tr}_M\rho\bigl(Q(u,v)\bigr)
  =\frac12\tau_M\langle u,v\rangle.
\end{equation*}
\end{lemma}

\begin{proof}
Using the orthonormal-basis form of
\eqref{eq:common-quadratic-term} and
\eqref{e7:eq:E7-Dynkin-index-definition}, we obtain
\begin{align}
  \operatorname{tr}_M\rho\bigl(Q(u,v)\bigr)
  &=2I_M\sum_{i=1}^{66}
    \langle[a_i,u],[v,a_i]\rangle.
  \label{e7:eq:E7-trace-Q-first-contraction}
\end{align}
The skew form \eqref{e7:eq:E7-half-spin-symplectic-form} is $D_6$-invariant,
so
\begin{equation*}
  \langle au,w\rangle+\langle u,aw\rangle=0,
  \qquad a\in D_6.
\end{equation*}
Since $[v,a_i]=-a_iv$, it follows that
\begin{align*}
  \sum_{i=1}^{66}\langle[a_i,u],[v,a_i]\rangle
  &=-\sum_{i=1}^{66}\langle a_i u,a_i v\rangle
   \notag\\
  &=\left\langle
      u,\sum_{i=1}^{66}a_i^2v
    \right\rangle.
\end{align*}
The Casimir acts on $L_{D_6}(\omega_6)$ by
\begin{equation*}
  c_2(\omega_6)
  =(\omega_6\mid\omega_6+2\rho_{D_6})
  =\frac{33}{2};
\end{equation*}
see \eqref{e7:eq:E7-half-spin-Casimir-data}.  Hence
\begin{equation}
  \sum_{i=1}^{66}\langle[a_i,u],[v,a_i]\rangle
  =\frac{33}{2}\langle u,v\rangle.
  \label{e7:eq:E7-quadratic-contraction-value}
\end{equation}
Substituting
\eqref{e7:eq:E7-Dynkin-index-Casimir-trace} and
\eqref{e7:eq:E7-quadratic-contraction-value} into
\eqref{e7:eq:E7-trace-Q-first-contraction}, we find
\[
  \operatorname{tr}_M\rho\bigl(Q(u,v)\bigr)
  =2\frac{\tau_M}{66}\frac{33}{2}\langle u,v\rangle
  =\frac12\tau_M\langle u,v\rangle.
\]
\end{proof}

\subsubsection{The uniform Casimir trace identity}

We now trace the Ramond relation
\eqref{e7:eq:E7-n-Ramond-Zhu-relation}.  Recall that the class
$L=[\omega]_{\mathrm R}$ is central in $B_n$.

\begin{proposition}[Uniform Casimir trace identity]
\label{e7:prop:E7-uniform-Casimir-trace-identity}
Let $M$ be a nonzero finite-dimensional $B_n$-module on which $L$ acts as
the scalar $h\in\mathbb C$.  Then
\begin{equation}
  \frac{\operatorname{tr}_M\rho(\Omega_{D_6})}{\dim M}
  =\frac23
   \left(
     (2n+28)h+\frac{n(n+2)}2
   \right).
  \label{e7:eq:E7-uniform-Casimir-trace-identity}
\end{equation}
\end{proposition}

\begin{proof}
Lemma~\ref{e7:lem:E7-trace-quadratic-term} gives $\gamma=1/2$.  Apply
Proposition~\ref{prop:common-Ramond-Casimir-trace} with
$k_n=n-4$, $h^\vee(E_7)=18$, and
$p_{E_7}(k_n)=n(n+2)$.  The result is
\[
 \frac32\frac{\tau_M}{N_M}
 =2(n+14)h+\frac{n(n+2)}2,
\]
which is equivalent to
\eqref{e7:eq:E7-uniform-Casimir-trace-identity}.
\end{proof}

For later use, we record the three specializations explicitly.

\begin{corollary}

Under the assumptions of
Proposition~\ref{e7:prop:E7-uniform-Casimir-trace-identity}, one has
\begin{align*}
  n=1:\qquad
  &\frac{\operatorname{tr}_M\rho(\Omega_{D_6})}{\dim M}
   =20h+1,\\
  n=2:\qquad
  &\frac{\operatorname{tr}_M\rho(\Omega_{D_6})}{\dim M}
   =\frac23(32h+4),\\
  n=3:\qquad
  &\frac{\operatorname{tr}_M\rho(\Omega_{D_6})}{\dim M}
   =\frac23\left(34h+\frac{15}{2}\right).
\end{align*}
\end{corollary}

The identity above is the only information about the full Ramond Zhu
algebra that will be required.  In particular, we shall not classify its
finite-dimensional simple modules.  Instead, the highest-root bound
\eqref{e7:eq:E7-Bn-constituent-level-bound} reduces the possible $D_6$-types
to a finite list, while Li spectral flow will determine the value of $h$
for the lowest space of a hypothetical nonzero ideal.  These two ingredients
will be combined in Sections~\ref{e7:sec:Li-spectral-flow} and
\ref{e7:sec:simplicity-reduced-quotients}.

\subsection{Li spectral flow and the lowest Ramond spectrum}
\label{e7:sec:Li-spectral-flow}

We now apply the Li twist fixed in
Subsection~\ref{e7:subsec:E7-Li-delta-operator} to the adjoint module of
\(\widetilde{\mathcal W}_n\).  The purpose of this section is twofold.  We
first determine the lattice containing the Ramond spectrum.  We then relate
shifted \(D_6\)-weights in the lowest Ramond space to the original conformal
weight and original \(D_6\)-weight.  These elementary relations will allow
us to exclude the large-Casimir constituents in
Section~\ref{e7:sec:simplicity-reduced-quotients} without determining the full
Ramond Zhu module structure.

Throughout the section,
\begin{equation*}
  n\in\{1,2,3\},
  \qquad
  k_n=n-4,
  \qquad
  x=\omega_6^\vee\in\mathfrak h^\natural.
\end{equation*}
Since \(D_6\) is simply laced, we identify \(x\) with the weight
\(\omega_6\) by means of the normalized invariant form.  In particular,
\begin{equation*}
  (x\mid x)=(\omega_6\mid\omega_6)=\frac32.
\end{equation*}
Let
\begin{equation*}
  \widetilde{\mathcal W}_n^{\mathrm R}
  =
  (\widetilde{\mathcal W}_n,Y^{\Delta})
\end{equation*}
be the Li-twisted adjoint module obtained from
\(H=J^{\{x\}}\).  By
\eqref{e7:eq:E7-Li-Ramond-conformal-shift} and
\eqref{e7:eq:E7-Li-Cartan-shift}, its conformal Hamiltonian and Cartan zero
modes are
\begin{align}
  L_0^{\mathrm R}
  &=L_0+x_0+\frac{3n}{4},
  \label{e7:eq:E7-section5-Ramond-Hamiltonian}\\
  J_0^{\{h\},\mathrm R}
  &=J_0^{\{h\}}+n(x\mid h),
  \qquad h\in\mathfrak h^\natural.
  \label{e7:eq:E7-section5-shifted-Cartan-action}
\end{align}
As observed in Section~\ref{e7:sec:preliminaries}, this twist is precisely the
Ramond twist: it fixes the \(D_6\)-currents and changes the sign of every
field \(G^{\{u\}}\).

\subsubsection{The Ramond spectral lattice}

We first show that the Li-twisted adjoint module is lower bounded and that
its spectrum lies in a single integral translate.

\begin{proposition}[Ramond spectral lattice]
\label{e7:prop:E7-Ramond-spectral-lattice}
The operator \(L_0^{\mathrm R}\) acts semisimply on
\(\widetilde{\mathcal W}_n^{\mathrm R}\), and
\begin{equation}
  \operatorname{Spec}L_0^{\mathrm R}
  \subset
  \frac{3n}{4}+\mathbb Z_{\ge0}.
  \label{e7:eq:E7-Ramond-spectrum-inclusion}
\end{equation}
Moreover, the lower endpoint is attained by the vacuum:
\begin{equation}
  L_0^{\mathrm R}\mathbf1
  =\frac{3n}{4}\mathbf1.
  \label{e7:eq:E7-Ramond-vacuum-energy}
\end{equation}
\end{proposition}

\begin{proof}
The ordinary conformal grading of the universal minimal \(W\)-algebra is
compatible with the \(D_6\)-action.  Each homogeneous subspace is
finite-dimensional and \(D_6\)-stable, and the same remains true after
passage to the quotient \(\widetilde{\mathcal W}_n\).  Since finite-dimensional
\(D_6\)-modules are completely reducible, \(L_0\) and \(x_0\) commute and
act semisimply on every homogeneous subspace.  Formula
\eqref{e7:eq:E7-section5-Ramond-Hamiltonian} therefore implies the
semisimplicity of \(L_0^{\mathrm R}\).

It remains to determine its possible eigenvalues.  The PBW theorem for the
universal minimal \(W\)-algebra gives ordered monomials in derivatives of
the generators
\[
  J^{\{a\}},\qquad G^{\{u\}},\qquad\omega.
\]
Their images span every conformal quotient, so it is enough to compute the
contribution of each factor to \(L_0+x_0\).

Let \(a_\alpha\in D_6\) have weight \(\alpha\), where \(\alpha\) is a root
or zero.  The state \(\partial^rJ^{\{a_\alpha\}}\) contributes
\begin{equation*}
  1+r+\alpha(x)
  \in\mathbb Z_{\ge0},
\end{equation*}
because \(x=\omega_6^\vee\) is minuscule and hence
\(\alpha(x)\in\{-1,0,1\}\).  If \(u_\mu\) is a weight vector of the
half-spin module \(L_{D_6}(\omega_6)\) of weight \(\mu\), then
\(\partial^rG^{\{u_\mu\}}\) contributes
\begin{equation*}
  \frac32+r+\mu(x)
  \in\mathbb Z_{\ge0},
\end{equation*}
since
\[
  \mu(x)\in
  \left\{-\frac32,-\frac12,\frac12,\frac32\right\}.
\]
Finally, \(\partial^r\omega\) has zero \(x_0\)-charge and contributes
\begin{equation*}
  2+r\in\mathbb Z_{>0}.
\end{equation*}
Thus every PBW monomial has an \(L_0+x_0\)-eigenvalue in
\(\mathbb Z_{\ge0}\).  Formula
\eqref{e7:eq:E7-section5-Ramond-Hamiltonian} proves
\eqref{e7:eq:E7-Ramond-spectrum-inclusion}.  Since
\(L_0\mathbf1=x_0\mathbf1=0\), equation
\eqref{e7:eq:E7-Ramond-vacuum-energy} follows.
\end{proof}

We shall write
\begin{equation*}
  T_n
  =\widetilde{\mathcal W}_n^{\mathrm R}
   \!\left[\frac{3n}{4}\right]
\end{equation*}
for the lowest Ramond eigenspace.  Proposition
\ref{e7:prop:E7-Ramond-spectral-lattice} shows that \(T_n\neq0\), since it
contains the vacuum.

\subsubsection{Shifted weights and original conformal degrees}

The Cartan shift in
\eqref{e7:eq:E7-section5-shifted-Cartan-action} has a simple weight-theoretic
interpretation.  If a vector has original \(D_6\)-weight \(\nu\), then its
shifted Ramond weight is
\begin{equation}
  \lambda=\nu+n\omega_6.
  \label{e7:eq:E7-original-to-shifted-weight}
\end{equation}
Indeed, for \(h\in\mathfrak h^\natural\), the additional eigenvalue is
\(n(x\mid h)=n\omega_6(h)\).

\begin{proposition}[Degree--weight relation at the Ramond bottom]
\label{e7:prop:E7-degree-weight-relation}
Let
\[
  0\ne v\in T_n
\]
be homogeneous of original conformal weight \(d\) and shifted
\(D_6\)-weight \(\lambda\).  Then its original \(D_6\)-weight is
\begin{equation}
  \nu=\lambda-n\omega_6,
  \label{e7:eq:E7-original-weight-from-shifted-weight}
\end{equation}
and
\begin{equation}
  d
  =\frac{3n}{2}-\lambda(x).
  \label{e7:eq:E7-degree-from-shifted-weight}
\end{equation}
Equivalently,
\begin{equation}
  d+\nu(x)=0.
  \label{e7:eq:E7-zero-excess-energy-relation}
\end{equation}
\end{proposition}

\begin{proof}
Equation \eqref{e7:eq:E7-original-weight-from-shifted-weight} is precisely
\eqref{e7:eq:E7-original-to-shifted-weight}.  Since
\(v\in T_n\), formula
\eqref{e7:eq:E7-section5-Ramond-Hamiltonian} gives
\[
  d+\nu(x)+\frac{3n}{4}=\frac{3n}{4},
\]
which proves \eqref{e7:eq:E7-zero-excess-energy-relation}.  Substituting
\(\nu=\lambda-n\omega_6\) and using
\[
  \omega_6(x)=(\omega_6\mid\omega_6)=\frac32
\]
yields \eqref{e7:eq:E7-degree-from-shifted-weight}.
\end{proof}

The extremal shifted weight is represented by a unique line.

\begin{corollary}[The extremal vacuum line]
\label{e7:cor:E7-extremal-vacuum-line}
For every \(n\in\{1,2,3\}\),
\begin{equation*}
  (T_n)_{n\omega_6}
  =\mathbb C\mathbf1.
\end{equation*}
\end{corollary}

\begin{proof}
The vacuum lies in the left-hand side by
\eqref{e7:eq:E7-Ramond-vacuum-energy} and
\eqref{e7:eq:E7-section5-shifted-Cartan-action}.  Conversely, let
\(v\in(T_n)_{n\omega_6}\).  Proposition
\ref{e7:prop:E7-degree-weight-relation} gives
\[
  \nu=0,
  \qquad
  d=\frac{3n}{2}-n\omega_6(x)=0.
\]
The degree-zero subspace of the nonzero conformal quotient
\(\widetilde{\mathcal W}_n\) is \(\mathbb C\mathbf1\).  Hence
\(v\in\mathbb C\mathbf1\).
\end{proof}

\subsubsection{The low original conformal degrees}

The next observation will be used repeatedly to exclude shifted highest
weights from \(T_n\).

\begin{lemma}[Low-degree weight restrictions]
\label{e7:lem:E7-low-degree-weight-restrictions}
The original conformal grading of \(\widetilde{\mathcal W}_n\) satisfies
\begin{align}
  (\widetilde{\mathcal W}_n)_0
  &=\mathbb C\mathbf1,
  \label{e7:eq:E7-Wtilde-degree-zero}\\
  (\widetilde{\mathcal W}_n)_{\frac12}
  &=0,
  \notag\\
  (\widetilde{\mathcal W}_n)_1
  &\text{ is a quotient of the adjoint }D_6\text{-module},
  \notag\\
  (\widetilde{\mathcal W}_n)_{\frac32}
  &\text{ is a quotient of }L_{D_6}(\omega_6).
  \label{e7:eq:E7-Wtilde-degree-three-halves}
\end{align}
Consequently, the possible \(D_6\)-weights in degree one are contained in
\begin{equation*}
  \{0\}\cup\{\pm\varepsilon_i\pm\varepsilon_j\mid i\ne j\},
\end{equation*}
whereas every weight in degree \(3/2\) is a half-spin weight of the form
\begin{equation*}
  \frac12(\pm\varepsilon_1\pm\cdots\pm\varepsilon_6)
\end{equation*}
with the parity appropriate to \(L_{D_6}(\omega_6)\).
\end{lemma}

\begin{proof}
The universal minimal \(W\)-algebra is generated by the vacuum together
with fields of conformal weights \(1\), \(3/2\), and \(2\); see
\eqref{e7:eq:E7-minimal-W-strong-generators}.  Hence its degree-zero space is
\(\mathbb C\mathbf1\), its degree-\(1/2\) space is zero, its degree-one
space is the span of the current states \(J^{\{a\}}\), and its
degree-\(3/2\) space is the span of the states \(G^{\{u\}}\).  Passing to
the nonzero quotient \(\widetilde{\mathcal W}_n\) cannot create new
vectors or new \(D_6\)-weights, which proves
\eqref{e7:eq:E7-Wtilde-degree-zero}--
\eqref{e7:eq:E7-Wtilde-degree-three-halves}.  The final two assertions are the
standard weight sets of the adjoint and half-spin representations of
\(D_6\).
\end{proof}

Combining the preceding results gives a convenient exclusion criterion.

\begin{corollary}[Lowest-space exclusion]
\label{e7:cor:E7-lowest-space-exclusion-criterion}
Suppose that an irreducible \(D_6\)-module \(L_{D_6}(\lambda)\) occurs in
\(T_n\), and let \(v_\lambda\) be a shifted highest-weight vector.  Put
\begin{equation*}
  d(\lambda)=\frac{3n}{2}-\lambda(x),
  \qquad
  \nu(\lambda)=\lambda-n\omega_6.
\end{equation*}
Then
\begin{equation}
  0\ne v_\lambda
  \in
  (\widetilde{\mathcal W}_n)_{d(\lambda)}
  [\nu(\lambda)].
  \label{e7:eq:E7-required-low-degree-weight-space}
\end{equation}
In particular, the occurrence of \(L_{D_6}(\lambda)\) is impossible if
that original degree--weight space vanishes.  For
\(d(\lambda)\in\{0,\frac12,1,\frac32\}\), the vanishing can be tested by
Corollary~\ref{e7:cor:E7-extremal-vacuum-line} and
Lemma~\ref{e7:lem:E7-low-degree-weight-restrictions}.
\end{corollary}

\begin{proof}
Choose a highest-weight vector in the indicated irreducible constituent.
Proposition~\ref{e7:prop:E7-degree-weight-relation} gives exactly the original
conformal degree and original \(D_6\)-weight in
\eqref{e7:eq:E7-required-low-degree-weight-space}.
\end{proof}

\subsubsection{Lowest spaces of twisted ideals}

By Proposition~\ref{e7:prop:E7-Ramond-spectral-lattice}, the
finite-dimensionality of $B_n$, and
Lemma~\ref{lem:common-lowest-Ramond-Zhu-module}, every nonzero ideal has a
finite-dimensional cyclic $B_n$-module at a lowest Ramond energy
$3n/4+m$, $m\in\mathbb Z_{\ge0}$.

\subsection{Simplicity of the reduced quotients}
\label{e7:sec:simplicity-reduced-quotients}

We now prove that the reduced quotients constructed in
Section~\ref{e7:sec:distinguished-singular-vectors} are already simple.  Recall
that, for
\[
  n\in\{1,2,3\},
  \qquad
  k_n=n-4,
\]
there is a natural surjection
\begin{equation}
  \pi_n:\widetilde{\mathcal W}_n
  \twoheadrightarrow
  \mathcal W_{k_n}(E_7,f_\theta).
  \label{e7:eq:E7-section6-natural-surjection}
\end{equation}
The proof combines the highest-root bound of
Lemma~\ref{e7:lem:E7-Bn-constituent-level-bound}, the trace identity of
Proposition~\ref{e7:prop:E7-uniform-Casimir-trace-identity}, and the low-degree
restrictions obtained from Li spectral flow in
Section~\ref{e7:sec:Li-spectral-flow}.

\subsubsection{Finite \texorpdfstring{\(D_6\)}{D6} estimates}

We first isolate the elementary root-theoretic information needed in the
three cases.  Write a dominant integral weight of \(D_6\) as
\begin{equation*}
  \lambda=\sum_{i=1}^{6}m_i\omega_i,
  \qquad m_i\in\mathbb Z_{\ge0}.
\end{equation*}
In the Bourbaki numbering,
\begin{equation*}
  \vartheta
  =\alpha_1+2\alpha_2+2\alpha_3+2\alpha_4+\alpha_5+\alpha_6,
\end{equation*}
and hence
\begin{equation}
  \langle\lambda,\vartheta^\vee\rangle
  =m_1+2m_2+2m_3+2m_4+m_5+m_6.
  \label{e7:eq:E7-D6-highest-root-level-formula}
\end{equation}
We use the normalized Casimir eigenvalue
\begin{equation}
  c_2(\lambda)
  =(\lambda\mid\lambda+2\rho_{D_6}).
  \label{e7:eq:E7-D6-Casimir-eigenvalue-definition}
\end{equation}

\begin{lemma}[Casimir bounds at levels one, two, and three]
\label{e7:lem:E7-D6-Casimir-small-levels}
The following statements hold.
\begin{enumerate}
\item If \(\langle\lambda,\vartheta^\vee\rangle\le1\), then
\begin{equation*}
  \lambda\in\{0,\omega_1,\omega_5,\omega_6\},
\end{equation*}
and
\begin{equation*}
  c_2(0)=0,
  \qquad
  c_2(\omega_1)=11,
  \qquad
  c_2(\omega_5)=c_2(\omega_6)=\frac{33}{2}.
\end{equation*}

\item If \(\langle\lambda,\vartheta^\vee\rangle\le2\), then
\begin{equation*}
  c_2(\lambda)\le36.
\end{equation*}
Moreover, the weights satisfying \(c_2(\lambda)>32\) are precisely
\begin{equation*}
  2\omega_5,
  \qquad
  2\omega_6,
  \qquad
  \omega_5+\omega_6,
\end{equation*}
with
\begin{equation*}
  c_2(2\omega_5)=c_2(2\omega_6)=36,
  \qquad
  c_2(\omega_5+\omega_6)=35.
\end{equation*}

\item If \(\langle\lambda,\vartheta^\vee\rangle\le3\), then
\begin{equation*}
  c_2(\lambda)\le\frac{117}{2}.
\end{equation*}
The weights satisfying \(c_2(\lambda)\ge56\) are precisely
\begin{equation*}
  3\omega_5,
  \qquad
  3\omega_6,
  \qquad
  2\omega_5+\omega_6,
  \qquad
  \omega_5+2\omega_6.
\end{equation*}
Their Casimir eigenvalues are
\begin{align*}
  c_2(3\omega_5)
  &=c_2(3\omega_6)=\frac{117}{2},\\
  c_2(2\omega_5+\omega_6)
  &=c_2(\omega_5+2\omega_6)=\frac{113}{2}.
\end{align*}
Every other allowed weight satisfies
\begin{equation*}
  c_2(\lambda)\le\frac{105}{2}.
\end{equation*}
\end{enumerate}
\end{lemma}

\begin{proof}
Equation~\eqref{e7:eq:E7-D6-highest-root-level-formula} reduces each assertion
to a finite enumeration of the nonnegative integer tuples
\((m_1,\ldots,m_6)\).  The Casimir values are then obtained from
\eqref{e7:eq:E7-D6-Casimir-eigenvalue-definition}, using the standard
orthogonal realization fixed in
\eqref{e7:eq:D6-spin-weight-coordinates}.  For example,
\begin{equation*}
  c_2(r\omega_6)
  =\frac32r^2+15r,
\end{equation*}
which gives \(33/2\), \(36\), and \(117/2\) for \(r=1,2,3\),
respectively.  Evaluating the remaining finitely many tuples gives the
lists and bounds above.  The complete enumeration is recorded in
Appendix~\ref{e7:app:D6-weight-Casimir-data}.
\end{proof}

\subsubsection{Uniform Casimir-gap argument}

The three levels are governed by the same numerical mechanism.  Put
$h_0=3n/4$ and let $A_n(h)$ denote the Casimir average from
Proposition~\ref{e7:prop:E7-uniform-Casimir-trace-identity}.  The relevant
data are
\begin{equation*}
\begin{array}{c|c|c|c|c}
 n&A_n(h)&C_{\max}&A_n(h_0)&C_{\mathrm{rest}}\\ \hline
 1&20h+1&33/2&16&11\\[1mm]
 2&\frac23(32h+4)&36&104/3&32\\[1mm]
 3&\frac23(34h+15/2)&117/2&56&105/2
\end{array}
\end{equation*}
Here $C_{\max}$ is the global maximum under the highest-root bound, while
$C_{\mathrm{rest}}$ is the largest value left after excluding the
non-extremal high-Casimir types.  In each row,
\begin{equation}
 A_n(h_0+1)>C_{\max},
 \qquad
 A_n(h_0)>C_{\mathrm{rest}}.
 \label{e7:eq:E7-two-uniform-Casimir-gaps}
\end{equation}

The exclusions needed to justify the last column follow uniformly from
Corollary~\ref{e7:cor:E7-lowest-space-exclusion-criterion}:
\begin{itemize}
\item for $n=1$, the weight $\omega_5$ would require original degree
$1/2$; the only remaining high-Casimir weight is the extremal
$\omega_6$;
\item for $n=2$, the weight $\omega_5+\omega_6$ would require degree
$1/2$, while $2\omega_5$ would require degree one and original weight
$-2\varepsilon_6$, which is not an adjoint weight; the remaining weight is
$2\omega_6$;
\item for $n=3$, the weights $\omega_5+2\omega_6$ and
$2\omega_5+\omega_6$ are excluded in degrees $1/2$ and $1$, respectively,
and $3\omega_5$ would require degree $3/2$ with original weight
$-3\varepsilon_6$, which is not a half-spin weight.  The remaining weight
is $3\omega_6$.
\end{itemize}
Thus every ground-energy Zhu module with the prescribed Casimir average
contains the extremal constituent $L_{D_6}(n\omega_6)$.

\begin{theorem}[Simplicity of the reduced quotients]
\label{e7:thm:E7-reduced-quotients-simple}
For every $n\in\{1,2,3\}$,
\begin{equation*}
  \widetilde{\mathcal W}_n
  \cong
  \mathcal W_{n-4}(E_7,f_\theta).
\end{equation*}
Equivalently,
\begin{equation*}
  \mathcal W^{n-4}(E_7,f_\theta)
  \Big/
  \left\langle:J^{\{e_\vartheta\}}{}^{n+1}:\right\rangle
  \cong
  \mathcal W_{n-4}(E_7,f_\theta).
\end{equation*}
\end{theorem}

\begin{proof}
Fix $n\in\{1,2,3\}$ and apply
Proposition~\ref{prop:common-Casimir-gap-simplicity} to the natural
surjection \eqref{e7:eq:E7-section6-natural-surjection}.  The spectral
lattice and extremal line are
Proposition~\ref{e7:prop:E7-Ramond-spectral-lattice} and
Corollary~\ref{e7:cor:E7-extremal-vacuum-line}; the finite Ramond Zhu
algebra and complete reducibility were established in
Section~\ref{e7:sec:Ramond-Casimir-identity}.  The global bounds and the
strict first inequality in \eqref{e7:eq:E7-two-uniform-Casimir-gaps} force
the lowest ideal energy to equal $h_0$.  For every simple quotient of a
nonzero ground Zhu submodule, complete reducibility over $D_6$ allows its
constituents to be viewed inside the ambient ground space.  The exclusions
above, together with the second inequality, then force the extremal
constituent $L_{D_6}(n\omega_6)$, whose highest line is
$\mathbb C\mathbf1$.  The common criterion proves that the surjection is
an isomorphism.
\end{proof}

Combining Theorem~\ref{e7:thm:E7-reduced-quotients-simple} with
Proposition~\ref{e7:prop:E7-reduction-of-Qn}, we obtain
\begin{equation*}
  H^0_{\mathrm{DS},f_\theta}(Q_n)
  \cong
  \mathcal W_{n-4}(E_7,f_\theta),
  \qquad n=1,2,3.
\end{equation*}
The affine kernel now vanishes by Proposition~\ref{prop:common-affine-lifting}.

\subsection{Maximal ideals of the affine vertex algebras}

We apply the common affine-lifting result to the three quotients $Q_n$.

\subsubsection{Uniform affine lifting}

For $n=1,2,3$, the singular-vector quotient admits the canonical surjection
\begin{equation*}
 q_n:Q_n\twoheadrightarrow L_{k_n}(E_7).
\end{equation*}
Theorem~\ref{e7:thm:E7-reduced-quotients-simple} identifies the induced map
on minimal reductions with an isomorphism.  Proposition~\ref{prop:common-affine-lifting}
therefore applies.

\begin{theorem}[Maximal ideals at levels \(-3,-2,-1\)]
\label{e7:thm:E7-main-maximal-ideal}
For every \(n\in\{1,2,3\}\),
\begin{equation*}
  Q_n\cong L_{n-4}(E_7).
\end{equation*}
Equivalently,
\begin{equation*}
  \ker\!\left(
    V^{n-4}(E_7)\longrightarrow L_{n-4}(E_7)
  \right)
  =
  \left\langle\sigma(w)^{n+1}\right\rangle.
\end{equation*}
\end{theorem}

\begin{proof}
Apply Proposition~\ref{prop:common-affine-lifting} to
$q_n:Q_n\twoheadrightarrow L_{k_n}(E_7)$.
\end{proof}

Writing the three cases separately, Theorem~\ref{e7:thm:E7-main-maximal-ideal}
gives
\begin{align*}
  \ker\!\left(V^{-3}(E_7)\longrightarrow L_{-3}(E_7)\right)
  &=\left\langle\sigma(w)^2\right\rangle,\\
  \ker\!\left(V^{-2}(E_7)\longrightarrow L_{-2}(E_7)\right)
  &=\left\langle\sigma(w)^3\right\rangle,\\
  \ker\!\left(V^{-1}(E_7)\longrightarrow L_{-1}(E_7)\right)
  &=\left\langle\sigma(w)^4\right\rangle.
\end{align*}

\begin{remark}[The boundary level]

The parameter value \(n=0\) corresponds to the boundary level \(k=-4\).
That case was already established in the proof of
\cite[Theorem~3.1]{ArakawaMoreau2018}.  Thus the known \(n=0\) result,
together with Theorem~\ref{e7:thm:E7-main-maximal-ideal}, settles the four
levels \(k=-4,-3,-2,-1\) arising from this exceptional-series family for
\(E_7\).
\end{remark}
\section{Type \texorpdfstring{$E_8$}{E8}}
\label{sec:case-e8}

\subsection{Type data}
\label{e8:sec:preliminaries}

Let $\mathfrak g=E_8$ and use the minimal grading associated with
$f_\theta=e_{-\theta}$.  Its degree-zero centralizer and half-graded module
are
\begin{equation*}
 \mathfrak g^\natural\cong E_7,
 \qquad
 U:=\mathfrak g_{-\frac12}\cong L_{E_7}(\omega_7),
\end{equation*}
where $L_{E_7}(\omega_7)$ is the $56$-dimensional minuscule module.  We use
\begin{equation*}
 h^\vee(E_8)=30,
 \qquad h^\vee(E_7)=18,
 \qquad
 (\omega_7\mid\omega_7)=\frac32,
 \qquad c_2(\omega_7)=\frac{57}{2},
\end{equation*}
and the invariant symplectic form
\begin{equation}
 \langle u,v\rangle=(e_\theta\mid[u,v]).
 \label{e8:eq:E8-half-piece-symplectic-form}
\end{equation}
The complete $E_7$ root, weight, and Casimir tables are recorded in
Appendix~\ref{e8:app:E8-E7-Casimir-data}; see
\cite{KacWakimoto2004,AdamovicKacMosenederFrajriaPapiPerse2018}.

\subsubsection{Level, Ramond, and spectral-flow data}

All formal conventions are those of Section~\ref{sec:common-framework}.  Set
\begin{equation}
 k_n=n-6,\qquad 1\le n\le5.
 \label{e8:eq:E8-five-levels-definition}
\end{equation}
The $E_7$ current level is $k^\natural=k+6=n$, and
\begin{equation}
 [J^{\{a\}}{}_{\lambda}J^{\{b\}}]
 =J^{\{[a,b]\}}+n(a\mid b)\lambda.
 \label{e8:eq:E8-E7-current-relation}
\end{equation}
The minimal $W$-algebra has generators
\begin{equation}
 J^{\{a\}},\ a\in E_7;
 \qquad G^{\{u\}},\ u\in L_{E_7}(\omega_7);
 \qquad \omega.
 \label{e8:eq:E8-minimal-W-generators}
\end{equation}
Since $p_{E_8}(k)=(k+6)(k+10)$, the Ramond relation at $k=k_n$ is
\begin{equation*}
 [u,v]=\langle u,v\rangle
 \left(\Omega_{E_7}-2(n+24)L-\frac{n(n+4)}2\right)+Q(u,v).
\end{equation*}
\label{e8:subsec:E8-exactness-delta-shifts}
The BRST Hamiltonian is
\begin{equation}
 L_0^{\mathcal W}=L_0^{\mathrm{aff}}-(x_\theta)_0+L_0^{\mathrm{gh}}.
 \label{e8:eq:E8-BRST-conformal-Hamiltonian}
\end{equation}
\label{e8:subsec:E8-Li-delta}
The Li twist is governed by
\eqref{eq:common-Li-Hamiltonian}--\eqref{eq:common-Li-mode-shifts}, and every
ideal is stable under the twisted adjoint action.

\subsection{The distinguished singular vectors and their exact reductions}

In this section we introduce the five affine singular vectors associated with
\eqref{e8:eq:E8-five-levels-definition} and identify their images under minimal
Drinfeld--Sokolov reduction.  The main result is the uniform quotient formula
in Proposition~\ref{e8:prop:E8-reduction-of-Qn}.

\subsubsection{The distinguished degree-two vector}

For $k\in\C$, let
\begin{equation}
  \sigma_k:S^2(\g)\longrightarrow V^k(\g)_2
  \label{e8:eq:E8-degree-two-symmetrization-map}
\end{equation}
be the $\g$-equivariant symmetrization map defined by
\begin{equation*}
  \sigma_k(ab)
  =\frac12\bigl(a(-1)b(-1)+b(-1)a(-1)\bigr)\Vac,
  \qquad a,b\in\g.
\end{equation*}
Following \cite[Section~4]{ArakawaMoreau2018}, let
\begin{equation*}
  \mathcal R\subset S^2(\g)
\end{equation*}
be the distinguished irreducible $E_8$-submodule whose highest weight is
\begin{equation*}
  \theta+\vartheta.
\end{equation*}
Fix a nonzero highest-weight vector
\begin{equation*}
  w\in\mathcal R_{\theta+\vartheta}.
\end{equation*}
The map \eqref{e8:eq:E8-degree-two-symmetrization-map} is independent of the
level as a linear map on the affine PBW space.  We therefore write
\begin{equation*}
  u=\sigma(w)\in V^k(\g)_2
\end{equation*}
when the ambient level is clear.

Under the PBW identification
\begin{equation*}
  V^k(\g)\cong
  U\!\left(\g[t^{-1}]t^{-1}\right)
\end{equation*}
of the universal vacuum module, powers of $u$ will always mean ordinary
associative powers in the enveloping algebra:
\begin{equation*}
  u^r=
  \underbrace{u\cdots u}_{r\text{ factors}},
  \qquad r\ge1.
\end{equation*}
This convention agrees with that of \cite{ArakawaMoreau2018} and avoids any
reassociation issue for iterated vertex-algebra $(-1)$-products.

\begin{lemma}
\label{e8:lem:E8-PBW-powers-nonzero}
For every $r\ge1$, the vector $u^r$ is nonzero.  It has affine conformal
degree $2r$ and finite $E_8$-weight $r(\theta+\vartheta)$.
\end{lemma}

\begin{proof}
Equip $V^k(\g)$ with the PBW filtration by the number of negative-current
factors.  Its associated graded algebra is
\begin{equation}
  \gr V^k(\g)
  \cong S\!\left(\g[t^{-1}]t^{-1}\right).
  \label{e8:eq:E8-PBW-associated-graded}
\end{equation}
The leading symbol of $u$ is the nonzero element obtained from $w$ by placing
both factors in loop degree $-1$.  Hence the leading symbol of $u^r$ is its
$r$th power.  Since the symmetric algebra in
\eqref{e8:eq:E8-PBW-associated-graded} is an integral domain, this symbol is
nonzero, and therefore $u^r\neq0$.  The degree and weight assertions follow
from the corresponding assertions for $u$.
\end{proof}

\subsubsection{The five affine singular vectors}

For $1\le n\le5$, set
\begin{equation*}
  r_n=n+1,
  \qquad
  s_n=u^{r_n}\in V^{k_n}(\g),
  \qquad
  k_n=n-6.
\end{equation*}

\begin{theorem}

For each $1\le n\le5$, the vector $s_n$ is a nonzero affine singular vector
in $V^{k_n}(\g)$.  Its affine conformal degree and finite $E_8$-weight are
\begin{equation*}
  \deg(s_n)=2(n+1),
  \qquad
  \wt_{E_8}(s_n)=(n+1)(\theta+\vartheta).
\end{equation*}
\end{theorem}

\begin{proof}
By \cite[Theorem~4.2(1)(b)]{ArakawaMoreau2018}, for a Lie algebra in the
Deligne exceptional series outside type $A$, the vector
$\sigma(w)^{r}$ is affine singular precisely at
\begin{equation*}
  k=r-1-\frac{h^{\vee}(\g)}{6}-1.
\end{equation*}
For $\g=E_8$, $h^{\vee}(\g)=30$, and $r=r_n=n+1$, the right-hand side is
\begin{equation*}
  n-\frac{30}{6}-1=n-6=k_n.
\end{equation*}
Thus $s_n$ is singular in $V^{k_n}(\g)$.  Its nonvanishing, degree, and
finite weight follow from Lemma~\ref{e8:lem:E8-PBW-powers-nonzero}.
\end{proof}

Let
\begin{equation*}
  I_n=\langle s_n\rangle_{V^{k_n}(\g)}
\end{equation*}
be the graded vertex-algebra ideal generated by $s_n$.

\begin{lemma}

For every $1\le n\le5$,
\begin{equation}
  I_n
  =U(\widehat{\g}^{\,\prime})s_n.
  \label{e8:eq:E8-In-equals-affine-submodule}
\end{equation}
Moreover, $I_n$ is proper.
\end{lemma}

\begin{proof}
The universal affine vertex algebra is generated by the current states
$a(-1)\Vac$, $a\in\g$.  Hence every vertex-algebra ideal containing $s_n$ is
stable under all current modes and therefore contains
$U(\widehat{\g}^{\,\prime})s_n$.  Conversely, the latter subspace is stable
under all modes of the generating currents, and hence is a vertex-algebra
ideal.  This proves \eqref{e8:eq:E8-In-equals-affine-submodule}.

Since $s_n$ is annihilated by the positive affine nilpotent subalgebra, the
PBW triangular decomposition gives
\begin{equation}
  U(\widehat{\g}^{\,\prime})s_n
  =U(\widehat{\mathfrak n}_{-})s_n.
  \label{e8:eq:E8-In-negative-triangular-decomposition}
\end{equation}
The zero-mode factors in $\widehat{\mathfrak n}_{-}$ preserve loop degree,
whereas negative loop modes increase it.  Every vector in the right-hand side
of \eqref{e8:eq:E8-In-negative-triangular-decomposition} therefore has loop
degree at least $2(n+1)>0$.  The vacuum vector cannot belong to $I_n$, so the
ideal is proper.
\end{proof}

Define
\begin{equation}
  Q_n=V^{k_n}(\g)/I_n.
  \label{e8:eq:E8-definition-Qn}
\end{equation}
Since $I_n$ is contained in the unique maximal graded ideal, the canonical
map onto the simple affine vertex algebra factors as
\begin{equation}
  V^{k_n}(\g)\twoheadrightarrow Q_n
  \overset{\pi_n}{\twoheadrightarrow}L_{k_n}(\g).
  \label{e8:eq:E8-Qn-to-simple-affine}
\end{equation}

\subsubsection{Nonvanishing of the reduced cyclic modules}

The actual extended affine highest weight of $s_n$ is
\begin{equation*}
  \widehat\lambda_n
  =k_n\Lambda_0
   +(n+1)(\theta+\vartheta)
   -2(n+1)\delta.
\end{equation*}
For the application of the Verma-module reduction theorem, introduce the
normalized weight
\begin{equation*}
  \widehat\lambda_n^{\circ}
  =\widehat\lambda_n+2(n+1)\delta
  =k_n\Lambda_0+(n+1)(\theta+\vartheta),
  \qquad
  \widehat\lambda_n^{\circ}(d)=0.
\end{equation*}
Equivalently, let $I_n^{\circ}$ denote the same
$\widehat{\g}^{\,\prime}$-module as $I_n$, with the action of $d$ shifted so
that $s_n$ has $d$-eigenvalue zero.  Then there are natural surjections
\begin{equation}
  M(\widehat\lambda_n^{\circ})
  \twoheadrightarrow I_n^{\circ}
  \twoheadrightarrow L(\widehat\lambda_n^{\circ}).
  \label{e8:eq:E8-Verma-In-simple-surjections}
\end{equation}

\begin{proposition}
\label{e8:prop:E8-reduction-In-cyclic-nonzero}
For every $1\le n\le5$, the module
$H^0_{\DS,f_{\theta}}(I_n)$ is nonzero and is generated, as a
$\W^{k_n}(\g,f_{\theta})$-module, by the reduced highest-weight vector
\begin{equation*}
  \overline{s}_n=[s_n]_{\DS}.
\end{equation*}
\end{proposition}

\begin{proof}
By the Verma-module reduction theorem recalled in
Section~\ref{e8:subsec:E8-exactness-delta-shifts},
$H^0_{\DS,f_{\theta}}(M(\widehat\lambda_n^{\circ}))$ is a Verma module over
$\W^{k_n}(\g,f_{\theta})$ generated by the class of its affine
highest-weight vector.  Exactness applied to the first map in
\eqref{e8:eq:E8-Verma-In-simple-surjections} shows that
$H^0_{\DS,f_{\theta}}(I_n^{\circ})$ is generated by the image of that class.
After forgetting the auxiliary shifted $d$-action, this is precisely the
class $\overline{s}_n$ in the reduction of $I_n$.

It remains to prove that this cyclic generator is nonzero.  Since
$\vartheta$ is a root of $\g^{\natural}$, it is orthogonal to $\theta$, and
therefore
\begin{equation*}
  \langle\vartheta,\theta^{\vee}\rangle=0.
\end{equation*}
Using $\alpha_0^{\vee}=K-\theta^{\vee}$, we obtain
\begin{align*}
  \widehat\lambda_n^{\circ}(\alpha_0^{\vee})
  &=k_n-
    \left\langle(n+1)(\theta+\vartheta),\theta^{\vee}\right\rangle
    \notag\\
  &=(n-6)-2(n+1)
   =-n-8.
\end{align*}
In particular,
$\widehat\lambda_n^{\circ}(\alpha_0^{\vee})\notin\Z_{\ge0}$.
The nonvanishing criterion
\eqref{eq:common-DS-nonvanishing} gives
\begin{equation*}
  H^0_{\DS,f_{\theta}}
  \!\left(L(\widehat\lambda_n^{\circ})\right)\neq0.
\end{equation*}
Exactness applied to the second map in
\eqref{e8:eq:E8-Verma-In-simple-surjections} yields a surjection
\begin{equation*}
  H^0_{\DS,f_{\theta}}(I_n^{\circ})
  \twoheadrightarrow
  H^0_{\DS,f_{\theta}}
  \!\left(L(\widehat\lambda_n^{\circ})\right).
\end{equation*}
Thus $H^0_{\DS,f_{\theta}}(I_n)$ is nonzero, and its cyclic generator
$\overline{s}_n$ is nonzero.
\end{proof}

\subsubsection{The exact reduced highest-weight vectors}

We next determine the conformal and $E_7$-weights of $\overline{s}_n$.
Since
\begin{equation*}
  \theta(x_{\theta})=1,
  \qquad
  \vartheta(x_{\theta})=0,
\end{equation*}
the $x_{\theta}$-eigenvalue of $s_n$ is $n+1$.  By
\eqref{e8:eq:E8-BRST-conformal-Hamiltonian}, its reduced conformal weight is
\begin{equation*}
  \wt_{\W}(\overline{s}_n)
  =2(n+1)-(n+1)=n+1.
\end{equation*}
The restriction of $\theta$ to the Cartan subalgebra of $\g^{\natural}$ is
zero, so the finite $E_7$-weight is
\begin{equation*}
  \wt_{E_7}(\overline{s}_n)=(n+1)\vartheta.
\end{equation*}
The Verma-module reduction theorem also shows that $\overline{s}_n$ is an
$E_7$-highest-weight vector.  Hence
\begin{equation}
  0\neq\overline{s}_n
  \in
  \W^{k_n}(\g,f_{\theta})_{n+1}
  [(n+1)\vartheta]^{\mathrm{hw}}.
  \label{e8:eq:E8-reduced-sn-extremal-space}
\end{equation}

\begin{lemma}
\label{e8:lem:E8-extremal-W-weight-space}
For every integer $r\ge1$ and every noncritical level $k$, one has
\begin{equation}
  \W^k(\g,f_{\theta})_{r}
  [r\vartheta]^{\mathrm{hw}}
  =\C\,J^{\{e_{\vartheta}\}}_{(-1)}{}^{r}\Vac
  =\C\,:J^{\{e_{\vartheta}\}}{}^{r}:.
  \label{e8:eq:E8-extremal-W-space-one-dimensional}
\end{equation}
\end{lemma}

\begin{proof}
Apply Proposition~\ref{prop:common-extremal-PBW-line} to
\(\mathfrak g^\natural\cong E_7\), with \(\eta=\vartheta\).  Since
\(E_7\) is simply laced, its current weights have
\(\vartheta^\vee\)-height at most two, with equality only at the highest
root.  The module
\(\mathfrak g_{-1/2}\cong L_{E_7}(\omega_7)\) is minuscule and all its
weights \(\nu\) satisfy
\(\langle\nu,\vartheta^\vee\rangle\le1\).  The common proposition gives
\eqref{e8:eq:E8-extremal-W-space-one-dimensional}.
\end{proof}

\begin{corollary}[Exact reduction of the singular vectors]
\label{e8:cor:E8-exact-reduction-sn}
For every $1\le n\le5$, there exists a scalar $c_n\in\C^{\times}$ such that
\begin{equation*}
  [s_n]_{\DS}
  =c_n\,:J^{\{e_{\vartheta}\}}{}^{n+1}:.
\end{equation*}
\end{corollary}

\begin{proof}
Combine Proposition~\ref{e8:prop:E8-reduction-In-cyclic-nonzero},
\eqref{e8:eq:E8-reduced-sn-extremal-space}, and
Lemma~\ref{e8:lem:E8-extremal-W-weight-space} with $r=n+1$.
\end{proof}

\subsubsection{Reduction of the affine quotients}

Set
\begin{equation}
  \widetilde{\W}_n
  =\frac{\W^{k_n}(\g,f_{\theta})}
  {\left\langle:J^{\{e_{\vartheta}\}}{}^{n+1}:\right\rangle}.
  \label{e8:eq:E8-definition-Wtilde-n}
\end{equation}

\begin{proposition}
\label{e8:prop:E8-reduction-of-Qn}
For every $1\le n\le5$, there is a canonical isomorphism of vertex algebras
\begin{equation}
  H^0_{\DS,f_{\theta}}(Q_n)
  \cong\widetilde{\W}_n.
  \label{e8:eq:E8-reduction-Qn-is-Wtilde-n}
\end{equation}
Moreover, the map $\pi_n$ in \eqref{e8:eq:E8-Qn-to-simple-affine} induces a
nonzero surjective homomorphism
\begin{equation}
  \rho_n:\widetilde{\W}_n
  \twoheadrightarrow\W_{k_n}(\g,f_{\theta}).
  \label{e8:eq:E8-natural-Wtilde-to-simple-W}
\end{equation}
\end{proposition}

\begin{proof}
For fixed $n$, the ideal $I_n$ is the cyclic affine submodule generated by
$s_n$, and its reduction is cyclic by the reduced Verma-module argument above.
Corollary~\ref{e8:cor:E8-exact-reduction-sn} identifies the cyclic generator
with $:J^{\{e_\vartheta\}}{}^{n+1}:$.  Proposition~\ref{prop:common-exact-generated-quotient}
therefore gives \eqref{e8:eq:E8-reduction-Qn-is-Wtilde-n}.  Reducing the
canonical map $Q_n\twoheadrightarrow L_{k_n}(E_8)$ gives
\eqref{e8:eq:E8-natural-Wtilde-to-simple-W}; its target is nonzero by
\eqref{eq:common-DS-nonvanishing}.
\end{proof}

The next section studies the associated varieties and Ramond Zhu algebras of
$\widetilde{\W}_n$.  In particular, we shall prove that each
$\widetilde{\W}_n$ is lisse and derive the uniform Casimir trace identity
needed for the spectral-flow argument.

\subsection{Lisse quotients and a Ramond Zhu trace identity}
\label{e8:sec:lisse-Ramond-trace}

For $1\le n\le5$, recall the affine quotient $Q_n$ from
\eqref{e8:eq:E8-definition-Qn} and the reduced quotient
$\widetilde{\W}_n$ from \eqref{e8:eq:E8-definition-Wtilde-n}.  In this section
we first establish the geometric and finiteness properties of
$\widetilde{\W}_n$.  We then pass the defining highest-root relation to the
Ramond Zhu algebra and derive a uniform Casimir trace identity.  The latter
will be combined with the finite-dimensional $E_7$ estimates of the next
section.

\subsubsection{Associated varieties and lisse reduction}

The associated variety of $Q_n$ is known explicitly.

\begin{proposition}
\label{e8:prop:E8-associated-variety-Qn}
For every $1\le n\le5$,
\begin{equation}
  X_{Q_n}=\overline{\mathbb O_{\min}},
  \label{e8:eq:E8-associated-variety-Qn}
\end{equation}
where $\mathbb O_{\min}\subset E_8$ is the minimal nonzero nilpotent orbit.
Consequently,
\begin{equation}
  X_{\widetilde{\W}_n}=\{f_{\theta}\}
  \label{e8:eq:E8-associated-variety-Wtilde-n}
\end{equation}
in the Slodowy slice $\mathcal S_{f_{\theta}}$.  In particular,
$\widetilde{\W}_n$ is lisse.
\end{proposition}

\begin{proof}
Equation \eqref{e8:eq:E8-associated-variety-Qn} is
\cite[Proposition~5.2]{ArakawaMoreau2018}.  By
Proposition~\ref{e8:prop:E8-reduction-of-Qn} and the compatibility of associated
varieties with minimal reduction recalled in
\eqref{eq:common-associated-variety},
\begin{equation*}
  X_{\widetilde{\W}_n}
  =X_{Q_n}\cap\mathcal S_{f_{\theta}}.
\end{equation*}
The closure of the minimal orbit meets the Slodowy slice through
$f_{\theta}$ only at $f_{\theta}$, and hence
\eqref{e8:eq:E8-associated-variety-Wtilde-n} follows.  Since this variety is
zero-dimensional, $\widetilde{\W}_n$ is lisse.
\end{proof}

\begin{remark}

Nonvanishing is independent of the associated-variety convention.  Indeed,
\eqref{e8:eq:E8-natural-Wtilde-to-simple-W} gives a surjection
\begin{equation*}
  \widetilde{\W}_n\twoheadrightarrow
  \W_{k_n}(\g,f_{\theta}),
\end{equation*}
and the target is nonzero by
\eqref{eq:common-DS-nonvanishing}.  Thus
\begin{equation*}
  \widetilde{\W}_n\neq0.
\end{equation*}
\end{remark}

Set
\begin{equation*}
  B_n=A_{\mathrm R}(\widetilde{\W}_n).
\end{equation*}
The subscript records the level $n$ of the affine $E_7$ currents.

\begin{corollary}
\label{e8:cor:E8-Bn-finite-dimensional}
For every $1\le n\le5$, the Ramond Zhu algebra $B_n$ is finite-dimensional.
\end{corollary}

\begin{proof}
By Proposition~\ref{e8:prop:E8-associated-variety-Qn}, the vertex algebra
$\widetilde{\W}_n$ is lisse, so its $C_2$-Poisson algebra
$R_{\widetilde{\W}_n}$ is finite-dimensional.  The filtered surjection
\eqref{eq:common-C2-Ramond-Zhu} gives
\begin{equation*}
  R_{\widetilde{\W}_n}\twoheadrightarrow\gr B_n.
\end{equation*}
Hence $\gr B_n$, and therefore $B_n$, is finite-dimensional.
\end{proof}

\subsubsection{The highest-root nilpotency relation}

Fix a nonzero highest-root vector
\begin{equation*}
  e=e_{\vartheta}\in(E_7)_{\vartheta}
\end{equation*}
and abbreviate
\begin{equation*}
  J=J^{\{e\}}.
\end{equation*}
By the current relation \eqref{e8:eq:E8-E7-current-relation},
\begin{equation*}
  [J_{\lambda}J]
  =J^{\{[e,e]\}}+n(e\mid e)\lambda=0.
\end{equation*}
Indeed, $[e,e]=0$ and the invariant form pairs the root space
$(E_7)_{\vartheta}$ only with $(E_7)_{-\vartheta}$.  Thus
\begin{equation}
  J_{(m)}J=0,
  \qquad m\ge0.
  \label{e8:eq:E8-J-nonnegative-products-zero}
\end{equation}

\begin{lemma}
\label{e8:lem:E8-Zhu-image-normal-power}
For every integer $r\ge1$,
\begin{equation*}
  \bigl[:J^r:\bigr]_{\mathrm R}
  =[J]_{\mathrm R}^{\,r}
\end{equation*}
in the Ramond Zhu algebra of every quotient in which the state
$:J^r:$ is defined.
\end{lemma}

\begin{proof}
The Ramond automorphism fixes $J$, and $J$ has conformal weight one.  Hence
its Ramond Zhu product with a fixed vector $v$ is
\begin{equation}
  [J]_{\mathrm R}*[v]_{\mathrm R}
  =[J_{(-1)}v]_{\mathrm R}+[J_{(0)}v]_{\mathrm R}.
  \label{e8:eq:E8-weight-one-Ramond-Zhu-product}
\end{equation}
Equation \eqref{e8:eq:E8-J-nonnegative-products-zero} and the noncommutative
Wick formula imply
\begin{equation*}
  J_{(m)}:J^r:=0,
  \qquad m\ge0.
\end{equation*}
Taking $v=:J^r:$ in
\eqref{e8:eq:E8-weight-one-Ramond-Zhu-product} gives
\begin{equation*}
  [J]_{\mathrm R}*\bigl[:J^r:\bigr]_{\mathrm R}
  =\bigl[:J^{r+1}:\bigr]_{\mathrm R},
\end{equation*}
and induction proves the assertion.
\end{proof}

Since the defining ideal of $\widetilde{\W}_n$ contains
$:J^{n+1}:$, Lemma~\ref{e8:lem:E8-Zhu-image-normal-power} gives
\begin{equation}
  e^{n+1}=0
  \qquad\text{in }B_n.
  \label{e8:eq:E8-highest-root-nilpotency-Bn}
\end{equation}
Here and below we use the same symbol $e$ for the image of the current class
in $B_n$.

\begin{lemma}
\label{e8:lem:E8-Bn-constituent-level-bound}
Let $M$ be a finite-dimensional $B_n$-module.  If
$L_{E_7}(\mu)$ occurs in the restriction of $M$ to $E_7$, then
\begin{equation}
  \langle\mu,\vartheta^{\vee}\rangle\le n.
  \label{e8:eq:E8-Bn-constituent-level-bound}
\end{equation}
\end{lemma}

\begin{proof}
Consider the $\mathfrak{sl}_2$-subalgebra generated by
$e_{\vartheta}$, $f_{\vartheta}$, and $\vartheta^{\vee}$, and set
\begin{equation*}
  m=\langle\mu,\vartheta^{\vee}\rangle.
\end{equation*}
If $v_{\mu}$ is an $E_7$-highest-weight vector, then
$f_{\vartheta}^{m}v_{\mu}\neq0$.  The standard
$\mathfrak{sl}_2$ formulas give, for every $0\le r\le m$,
\begin{equation}
  e_{\vartheta}^{r}f_{\vartheta}^{m}v_{\mu}
  =c_{m,r}f_{\vartheta}^{m-r}v_{\mu},
  \qquad c_{m,r}\neq0.
  \label{e8:eq:E8-highest-root-string-nonzero}
\end{equation}
If $m\ge n+1$, taking $r=n+1$ in
\eqref{e8:eq:E8-highest-root-string-nonzero} shows that
$e_{\vartheta}^{n+1}$ acts nontrivially on $L_{E_7}(\mu)$, contradicting
\eqref{e8:eq:E8-highest-root-nilpotency-Bn}.  Hence $m\le n$, which is
\eqref{e8:eq:E8-Bn-constituent-level-bound}.
\end{proof}

\subsubsection{The trace of the quadratic term}

Let $M$ be a finite-dimensional $B_n$-module and denote the resulting
$E_7$-representation by
\begin{equation*}
  \rho:E_7\longrightarrow\operatorname{End}_{\C}(M).
\end{equation*}
Put
\begin{equation*}
  N_M=\dim M,
  \qquad
  \tau_M=\tr_M\rho(\Omega_{E_7}).
\end{equation*}
Since $E_7$ is simple, the trace form of $\rho$ is a scalar multiple of the
normalized invariant form.  Thus there exists $I_M\in\C$ such that
\begin{equation}
  \tr_M\bigl(\rho(a)\rho(b)\bigr)
  =I_M(a\mid b),
  \qquad a,b\in E_7.
  \label{e8:eq:E8-Dynkin-index-definition}
\end{equation}
Taking the trace of the Casimir in an orthonormal basis gives
\begin{equation}
  \tau_M=133I_M,
  \qquad
  I_M=\frac{\tau_M}{133}.
  \label{e8:eq:E8-Dynkin-index-Casimir-trace}
\end{equation}

\begin{lemma}
\label{e8:lem:E8-trace-quadratic-term}
For all $u,v\in\g_{-\frac12}\cong L_{E_7}(\omega_7)$,
\begin{equation}
  \tr_M\rho\bigl(Q(u,v)\bigr)
  =\frac{57}{133}\,\tau_M\,\langle u,v\rangle.
  \label{e8:eq:E8-trace-quadratic-term}
\end{equation}
\end{lemma}

\begin{proof}
Choose an orthonormal basis $\{a_i\}_{i=1}^{133}$ of $E_7$.  Then
$a^i=a_i$, and \eqref{e8:eq:E8-Dynkin-index-definition} gives
\begin{align}
  \tr_M\rho\bigl(Q(u,v)\bigr)
  &=2I_M\sum_{i=1}^{133}
    \langle[a_i,u],[v,a_i]\rangle.
  \label{e8:eq:E8-trace-Q-first-contraction}
\end{align}
The symplectic form \eqref{e8:eq:E8-half-piece-symplectic-form} is
$E_7$-invariant, so
\begin{equation*}
  \langle au,w\rangle+\langle u,aw\rangle=0,
  \qquad a\in E_7.
\end{equation*}
Since $[v,a_i]=-a_iv$, we obtain
\begin{align*}
  \sum_{i=1}^{133}\langle[a_i,u],[v,a_i]\rangle
  &=-\sum_{i=1}^{133}\langle a_iu,a_iv\rangle
   \notag\\
  &=\left\langle u,\sum_{i=1}^{133}a_i^2v\right\rangle.
\end{align*}
The Casimir eigenvalue on $L_{E_7}(\omega_7)$ is
\begin{equation*}
  c_2(\omega_7)
  =(\omega_7\mid\omega_7+2\rho_{E_7})
  =\frac{57}{2}.
\end{equation*}
Therefore
\begin{equation}
  \sum_{i=1}^{133}\langle[a_i,u],[v,a_i]\rangle
  =\frac{57}{2}\langle u,v\rangle.
  \label{e8:eq:E8-quadratic-contraction-value}
\end{equation}
Substitution of
\eqref{e8:eq:E8-Dynkin-index-Casimir-trace} and
\eqref{e8:eq:E8-quadratic-contraction-value} into
\eqref{e8:eq:E8-trace-Q-first-contraction} proves
\eqref{e8:eq:E8-trace-quadratic-term}.
\end{proof}

\subsubsection{The uniform Casimir trace identity}

\begin{proposition}[Uniform Casimir trace identity]
\label{e8:prop:E8-uniform-Casimir-trace-identity}
Let $M$ be a nonzero finite-dimensional $B_n$-module on which the central
class $L$ acts as the scalar $h\in\C$.  Then
\begin{equation}
  \frac{\tr_M\rho(\Omega_{E_7})}{\dim M}
  =\frac{133}{190}
   \left(
     2(n+24)h+\frac{n(n+4)}2
   \right).
  \label{e8:eq:E8-uniform-Casimir-trace-identity}
\end{equation}
\end{proposition}

\begin{proof}
Lemma~\ref{e8:lem:E8-trace-quadratic-term} gives
$\gamma=57/133$.  Proposition~\ref{prop:common-Ramond-Casimir-trace},
with $k_n=n-6$, $h^\vee(E_8)=30$, and
$p_{E_8}(k_n)=n(n+4)$, gives
\[
 \frac{190}{133}\frac{\tau_M}{N_M}
 =2(n+24)h+\frac{n(n+4)}2.
\]
This is precisely \eqref{e8:eq:E8-uniform-Casimir-trace-identity}.
\end{proof}

\begin{remark}

For $n=3$, Proposition~\ref{e8:prop:E8-uniform-Casimir-trace-identity}
specializes to
\begin{equation*}
  \frac{\tr_M\rho(\Omega_{E_7})}{\dim M}
  =\frac{133}{190}\left(54h+\frac{21}{2}\right),
\end{equation*}
recovering the level-$-3$ formula.
\end{remark}

The two outputs of this section are the highest-root level constraint
\eqref{e8:eq:E8-Bn-constituent-level-bound} and the trace identity
\eqref{e8:eq:E8-uniform-Casimir-trace-identity}.  The next section determines
the largest and second-largest $E_7$ Casimir eigenvalues under precisely this
level constraint.

\subsection{Casimir bounds for low-level \texorpdfstring{$E_{7}$}{E7}-weights}
\label{e8:sec:E8-E7-Casimir-bounds}

The Ramond Zhu relation from Section~\ref{e8:sec:lisse-Ramond-trace}
restricts every finite-dimensional $E_7$-constituent to highest-root level
at most $n$.  In this section we determine, uniformly for $1\le n\le5$,
the largest and second-largest quadratic Casimir eigenvalues under this
constraint.

\subsubsection{The level functional and the Casimir form}

Write a dominant integral weight of $E_7$ as
\begin{equation*}
  \mu=\sum_{i=1}^{7}m_i\omega_i,
  \qquad m_i\in\Z_{\ge0}.
\end{equation*}
In the Bourbaki numbering, the highest root is
\begin{equation}
  \vartheta
  =2\alpha_1+2\alpha_2+3\alpha_3+4\alpha_4
   +3\alpha_5+2\alpha_6+\alpha_7.
  \label{e8:eq:E8-E7-highest-root-expansion}
\end{equation}
Since $E_7$ is simply laced,
\begin{equation}
  \ell(\mu)
  :=\langle\mu,\vartheta^\vee\rangle
  =2m_1+2m_2+3m_3+4m_4+3m_5+2m_6+m_7.
  \label{e8:eq:E8-E7-level-functional}
\end{equation}
We shall consider the finite set
\begin{equation*}
  \mathcal P_n
  =\{\mu\in P_+(E_7)\mid \ell(\mu)\le n\},
  \qquad 1\le n\le5.
\end{equation*}

For a dominant integral weight $\mu$, let
\begin{equation*}
  c_2(\mu)
  =(\mu\mid\mu+2\rho_{E_7})
\end{equation*}
be the eigenvalue of the normalized quadratic Casimir $\Omega_{E_7}$ on
$L_{E_7}(\mu)$.  With the convention of
Section~\ref{e8:sec:preliminaries}, the inverse Cartan matrix is
\begin{equation}
 A_{E_7}^{-1}
 =\begin{pmatrix}
 2&2&3&4&3&2&1\\
 2&\frac72&4&6&\frac92&3&\frac32\\
 3&4&6&8&6&4&2\\
 4&6&8&12&9&6&3\\
 3&\frac92&6&9&\frac{15}{2}&5&\frac52\\
 2&3&4&6&5&4&2\\
 1&\frac32&2&3&\frac52&2&\frac32
 \end{pmatrix}.
 \label{e8:eq:E8-E7-inverse-Cartan-matrix}
\end{equation}
Thus, for
$\mathbf m=(m_1,\dots,m_7)^{\mathsf T}$ and
$\mathbf 1=(1,\dots,1)^{\mathsf T}$,
\begin{equation}
  c_2(\mu)
  =\mathbf m^{\mathsf T}A_{E_7}^{-1}
    (\mathbf m+2\mathbf 1).
  \label{e8:eq:E8-E7-Casimir-matrix-formula}
\end{equation}
The root and weight conventions in
\eqref{e8:eq:E8-E7-highest-root-expansion}--
\eqref{e8:eq:E8-E7-Casimir-matrix-formula} agree with
\cite[Planche~VI]{Bourbaki2002}.

\begin{lemma}
\label{e8:lem:E8-Casimir-increases-in-omega7}
For every dominant integral weight $\mu$,
\begin{equation*}
  c_2(\mu+\omega_7)-c_2(\mu)
  =2(\mu\mid\omega_7)+\frac{57}{2}>0.
\end{equation*}
\end{lemma}

\begin{proof}
Using $c_2(\mu)=(\mu\mid\mu+2\rho_{E_7})$, one obtains
\begin{align*}
 c_2(\mu+\omega_7)-c_2(\mu)
 &=2(\mu\mid\omega_7)
   +(\omega_7\mid\omega_7+2\rho_{E_7})
 \notag\\
 &=2(\mu\mid\omega_7)+\frac{57}{2}.
\end{align*}
All entries in the seventh column of
$A_{E_7}^{-1}$ are positive, so $(\mu\mid\omega_7)\ge0$.
\end{proof}

\subsubsection{The largest and second-largest eigenvalues}

For $\nu=\sum_{i=1}^{6}m_i\omega_i$, put
\begin{equation}
 q(\nu)=2m_1+2m_2+3m_3+4m_4+3m_5+2m_6.
 \label{e8:eq:E8-E7-nonomega7-level}
\end{equation}
If $q(\nu)\le n$, define the saturated weight
\begin{equation*}
 \mu_n(\nu)=\nu+(n-q(\nu))\omega_7
\end{equation*}
and its Casimir deficit from $n\omega_7$ by
\begin{equation*}
 D_n(\nu)
 =c_2(n\omega_7)-c_2(\mu_n(\nu)).
\end{equation*}

\begin{lemma}
\label{e8:lem:E8-E7-deficit-classification}
Let $1\le n\le5$ and let $\nu\ne0$ satisfy $q(\nu)\le n$.
Then $q(\nu)\in\{2,3,4,5\}$.  For each possible value of $q(\nu)$,
the minimum of $D_n(\nu)$ is as follows:
\begin{equation*}
\begin{array}{c|c|c}
q(\nu)&\min D_n(\nu)&\text{unique minimizing }\nu\\
\hline
2&2n&\omega_6\\
3&4n&\omega_5\\
4&4(n-1)&2\omega_6\\
5&6(n-1)&\omega_5+\omega_6.
\end{array}
\end{equation*}
The rows $q=4$ and $q=5$ occur only for $n\ge4$ and $n=5$,
respectively.
\end{lemma}

\begin{proof}
The possible nonzero $\nu$ of level at most five are obtained directly from
\eqref{e8:eq:E8-E7-nonomega7-level}.  At level two they are
\begin{equation*}
  \omega_1,\quad\omega_2,\quad\omega_6,
\end{equation*}
and substitution into
\eqref{e8:eq:E8-E7-Casimir-matrix-formula} gives, respectively,
\begin{equation*}
  4(n+4),\qquad
  \frac{3(2n+1)}2,
  \qquad 2n.
\end{equation*}
At level three the possibilities are $\omega_3$ and $\omega_5$, with
deficits
\begin{equation*}
  \frac{5(2n+3)}2,
  \qquad 4n.
\end{equation*}

At level four the possible $\nu$, in the order displayed below, are
\begin{equation*}
 \begin{split}
  &2\omega_6,\ \omega_4,\ \omega_2+\omega_6,\ 2\omega_2,\\
  &\omega_1+\omega_6,\ \omega_1+\omega_2,\ 2\omega_1,
 \end{split}
\end{equation*}
and their deficits are
\begin{equation*}
 \begin{split}
  &4(n-1),\ 6n,\ \frac{5(2n-1)}2,\ 2(3n-2),\\
  &6(n+2),\ \frac{14n+23}{2},\ 8(n+3).
 \end{split}
\end{equation*}
Finally, at level five the possible $\nu$ are
\begin{equation*}
 \begin{split}
  &\omega_5+\omega_6,\ \omega_3+\omega_6,\
    \omega_2+\omega_5,\\
  &\omega_2+\omega_3,\ \omega_1+\omega_5,\
    \omega_1+\omega_3,
 \end{split}
\end{equation*}
and the corresponding deficits are
\begin{equation*}
 \begin{split}
  &6(n-1),\ \frac{14n+3}{2},\ \frac{14n-13}{2},\\
  &8n,\ 8(n+1),\ \frac{9(2n+3)}2.
 \end{split}
\end{equation*}
The asserted minima and their uniqueness are immediate from these lists,
using $n\ge q(\nu)$ in each row.
\end{proof}

\begin{theorem}
\label{e8:thm:E8-E7-Casimir-maxima}
Let $1\le n\le5$ and $\mu\in\mathcal P_n$.
Then
\begin{equation*}
  c_2(\mu)
  \le C_{\max}(n)
  :=\frac{3n(n+18)}2,
\end{equation*}
and equality holds if and only if
\begin{equation*}
  \mu=n\omega_7.
\end{equation*}
For $2\le n\le5$, if $\mu\ne n\omega_7$, then
\begin{equation*}
  c_2(\mu)
  \le C_{\mathrm{second}}(n)
  :=\frac{n(3n+50)}2,
\end{equation*}
and equality holds if and only if
\begin{equation*}
  \mu=\omega_6+(n-2)\omega_7.
\end{equation*}
For $n=1$, the only weights in $\mathcal P_1$ are $0$ and $\omega_7$,
so the second-largest Casimir eigenvalue is $0$.
\end{theorem}

\begin{proof}
The seventh row of \eqref{e8:eq:E8-E7-inverse-Cartan-matrix} gives
\begin{equation*}
  (\omega_7\mid\omega_7)=\frac32,
  \qquad
  (\omega_7\mid\rho_{E_7})=\frac{27}{2}.
\end{equation*}
Hence
\begin{equation*}
 c_2(n\omega_7)
 =\frac32n^2+27n
 =\frac{3n(n+18)}2.
\end{equation*}

Write $\mu=\nu+m_7\omega_7$ with
$\nu=\sum_{i=1}^{6}m_i\omega_i$.  If $\nu\ne0$ and
$\ell(\mu)<n$, Lemma~\ref{e8:lem:E8-Casimir-increases-in-omega7} allows us
to increase $m_7$ until the level is saturated.  The resulting weight is
still different from $n\omega_7$ and has strictly larger Casimir eigenvalue.
Thus every non-pure candidate for the second-largest value is bounded by a
saturated weight $\mu_n(\nu)$.

If $\nu=0$, the largest weight different from $n\omega_7$ is
$(n-1)\omega_7$, and
\begin{equation*}
 c_2(n\omega_7)-c_2((n-1)\omega_7)
 =\frac{3(2n+17)}2>2n.
\end{equation*}
For saturated weights with $\nu\ne0$, Lemma~\ref{e8:lem:E8-E7-deficit-classification}
shows that the smallest possible positive deficit is $2n$, uniquely attained
at $q(\nu)=2$ and $\nu=\omega_6$.  Indeed,
\begin{equation*}
  4n>2n,
  \qquad
  4(n-1)>2n\quad(n=4,5),
  \qquad
  6(n-1)>2n\quad(n=5).
\end{equation*}
It follows that $n\omega_7$ is the unique maximum and that, for $n\ge2$,
the unique second maximizer is
$\omega_6+(n-2)\omega_7$.  Its eigenvalue is
\begin{align*}
 c_2(\omega_6+(n-2)\omega_7)
 &=c_2(n\omega_7)-2n
 \notag\\
 &=\frac{n(3n+50)}2.
\end{align*}
The case $n=1$ follows immediately from
\eqref{e8:eq:E8-E7-level-functional}.
\end{proof}

\begin{corollary}
\label{e8:cor:E8-Casimir-average-bounds}
Let $M$ be a nonzero finite-dimensional $B_n$-module, where
$1\le n\le5$.  Then
\begin{equation*}
  \frac{\tr_M\rho(\Omega_{E_7})}{\dim M}
  \le \frac{3n(n+18)}2.
\end{equation*}
If $2\le n\le5$ and the restriction of $M$ to $E_7$ contains no copy of
$L_{E_7}(n\omega_7)$, then
\begin{equation*}
  \frac{\tr_M\rho(\Omega_{E_7})}{\dim M}
  \le \frac{n(3n+50)}2.
\end{equation*}
For $n=1$, the corresponding upper bound in the absence of
$L_{E_7}(\omega_7)$ is $0$.
\end{corollary}

\begin{proof}
By Lemma~\ref{e8:lem:E8-Bn-constituent-level-bound}, every irreducible
$E_7$-constituent of $M$ belongs to $\mathcal P_n$.  Since finite-dimensional
$E_7$-modules are completely reducible, the Casimir average is a
dimension-weighted average of the eigenvalues $c_2(\mu)$.  The result follows
from Theorem~\ref{e8:thm:E8-E7-Casimir-maxima}.
\end{proof}

For later use, define the Casimir average forced by the trace identity at the
spectral-flow ground energy by
\begin{equation}
  A_n:=\frac{7n(n+19)}5.
  \label{e8:eq:E8-ground-state-Casimir-average}
\end{equation}
For $2\le n\le5$, direct subtraction gives
\begin{equation}
  C_{\max}(n)-A_n=\frac{n(n+4)}{10}>0,
  \qquad
  A_n-C_{\mathrm{second}}(n)=\frac{n(16-n)}{10}>0.
  \label{e8:eq:E8-Casimir-strict-gap}
\end{equation}
For $n=1$, one has
\begin{equation}
  0<A_1=28<\frac{57}{2}=C_{\max}(1).
  \label{e8:eq:E8-Casimir-gap-n1}
\end{equation}
These strict inequalities are the numerical input for the spectral-flow
simplicity argument in the next section.

\subsection{Spectral flow and simplicity of the reduced \texorpdfstring{$W$}{W}-quotients}

In this section we prove that the natural homomorphisms
\eqref{e8:eq:E8-natural-Wtilde-to-simple-W} are isomorphisms for all
$1\le n\le5$.  The argument combines Li's delta operator, the Ramond Zhu
trace identity of Section~\ref{e8:sec:lisse-Ramond-trace}, and the Casimir gap
established in Section~\ref{e8:sec:E8-E7-Casimir-bounds}.

\subsubsection{The Ramond spectral-flow twist}

Let
\begin{equation*}
  x=\omega_7^\vee\in\mathfrak h_{E_7}
\end{equation*}
and let
\begin{equation*}
  H=J^{\{x\}}\in (\widetilde{\W}_n)_1.
\end{equation*}
Since the $E_7$ current algebra in $\widetilde{\W}_n$ has level $n$, the
current relations give
\begin{equation}
  H_{(1)}H=n(x\mid x)\Vac=\frac{3n}{2}\Vac.
  \label{e8:eq:E8-H-one-H-value}
\end{equation}

\begin{lemma}

The inner automorphism
\begin{equation*}
  \exp(2\pi iH_{(0)})
\end{equation*}
agrees with the Ramond automorphism $\sigma_{\mathrm R}$ on
$\widetilde{\W}_n$.
\end{lemma}

\begin{proof}
For a root vector $a_\alpha\in(E_7)_\alpha$, one has
$H_{(0)}J^{\{a_\alpha\}}=\alpha(x)J^{\{a_\alpha\}}$.  Since $x$ is a
coweight, $\alpha(x)\in\Z$, and hence the inner automorphism fixes every
$E_7$ current.

Every weight $\mu$ of the minuscule module $L_{E_7}(\omega_7)$ belongs to
$\omega_7+Q(E_7)$.  Therefore
\begin{equation}
  \mu(x)\equiv\omega_7(x)
  =(\omega_7\mid\omega_7)=\frac32\pmod{\Z}.
  \label{e8:eq:E8-minuscule-x-charge-congruence}
\end{equation}
Thus the inner automorphism acts by $-1$ on every field $G^{\{u\}}$ and
fixes the conformal vector.  The images of the fields in
\eqref{e8:eq:E8-minimal-W-generators} generate $\widetilde{\W}_n$, so the two
automorphisms agree.
\end{proof}

Apply the Li construction of Subsection~\ref{e8:subsec:E8-Li-delta} to the
adjoint module of $\widetilde{\W}_n$, and denote the resulting Ramond-twisted
module by
\begin{equation*}
  \widetilde{\W}_n^{\mathrm R}.
\end{equation*}

\begin{lemma}

On $\widetilde{\W}_n^{\mathrm R}$, the conformal Hamiltonian and the shifted
Cartan zero modes are
\begin{equation}
  L_0^{\mathrm R}=L_0+H_0+\frac{3n}{4}
  \label{e8:eq:E8-Li-shift-L0}
\end{equation}
and
\begin{equation}
  J_0^{\{h\},\mathrm R}
  =J_0^{\{h\}}+n(x\mid h),
  \qquad h\in\mathfrak h_{E_7}.
  \label{e8:eq:E8-Li-shift-Cartan-zero-modes}
\end{equation}
\end{lemma}

\begin{proof}
Equation~\eqref{e8:eq:E8-Li-shift-L0} follows from
\eqref{eq:common-Li-Hamiltonian} and
\eqref{e8:eq:E8-H-one-H-value}.  For $h\in\mathfrak h_{E_7}$, the affine
current relations imply
\begin{equation*}
  H_{(0)}J^{\{h\}}=0,
  \qquad
  H_{(1)}J^{\{h\}}=n(x\mid h)\Vac,
  \qquad
  H_{(m)}J^{\{h\}}=0\quad(m\ge2).
\end{equation*}
Substitution into Li's delta operator gives
\begin{equation*}
  \Delta(H,z)J^{\{h\}}
  =J^{\{h\}}+n(x\mid h)z^{-1}\Vac,
\end{equation*}
which yields \eqref{e8:eq:E8-Li-shift-Cartan-zero-modes}.
\end{proof}

\subsubsection{The integral spectrum and the extremal vacuum line}

The minuscule property of $x=\omega_7^\vee$ gives
\begin{equation}
  \alpha(x)\in\{-1,0,1\}
  \label{e8:eq:E8-root-x-charges}
\end{equation}
for every root $\alpha$ of $E_7$.  The weights of
$L_{E_7}(\omega_7)$ form one Weyl orbit, and
\eqref{e8:eq:E8-minuscule-x-charge-congruence} gives
\begin{equation}
  \mu(x)\in
  \left\{-\frac32,-\frac12,\frac12,\frac32\right\}.
  \label{e8:eq:E8-minuscule-x-charges}
\end{equation}

\begin{proposition}
\label{e8:prop:E8-Ramond-spectrum-vacuum-line}
For $1\le n\le5$, the operator $L_0^{\mathrm R}$ acts semisimply on
$\widetilde{\W}_n^{\mathrm R}$, and
\begin{equation}
  \Spec_{\widetilde{\W}_n^{\mathrm R}}L_0^{\mathrm R}
  \subset \frac{3n}{4}+\Z_{\ge0}.
  \label{e8:eq:E8-Ramond-integral-spectrum}
\end{equation}
Moreover,
\begin{equation}
  \left(\widetilde{\W}_n^{\mathrm R}
  \left[\frac{3n}{4}\right]\right)_{n\omega_7}
  =\C\Vac,
  \label{e8:eq:E8-extremal-vacuum-line}
\end{equation}
where the subscript denotes the weight for the shifted $E_7$ zero-mode
action in \eqref{e8:eq:E8-Li-shift-Cartan-zero-modes}.
\end{proposition}

\begin{proof}
Every ordinary conformal homogeneous subspace of $\widetilde{\W}_n$ is
finite-dimensional and stable under the $E_7$-zero-mode action.  Complete
reducibility of finite-dimensional $E_7$-modules implies that $H_0$ acts
semisimply on each such subspace and commutes with $L_0$.  Therefore
$L_0^{\mathrm R}=L_0+H_0+3n/4$ acts semisimply on the twisted adjoint
module.

The quotient $\widetilde{\W}_n$ is spanned by normally ordered monomials in
factors of the form
\begin{equation*}
  \partial^rJ^{\{a_\alpha\}},
  \qquad
  \partial^sG^{\{u_\mu\}},
  \qquad
  \partial^t\omega,
\end{equation*}
where $r,s,t\in\Z_{\ge0}$ and the vectors are chosen to be $H_0$-eigenvectors.
Here $a_{\alpha}$ ranges over an $x$-weight basis of $E_7$, and the value
$\alpha=0$ is allowed; these zero-weight vectors include the Cartan currents.
By \eqref{e8:eq:E8-Li-shift-L0}, their contributions to
$L_0^{\mathrm R}-3n/4=L_0+H_0$ are, respectively,
\begin{equation}
  1+r+\alpha(x),
  \qquad
  \frac32+s+\mu(x),
  \qquad
  2+t.
  \label{e8:eq:E8-Li-energy-contributions}
\end{equation}
Equations~\eqref{e8:eq:E8-root-x-charges} and
\eqref{e8:eq:E8-minuscule-x-charges} show that all three quantities in
\eqref{e8:eq:E8-Li-energy-contributions} are nonnegative integers.  This proves
\eqref{e8:eq:E8-Ramond-integral-spectrum}.

The vacuum has twisted conformal weight $3n/4$.  Its shifted $E_7$ weight is
$n\omega_7$, because \eqref{e8:eq:E8-Li-shift-Cartan-zero-modes} gives
\begin{equation*}
  J_0^{\{h\},\mathrm R}\Vac
  =n(x\mid h)\Vac
  =(n\omega_7\mid h)\Vac.
\end{equation*}
A nonempty spanning monomial can have twisted conformal weight $3n/4$ only
if every contribution in \eqref{e8:eq:E8-Li-energy-contributions} is zero.
Every current factor then has $r=0$ and $\alpha(x)=-1$, every $G$-factor has
$s=0$ and $\mu(x)=-3/2$, and no Virasoro factor occurs.  Hence every such
nonempty monomial has strictly negative original $x$-charge.  If its shifted
$E_7$ weight were $n\omega_7$, its original $E_7$ weight would be zero and
its original $x$-charge would also be zero, a contradiction.  Thus only the
empty monomial contributes to the weight space in
\eqref{e8:eq:E8-extremal-vacuum-line}.
\end{proof}

\subsubsection{Simplicity of the reduced quotients}

Recall the natural nonzero surjection
\begin{equation}
  \rho_n\colon\widetilde{\W}_n
  \twoheadrightarrow\W_{k_n}(E_8,f_\theta)
  \label{e8:eq:E8-rho-n-W-surjection}
\end{equation}
from \eqref{e8:eq:E8-natural-Wtilde-to-simple-W}.

\begin{theorem}
\label{e8:thm:E8-Wtilde-simple-all-n}
For every $1\le n\le5$, the homomorphism \eqref{e8:eq:E8-rho-n-W-surjection}
is an isomorphism.  Equivalently,
\begin{equation*}
  \widetilde{\W}_n
  \cong\W_{n-6}(E_8,f_\theta).
\end{equation*}
\end{theorem}

\begin{proof}
Fix $1\le n\le5$ and apply
Proposition~\ref{prop:common-Casimir-gap-simplicity} to the surjection
\eqref{e8:eq:E8-rho-n-W-surjection}.  Proposition
\ref{e8:prop:E8-Ramond-spectrum-vacuum-line} gives
\[
 h_0=\frac{3n}{4},
 \qquad
 \lambda_0=n\omega_7,
 \qquad
 \left(\widetilde{\W}_n^{\mathrm R}[h_0]\right)_{\lambda_0}
 =\C\Vac,
\]
and Corollary~\ref{e8:cor:E8-Bn-finite-dimensional} gives the finite
Ramond Zhu algebra.  The weighted Casimir average is
\[
 A_n(h)=\frac{133}{190}
 \left(2(n+24)h+\frac{n(n+4)}2\right)
\]
by Proposition~\ref{e8:prop:E8-uniform-Casimir-trace-identity}, whereas
Theorem~\ref{e8:thm:E8-E7-Casimir-maxima} gives
$C_{\max}=3n(n+18)/2$.  A direct simplification gives
\[
 A_n(h_0+1)-C_{\max}
 =\frac{(24-n)(n+14)}{10}>0,
 \qquad 1\le n\le5.
\]
Hence the lowest energy of a nonzero ideal would have to be $h_0$.

At this energy the trace average equals
$A_n=7n(n+19)/5$ by
\eqref{e8:eq:E8-ground-state-Casimir-average}.  For $2\le n\le5$, the
strict gap \eqref{e8:eq:E8-Casimir-strict-gap} forces every simple Zhu
quotient to contain $L_{E_7}(n\omega_7)$.  For $n=1$,
Corollary~\ref{e8:cor:E8-Casimir-average-bounds} shows that a Zhu quotient
containing no copy of $L_{E_7}(\omega_7)$ has Casimir average at most zero,
whereas \eqref{e8:eq:E8-Casimir-gap-n1} gives $A_1=28>0$.  Thus the same
extremal constituent is forced in the level-one case.  Thus all hypotheses of the
common criterion hold, and \eqref{e8:eq:E8-rho-n-W-surjection} is an
isomorphism.
\end{proof}

Combining Theorem~\ref{e8:thm:E8-Wtilde-simple-all-n} with
Proposition~\ref{e8:prop:E8-reduction-of-Qn}, we obtain the uniform reduction
formula
\begin{equation}
  H^0_{\DS,f_\theta}(Q_n)
  \cong\W_{n-6}(E_8,f_\theta),
  \qquad 1\le n\le5.
  \label{e8:eq:E8-reduction-Qn-simple-W-all-n}
\end{equation}
Proposition~\ref{prop:common-affine-lifting} lifts this isomorphism to the
affine quotients $Q_n$.

\subsection{Maximal ideals of the affine vertex algebras}

We now apply Proposition~\ref{prop:common-affine-lifting}.  Recall that
for $1\le n\le5$,
\begin{equation*}
  k_n=n-6,
  \qquad
  Q_n
  =V^{k_n}(E_8)/\left\langle\sigma_{k_n}(w)^{n+1}\right\rangle,
\end{equation*}
and that the canonical projection onto the simple affine vertex algebra
factors as
\begin{equation}
  V^{k_n}(E_8)
  \twoheadrightarrow Q_n
  \overset{\pi_n}{\twoheadrightarrow}L_{k_n}(E_8).
  \label{e8:eq:E8-factorization-through-Qn-final}
\end{equation}
\subsubsection{Uniform affine lifting}

For $1\le n\le5$, the quotient $Q_n$ admits the canonical surjection
$\pi_n:Q_n\twoheadrightarrow L_{k_n}(E_8)$ from
\eqref{e8:eq:E8-factorization-through-Qn-final}.
Theorem~\ref{e8:thm:E8-Wtilde-simple-all-n} and
\eqref{e8:eq:E8-reduction-Qn-is-Wtilde-n} identify the induced map on minimal
reductions with an isomorphism.  Hence
Proposition~\ref{prop:common-affine-lifting} applies.

\begin{theorem}
\label{e8:thm:E8-maximal-ideals-all-negative-cases}
For every integer $n$ with $1\le n\le5$, the kernel of the canonical
homomorphism
\begin{equation*}
  V^{n-6}(E_8)\longrightarrow L_{n-6}(E_8)
\end{equation*}
is generated by the Arakawa--Moreau singular vector
$\sigma_{n-6}(w)^{n+1}$.  Equivalently,
\begin{equation}
  \ker\!\left(
    V^{n-6}(E_8)\longrightarrow L_{n-6}(E_8)
  \right)
  =\left\langle\sigma_{n-6}(w)^{n+1}\right\rangle.
  \label{e8:eq:E8-maximal-ideal-uniform-formula}
\end{equation}
In particular,
\begin{equation*}
  V^{n-6}(E_8)
  \Big/\left\langle\sigma_{n-6}(w)^{n+1}\right\rangle
  \cong L_{n-6}(E_8).
\end{equation*}
\end{theorem}

\begin{proof}
Apply Proposition~\ref{prop:common-affine-lifting} to
$\pi_n:Q_n\twoheadrightarrow L_{k_n}(E_8)$.
\end{proof}

For clarity, the five instances of
\eqref{e8:eq:E8-maximal-ideal-uniform-formula} are
\begin{align*}
  \ker\!\left(V^{-5}(E_8)\to L_{-5}(E_8)\right)
  &=\left\langle\sigma_{-5}(w)^2\right\rangle,
  \notag\\
  \ker\!\left(V^{-4}(E_8)\to L_{-4}(E_8)\right)
  &=\left\langle\sigma_{-4}(w)^3\right\rangle,
  \notag\\
  \ker\!\left(V^{-3}(E_8)\to L_{-3}(E_8)\right)
  &=\left\langle\sigma_{-3}(w)^4\right\rangle,
  \notag\\
  \ker\!\left(V^{-2}(E_8)\to L_{-2}(E_8)\right)
  &=\left\langle\sigma_{-2}(w)^5\right\rangle,
  \notag\\
  \ker\!\left(V^{-1}(E_8)\to L_{-1}(E_8)\right)
  &=\left\langle\sigma_{-1}(w)^6\right\rangle.
\end{align*}

\subsubsection{Consequences for the Arakawa--Moreau conjectures}

\begin{corollary}

For $E_8$ and every $1\le n\le5$, the Arakawa--Moreau affine quotient is
simple and its minimal reduction is the simple minimal $W$-algebra:
\begin{align}
  V^{n-6}(E_8)
  \Big/\left\langle\sigma_{n-6}(w)^{n+1}\right\rangle
  &\cong L_{n-6}(E_8),
  \label{e8:eq:E8-AM-conjecture-one-all-n}\\
  H^0_{\DS,f_\theta}\!\left(
    V^{n-6}(E_8)
    \Big/\left\langle\sigma_{n-6}(w)^{n+1}\right\rangle
  \right)
  &\cong\W_{n-6}(E_8,f_\theta).
  \label{e8:eq:E8-AM-conjecture-three-all-n}
\end{align}
Equivalently,
\begin{equation*}
  \W^{n-6}(E_8,f_\theta)
  \Big/\left\langle:J^{\{e_\vartheta\}}{}^{n+1}:\right\rangle
  \cong\W_{n-6}(E_8,f_\theta).
\end{equation*}
\end{corollary}

\begin{proof}
Equation \eqref{e8:eq:E8-AM-conjecture-one-all-n} is
Theorem~\ref{e8:thm:E8-maximal-ideals-all-negative-cases}.  Equation
\eqref{e8:eq:E8-AM-conjecture-three-all-n} follows by applying minimal
reduction, or directly from \eqref{e8:eq:E8-reduction-Qn-simple-W-all-n}.
The last assertion is Theorem~\ref{e8:thm:E8-Wtilde-simple-all-n}.
\end{proof}

\begin{corollary}

For every $1\le n\le5$,
\begin{equation*}
  X_{L_{n-6}(E_8)}
  =\overline{\mathbb O_{\min}}.
\end{equation*}
Moreover, the simple minimal $W$-algebra
$\W_{n-6}(E_8,f_\theta)$ is lisse.
\end{corollary}

\begin{proof}
By Theorem~\ref{e8:thm:E8-maximal-ideals-all-negative-cases}, the simple affine
vertex algebra is isomorphic to $Q_n$.  The associated-variety identity now
follows from \eqref{e8:eq:E8-associated-variety-Qn}.  The lisse assertion follows from
Proposition~\ref{e8:prop:E8-associated-variety-Qn} and
Theorem~\ref{e8:thm:E8-Wtilde-simple-all-n}.
\end{proof}

\begin{remark}

Together with the previously established $n=0$ case in
\cite{ArakawaMoreau2018}, Theorem~\ref{e8:thm:E8-maximal-ideals-all-negative-cases}
settles all negative members of the Arakawa--Moreau family for type $E_8$.
The proof requires neither a classification of finite-dimensional simple
modules over the Ramond Zhu algebras $B_n$ nor a direct determination of the
full affine maximal ideals by singular-vector calculations.
\end{remark}
\section{A rank-reduction application at collapsing levels}
\label{sec:collapsing-application}

The maximal-ideal theorem in this section is already known from
\cite[Section~8]{AdamovicKacMosenederFrajriaPapiPerse2020}.  Our purpose is
not to reprove a new classification result, but to give an alternative
reduction-theoretic proof which lies outside the Arakawa--Moreau family and
shows that Theorem~\ref{thm:common-minimal-reduction-maximality} is useful
independently of the case analysis in Sections~\ref{sec:case-dl}--
\ref{sec:case-e8}.  The proof lowers the rank by two under minimal
Drinfeld--Sokolov reduction.

\subsection{The collapsing family and its singular vectors}

Fix an integer $r\ge2$ and set
\begin{equation*}
  \mathfrak g_r=D_{2r},
  \qquad
  k_r=2-2r.
\end{equation*}
Use the orthogonal realization with roots
$\pm\varepsilon_i\pm\varepsilon_j$,
$1\le i<j\le2r$.  We use the matrix root-vector convention specified in
Lemma~\ref{lem:collapsing-quadratic-relation} below and choose
\begin{equation*}
  \theta=\varepsilon_1+\varepsilon_2,
  \qquad
  f_\theta=-e_{-\theta}.
\end{equation*}
The singular vectors used in
\cite{AdamovicKacMosenederFrajriaPapiPerse2020} are the quadratic vector
\begin{equation}
  v_r=
  \sum_{i=2}^{2r}
  e_{\varepsilon_1-\varepsilon_i}(-1)
  e_{\varepsilon_1+\varepsilon_i}(-1)\Vac
  \label{eq:collapsing-vr}
\end{equation}
and two spinor Pfaffian vectors $w_r^+$ and $w_r^-$.  To fix notation, let
$\mathcal P(I)$ denote the set of perfect matchings of an even ordered set
$I$.  Up to a nonzero scalar,
\begin{equation*}
  w_r^+
  =
  \sum_{p\in\mathcal P(\{1,\ldots,2r\})}
  \operatorname{sgn}(p)
  \prod_{\{i,j\}\in p}
  e_{\varepsilon_i+\varepsilon_j}(-1)\Vac,
\end{equation*}
where the factors may be taken in any fixed order.  The vector $w_r^-$ is
obtained from $w_r^+$ by the diagram automorphism
$\varepsilon_{2r}\mapsto-\varepsilon_{2r}$.  Their finite highest weights
are $2\omega_{2r}$ and $2\omega_{2r-1}$, respectively, and their conformal
weight is $r$.

Define
\begin{equation}
  Q_r
  =
  V^{k_r}(D_{2r})
  \big/
  \langle v_r,w_r^+,w_r^-\rangle.
  \label{eq:collapsing-Qr}
\end{equation}
The three vectors belong to the maximal proper graded ideal, so there is a
canonical surjection
\begin{equation*}
  Q_r\twoheadrightarrow L_{k_r}(D_{2r}).
\end{equation*}

For the minimal grading determined by $\theta$ one has
\begin{equation*}
  \mathfrak g_r^\natural
  \cong
  \mathfrak s\oplus\mathfrak d_r,
  \qquad
  \mathfrak s\cong A_1,
  \qquad
  \mathfrak d_r\cong D_{2r-2},
\end{equation*}
where the root of $\mathfrak s$ is
$\theta_A=\varepsilon_1-\varepsilon_2$.  The standard current-level formula of
\cite{AdamovicKacMosenederFrajriaPapiPerse2018} gives
\begin{equation}
  k^\natural_{\mathfrak s}=0,
  \qquad
  k^\natural_{\mathfrak d_r}=4-2r=k_{r-1}.
  \label{eq:collapsing-current-levels}
\end{equation}
Moreover,
\begin{equation}
  (\mathfrak g_r)_{-1/2}
  \cong
  L_{A_1}(\varpi)\boxtimes
  L_{D_{2r-2}}(\varpi'_1).
  \label{eq:collapsing-half-module}
\end{equation}
Here and below primes refer to the root datum of $D_{2r-2}$ on the
coordinates $\varepsilon_3,\ldots,\varepsilon_{2r}$.

\subsection{Exact reduction of the three generators}

Let
\begin{equation*}
  \mathcal U_r
  =
  \mathcal W^{k_r}(D_{2r},f_\theta).
\end{equation*}
Write $J_A=J^{\{e_{\theta_A}\}}$ for the highest-root current of the
$A_1$-factor.  The next proposition is the local input for the
rank-reduction argument.

\begin{proposition}[Exact rank-reducing BRST images]
\label{prop:collapsing-exact-images}
For every $r\ge3$, there are nonzero constants
$c_r,a_r^+,a_r^-\in\mathbb C^\times$ such that
\begin{equation}
  [v_r]_{\mathrm{DS}}=c_rJ_A,
  \qquad
  [w_r^+]_{\mathrm{DS}}=a_r^+w_{r-1}^+,
  \qquad
  [w_r^-]_{\mathrm{DS}}=a_r^-w_{r-1}^-.
  \label{eq:collapsing-exact-images}
\end{equation}
On the right, the Pfaffian states are formed from the surviving
$D_{2r-2}$ currents of level $k_{r-1}$.
\end{proposition}

\begin{proof}
Let
\[
 N_r^{v}=U(\widehat{\mathfrak g}_r^{\prime})v_r,
 \qquad
 N_r^{\pm}=U(\widehat{\mathfrak g}_r^{\prime})w_r^{\pm}.
\]
The three finite highest weights are respectively
$2\varepsilon_1$, $2\omega_{2r}$, and $2\omega_{2r-1}$.  For
$\theta=\varepsilon_1+\varepsilon_2$ one has
\[
 \langle2\varepsilon_1,\theta^\vee\rangle
 =\langle2\omega_{2r},\theta^\vee\rangle
 =\langle2\omega_{2r-1},\theta^\vee\rangle=2.
\]
Thus the affine highest weight $\widehat\lambda$ of each of the three
singular vectors satisfies
\[
 \widehat\lambda(\alpha_0^\vee)=k_r-2=(2-2r)-2=-2r<0.
\]
Lemma~\ref{lem:common-reduced-singular-generator} therefore shows that all
three reduced cyclic generators are nonzero.

The vector $v_r$ has affine conformal weight $2$ and finite highest weight
$2\varepsilon_1$.  Since
\[
  x_\theta=\frac{\varepsilon_1+\varepsilon_2}{2},
  \qquad
  (2\varepsilon_1)(x_\theta)=1,
\]
its reduced conformal weight is $1$, while its restriction to
$\mathfrak h^\natural$ is
$\varepsilon_1-\varepsilon_2=\theta_A$.  The corresponding weight-one
highest-weight line in the universal minimal $W$-algebra is
$\mathbb C J_A$.  Since $[v_r]_{\mathrm{DS}}\ne0$, the first formula in
\eqref{eq:collapsing-exact-images} follows.

We next treat $w_r^+$.  Its reduced conformal weight is $r-1$, and its
restricted $D_{2r-2}$-weight is
\[
  \varepsilon_3+\cdots+\varepsilon_{2r}
  =2\omega'_{2r-2}.
\]
Define
\[
  q\!\left(\sum_{i=3}^{2r}b_i\varepsilon_i\right)
  =\sum_{i=3}^{2r}b_i
\]
and extend $q$ by zero on the $A_1$-factor.  Every PBW monomial $X$ in the
standard strong generators of $\mathcal U_r$ satisfies
\[
  q(\operatorname{wt}X)\le2\Delta(X).
\]
Equality can occur only for a product of undifferentiated currents
$J^{\{e_{\varepsilon_i+\varepsilon_j}\}}$ with $3\le i<j\le2r$.
For the simultaneous weight under consideration both sides equal
$2r-2$.  Hence the whole weight space is spanned by the perfect-matching
current monomials on $I=\{3,\ldots,2r\}$.

Let $E$ be the vector space with basis indexed by $I$.  Under the natural
$\mathfrak{sl}(E)\subset D_{2r-2}$ action, the span of the root vectors
$e_{\varepsilon_i+\varepsilon_j}$ identifies with $\bigwedge^2E$.  The
target weight is the determinant weight.  Thus a vector in the
perfect-matching space which is killed by all positive difference-root
operators is an $\mathfrak{sl}(E)$-highest vector of highest weight zero,
and hence an $\mathfrak{sl}(E)$-invariant.  The invariant space
\[
 \left(S^{r-1}(\bigwedge^2E)\right)^{\mathfrak{sl}(E)}
\]
is one-dimensional and is generated by the Pfaffian.  Indeed, complete
reducibility identifies its dimension with the multiplicity of the trivial
module in $S^{r-1}(\bigwedge^2E^*)$.  This is the degree-$(r-1)$ piece of
the $SL(E)$-invariant polynomial ring on skew forms, and that piece is
spanned by the Pfaffian.  The remaining positive simple root
$\varepsilon_{2r-1}+\varepsilon_{2r}$ also annihilates every factor
$e_{\varepsilon_i+\varepsilon_j}$, since the sum of the two roots is never
a root.  Therefore the $D_{2r-2}$-highest-weight subspace is exactly the
one-dimensional Pfaffian line.  Since $[w_r^+]_{\mathrm{DS}}\ne0$, it is a
nonzero scalar multiple of $w_{r-1}^+$.

The diagram automorphism defining $w_r^-$ fixes the chosen minimal
$\mathfrak{sl}_2$-triple and commutes with reduction.  It induces the
corresponding spin-node automorphism of $D_{2r-2}$, so the third formula
follows from the second.
\end{proof}

\subsection{The rank-reducing current quotient}

Put
\begin{equation*}
  R_r=H^0_{\mathrm{DS},f_\theta}(Q_r).
\end{equation*}
By Lemma~\ref{lem:common-graded-subquotient-category-O}, the defining
short exact sequence of $Q_r$ lies in the category on which minimal
reduction is exact.  Hence there is a surjection
$\mathcal U_r\twoheadrightarrow R_r$, and the three classes in
\eqref{eq:collapsing-exact-images} vanish in $R_r$.

\begin{lemma}[Generation by the surviving current algebra]
\label{lem:collapsing-surviving-generation}
The vertex algebra $R_r$ is generated by the currents of the
$D_{2r-2}$-factor.
\end{lemma}

\begin{proof}
Since $J_A=0$ in $R_r$, stability of the kernel under the $A_1$ current
zero modes implies that every $A_1$ current vanishes.  The universal
current--$G$ relation
\begin{equation*}
  [J^{\{a\}}{}_{\lambda}G^{\{u\}}]
  =G^{\{[a,u]\}}
\end{equation*}
and \eqref{eq:collapsing-half-module} then imply that all $G$-generators
vanish as well.

It remains to eliminate the independent conformal generator.  Choose
$u,v\in(\mathfrak g_r)_{-1/2}$ with
$\langle u,v\rangle\ne0$.  The standard $G$--$G$ relation of
\cite[Lemma~3.1]{AdamovicKacMosenederFrajriaPapiPerse2018} has constant term
\begin{equation*}
  -2(k_r+h^\vee)\langle u,v\rangle\omega
  +\Phi_{u,v},
\end{equation*}
where $\Phi_{u,v}$ belongs to the differential vertex subalgebra generated
by the $\mathfrak g_r^\natural$ currents.  In $R_r$ the left-hand side and
all $A_1$-current contributions vanish.  Since
\begin{equation*}
  k_r+h^\vee(D_{2r})
  =(2-2r)+(4r-2)=2r\ne0,
\end{equation*}
the resulting identity expresses the image of $\omega$ in terms of the
surviving $D_{2r-2}$ currents.  The universal minimal $W$-algebra is
strongly generated by the currents, the $G$-fields, and $\omega$, so the
claim follows.
\end{proof}

The next relation is responsible for the strict decrease in rank.

\begin{lemma}[The rank-reducing quadratic relation]
\label{lem:collapsing-quadratic-relation}
The image in $R_r$ of the quadratic singular vector $v_{r-1}$ of
$V^{k_{r-1}}(D_{2r-2})$ is zero.
\end{lemma}

\begin{proof}
Take
\begin{equation*}
  u=e_{\varepsilon_3-\varepsilon_2},
  \qquad
  v=e_{\varepsilon_3-\varepsilon_1}
  \in(\mathfrak g_r)_{-1/2}.
\end{equation*}
Then $\langle u,v\rangle=0$ and
$[[e_\theta,u],v]^\natural=0$.  By the formula of
\cite[Lemma~3.1]{AdamovicKacMosenederFrajriaPapiPerse2018}, the constant
term of $[G^{\{u\}}{}_{\lambda}G^{\{v\}}]$ is therefore
\begin{equation}
 \sum_{\gamma}
 :J^{\{[u,y^\gamma]^\natural\}}
  J^{\{[y_\gamma,v]^\natural\}}:,
 \label{eq:collapsing-GG-contraction}
\end{equation}
where $\{y_\gamma\}$ and $\{y^\gamma\}$ are dual bases of
$(\mathfrak g_r)_{1/2}$ with respect to
$\langle x,y\rangle_{\mathrm{ne}}=(f_\theta\mid[x,y])$.
The expression in \eqref{eq:collapsing-GG-contraction} has conformal
weight $2$, $D_{2r-2}$-weight $2\varepsilon_3$, and is killed by the
positive $D_{2r-2}$ current zero modes, since both $u$ and $v$ restrict to
the highest weight $\varepsilon_3$ of the vector representation.

We first determine the relevant highest-weight line.  Put
\begin{equation*}
 X_j=
 :J^{\{e_{\varepsilon_3-\varepsilon_j}\}}
  J^{\{e_{\varepsilon_3+\varepsilon_j}\}}:,
 \qquad 4\le j\le2r.
\end{equation*}
If roots $\alpha,\beta\in\Delta(D_{2r-2})$ satisfy
$\alpha+\beta=2\varepsilon_3$, then necessarily
\[
  \{\alpha,\beta\}
  =\{\varepsilon_3-\varepsilon_j,
       \varepsilon_3+\varepsilon_j\}
\]
for a unique $j$ with $4\le j\le2r$.  Moreover,
$2\varepsilon_3$ is not a root, so no derivative of a single current has
this weight.  It follows that every quadratic current state of weight
$2\varepsilon_3$ is a linear combination
$\sum_{j=4}^{2r}c_jX_j$.  For
$4\le i\le2r-1$, using the standard matrix root-vector normalization
displayed below, the Chevalley relation for the positive difference root
$\varepsilon_i-\varepsilon_{i+1}$ gives
\begin{equation*}
 e_{\varepsilon_i-\varepsilon_{i+1}}(0)
 \sum_{j=4}^{2r}c_jX_j
 =(c_{i+1}-c_i)
 :J^{\{e_{\varepsilon_3-\varepsilon_{i+1}}\}}
  J^{\{e_{\varepsilon_3+\varepsilon_i}\}}:.
\end{equation*}
The last positive simple root
$\varepsilon_{2r-1}+\varepsilon_{2r}$ gives the same condition
$c_{2r-1}=c_{2r}$.  Hence all coefficients are equal, and the
highest-weight line is one-dimensional, generated by
\begin{equation*}
 v_{r-1}
 =\sum_{j=4}^{2r}X_j
 =\sum_{j=4}^{2r}
 e_{\varepsilon_3-\varepsilon_j}(-1)
 e_{\varepsilon_3+\varepsilon_j}(-1)\Vac.
\end{equation*}

It remains to show that the contraction is nonzero.  Use the standard
matrix root vectors
\begin{equation*}
 e_{\varepsilon_i-\varepsilon_j}=E_{i,j}-E_{-j,-i},
 \qquad
 e_{\varepsilon_i+\varepsilon_j}=E_{i,-j}-E_{j,-i},
\end{equation*}
with $(x\mid y)=\frac12\operatorname{tr}(xy)$ and
$f_\theta=-e_{-\varepsilon_1-\varepsilon_2}$.  With this choice,
$[e_\theta,f_\theta]=\theta^\vee$ and
$(e_\theta\mid f_\theta)=1$.  In
$(\mathfrak g_r)_{1/2}$ the two dual pairs relevant to the coefficient of
$X_4$ are
\begin{equation*}
 e_{\varepsilon_1+\varepsilon_4}
 \longleftrightarrow -e_{\varepsilon_2-\varepsilon_4},
 \qquad
 e_{\varepsilon_1-\varepsilon_4}
 \longleftrightarrow -e_{\varepsilon_2+\varepsilon_4}.
\end{equation*}
A direct matrix commutator calculation gives
\begin{align*}
 [u,-e_{\varepsilon_2-\varepsilon_4}]
 &=-e_{\varepsilon_3-\varepsilon_4},
 &
 [e_{\varepsilon_1+\varepsilon_4},v]
 &=-e_{\varepsilon_3+\varepsilon_4},\\
 [u,-e_{\varepsilon_2+\varepsilon_4}]
 &=-e_{\varepsilon_3+\varepsilon_4},
 &
 [e_{\varepsilon_1-\varepsilon_4},v]
 &=-e_{\varepsilon_3-\varepsilon_4}.
\end{align*}
These are the only dual-basis terms which can contribute to $X_4$, as is
seen by adding the corresponding roots.  The two resulting normally
ordered products are equal: the two current root vectors commute and have
zero invariant pairing.  Thus the coefficient of $X_4$ in
\eqref{eq:collapsing-GG-contraction} is $2$.  Since the contraction lies in
the one-dimensional highest-weight line, we obtain
\begin{equation*}
 G^{\{u\}}_{(0)}G^{\{v\}}=2v_{r-1}.
\end{equation*}
All $G$-fields vanish in $R_r$ by
Lemma~\ref{lem:collapsing-surviving-generation}; therefore
$v_{r-1}=0$ in $R_r$.
\end{proof}

\begin{proposition}[Rank-reduction surjection]
\label{prop:collapsing-rank-surjection}
For every $r\ge3$, there is a natural surjective homomorphism
\begin{equation*}
  Q_{r-1}\twoheadrightarrow R_r.
\end{equation*}
\end{proposition}

\begin{proof}
By Lemma~\ref{lem:collapsing-surviving-generation} and
\eqref{eq:collapsing-current-levels}, the universal property of the affine
vertex algebra gives a surjection
\begin{equation*}
  V^{k_{r-1}}(D_{2r-2})\twoheadrightarrow R_r.
\end{equation*}
Lemma~\ref{lem:collapsing-quadratic-relation} places $v_{r-1}$ in its
kernel.  Proposition~\ref{prop:collapsing-exact-images} places
$w_{r-1}^+$ and $w_{r-1}^-$ in the kernel as well.  The map therefore
factors through the quotient $Q_{r-1}$ defined in
\eqref{eq:collapsing-Qr} with $r$ replaced by $r-1$.
\end{proof}

\subsection{The induction and the independent application}

At the levels $k_r$, the simple minimal $W$-algebra is collapsing:
\begin{equation}
  H^0_{\mathrm{DS},f_\theta}
  \bigl(L_{k_r}(D_{2r})\bigr)
  \cong
  L_{k_{r-1}}(D_{2r-2}).
  \label{eq:collapsing-simple-reduction}
\end{equation}
This is the type-$D$ collapsing family of
\cite{AdamovicKacMosenederFrajriaPapiPerse2018,
AdamovicKacMosenederFrajriaPapiPerse2020}.

\begin{theorem}[Rank-reduction proof at the collapsing levels]
\label{thm:collapsing-rank-induction}
For every integer $r\ge2$,
\begin{equation}
  \operatorname{Rad}V^{2-2r}(D_{2r})
  =
  \langle v_r,w_r^+,w_r^-\rangle.
  \label{eq:collapsing-maximal-ideal}
\end{equation}
Equivalently,
$Q_r\cong L_{2-2r}(D_{2r})$.
\end{theorem}

\begin{proof}
We argue by induction on $r$.  For $r=2$, the three vectors become the
triality-related conformal-weight-two singular vectors of
$V^{-2}(D_4)$.  They generate the maximal ideal by
\cite[Theorem~4.2]{Perse2013}.

Assume that
\begin{equation*}
  Q_{r-1}\cong L_{k_{r-1}}(D_{2r-2})
\end{equation*}
is simple.  Proposition~\ref{prop:collapsing-rank-surjection} gives a
surjection
\begin{equation*}
  Q_{r-1}\twoheadrightarrow R_r.
\end{equation*}
On the other hand, reduction of the canonical map
$Q_r\twoheadrightarrow L_{k_r}(D_{2r})$, together with
\eqref{eq:collapsing-simple-reduction}, gives a surjection
\begin{equation*}
  R_r\twoheadrightarrow L_{k_{r-1}}(D_{2r-2}).
\end{equation*}
Thus there is a chain
\begin{equation*}
  L_{k_{r-1}}(D_{2r-2})
  \cong Q_{r-1}
  \twoheadrightarrow R_r
  \twoheadrightarrow L_{k_{r-1}}(D_{2r-2}).
\end{equation*}
Its composite is a nonzero unital endomorphism of a simple vertex algebra.
Consequently both arrows are isomorphisms, and $R_r$ is nonzero and
simple.  The minimal-reduction maximality principle,
Theorem~\ref{thm:common-minimal-reduction-maximality}, now gives
\begin{equation*}
  Q_r\cong L_{k_r}(D_{2r}).
\end{equation*}
This completes the induction.
\end{proof}

\begin{remark}
The conclusion \eqref{eq:collapsing-maximal-ideal} was obtained in
\cite{AdamovicKacMosenederFrajriaPapiPerse2020} by a different method.
The point of Theorem~\ref{thm:collapsing-rank-induction} is that the
minimal reduction produces the previous member of the same family and the
general maximality principle lifts the resulting simplicity back to the
affine quotient.  Thus the mechanism isolated in Section~\ref{sec:common-framework}
is not specific to the Arakawa--Moreau conjecture.
\end{remark}
\appendix
\section{Root and Casimir data for \texorpdfstring{$D_\ell$}{D-l}}

\subsection{Root, weight, and Casimir data}
\label{dl:app:Dl-root-weight-Casimir-data}

This appendix records the elementary root-theoretic and Casimir
computations used in
Sections~\ref{dl:sec:Dl-singular-vectors-exact-reductions},
\ref{dl:sec:Dl-Ramond-weighted-Casimir},
\ref{dl:sec:Dl-weight-cosets-lowest-Ramond},
\ref{dl:sec:Dl-simplicity-reduced-quotient}, and
\ref{dl:sec:Dl-level-minus-two}.  All roots are normalized to have
squared length \(2\).

\subsubsection{The root and weight lattices of \(D_\ell\)}

Let
\[
  E_\ell
  =
  \bigoplus_{i=1}^{\ell}\mathbb R\varepsilon_i,
  \qquad
  (\varepsilon_i\mid\varepsilon_j)=\delta_{ij}.
\]
The root system is
\begin{equation*}
  \Delta(D_\ell)
  =
  \left\{
    \pm\varepsilon_i\pm\varepsilon_j
    \,\middle|\,
    1\le i<j\le\ell
  \right\}.
\end{equation*}
With the simple roots from
\eqref{dl:eq:Dl-simple-roots}, the fundamental weights are
\begin{equation*}
  \omega_i
  =
  \varepsilon_1+\cdots+\varepsilon_i,
  \qquad
  1\le i\le\ell-2,
\end{equation*}
and
\begin{equation*}
\begin{aligned}
  \omega_{\ell-1}
  &=
  \frac12
  \left(
    \varepsilon_1+\cdots+\varepsilon_{\ell-1}
    -\varepsilon_\ell
  \right),\\
  \omega_\ell
  &=
  \frac12
  \left(
    \varepsilon_1+\cdots+\varepsilon_{\ell-1}
    +\varepsilon_\ell
  \right).
\end{aligned}
\end{equation*}
The highest root is
\begin{equation*}
  \theta
  =
  \varepsilon_1+\varepsilon_2
  =
  \alpha_1+2\alpha_2+\cdots+
  2\alpha_{\ell-2}+\alpha_{\ell-1}+\alpha_\ell.
\end{equation*}
In particular,
\begin{equation*}
  \langle\omega_1,\theta^\vee\rangle=1,
  \qquad
  \langle\omega_i,\theta^\vee\rangle
  =
  \begin{cases}
    2,&2\le i\le\ell-2,\\
    1,&i=\ell-1,\ell.
  \end{cases}
\end{equation*}

The two distinguished roots from
\eqref{dl:eq:Dl-two-natural-highest-roots} satisfy
\begin{equation*}
  (\theta\mid\theta_A)=0,
  \qquad
  (\theta\mid\theta_D)=0,
  \qquad
  (\theta_A\mid\theta_D)=0.
\end{equation*}
Their associated Kostant highest weights are
\begin{equation*}
  \theta+\theta_A
  =
  2\varepsilon_1
  =
  2\omega_1
\end{equation*}
and
\begin{equation*}
  \theta+\theta_D
  =
  \varepsilon_1+\varepsilon_2+\varepsilon_3+\varepsilon_4
  =
  \begin{cases}
    \omega_4+\omega_5,&\ell=5,\\
    \omega_4,&\ell\ge6.
  \end{cases}
\end{equation*}
These identities justify
\eqref{dl:eq:Dl-first-Kostant-component} and
\eqref{dl:eq:Dl-second-Kostant-highest-weight-cases}.

Since
\(x_\theta=\theta^\vee/2\), one has
\begin{equation*}
  \theta(x_\theta)=1,
  \qquad
  \theta_A(x_\theta)=0,
  \qquad
  \theta_D(x_\theta)=0.
\end{equation*}
Consequently,
\begin{equation*}
\begin{aligned}
  \bigl(m(\theta+\theta_A)\bigr)(x_\theta)&=m,\\
  \bigl(2(\theta+\theta_D)\bigr)(x_\theta)&=2,
\end{aligned}
\end{equation*}
which are the shifts used in
\eqref{dl:eq:Dl-reduced-singular-conformal-weights}.

\subsubsection{The natural \(A_1\oplus D_m\) subsystem}

Recall that \(m=\ell-2\).  For the \(D_m\)-factor, set
\begin{equation*}
  \delta_i=\varepsilon_{i+2},
  \qquad
  1\le i\le m.
\end{equation*}
Then
\begin{equation*}
  \Delta(D_m)
  =
  \left\{
    \pm\delta_i\pm\delta_j
    \,\middle|\,
    1\le i<j\le m
  \right\}.
\end{equation*}
Its simple roots are
\begin{equation*}
  \beta_i=\delta_i-\delta_{i+1}
  \quad(1\le i\le m-1),
  \qquad
  \beta_m=\delta_{m-1}+\delta_m.
\end{equation*}
Under the embedding
\(D_m\subset D_\ell\), these correspond to
\[
  \beta_i=\alpha_{i+2}
  \quad(1\le i\le m-1),
  \qquad
  \beta_m=\alpha_\ell.
\]
The highest root is
\begin{equation}
  \theta_D
  =
  \delta_1+\delta_2
  =
  \beta_1+2\beta_2+\cdots+
  2\beta_{m-2}+\beta_{m-1}+\beta_m.
  \label{dl:eq:Dl-appendix-Dm-highest-root}
\end{equation}

For \(m\ge4\), the fundamental weights of \(D_m\) are
\begin{equation*}
  \eta_i
  =
  \delta_1+\cdots+\delta_i,
  \qquad
  1\le i\le m-2,
\end{equation*}
and
\begin{equation*}
\begin{aligned}
  \eta_{m-1}
  &=
  \frac12
  \left(
    \delta_1+\cdots+\delta_{m-1}-\delta_m
  \right),\\
  \eta_m
  &=
  \frac12
  \left(
    \delta_1+\cdots+\delta_{m-1}+\delta_m
  \right).
\end{aligned}
\end{equation*}
The vector representation has highest weight
\begin{equation*}
  \eta_1=\delta_1=\varepsilon_3.
\end{equation*}

The \(A_1\)-factor has simple root
\(\theta_A=\varepsilon_1-\varepsilon_2\) and fundamental weight
\begin{equation*}
  \varpi=\frac12\theta_A,
  \qquad
  (\varpi\mid\varpi)=\frac12.
\end{equation*}
The decomposition
\[
  \mathfrak g^\natural
  =
  \mathfrak s\oplus\mathfrak d
  \cong
  A_1\oplus D_m
\]
is orthogonal with respect to the invariant form.

The restrictions of the two Kostant weights to
\(\mathfrak h^\natural\) are
\begin{equation*}
  (\theta+\theta_A)\big|_{\mathfrak h^\natural}
  =
  \theta_A,
  \qquad
  (\theta+\theta_D)\big|_{\mathfrak h^\natural}
  =
  \theta_D.
\end{equation*}
The restriction of \(\theta\) is zero because
\(\mathfrak h^\natural\) centralizes the minimal
\(\mathfrak{sl}_2\)-triple.

Finally,
\begin{equation*}
  U=\mathfrak g_{-\frac12}
  \cong
  L_{A_1}(\varpi)
  \boxtimes
  L_{D_m}(\eta_1).
\end{equation*}
Its \(A_1\)-weights are \(\pm\varpi\), and its \(D_m\)-weights are
\begin{equation*}
  \operatorname{wt}L_{D_m}(\eta_1)
  =
  \{\pm\delta_i\mid1\le i\le m\}.
\end{equation*}
For \(x=\varpi^\vee=\varpi\),
\begin{equation*}
  \mu(x)\in
  \left\{
    -\frac12,\frac12
  \right\}
  \qquad
  (\mu\in\operatorname{wt}U).
\end{equation*}

\subsubsection{The \(D_m\) root lattice and its four cosets}

For \(m\ge4\), the root lattice is
\begin{equation*}
  Q_D
  =
  \left\{
    \sum_{i=1}^{m}z_i\delta_i
    \,\middle|\,
    z_i\in\mathbb Z,\quad
    \sum_{i=1}^{m}z_i\in2\mathbb Z
  \right\}.
\end{equation*}
The four classes in \(P_D/Q_D\) may be represented by
\begin{equation*}
  0,
  \qquad
  \eta_1,
  \qquad
  \eta_{m-1},
  \qquad
  \eta_m.
\end{equation*}
Thus
\begin{equation*}
  P_D
  =
  Q_D
  \sqcup
  (\eta_1+Q_D)
  \sqcup
  (\eta_{m-1}+Q_D)
  \sqcup
  (\eta_m+Q_D).
\end{equation*}
The first two are the root and vector cosets; the last two are the
spinor cosets.

The vector class has order two:
\begin{equation*}
  2\eta_1=2\delta_1\in Q_D.
\end{equation*}
Hence
\begin{equation}
  Q_D+\mathbb Z\eta_1
  =
  Q_D
  \sqcup
  (\eta_1+Q_D),
  \label{dl:eq:Dl-appendix-root-vector-subgroup}
\end{equation}
which is the lattice statement used in
Lemma~\ref{dl:lem:Dl-adjoint-D-weight-cosets}.

For \(m=3\), identify \(D_3\) with \(A_3\).  The vector representation
is \(L_{A_3}(\omega_2)\), whereas the two spinor representations are
\(L_{A_3}(\omega_1)\) and \(L_{A_3}(\omega_3)\).  Since
\begin{equation*}
  P(A_3)/Q(A_3)\cong\mathbb Z/4\mathbb Z,
\end{equation*}
the class of \(\omega_2\) has order two and
\begin{equation*}
  Q(A_3)+\mathbb Z\omega_2
  =
  Q(A_3)
  \sqcup
  \bigl(\omega_2+Q(A_3)\bigr).
\end{equation*}
The classes of \(\omega_1\) and \(\omega_3\) are excluded.

\subsubsection{Level-one dominant weights}

For \(m\ge4\), the pairings of the fundamental weights with the
highest coroot are
\begin{equation*}
  \langle\eta_i,\theta_D^\vee\rangle
  =
  \begin{cases}
    1,&i=1,m-1,m,\\
    2,&2\le i\le m-2.
  \end{cases}
\end{equation*}
Indeed, these are the coefficients of the simple roots in
\eqref{dl:eq:Dl-appendix-Dm-highest-root}.  Therefore, if
\[
  \lambda=\sum_{i=1}^{m}a_i\eta_i
  \in P_+(D_m)
\]
and
\[
  \langle\lambda,\theta_D^\vee\rangle\le1,
\]
then
\begin{equation*}
  \lambda
  \in
  \{0,\eta_1,\eta_{m-1},\eta_m\}.
\end{equation*}
Combining this list with
\eqref{dl:eq:Dl-appendix-root-vector-subgroup} leaves only
\begin{equation*}
  \lambda=0
  \quad\text{or}\quad
  \lambda=\eta_1
\end{equation*}
in the adjoint module.

For \(D_3\cong A_3\), the dominant weights satisfying the analogous
level-one condition are
\begin{equation*}
  0,\qquad
  \omega_1,\qquad
  \omega_2,\qquad
  \omega_3.
\end{equation*}
The root-vector coset restriction leaves only
\(0\) and \(\omega_2\).

\subsubsection{Casimir eigenvalues}

Let \(\Omega_A\) and \(\Omega_D\) be the Casimir elements defined using
dual bases with respect to the normalized invariant form.

For the \(A_1\)-factor, the Weyl vector is \(\varpi\), and
\begin{equation*}
  c_A(a\varpi)
  =
  (a\varpi\mid a\varpi+2\varpi)
  =
  \frac{a(a+2)}2.
\end{equation*}
In particular,
\begin{equation*}
  c_A(\varpi)=\frac32,
  \qquad
  c_A(K\varpi)
  =
  \frac{K(K+2)}2
  =
  \frac{m^2-1}{2}.
\end{equation*}

For the \(D_m\)-factor, the Weyl vector is
\begin{equation*}
  \rho_D
  =
  \sum_{i=1}^{m}(m-i)\delta_i.
\end{equation*}
The vector representation has
\begin{align}
  c_D(\eta_1)
  &=
  (\delta_1\mid\delta_1+2\rho_D)
  \notag\\
  &=
  1+2(m-1)
  =
  2m-1.
  \label{dl:eq:Dl-appendix-vector-Casimir}
\end{align}
For either spinor highest weight,
\begin{align*}
  c_D(\eta_{m-1})
  =
  c_D(\eta_m)
  &=
  \frac m4
  +
  \sum_{i=1}^{m-1}(m-i)
  \notag\\
  &=
  \frac{m(2m-1)}4.
\end{align*}
The spinor value is recorded to emphasize that the level-one
condition by itself does not provide the uniform vector bound used in
Section~\ref{dl:sec:Dl-simplicity-reduced-quotient}; the weight-coset
argument is essential.

When \(m=3\), in the \(A_3\) notation,
\begin{equation*}
  c_2(\omega_2)=5,
  \qquad
  c_2(\omega_1)=c_2(\omega_3)=\frac{15}{4}.
\end{equation*}
The identity
\(c_2(\omega_2)=5=2m-1\) agrees with
\eqref{dl:eq:Dl-appendix-vector-Casimir}.

The dimensions and contraction coefficients used in
Lemma~\ref{dl:lem:Dl-trace-quadratic-Ramond-term} are
\begin{equation*}
\begin{array}{c|c|c|c}
  \mathfrak a
  &
  \dim\mathfrak a
  &
  c_U(\mathfrak a)
  &
  \displaystyle
  \frac{2c_U(\mathfrak a)}{\dim\mathfrak a}
  \\ \hline
  A_1
  &
  3
  &
  \frac32
  &
  1
  \\[2mm]
  D_m
  &
  m(2m-1)
  &
  2m-1
  &
  \frac2m
\end{array}
\end{equation*}
where \(c_U(\mathfrak a)\) denotes the Casimir eigenvalue of the
\(\mathfrak a\)-factor on
\(U\cong\mathbf2\boxtimes V_{\mathrm{vec}}\).

\subsubsection{Summary of the numerical bounds}

For a finite-dimensional Ramond Zhu module occurring inside the
twisted adjoint module, the relations
\(e_A^m=0\) and \(e_D^2=0\), together with the coset exclusion, imply
\begin{equation*}
  0\le a\le m-1,
  \qquad
  \lambda_D\in\{0,\eta_1\}.
\end{equation*}
Consequently,
\begin{equation*}
  c_A(a\varpi)
  \le
  \frac{m^2-1}{2},
  \qquad
  c_D(\lambda_D)
  \le
  2m-1.
\end{equation*}
These are precisely the two sharp estimates inserted into the weighted
trace identity in
\eqref{dl:eq:Dl-minus-one-Cmax}.
\section{Root and Casimir data for \texorpdfstring{$D_4$}{D4}}

\subsection{Root, weight, and Casimir data}
\label{d4:app:D4-root-weight-Casimir-data}

This appendix records the elementary \(D_4\) and
\(A_1^{(1)}\oplus A_1^{(3)}\oplus A_1^{(4)}\) calculations used in
Sections~\ref{d4:sec:D4-triality-singular-vectors},
\ref{d4:sec:D4-Ramond-Casimir}, and
\ref{d4:sec:D4-Li-spectral-flow-simplicity}.  All roots are normalized to
have squared length \(2\).

\subsubsection{The \texorpdfstring{\(D_4\)}{D4} root system}

Let \(E=\bigoplus_{i=1}^{4}\mathbb R\varepsilon_i\) with
\((\varepsilon_i\mid\varepsilon_j)=\delta_{ij}\).  The roots of \(D_4\)
are
\begin{equation*}
  \Delta(D_4)
  =
  \{\pm\varepsilon_i\pm\varepsilon_j
    \mid 1\le i<j\le4\}.
\end{equation*}
We use the Bourbaki simple roots
\begin{equation*}
\begin{aligned}
  \alpha_1&=\varepsilon_1-\varepsilon_2,\\
  \alpha_2&=\varepsilon_2-\varepsilon_3,\\
  \alpha_3&=\varepsilon_3-\varepsilon_4,\\
  \alpha_4&=\varepsilon_3+\varepsilon_4.
\end{aligned}
\end{equation*}
The fundamental weights are
\begin{equation*}
\begin{aligned}
  \omega_1&=\varepsilon_1,\\
  \omega_2&=\varepsilon_1+\varepsilon_2,\\
  \omega_3&=\frac12
    (\varepsilon_1+\varepsilon_2+\varepsilon_3-\varepsilon_4),\\
  \omega_4&=\frac12
    (\varepsilon_1+\varepsilon_2+\varepsilon_3+\varepsilon_4).
\end{aligned}
\end{equation*}
The highest root is
\begin{equation*}
  \theta
  =
  \varepsilon_1+\varepsilon_2
  =
  \alpha_1+2\alpha_2+\alpha_3+\alpha_4.
\end{equation*}
Consequently,
\begin{equation*}
  \langle\omega_a,\theta^\vee\rangle=1,
  \qquad
  a\in\{1,3,4\}.
\end{equation*}
This is the numerical input in
\eqref{d4:eq:D4-singular-highest-weight-negative-wall} and
\eqref{eq:common-DS-nonvanishing}.

The Cartan matrix is
\begin{equation*}
  A_{D_4}
  =
  \begin{pmatrix}
    2&-1&0&0\\
    -1&2&-1&-1\\
    0&-1&2&0\\
    0&-1&0&2
  \end{pmatrix},
\end{equation*}
and its inverse is
\begin{equation*}
  A_{D_4}^{-1}
  =
  \begin{pmatrix}
    1&1&\frac12&\frac12\\
    1&2&1&1\\
    \frac12&1&1&\frac12\\
    \frac12&1&\frac12&1
  \end{pmatrix}.
\end{equation*}
Since the root system is simply laced,
\((\omega_i\mid\omega_j)=(A_{D_4}^{-1})_{ij}\).

The diagram automorphism group \(S_3\) fixes the central node \(2\) and
permutes the outer nodes
\begin{equation*}
  \mathcal A=\{1,3,4\}.
\end{equation*}
It therefore permutes
\(\omega_1,\omega_3,\omega_4\), the three irreducible components of the
Kostant module in Section~\ref{d4:sec:D4-triality-singular-vectors}, and
the three singular vectors
\(\sigma(w_1)^2,\sigma(w_3)^2,\sigma(w_4)^2\).

\subsubsection{The natural \texorpdfstring{\(A_1^3\)}{A1 cubed} subalgebra}

The simple roots of
\(\mathfrak g^\natural\cong
 A_1^{(1)}\oplus A_1^{(3)}\oplus A_1^{(4)}\)
are
\begin{equation*}
  \beta_1=\alpha_1,\qquad
  \beta_3=\alpha_3,\qquad
  \beta_4=\alpha_4.
\end{equation*}
They are pairwise orthogonal:
\begin{equation*}
  (\beta_a\mid\beta_b)=2\delta_{ab},
  \qquad a,b\in\mathcal A.
\end{equation*}
The fundamental weight of the \(a\)-th \(A_1\)-factor is
\begin{equation*}
  \varpi_a=\frac12\beta_a,
  \qquad
  (\varpi_a\mid\varpi_b)=\frac12\delta_{ab}.
\end{equation*}
The restrictions of the three outer \(D_4\)-fundamental weights are
\begin{equation*}
  \omega_a\big|_{\mathfrak h^\natural}=\varpi_a,
  \qquad a\in\mathcal A.
\end{equation*}
Indeed,
\[
  \langle\omega_a,\beta_b^\vee\rangle=\delta_{ab}.
\]

The minimal half-space is
\begin{equation*}
  U=\mathfrak g_{-\frac12}
  \cong
  L_{\mathfrak s_1}(\varpi_1)
  \boxtimes
  L_{\mathfrak s_3}(\varpi_3)
  \boxtimes
  L_{\mathfrak s_4}(\varpi_4).
\end{equation*}
Its eight weights are
\begin{equation*}
  \epsilon_1\varpi_1+
  \epsilon_3\varpi_3+
  \epsilon_4\varpi_4,
  \qquad
  \epsilon_a\in\{\pm1\}.
\end{equation*}
For the Li element \(x=\varpi_1^\vee=\varpi_1\),
\begin{equation*}
  \gamma(x)\in\{-1,0,1\}
  \quad(\gamma\in\Delta(\mathfrak g^\natural)),
  \qquad
  \mu(x)\in\left\{-\frac12,\frac12\right\}
  \quad(\mu\in\operatorname{wt}U).
\end{equation*}
These are the charge sets used in the proof of
Proposition~\ref{d4:prop:D4-Ramond-spectral-lattice}.

\subsubsection{Casimir normalization for each \texorpdfstring{\(A_1\)}{A1} factor}

For the \(a\)-th \(A_1\)-factor, let
\(\Omega_a\) be the quadratic Casimir defined using dual bases with
respect to the form for which
\((\beta_a\mid\beta_a)=2\).  The Weyl vector is \(\rho_a=\varpi_a\).
Hence \(\Omega_a\) acts on the irreducible module
\(L_{\mathfrak s_a}(m\varpi_a)\) by
\begin{equation*}
  c_2(m\varpi_a)
  =
  (m\varpi_a\mid m\varpi_a+2\varpi_a)
  =
  \frac{m(m+2)}2.
\end{equation*}
In particular,
\begin{equation*}
  c_2(0)=0,
  \qquad
  c_2(\varpi_a)=\frac32,
  \qquad
  c_2(2\varpi_a)=4.
\end{equation*}
Only the first two values occur in finite-dimensional modules over the
Ramond Zhu quotient \(B\), because
\(e_a^2=0\) for every \(a\in\mathcal A\).

Set
\begin{equation*}
  \Omega^\natural=\Omega_1+\Omega_3+\Omega_4.
\end{equation*}
On
\begin{equation*}
  L(m_1\varpi_1,m_3\varpi_3,m_4\varpi_4)
  :=
  \boxtimes_{a\in\mathcal A}
  L_{\mathfrak s_a}(m_a\varpi_a),
\end{equation*}
the total Casimir acts by
\begin{equation*}
  c_2^\natural(m_1,m_3,m_4)
  =
  \frac12
  \sum_{a\in\mathcal A}m_a(m_a+2).
\end{equation*}
When \(m_a\in\{0,1\}\), this simplifies to
\begin{equation*}
  c_2^\natural(m_1,m_3,m_4)
  =
  \frac32(m_1+m_3+m_4).
\end{equation*}

\subsubsection{The eight allowed highest weights}

Lemma~\ref{d4:lem:D4-allowed-A1-cubed-types} shows that every irreducible
\(\mathfrak g^\natural\)-constituent of a finite-dimensional
\(B\)-module has \(m_a\in\{0,1\}\).  The complete list is therefore:
\begin{equation*}
\begin{array}{c|c|c}
 (m_1,m_3,m_4)
 &
 \lambda=m_1\varpi_1+m_3\varpi_3+m_4\varpi_4
 &
 c_2^\natural(\lambda)
 \\ \hline
 (0,0,0)&0&0\\
 (1,0,0)&\varpi_1&\frac32\\
 (0,1,0)&\varpi_3&\frac32\\
 (0,0,1)&\varpi_4&\frac32\\
 (1,1,0)&\varpi_1+\varpi_3&3\\
 (1,0,1)&\varpi_1+\varpi_4&3\\
 (0,1,1)&\varpi_3+\varpi_4&3\\
 (1,1,1)&\varpi_1+\varpi_3+\varpi_4&\frac92
\end{array}
\end{equation*}
Thus
\begin{equation*}
  0\le c_2^\natural(\lambda)\le\frac92,
\end{equation*}
with equality only for
\(\lambda=\varpi_1+\varpi_3+\varpi_4\).
This is the global upper bound used in the Casimir-gap proof of
Theorem~\ref{d4:thm:D4-reduced-quotient-simple}.

For the Li twist based on \(x=\varpi_1^\vee\), the Ramond-bottom
degree formula
\[
  d=\frac12-\lambda(x)
\]
takes the particularly simple form
\begin{equation*}
  d=\frac{1-m_1}{2}.
\end{equation*}
Consequently, the eight types separate into two classes:
\begin{equation*}
\begin{array}{c|c|c}
 m_1&d&\text{consequence}\\ \hline
 0&\frac12&
 \widetilde{\mathcal W}_{1/2}=0\\[1mm]
 1&0&
 \widetilde{\mathcal W}_0=\mathbb C\mathbf1
\end{array}
\end{equation*}
This is precisely the dichotomy used in the proof of
Theorem~\ref{d4:thm:D4-reduced-quotient-simple}.
\section{Explicit calculations for \texorpdfstring{$E_6$}{E6}}

\subsection{Explicit calculations for \texorpdfstring{$E_6$}{E6} and \texorpdfstring{$A_5$}{A5}}
\label{e6:app:explicit-calculations}

This appendix collects the root-theoretic and finite-dimensional calculations
used in Sections~\ref{e6:sec:singular-vectors}--\ref{e6:sec:Ramond-spectral-flow-simplicity}.
The purpose is to make the normalizations and the numerical bounds independently
checkable without interrupting the main argument.

\subsubsection{The root calculation for the Joseph vector}
\label{e6:app:Joseph-root-calculation}

We retain the notation
\(
[c_1c_2c_3c_4c_5c_6]=\sum_i c_i\alpha_i
\)
from Section~\ref{e6:sec:singular-vectors}.  Put
\[
  \gamma_j=\beta_j+\theta_1,
  \qquad
  \varepsilon_j=\delta_j+\theta_1
  \qquad (j=1,2,3).
\]
The roots entering the quadratic vector \eqref{e6:eq:def-w} are listed below.
\begin{equation}\label{e6:eq:appendix-E6-root-table}
\begin{array}{c|c@{\qquad}c|c}
\text{root}&\text{coordinates}&\text{root}&\text{coordinates}\\ \hline
\theta &[122321]&\theta_1 &[101111]\\
\gamma_1 &[111111]&\varepsilon_1 &[112321]\\
\gamma_2 &[111211]&\varepsilon_2 &[112221]\\
\gamma_3 &[111221]&\varepsilon_3 &[112211]
\end{array}
\end{equation}
Each vector in \eqref{e6:eq:appendix-E6-root-table} has squared length two, and
\[
  \gamma_j+\varepsilon_j=\theta+\theta_1
  \qquad (j=1,2,3).
\]
Moreover,
\begin{equation}\label{e6:eq:appendix-weight-pairings}
  \bigl\langle\theta+\theta_1,\alpha_i^\vee\bigr\rangle
  =\delta_{i1}+\delta_{i6},
\end{equation}
so \(\theta+\theta_1=\omega_1+\omega_6\).

After compatible rescaling of the nonsimple root vectors, the relative signs may
be chosen so that the nonzero actions of the positive simple-root vectors on the
root vectors occurring in \(w\) are
\begin{equation}\label{e6:eq:appendix-raising-table}
\begin{array}{c|cccc}
 & e_2&e_3&e_4&e_5\\ \hline
 e_{\theta_1}     &e_{\gamma_1}&0&0&0\\
 e_{\gamma_1}     &0&0&e_{\gamma_2}&0\\
 e_{\varepsilon_1}&e_\theta&0&0&0\\
 e_{\gamma_2}     &0&e_{\varepsilon_3}&0&e_{\gamma_3}\\
 e_{\varepsilon_2}&0&0&-e_{\varepsilon_1}&0\\
 e_{\gamma_3}     &0&-e_{\varepsilon_2}&0&0\\
 e_{\varepsilon_3}&0&0&0&-e_{\varepsilon_2}.
\end{array}
\end{equation}
The signs can be realized recursively.  Starting from the fixed simple-root
vectors and $e_{\theta_1}$, take
\begin{align*}
 e_{\gamma_1}&=[e_2,e_{\theta_1}],
 &e_{\gamma_2}&=[e_4,e_{\gamma_1}],
 &e_{\gamma_3}&=[e_5,e_{\gamma_2}],\\
 e_{\varepsilon_3}&=[e_3,e_{\gamma_2}],
 &e_{\varepsilon_2}&=-[e_3,e_{\gamma_3}],
 &e_{\varepsilon_1}&=-[e_4,e_{\varepsilon_2}],\\
 e_\theta&=[e_2,e_{\varepsilon_1}].
\end{align*}
The only entry obtained by two visibly different paths is consistent by the
Jacobi identity:
\begin{equation*}
 [e_5,e_{\varepsilon_3}]
 =[e_5,[e_3,e_{\gamma_2}]]
 =[e_3,[e_5,e_{\gamma_2}]]
 =[e_3,e_{\gamma_3}]
 =-e_{\varepsilon_2},
\end{equation*}
because $[e_3,e_5]=0$.  Direct addition of the simple-root coordinates in
\eqref{e6:eq:appendix-E6-root-table} gives exactly the seven nonzero root
moves displayed in \eqref{e6:eq:appendix-raising-table}; every omitted sum
$\rho+\alpha_i$ is not a root.  Thus the recursive choice realizes the table
and is compatible with \eqref{e6:eq:Chevalley-normalization-w} after the
remaining one-dimensional root spaces are rescaled accordingly.

All omitted entries vanish; in particular, \(e_1\) and \(e_6\) annihilate every
factor that can contribute to the action on \(w\).  Since the product in
\(S^2(\mathfrak g)\) is commutative, the table gives
\begin{align*}
 e_2w
 &=e_\theta e_{\gamma_1}-e_{\gamma_1}e_\theta=0,\\
 e_4w
 &=-e_{\gamma_2}e_{\varepsilon_1}
   +e_{\gamma_2}e_{\varepsilon_1}=0,\\
 e_3w
 &=-e_{\varepsilon_3}e_{\varepsilon_2}
   +e_{\varepsilon_2}e_{\varepsilon_3}=0,\\
 e_5w
 &=-e_{\gamma_3}e_{\varepsilon_2}
   +e_{\gamma_3}e_{\varepsilon_2}=0.
\end{align*}
Together with \eqref{e6:eq:appendix-weight-pairings}, this verifies directly that
\(w\) is a highest-weight vector of weight \(\omega_1+\omega_6\).

\subsubsection{The affine mode calculation}
\label{e6:app:affine-mode-calculation}

Let
\begin{equation*}
  v=e_{\theta_1}(-1)\mathbf1.
\end{equation*}

\begin{lemma}\label{e6:lem:regularity-u-v}
Let
\[
  \mathcal R
  =\{\theta,\theta_1,\gamma_1,\varepsilon_1,
      \gamma_2,\varepsilon_2,\gamma_3,\varepsilon_3\}.
\]
For all $\rho,\tau\in\mathcal R$ and all $m,n\in\mathbb Z$, one has
\begin{equation}\label{e6:eq:pairwise-commuting-root-modes}
  [e_\rho(m),e_\tau(n)]=0.
\end{equation}
Consequently, the current fields occurring in $u$ and
$v=e_{\theta_1}(-1)\mathbf1$ have mutually regular operator products.  In
particular,
\begin{equation}\label{e6:eq:regularity-u-v-products}
  u_{(j)}u=0,
  \qquad
  v_{(j)}u=0
  \qquad(j\geq0).
\end{equation}
Hence the differential vertex subalgebra generated by these currents is
commutative, and its $(-1)$-product is associative and commutative.  In this
subalgebra, the state $u^r$ agrees with the ordinary $r$th power of $u$ in
$U(\mathfrak g[t^{-1}]t^{-1})$.
\end{lemma}

\begin{proof}
Using the root coordinates in \eqref{e6:eq:appendix-E6-root-table} and the
$E_6$ Cartan matrix, one obtains
\begin{equation*}
  (\rho+\tau\mid\rho+\tau)\in\{4,6,8\}
  \qquad(\rho,\tau\in\mathcal R).
\end{equation*}
Since every root of $E_6$ has squared length $2$, the sum $\rho+\tau$ is
never a root.  Moreover, $\rho$ and $\tau$ are positive roots, so
$(e_\rho\mid e_\tau)=0$.  The affine commutation relation therefore gives
\eqref{e6:eq:pairwise-commuting-root-modes}.  Thus the corresponding current
fields have no singular operator-product terms.  The same is then true for
the normally ordered composites $u$ and $v$, which proves
\eqref{e6:eq:regularity-u-v-products}.  A vertex algebra generated by mutually
regular fields is a commutative differential vertex algebra, and on such an
algebra the $(-1)$-product is ordinary associative multiplication.
\end{proof}

We spell out the calculation underlying
Proposition~\ref{e6:prop:f-theta-on-u-powers}.  The affine commutator relation and
the root table above give the following three identities:
\begin{align}
  f_\theta(1)u&=(k+3)v,
  \label{e6:eq:appendix-base-mode-identity}\\
  \bigl(f_\theta(0)u\bigr)_{(0)}u
  &=-2u_{(-1)}v,
  \label{e6:eq:appendix-zero-mode-on-u}\\
  v_{(-1)}u&=u_{(-1)}v.
  \label{e6:eq:appendix-v-commutes-u}
\end{align}
We verify the sign and coefficient in
\eqref{e6:eq:appendix-base-mode-identity}.  For every root $\rho$ occurring
below, normalize $f_\rho\in\mathfrak g_{-\rho}$ by
$(e_\rho\mid f_\rho)=1$.  Fix $j$ and abbreviate
$\gamma=\gamma_j$, $\varepsilon=\varepsilon_j$,
$\beta=\beta_j$, and $\delta=\delta_j$.  The normalization
\eqref{e6:eq:Chevalley-normalization-w} and the Jacobi identity give
\begin{equation*}
 [e_\delta,e_\gamma]=e_\theta,
 \qquad
 [e_\beta,e_\varepsilon]=e_\theta.
\end{equation*}
Invariance of the form then yields
\begin{equation*}
 [f_\theta,e_\gamma]=-f_\delta,
 \qquad
 [f_\theta,e_\varepsilon]=-f_\beta.
\end{equation*}
For example,
\begin{equation*}
 ([f_\theta,e_\gamma]\mid e_\delta)
 =(f_\theta\mid[e_\gamma,e_\delta])=-1.
\end{equation*}
Since $\langle\delta,\theta_1^\vee\rangle=-1$ and
$[e_\delta,e_{\theta_1}]=e_\varepsilon$, another Jacobi calculation gives
$[f_\delta,e_\varepsilon]=e_{\theta_1}$; similarly
$[f_\beta,e_\gamma]=e_{\theta_1}$.  Consequently,
\begin{equation}
 [[f_\theta,e_\gamma],e_\varepsilon]
 =[[f_\theta,e_\varepsilon],e_\gamma]
 =-e_{\theta_1}.
 \label{e6:eq:appendix-double-bracket-sign-check}
\end{equation}

Now apply $f_\theta(1)$ to the symmetrized quadratic expression
\eqref{e6:eq:explicit-u}.  The $e_\theta e_{\theta_1}$ term contributes
$kv$: the zero-mode part $-\theta^\vee(0)$ annihilates $v$ because
$(\theta\mid\theta_1)=0$, and the affine central term contributes $k$.
For each $j$, equation
\eqref{e6:eq:appendix-double-bracket-sign-check} shows that
\begin{equation*}
 f_\theta(1)\,\sigma(e_{\gamma_j}e_{\varepsilon_j})=-v.
\end{equation*}
The minus sign in the definition of $u$ therefore turns each of the three
root-pair contributions into $+v$.  Hence
$f_\theta(1)u=(k+3)v$.

In \eqref{e6:eq:appendix-zero-mode-on-u}, all mixed-weight terms vanish, and the
remaining Cartan term acts on the highest-weight vector \(u\) by
\[
  \langle\theta+\theta_1,\theta^\vee\rangle=2;
\]
the sign is determined by \([f_\theta,e_\theta]=-\theta^\vee\).
Equation \eqref{e6:eq:appendix-v-commutes-u}, as well as all reassociations of
$(-1)$-products used below, follows from
Lemma~\ref{e6:lem:regularity-u-v}.

By Lemma~\ref{e6:lem:regularity-u-v}, the $(-1)$-products below associate
without correction terms.  Since every
zero mode is a derivation of all vertex products,
\eqref{e6:eq:appendix-zero-mode-on-u} and
\eqref{e6:eq:appendix-v-commutes-u} imply, for \(m\geq1\),
\begin{equation}\label{e6:eq:appendix-zero-mode-on-power}
  \bigl(f_\theta(0)u\bigr)_{(0)}u^m
  =-2m\,u^m_{(-1)}v.
\end{equation}
We now prove the power formula by induction.  The commutator formula for the mode
\(f_\theta(1)\) gives
\begin{align}
 f_\theta(1)u^{r+1}
 &=f_\theta(1)\bigl(u_{(-1)}u^r\bigr)\notag\\
 &=\bigl(f_\theta(0)u\bigr)_{(0)}u^r
   +\bigl(f_\theta(1)u\bigr)_{(-1)}u^r
   +u_{(-1)}f_\theta(1)u^r.
 \label{e6:eq:appendix-affine-induction-step}
\end{align}
Assume
\[
  f_\theta(1)u^r=r(k-r+4)u^{r-1}_{(-1)}v.
\]
Substitution of \eqref{e6:eq:appendix-base-mode-identity} and
\eqref{e6:eq:appendix-zero-mode-on-power} into
\eqref{e6:eq:appendix-affine-induction-step} yields
\begin{align*}
 f_\theta(1)u^{r+1}
 &=\bigl[-2r+(k+3)+r(k-r+4)\bigr]u^r_{(-1)}v\\
 &=(r+1)(k-r+3)u^r_{(-1)}v.
\end{align*}
This is precisely \eqref{e6:eq:f-theta-on-u-powers} with \(r\) replaced by
\(r+1\).  The initial case is
\eqref{e6:eq:appendix-base-mode-identity}.

\subsubsection{Classical restriction of the PBW symbol}
\label{e6:app:BRST-filtration-bookkeeping}

For completeness, we record the classical restriction of the PBW symbol
appearing in the $E_6$ singular vectors.  This calculation is a useful
consistency check on the expected reduced highest weight, but the
nonvanishing of the BRST class is proved module-theoretically in
Proposition~\ref{e6:prop:BRST-image-s-n}; we do not infer BRST
nonvanishing solely from an ordinary affine PBW leading term.

Use the invariant form to identify
\(S(\mathfrak g)=\mathbb C[\mathfrak g]\), so that a vector
\(a\in\mathfrak g\) is the linear function \(y\mapsto(a\mid y)\).  For the
minimal triple one has
\[
  \mathfrak g^{e_\theta}
  =\mathfrak g^\natural\oplus\mathfrak g_{\frac12}
   \oplus\mathbb C e_\theta,
  \qquad
  \mathcal S_{f_\theta}=f_\theta+\mathfrak g^{e_\theta}.
\]
Let \(y=f_\theta+z\) with \(z\in\mathfrak g^{e_\theta}\).  Since the invariant
form pairs \(\mathfrak g_j\) only with \(\mathfrak g_{-j}\),
\[
  e_\theta(y)=1,
  \qquad
  e_{\gamma_j}(y)=e_{\varepsilon_j}(y)=0
  \quad(j=1,2,3).
\]
Moreover, \(e_{\theta_1}(y)=(e_{\theta_1}\mid z)\) is the linear coordinate on
the \(\mathfrak g^\natural\)-factor corresponding, under
$\theta_1=\vartheta$, to \(e_\vartheta\).  Substituting these relations in
the explicit formula for \(w\) gives
\[
  \left.w\right|_{\mathcal S_{f_\theta}}=e_\vartheta,
  \qquad
  \left.w^r\right|_{\mathcal S_{f_\theta}}
  =e_\vartheta^r\neq0.
\]
This agrees with the exact reduced class
$[u^r]_{\mathrm{DS}}\in\mathbb C^\times
:J^{\{e_\vartheta\}}{}^r:$ obtained in
Proposition~\ref{e6:prop:BRST-image-s-n} for $r=2,3$.

\subsubsection{The complete level-two \texorpdfstring{$A_5$}{A5} Casimir table}
\label{e6:app:A5-Casimir-table}

Write \(\mu=\sum_{i=1}^5m_i\varpi_i\).  From
\eqref{e6:eq:A5-inverse-Cartan}, the Casimir eigenvalue
\(c_2(\mu)=(\mu\mid\mu+2\rho_{A_5})\) is
\begin{align*}
 c_2(\mu)={}&
 \frac56m_1^2+\frac43m_1m_2+m_1m_3
 +\frac23m_1m_4+\frac13m_1m_5+5m_1\notag\\
 &+\frac43m_2^2+2m_2m_3+\frac43m_2m_4
 +\frac23m_2m_5+8m_2\notag\\
 &+\frac32m_3^2+2m_3m_4+m_3m_5+9m_3\notag\\
 &+\frac43m_4^2+\frac43m_4m_5+8m_4
 +\frac56m_5^2+5m_5.
\end{align*}
The integrability restriction \(m_1+\cdots+m_5\le2\) leaves exactly twenty-one
weights.  Their values are as follows.
\begin{equation*}
\begin{array}{c|c@{\qquad}c|c}
\mu&c_2(\mu)&\mu&c_2(\mu)\\ \hline
0&0
 &\varpi_1,\ \varpi_5&35/6\\
\varpi_2,\ \varpi_4&28/3
 &\varpi_3&21/2\\ \hline
2\varpi_1,\ 2\varpi_5&40/3
 &\varpi_1+\varpi_5&12\\
\varpi_1+\varpi_2,\ \varpi_4+\varpi_5&33/2
 &\varpi_1+\varpi_3,\ \varpi_3+\varpi_5&52/3\\
\varpi_1+\varpi_4,\ \varpi_2+\varpi_5&95/6
 &2\varpi_2,\ 2\varpi_4&64/3\\
\varpi_2+\varpi_3,\ \varpi_3+\varpi_4&131/6
 &\varpi_2+\varpi_4&20\\
2\varpi_3&24&&
\end{array}
\end{equation*}
The largest level-one value is \(21/2\), attained only at \(\varpi_3\); the
second largest is \(28/3\).  The largest level-two value is \(24\), attained only
at \(2\varpi_3\); the second largest is \(131/6\).  These are the sharp bounds
used in Lemma~\ref{e6:lem:A5-Casimir-bounds}.

\subsubsection{Charge multiplicities and the Li shift}
\label{e6:app:Li-charge-calculation}

The coweight \(\eta=\varpi_3^\vee\) may be represented on the defining module
\(\mathbb C^6\) by
\begin{equation*}
  \eta=\frac12\operatorname{diag}(1,1,1,-1,-1,-1).
\end{equation*}
It follows immediately that the roots of \(A_5\) have \(\eta\)-charges
\(-1,0,1\).  On
\(L_{A_5}(\varpi_3)=\bigwedge^3\mathbb C^6\), a basis vector containing exactly
\(a\) entries from the first three coordinate vectors has charge
\(a-3/2\).  Hence the charge multiplicities are
\begin{equation*}
\begin{array}{c|rrrr}
q&-3/2&-1/2&1/2&3/2\\ \hline
\dim U[q]&1&9&9&1.
\end{array}
\end{equation*}

For the weight-one current \(H=J^{\{\eta\}}\), one has
\(H_{(1)}H=(3n/2)\mathbf1\).  Expanding Li's operator gives
\begin{align*}
 \Delta(H,z)\omega
 &=\omega+z^{-1}H+\frac12H_{(1)}H\,z^{-2}\mathbf1,\\
 \Delta(H,z)J^{\{h\}}
 &=J^{\{h\}}+n(\eta\mid h)z^{-1}\mathbf1.
\end{align*}
Comparing modes proves
\[
  L_0^{\mathrm R}=L_0+H_0+\frac{3n}{4},
  \qquad
  J_0^{\{h\},\mathrm R}=J_0^{\{h\}}+n(\eta\mid h).
\]
For a derivative of order \(r,s,t\), respectively, the contributions to
\(L_0^{\mathrm R}-3n/4\) are
\begin{equation*}
\begin{array}{c|c|c}
\text{factor}&\eta\text{-charge}&\text{shifted energy}\\ \hline
\partial^rJ^{\{a_\alpha\}}&\alpha(\eta)\in\{-1,0,1\}
 &1+r+\alpha(\eta)\\
\partial^sG^{\{v_\mu\}}&\mu(\eta)\in\{-3/2,-1/2,1/2,3/2\}
 &3/2+s+\mu(\eta)\\
\partial^t\omega&0&2+t.
\end{array}
\end{equation*}
Every entry in the last column is a nonnegative integer.  Zero energy is possible
only for an undifferentiated current of charge \(-1\) or an undifferentiated
\(G\)-field of charge \(-3/2\).  Every nonempty zero-energy monomial therefore
has strictly negative original \(\eta\)-charge.  This is the numerical content
of the extremal-vacuum-line argument in
Proposition~\ref{e6:prop:E6-Ramond-spectrum-vacuum-line}.
\section{Weight and Casimir data for \texorpdfstring{$E_7$}{E7}}

\subsection{\texorpdfstring{$D_6$}{D6} weight and Casimir data}
\label{e7:app:D6-weight-Casimir-data}

This appendix records the root-theoretic calculations used in
Sections~\ref{e7:sec:Ramond-Casimir-identity} and
\ref{e7:sec:simplicity-reduced-quotients}.  We use the Bourbaki numbering
for $D_6$ and normalize the invariant form so that every root has squared
length two; see \cite[Planche~IV]{Bourbaki2002}.  Thus the Dynkin diagram
is
\[
  1-2-3-4,
  \qquad
  4-5,
  \qquad
  4-6.
\]

\subsubsection{Roots, fundamental weights, and the level function}

Let $\varepsilon_1,\ldots,\varepsilon_6$ be an orthonormal basis.  We take
\begin{align*}
  \alpha_i&=\varepsilon_i-\varepsilon_{i+1},
  &&1\le i\le4,\\
  \alpha_5&=\varepsilon_5-\varepsilon_6,
  &\alpha_6&=\varepsilon_5+\varepsilon_6.
\end{align*}
The fundamental weights are
\begin{align*}
  \omega_j&=\varepsilon_1+\cdots+\varepsilon_j,
  &&1\le j\le4,\\
  \omega_5
  &=\frac12(\varepsilon_1+\cdots+\varepsilon_5-\varepsilon_6),
  &
  \omega_6
  &=\frac12(\varepsilon_1+\cdots+\varepsilon_5+\varepsilon_6).
\end{align*}
In particular,
\begin{equation*}
  (\omega_5\mid\omega_5)
  =(\omega_6\mid\omega_6)=\frac32,
  \qquad
  (\omega_5\mid\omega_6)=1.
\end{equation*}

The Cartan matrix and its inverse are
\begin{equation*}
  A_{D_6}
  =
  \begin{pmatrix}
    2&-1&0&0&0&0\\
    -1&2&-1&0&0&0\\
    0&-1&2&-1&0&0\\
    0&0&-1&2&-1&-1\\
    0&0&0&-1&2&0\\
    0&0&0&-1&0&2
  \end{pmatrix},
  \qquad
  A_{D_6}^{-1}
  =
  \begin{pmatrix}
    1&1&1&1&\frac12&\frac12\\
    1&2&2&2&1&1\\
    1&2&3&3&\frac32&\frac32\\
    1&2&3&4&2&2\\
    \frac12&1&\frac32&2&\frac32&1\\
    \frac12&1&\frac32&2&1&\frac32
  \end{pmatrix}.
\end{equation*}
Since $D_6$ is simply laced,
\begin{equation*}
  (\omega_i\mid\omega_j)
  =(A_{D_6}^{-1})_{ij}.
\end{equation*}

The highest root is
\begin{equation*}
  \vartheta
  =\varepsilon_1+\varepsilon_2
  =\alpha_1+2\alpha_2+2\alpha_3+2\alpha_4
   +\alpha_5+\alpha_6.
\end{equation*}
For a dominant integral weight
\[
  \lambda=\sum_{i=1}^{6}m_i\omega_i,
  \qquad m_i\in\mathbb Z_{\ge0},
\]
set
\begin{equation}
  \ell(\lambda)
  :=\langle\lambda,\vartheta^\vee\rangle
  =m_1+2m_2+2m_3+2m_4+m_5+m_6.
  \label{e7:eq:E7-appendix-D6-level-function}
\end{equation}
Thus the finite set used in the body of the paper is
\begin{equation*}
  P_n^+(D_6)
  =\{\lambda\in P_+(D_6)\mid \ell(\lambda)\le n\}.
\end{equation*}

\subsubsection{Casimir formula}

In the orthogonal realization above,
\begin{equation*}
  \rho_{D_6}
  =5\varepsilon_1+4\varepsilon_2+3\varepsilon_3
   +2\varepsilon_4+\varepsilon_5.
\end{equation*}
The normalized quadratic Casimir acts on $L_{D_6}(\lambda)$ by
\begin{equation*}
  c_2(\lambda)
  =(\lambda\mid\lambda+2\rho_{D_6}).
\end{equation*}
Equivalently, if
$\mathbf m=(m_1,\ldots,m_6)^{\mathsf T}$ and
$\mathbf1=(1,\ldots,1)^{\mathsf T}$, then
\begin{equation*}
  c_2(\lambda)
  =\mathbf m^{\mathsf T}A_{D_6}^{-1}
   (\mathbf m+2\mathbf1).
\end{equation*}
For the two half-spin rays one obtains the useful closed formula
\begin{equation*}
  c_2(r\omega_5)=c_2(r\omega_6)
  =\frac32r^2+15r,
  \qquad r\in\mathbb Z_{\ge0}.
\end{equation*}
In particular,
\begin{equation*}
  c_2(\omega_6)=\frac{33}{2},
  \qquad
  c_2(2\omega_6)=36,
  \qquad
  c_2(3\omega_6)=\frac{117}{2}.
\end{equation*}

\subsubsection{Complete enumeration through level three}

Equation~\eqref{e7:eq:E7-appendix-D6-level-function} reduces the enumeration
to the nonnegative integer solutions of
\[
  m_1+2m_2+2m_3+2m_4+m_5+m_6\le3.
\]
There are $32$ solutions: one of exact level zero, three of exact level
one, nine of exact level two, and nineteen of exact level three.  The
complete list and the corresponding Casimir eigenvalues are as follows.

\medskip\noindent\emph{Exact levels zero and one.}
\begin{equation*}
\begin{array}{c|c|c}
  \ell(\lambda)&\lambda&c_2(\lambda)\\
  \hline
  0&0&0\\
  1&\omega_1&11\\
  1&\omega_5&\dfrac{33}{2}\\[1mm]
  1&\omega_6&\dfrac{33}{2}
\end{array}
\end{equation*}

\medskip\noindent\emph{Exact level two.}
\begin{equation*}
\begin{array}{c|c@{\qquad}c|c}
  \lambda&c_2(\lambda)&\lambda&c_2(\lambda)\\
  \hline
  \omega_2&20
  &2\omega_1&24\\
  \omega_3&27
  &\omega_1+\omega_5&\dfrac{57}{2}\\[1mm]
  \omega_1+\omega_6&\dfrac{57}{2}
  &\omega_4&32\\
  \omega_5+\omega_6&35
  &2\omega_5&36\\
  2\omega_6&36
  &&
\end{array}
\end{equation*}
Consequently, among all weights of level at most two, the maximum is
$36$, and the weights with Casimir eigenvalue strictly larger than $32$
are precisely
\begin{equation*}
  2\omega_5,
  \qquad
  2\omega_6,
  \qquad
  \omega_5+\omega_6.
\end{equation*}

\medskip\noindent\emph{Exact level three.}
\begin{equation*}
\begin{array}{c|c@{\qquad}c|c}
  \lambda&c_2(\lambda)&\lambda&c_2(\lambda)\\
  \hline
  \omega_1+\omega_2&33
  &\omega_2+\omega_5&\dfrac{77}{2}\\[1mm]
  \omega_2+\omega_6&\dfrac{77}{2}
  &3\omega_1&39\\
  \omega_1+\omega_3&40
  &2\omega_1+\omega_5&\dfrac{85}{2}\\[1mm]
  2\omega_1+\omega_6&\dfrac{85}{2}
  &\omega_1+\omega_4&45\\
  \omega_3+\omega_5&\dfrac{93}{2}
  &\omega_3+\omega_6&\dfrac{93}{2}\\[1mm]
  \omega_1+\omega_5+\omega_6&48
  &\omega_1+2\omega_5&49\\
  \omega_1+2\omega_6&49
  &\omega_4+\omega_5&\dfrac{105}{2}\\[1mm]
  \omega_4+\omega_6&\dfrac{105}{2}
  &2\omega_5+\omega_6&\dfrac{113}{2}\\[1mm]
  \omega_5+2\omega_6&\dfrac{113}{2}
  &3\omega_5&\dfrac{117}{2}\\[1mm]
  3\omega_6&\dfrac{117}{2}
  &&
\end{array}
\end{equation*}
It follows that the weights in $P_3^+(D_6)$ with
$c_2(\lambda)\ge56$ are exactly
\begin{equation*}
  3\omega_5,
  \qquad
  3\omega_6,
  \qquad
  2\omega_5+\omega_6,
  \qquad
  \omega_5+2\omega_6.
\end{equation*}
Their Casimir values are respectively $117/2,117/2,113/2,113/2$.
After these four weights are removed, the largest remaining value is
\begin{equation*}
  c_2(\omega_4+\omega_5)
  =c_2(\omega_4+\omega_6)
  =\frac{105}{2}<56.
\end{equation*}
This proves all finite $D_6$ estimates collected in
Lemma~\ref{e7:lem:E7-D6-Casimir-small-levels}.
\section{Root and Casimir data for \texorpdfstring{$E_8$}{E8}}

\subsection{Explicit \texorpdfstring{$E_{7}$}{E7} root and Casimir data}
\label{e8:app:E8-E7-Casimir-data}

This appendix records the finite root-theoretic calculations used in
Section~\ref{e8:sec:E8-E7-Casimir-bounds}.  All roots and fundamental weights
are numbered as in Bourbaki, Planche~VI \cite{Bourbaki2002}.  The purpose is
to make the Casimir comparison independently checkable without invoking a
computer calculation.

\subsubsection{Cartan data and the Casimir formula}

In the Bourbaki numbering, the Cartan matrix of $E_{7}$ is
\begin{equation*}
 A_{E_7}
 =\begin{pmatrix}
 2&0&-1&0&0&0&0\\
 0&2&0&-1&0&0&0\\
 -1&0&2&-1&0&0&0\\
 0&-1&-1&2&-1&0&0\\
 0&0&0&-1&2&-1&0\\
 0&0&0&0&-1&2&-1\\
 0&0&0&0&0&-1&2
 \end{pmatrix}.
\end{equation*}
Its inverse is
\begin{equation*}
 A_{E_7}^{-1}
 =\begin{pmatrix}
 2&2&3&4&3&2&1\\
 2&\frac72&4&6&\frac92&3&\frac32\\
 3&4&6&8&6&4&2\\
 4&6&8&12&9&6&3\\
 3&\frac92&6&9&\frac{15}{2}&5&\frac52\\
 2&3&4&6&5&4&2\\
 1&\frac32&2&3&\frac52&2&\frac32
 \end{pmatrix}.
\end{equation*}
Direct matrix multiplication gives
\begin{equation*}
 A_{E_7}A_{E_7}^{-1}=I_7.
\end{equation*}
Since $E_7$ is simply laced, the entries of
$A_{E_7}^{-1}$ are the inner products
$(\omega_i\mid\omega_j)$.  The row sums are
\begin{equation*}
 17,\quad \frac{49}{2},\quad 33,\quad 48,
 \quad \frac{75}{2},\quad 26,\quad \frac{27}{2}.
\end{equation*}
Consequently,
\begin{equation*}
 (\omega_7\mid\omega_7)=\frac32,
 \qquad
 (\omega_7\mid\rho_{E_7})=\frac{27}{2},
 \qquad
 c_2(\omega_7)=\frac{57}{2}.
\end{equation*}

The highest root is
\begin{equation*}
 \vartheta
 =2\alpha_1+2\alpha_2+3\alpha_3+4\alpha_4
  +3\alpha_5+2\alpha_6+\alpha_7,
\end{equation*}
so for
$\mu=\sum_{i=1}^{7}m_i\omega_i$ one has
\begin{equation*}
 \ell(\mu)=\langle\mu,\vartheta^\vee\rangle
 =2m_1+2m_2+3m_3+4m_4+3m_5+2m_6+m_7.
\end{equation*}
Writing $\mathbf m=(m_1,\ldots,m_7)^{\mathsf T}$ and
$\mathbf 1=(1,\ldots,1)^{\mathsf T}$, the quadratic Casimir eigenvalue is
\begin{equation}
 c_2(\mu)
 =(\mu\mid\mu+2\rho_{E_7})
 =\mathbf m^{\mathsf T}A_{E_7}^{-1}
  (\mathbf m+2\mathbf 1).
 \label{e8:eq:app-E8-E7-Casimir-formula}
\end{equation}

\subsubsection{A uniform deficit formula}

Let
\begin{equation*}
 \nu=\sum_{i=1}^{6}m_i\omega_i,
 \qquad
 q=q(\nu)=2m_1+2m_2+3m_3+4m_4+3m_5+2m_6,
\end{equation*}
and suppose $q\le n$.  Set
\begin{equation*}
 \mu_n(\nu)=\nu+(n-q)\omega_7,
 \qquad
 D_n(\nu)=c_2(n\omega_7)-c_2(\mu_n(\nu)).
\end{equation*}
Put
\begin{equation*}
 a(\nu)=(\nu\mid\omega_7).
\end{equation*}
Expanding \eqref{e8:eq:app-E8-E7-Casimir-formula} gives the affine-linear
formula
\begin{equation}
 D_n(\nu)
 =n\bigl(3q-2a(\nu)\bigr)
  -\frac32q^2+27q-c_2(\nu)+2q\,a(\nu).
 \label{e8:eq:app-E8-uniform-deficit-formula}
\end{equation}
Thus every entry in the tables below follows from three pieces of data:
$q(\nu)$, $a(\nu)$, and $c_2(\nu)$.

\subsubsection{The complete lists for \texorpdfstring{$q\le5$}{q<=5}}

There is no nonzero $\nu$ with $q(\nu)=1$.  The complete lists for
$q=2,3,4,5$ are as follows.  In every table the last column is obtained by
substitution into \eqref{e8:eq:app-E8-uniform-deficit-formula}.

\begin{center}
\renewcommand{\arraystretch}{1.18}
\begin{tabular}{c|c|c|c}
$\nu$ & $a(\nu)$ & $c_2(\nu)$ & $D_n(\nu)$\\
\hline
$\omega_6$ & $2$ & $56$ & $2n$\\
$\omega_2$ & $\frac32$ & $\frac{105}{2}$
  & $3n+\frac32$\\
$\omega_1$ & $1$ & $36$ & $4n+16$
\end{tabular}
\par\smallskip
\textit{Table 1. All nonzero weights with $q(\nu)=2$.}
\end{center}

\begin{center}
\renewcommand{\arraystretch}{1.18}
\begin{tabular}{c|c|c|c}
$\nu$ & $a(\nu)$ & $c_2(\nu)$ & $D_n(\nu)$\\
\hline
$\omega_5$ & $\frac52$ & $\frac{165}{2}$ & $4n$\\
$\omega_3$ & $2$ & $72$ & $5n+\frac{15}{2}$
\end{tabular}
\par\smallskip
\textit{Table 2. All nonzero weights with $q(\nu)=3$.}
\end{center}

\begin{center}
\renewcommand{\arraystretch}{1.18}
\begin{tabular}{c|c|c|c}
$\nu$ & $a(\nu)$ & $c_2(\nu)$ & $D_n(\nu)$\\
\hline
$2\omega_6$ & $4$ & $120$ & $4n-4$\\
$\omega_4$ & $3$ & $108$ & $6n$\\
$\omega_2+\omega_6$ & $\frac72$ & $\frac{229}{2}$
  & $5n-\frac52$\\
$2\omega_2$ & $3$ & $112$ & $6n-4$\\
$\omega_1+\omega_6$ & $3$ & $96$ & $6n+12$\\
$\omega_1+\omega_2$ & $\frac52$ & $\frac{185}{2}$
  & $7n+\frac{23}{2}$\\
$2\omega_1$ & $2$ & $76$ & $8n+24$
\end{tabular}
\par\smallskip
\textit{Table 3. All nonzero weights with $q(\nu)=4$.}
\end{center}

\begin{center}
\renewcommand{\arraystretch}{1.18}
\begin{tabular}{c|c|c|c}
$\nu$ & $a(\nu)$ & $c_2(\nu)$ & $D_n(\nu)$\\
\hline
$\omega_5+\omega_6$ & $\frac92$ & $\frac{297}{2}$
  & $6n-6$\\
$\omega_3+\omega_6$ & $4$ & $136$ & $7n+\frac32$\\
$\omega_2+\omega_5$ & $4$ & $144$ & $7n-\frac{13}{2}$\\
$\omega_2+\omega_3$ & $\frac72$ & $\frac{265}{2}$
  & $8n$\\
$\omega_1+\omega_5$ & $\frac72$ & $\frac{249}{2}$
  & $8n+8$\\
$\omega_1+\omega_3$ & $3$ & $114$ & $9n+\frac{27}{2}$
\end{tabular}
\par\smallskip
\textit{Table 4. All nonzero weights with $q(\nu)=5$.}
\end{center}

The minima in the four displayed tables are, respectively,
\begin{equation*}
 2n,\qquad 4n,\qquad 4(n-1),\qquad 6(n-1),
\end{equation*}
with unique minimizers
\begin{equation*}
 \omega_6,
 \qquad
 \omega_5,
 \qquad
 2\omega_6,
 \qquad
 \omega_5+\omega_6.
\end{equation*}
For the relevant ranges $n\ge q$, the smallest positive deficit among all
nonzero $\nu$ is therefore $2n$, uniquely attained by $\nu=\omega_6$.
This independently verifies Lemma~\ref{e8:lem:E8-E7-deficit-classification}
and the second-largest Casimir statement in
Theorem~\ref{e8:thm:E8-E7-Casimir-maxima}.

For completeness, the resulting extremal values are
\begin{equation*}
\begin{array}{c|c|c|c}
 n & \text{unique maximum weight} & C_{\max}(n) & C_{\mathrm{second}}(n)\\
\hline
 1&\omega_7&\frac{57}{2}&0\\
 2&2\omega_7&60&56\\
 3&3\omega_7&\frac{189}{2}&\frac{177}{2}\\
 4&4\omega_7&132&124\\
 5&5\omega_7&\frac{345}{2}&\frac{325}{2}
\end{array}
\end{equation*}
\section*{Statements and Declarations}
\noindent\textbf{Funding.} No funding was received for this work.

\smallskip
\noindent\textbf{Competing interests.} The author declares no competing interests.

\printbibliography[title={References}]
\end{document}